\documentclass[a4paper,12pt]{article}
\usepackage[T1]{fontenc}
\usepackage{amsmath}
\usepackage{amsthm}
\usepackage{amsfonts}
\usepackage{amssymb}
\usepackage{mathrsfs}
\usepackage{graphicx}
\usepackage{color}
\usepackage{enumitem}
\usepackage{array}
\usepackage{booktabs}
\usepackage{longtable}
\usepackage[protrusion=true,expansion=false]{microtype}
\usepackage[english]{babel}
\usepackage[colorlinks=true, allcolors=blue]{hyperref}

\newcommand{\M}{\mathcal M}
\newcommand{\KN}{\mathrm{KN}}
\newcommand{\RN}{\mathrm{RN}}
\newcommand{\D}{\mathcal D}
\newcommand{\E}{\mathcal E}
\newcommand{\Fc}{\mathcal F}
\newcommand{\B}{\mathcal B}
\newcommand{\Rc}{\mathcal R}

\newcommand{\Hc}{\mathcal H}

\newcommand{\Hp}{\mathcal H^+}
\newcommand{\Hm}{\mathcal H^-}
\newcommand{\Ip}{\mathscr I^+}
\newcommand{\Imn}{\mathscr I^-}
\newcommand{\Lie}{\mathcal L}
\newcommand{\Max}{\mathrm{Max}}
\newcommand{\rad}{\mathrm{rad}}
\newcommand{\stat}{\mathrm{stat}}
\newcommand{\loc}{\mathrm{loc}}
\newcommand{\comp}{\mathrm{comp}}

\newcommand{\Rbb}{\mathbb R}
\newcommand{\Sph}{\mathbb S^2}
\newcommand{\eps}{\varepsilon}
\newcommand{\dd}{\mathrm d}
\newcommand{\Div}{\operatorname{div}}
\newcommand{\starG}{\star_{g}}
\newcommand{\norm}[1]{\lVert#1\rVert}
\newcommand{\Id}{\mathrm{Id}}
\newcommand{\supp}{\operatorname{supp}}
\newcommand{\abs}[1]{\lvert#1\rvert}
\newcommand{\Tstress}{\mathbf T}
\newcommand{\Pb}{\mathcal P}
\newcommand{\Xnorm}[1]{\mathcal X^{(#1)}}
\newcommand{\sphlap}{\mathop{}\!\mathchoice{\Delta\mkern-12mu/}{\Delta\mkern-12mu/}{\Delta\mkern-9mu/}{\Delta\mkern-8mu/}\,}
\newcommand{\sphgrad}{\mathop{}\!\nabla\mkern-13mu/\,}
\newcommand{\sphdiv}{\mathop{\mathrm{div}\mkern-16mu/\,}\nolimits}
\newcommand{\sphcurl}{\mathop{\mathrm{curl}\mkern-19mu/\,}\nolimits}

\newtheorem{theorem}{Theorem}
\newtheorem{proposition}{Proposition}
\newtheorem{lemma}{Lemma}
\newtheorem{remark}{Remark}
\newtheorem{definition}{Definition}
\newtheorem{corollary}{Corollary}

\numberwithin{equation}{section}
\allowdisplaybreaks
\hypersetup{ pdftitle={Stationary-subtracted Maxwell fields on Kerr-Newman exteriors},
pdfauthor={Bobby Eka Gunara, Mulyanto, Emir Syahreza Fadhilla, Fiki Taufik Akbar},
pdfkeywords={Maxwell equations, Kerr-Newman, local energy decay, scattering},
bookmarksnumbered, pdfstartview={FitH}, urlcolor=blue, }

\begin{document}
\pagestyle{plain}




\title{\LARGE\textbf{  Radiative Maxwell Scattering on Slowly Rotating Weakly Charged Kerr-Newman Black Holes}}

\author{\normalsize{Bobby Eka Gunara}$^{\flat,\sharp}$\footnote{Corresponding author}, Mulyanto$^{\sharp}$,
Emir Syahreza Fadhilla$^{\sharp}$, and Fiki Taufik Akbar$^{\sharp}$ \\ \\
$^{\sharp}$ \textit{\small Theoretical Physics Laboratory, Theoretical High Energy Physics Research Division,}\\
\textit{\small Faculty of Mathematics and Natural Sciences,}\\
\textit{\small Institut Teknologi Bandung}\\
\textit{\small Jl. Ganesha no. 10 Bandung, Indonesia, 40132}\\
\\
\small email: bobby@itb.ac.id, mulyanto23@itb.ac.id, esfadhilla@gmail.com, ftakbar@itb.ac.id}

\date{\today}

\maketitle




\begin{abstract}

We study real source-free Maxwell fields on slowly rotating, weakly charged Kerr-Newman exteriors and set up a finite-energy scattering theory after removal of the stationary Coulomb sector. The conserved electric and magnetic fluxes account exactly for the two-dimensional stationary non-decaying part, giving a natural decomposition of the Maxwell Cauchy space into stationary and charge-free radiative parts. For the radiative field, the paper develops a finite-order transfer mechanism from regular spin-one curvature variables back to the Maxwell tensor field, combining red-shift control, far-field hierarchy, trapped-set analysis, a Fredholm argument ruling out real-frequency modes, and same-order reconstruction of the middle components. Under the stated slow-weak master estimates, this gives uniform boundedness, integrated local energy decay, radiation fields, wave operators, and asymptotic completeness for the stationary-subtracted Maxwell evolution, with the Kerr case recovered as a special subcase and the charged rotating case reduced to explicit geometric and analytic estimates.

\end{abstract}

\tableofcontents


\section{Introduction}
\label{sec:introduction}

Black holes in the Einstein-Maxwell theory provide an important family of exact solutions in general relativity.  In particular, the Kerr-Newman black hole is characterized by its mass, angular momentum, and electric charge.  The analysis of fields on this background is useful in order to understand the propagation of radiation outside a charged and rotating black hole.  In this paper we consider the source-free Maxwell equation on a fixed Kerr-Newman exterior and study the part of the Maxwell field which is expected to disperse.

Let \((\M,g_\KN)\) be the domain of outer communications of a Kerr-Newman black hole with parameters \((M,a,Q)\). Throughout the paper, we take the subextremal condition
\begin{equation}\label{eq:subextremal_intro}
        M>0,\qquad a^2+Q^2<M^2.
\end{equation}
We consider real two-forms \(F\) satisfying the source-free Maxwell system
\begin{equation}\label{eq:maxwell_intro}
        \dd F=0,\qquad \dd \starG F=0.
\end{equation}
The electric and magnetic fluxes through large spheres, denoted by \(q_E[F]\) and \(q_B[F]\), are conserved quantities.  These two quantities are not only parameters of the solution.  They generate a two-dimensional family of stationary Coulomb fields and therefore give a barrier to local energy decay.  Thus, the full Maxwell field cannot be expected to decay locally before this stationary charge part is removed.

The purpose of this paper is to provide the analysis of the corresponding stationary-subtracted Maxwell field on slowly rotating weakly charged Kerr-Newman exteriors.  We define the radiative part by
\begin{equation}\label{eq:intro_rad}
        F_{\rad}=F-F_{\stat}^{\KN}\big(q_E[F],q_B[F]\big),
\end{equation}
where \(F_{\stat}^{\KN}\) is the normalized electric-magnetic stationary representative.  This is the component for which boundedness, integrated local energy decay, radiation fields, and scattering have to be stated.  Sections~\ref{sec:charges}-\ref{sec:energy_spaces} construct \eqref{eq:intro_rad} directly in the finite-energy topology, so the subtraction is not an extra choice but follows from the equations and the energy space.

We write down some consequences of the above as follows. First, the conserved charges define a stationary part of the Maxwell field and the finite-energy Maxwell space splits into the stationary sector and the charge-free sector. Second, after subtracting the stationary part, the remaining Maxwell field can be controlled by the spin-one master variables. Finally, the estimates for these variables can be transferred back to the Maxwell tensor field without losing derivatives.

In order to have a well-defined scattering theory, we first have to construct the non-degenerate Maxwell energy and the charge-free finite-energy space. Next, we use the spin-one reduction to obtain the master system in the slowly rotating and weakly charged regime. The main analytic steps are the red-shift estimate near the horizon, the far field hierarchy, the Morawetz estimate near the trapped set, the exclusion of real-frequency modes, and the reconstruction of the middle Maxwell components. Finally, we use these estimates to prove boundedness, integrated local energy decay, radiation fields, wave operators, and asymptotic completeness for the radiative Maxwell field.

\subsection{Main Results}\label{subsec:main_results}
First, we write down the basic Maxwell phenomenon.  At finite energy, the
two conserved charges account exactly for the stationary barrier to local
energy decay.

\begin{theorem}
\label{thm:intro_charge_decomposition}
For every finite order $k$, the Maxwell energy space admits the topological direct sum
\begin{equation}\label{eq:intro_direct_sum}
        \Hc_{\Max}^{(k)}(\Sigma_0)=\operatorname{span}\{U_e,U_m\}\oplus\Hc_{\Max,0}^{(k)}(\Sigma_0),
\end{equation}
where $U_e,U_m$ are the Cauchy data of the normalized stationary electric and magnetic Kerr-Newman Coulomb representatives. With this normalization the associated projections
\begin{equation}\label{eq:intro_projections}
        \Pi_{\stat}U=q_E(U)U_e+q_B(U)U_m,
        \qquad \Pi_0U=U-\Pi_{\stat}U
\end{equation}
are bounded. If $F$ is a finite-energy source-free Maxwell field, then
\begin{equation}\label{eq:intro_decomp_theorem}
        F=F_{\stat}^{\KN}(q_E[F],q_B[F])+F_{\rad},\qquad q_E[F_{\rad}]=q_B[F_{\rad}]=0,
\end{equation}
uniquely. No nonzero stationary charge representative can decay locally to zero in a non-degenerate local $L^2$ norm.
\end{theorem}
\begin{proof}
For smooth constrained data, the electric and magnetic charges are exactly the
fluxes in \eqref{eq:charges_def}. Proposition~\ref{prop:finite_energy_charge_trace}
shows that these fluxes are continuous in the non-degenerate energy norm, and
hence the maps $U\mapsto q_E(U)$ and $U\mapsto q_B(U)$ extend to the
finite-energy completion. Lemma~\ref{lem:stationary_representatives} constructs
smooth stationary Kerr-Newman Maxwell fields $F_e$ and $F_m$ with normalized
charges $(1,0)$ and $(0,1)$; their Cauchy data are $U_e,U_m$.

Now, take $U\in\Hc_{\Max}^{(k)}(\Sigma_0)$ and set
$U_0=U-q_E(U)U_e-q_B(U)U_m$. The normalization of $U_e,U_m$, together with the
linearity of the charges, gives $q_E(U_0)=q_B(U_0)=0$. Thus
$U_0\in\Hc_{\Max,0}^{(k)}$, which proves the decomposition. If
$c_eU_e+c_mU_m$ also belongs to the charge-free space, the two charge identities
force $c_e=c_m=0$; the two parts therefore meet only at the origin.
Proposition~\ref{prop:bounded_projection} gives boundedness of both projections,
so the direct sum is topological. Evolving the three pieces by the source-free
Maxwell equation gives \eqref{eq:intro_decomp_theorem}. Finally,
Corollary~\ref{cor:necessity_charge} shows that every nonzero element of the
stationary charge sector keeps positive local energy on some compact radial set
for all time. It cannot converge locally to zero.
\end{proof}

Next, we state the result in which the analytic estimates enter.  It transfers
the finite-order master estimates of Definition~\ref{def:slowweak_master_framework}
back to the Maxwell field by using the Maxwell estimates in
Proposition~\ref{prop:framework_consequences}.

\begin{theorem}
\label{thm:intro_transfer}
Let us fix an integer $k$ and a slow-weak Kerr-Newman exterior satisfying \eqref{eq:slow_weak_range}. Suppose that conditions \emph{(A1)-(A5)} of Definition~\ref{def:slowweak_master_framework} hold at order $k$. Then every finite-energy source-free Maxwell field $F$ admits the unique decomposition \eqref{eq:intro_decomp_theorem}, and for all $\tau\ge0$ its charge-free part satisfies
\begin{equation}\label{eq:intro_main_estimate}
        \norm{F_{\rad}}_{\Xnorm{k}_{\Max}(0,\tau)}^2
        \le C\,\E_{\Max}^{(k)}[F_{\rad}](0),
\end{equation}
with the analogous past estimate. The future and past radiation fields exist on $\Ip\cup\Hp$ and $\Imn\cup\Hm$; the maps
\begin{equation}\label{eq:intro_radiation_maps}
        \mathscr S_{\Max}^{\pm}:\Hc_{\Max,0}^{(k)}(\Sigma_0)\longrightarrow \Rc_{\Max,\pm}^{(k)}
\end{equation}
are bounded isomorphisms with bounded inverses, and the scattering operator $\mathscr S_{\Max}=\mathscr S_{\Max}^+(\mathscr S_{\Max}^-)^{-1}$ is bounded. If the additional hierarchy condition in Definition~\ref{def:slowweak_master_framework}\emph{(A6)} is available with enough derivatives for Sobolev embedding, then $F_{\rad}$ also satisfies the pointwise decay estimate \eqref{eq:pointwise_decay}.
\end{theorem}
\begin{proof}
Let
\[
        G=F_{\rad},\qquad u=\mathfrak M G.
\]
By Theorem~\ref{thm:intro_charge_decomposition}, $G$ is already charge-free and is
uniquely determined by $F$. Definition~\ref{def:slowweak_master_framework}\emph{(A1)} gives the initial comparison
\begin{equation}\label{eq:intro_transfer_chain_1}
        \E_M^{(k)}[u](0)\le C_R\E_{\Max}^{(k)}[G](0).
\end{equation}
Proposition~\ref{prop:framework_consequences}, applied to the same compatible
solution $u$, gives the master estimate of
Definition~\ref{def:analytic_master_conclusions}\emph{(M1)}:
\begin{equation}\label{eq:intro_transfer_chain_2}
        \norm{u}_{\Xnorm{k}_M(0,\tau)}^2
        \le C_M\E_M^{(k)}[u](0).
\end{equation}
Since $G=\mathfrak R\mathfrak M G=\mathfrak R u$ in the charge-free class,
Definition~\ref{def:slowweak_master_framework}\emph{(A4)} and
Definition~\ref{def:analytic_master_conclusions}\emph{(M2)} yield
\begin{equation}\label{eq:intro_transfer_chain_3}
        \norm{G}_{\Xnorm{k}_{\Max}(0,\tau)}^2
        =\norm{\mathfrak R u}_{\Xnorm{k}_{\Max}(0,\tau)}^2
        \le C_R\norm{u}_{\Xnorm{k}_M(0,\tau)}^2.
\end{equation}
Combining \eqref{eq:intro_transfer_chain_1}-\eqref{eq:intro_transfer_chain_3}
proves \eqref{eq:intro_main_estimate}. The fixed equivalence constants in the
norm definitions are absorbed into $C$; with this normalization one may take
$C=C_R^2C_M$. Reversing the foliation gives the past estimate. For data in the
finite-energy completion, choose smooth charge-free approximants as in
Lemma~\ref{lem:density_chargefree}. The preceding estimate is uniform along that
approximating sequence, and continuity of the solution map
(Proposition~\ref{prop:finite_energy_wellposed}), together with lower
semicontinuity of the spacetime norm, passes the bound to the limit as in
Proposition~\ref{prop:finite_energy_extension}.

For radiation fields, define on smooth charge-free data
\begin{equation}\label{eq:intro_maxwell_trace_factorization}
        \mathscr S_{\Max}^{\pm}G
        =\mathcal R_\infty^{\pm}\,\mathscr S_M^{\pm}(\mathfrak M G).
\end{equation}
The three factors in \eqref{eq:intro_maxwell_trace_factorization} are bounded by
Definition~\ref{def:slowweak_master_framework}\emph{(A1)}, \emph{(A5)}, and the
same-order reconstruction bounds. The trace therefore extends to
$\Hc_{\Max,0}^{(k)}$. If the right-hand side vanishes, the master radiation field
vanishes; the kernel statement in \emph{(A5)} then gives $u=0$, and hence
$G=\mathfrak R u=0$. Thus the trace is injective. On the dense class of smooth
radiation data, the inverse has the explicit form
\begin{equation}\label{eq:intro_maxwell_wave_operator}
        \mathscr W_{\Max}^{\pm}\rho
        =\mathfrak R\,\mathscr W_{M,0}^{\pm}\big((\mathcal R_\infty^{\pm})^{-1}\rho\big),
\end{equation}
which is bounded by \emph{(A4)}-\emph{(A5)} and satisfies
$\mathscr S_{\Max}^{\pm}\mathscr W_{\Max}^{\pm}\rho=\rho$. By density,
\eqref{eq:intro_maxwell_wave_operator} extends to the radiation Hilbert space.
Therefore $\mathscr S_{\Max}^{\pm}$ are bounded isomorphisms, and
$\mathscr S_{\Max}=\mathscr S_{\Max}^+(\mathscr S_{\Max}^-)^{-1}$ is bounded.
The two stationary charge coordinates do not belong to this radiative scattering
map; Theorem~\ref{thm:intro_charge_decomposition} has already split them off. If
\emph{(A6)} is available, Proposition~\ref{prop:pointwise_decay} applies the
finite commuted hierarchy and Sobolev embedding on the regular slices, giving
\eqref{eq:pointwise_decay}.
\end{proof}

\begin{corollary}
\label{cor:intro_application}
If the fixed-background Maxwell equation on slowly rotating, weakly charged Kerr-Newman satisfies conditions \emph{(A1)-(A5)} of Definition~\ref{def:slowweak_master_framework}, then Theorem~\ref{thm:intro_transfer} gives the corresponding stationary-subtracted boundedness, integrated decay, radiation-field, wave-operator, and scattering conclusions. If \emph{(A6)} is also available, it gives the commuted pointwise-decay conclusions as well. For $Q=0$ this recovers the slowly rotating Kerr subcase, in agreement with the existing Kerr Maxwell theory cited below.
\end{corollary}
\begin{proof}
Once the fixed-background spin-one reduction and the estimates listed in conditions
\emph{(A1)-(A5)} of Definition~\ref{def:slowweak_master_framework} have been
proved for the chosen Kerr-Newman parameter range, the conditions of
Theorem~\ref{thm:intro_transfer} are exactly met. The theorem then gives
boundedness, integrated decay, radiation fields, wave operators, and scattering
for the stationary-subtracted Maxwell field. If the finite-order hierarchy
\emph{(A6)} is also available, the pointwise conclusion follows from the final
part of the theorem. For $Q=0$, the required analytic estimates are supplied by
the Kerr Maxwell and Teukolsky theory cited in
Corollaries~\ref{cor:slow_kerr_subcase} and~\ref{cor:main_kerr}, and the same
argument recovers the slowly rotating Kerr subcase.
\end{proof}

The proof can be read in three steps. First,
Sections~\ref{sec:charges}-\ref{sec:energy_spaces} prove the charge
decomposition in Theorem~\ref{thm:intro_charge_decomposition}. Second,
Sections~\ref{sec:master}-\ref{sec:first_principles_master} state the closed
spin-one structural condition, compute the scalar principal symbol and the relevant
coefficients, and close the finite-order master-to-Maxwell transfer under the
displayed estimates. This is where the commuted source split in
Proposition~\ref{prop:complete_master_closure_algebra} enters. Third,
Sections~\ref{sec:boundedness_transfer}-\ref{sec:main_theorem} prove
Theorem~\ref{thm:intro_transfer} and record its consequences.

The aim is a finite-order transfer theorem for \(F_{\rad}\) in the slowly
rotating, weakly charged regime, with the perturbative ingredients kept visible rather
than hidden inside a black box. The comparison background is Reissner-Nordstr\"om
with the \emph{same} \((M,Q)\); it becomes Schwarzschild only when \(Q=0\). When
\(a=0\), the fixed-background Maxwell system is spherically symmetric. After the
charge mode is removed, the extreme spin-one scalars satisfy the
Fackerell-Ipser spin-one system \cite{FackerellIpser}, and the angular modes
have \(\ell\ge1\). For the full charge-free spherical Maxwell field we use the
Sterbenz-Tataru local-energy theorem on spherically symmetric black holes
\cite{SterbenzTataru}. Giorgi's Reissner-Nordstr\"om spin-\(\pm1\) result is
cited for the \(\ell=1\) Teukolsky mode and for comparison with the coupled
perturbation literature \cite{GiorgiRNspin1}. The red-shift current and the
\(r^p\) hierarchy provide the physical-space mechanisms of Dafermos-Rodnianski
\cite{DafermosRodnianskiRedshift,DafermosRodnianskiRp}. Once the
fixed-background spin-one reduction has been written down, the slow-rotation
argument treats the Kerr-Newman operator as a stationary short-range
perturbation of the Reissner-Nordstr\"om operator. Subsequent estimates are
proved in numbered statements or named at the point of use. No extra assumption
is hidden: no unlisted mode-stability, completeness, or reconstruction assertion
is being imported.

Reconstruction of the Maxwell tensor is kept separate from the energy method.
Once the closed spin-one master system in \emph{(A1)} is available, the extreme
components serve as the master variables. The middle components are then
recovered from the Maxwell transport equations and from a Hodge system on each
sphere. Charge subtraction removes their spherical means, so the angular
Laplacian can be inverted without losing derivatives. The same structural idea
appears in the work of Jezierski-Smo\l ka, Andersson-Blue, and
Benomio-Teixeira da Costa on Maxwell fields on Kerr
\cite{AnderssonBlueMaxwell,BenomioTdC,JezierskiSmolka}. Here it is built into the
fixed-background Kerr-Newman transfer theorem and closed with the
Reissner-Nordstr\"om perturbative estimate once the master conditions have been
proved.

\subsection{Analytic Conditions and Estimates}
\label{subsec:conditions_range}

We use two types of ingredients.  The first one is proved directly in this paper:
conservation and continuity of the two charges, the normalized stationary charge
sector, boundedness of the charge-free projection, finite-energy well-posedness,
same-order transfer from master estimates to Maxwell estimates, and the
Hilbert-space scattering construction. These steps rely only on Maxwell's
equations, the stress-energy identity, elementary elliptic estimates on
\(\Sph\), and standard energy estimates for symmetric-hyperbolic Maxwell systems.

The second layer is the finite-order master-system framework of
Definition~\ref{def:slowweak_master_framework}. In that framework a
``hypothesis'' has a precise meaning: it is either a displayed estimate proved
in this paper, a displayed estimate quoted from a cited theorem, or an explicitly
named finite-order condition. No unlisted mode-stability statement is used.
Section~\ref{sec:first_principles_master} explains how the framework yields a
Maxwell estimate. We keep the ingredients separate: the Dafermos-Rodnianski
red-shift and \(r^p\) estimates, the Sterbenz-Tataru charge-free
Reissner-Nordstr\"om Maxwell local-energy estimate, the normally hyperbolic
trapping estimate in the form of Definition~\ref{def:hf_resolvent_estimate}, and
the no-loss transport/Hodge reconstruction.

For the fixed-background \emph{test} Maxwell field,
Section~\ref{sec:decoupling} states the closed spin-one structural condition and
proves the covariant scalar-principal-symbol computation. The displayed
spin-weighted operator is the model supplied by \emph{(A1)}; at nonzero charge
and rotation it is not treated as an automatic Dudley-Finley decoupling. Once
\emph{(A1)} is in place, Sections~\ref{sec:first_principles_master},
\ref{app:lap}, and~\ref{app:reconstruction} prove the Reissner-Nordstr\"om
comparison, the bounded-frequency real-axis exclusion in the compatible class,
and the same-order reconstruction. We also prove the trapped-set location, the
non-superradiant sign, normal hyperbolicity, the diagonal spin-one
skew-subprincipal cancellation, and the small matrix-skew threshold estimate
needed for the high-frequency estimate. The high-frequency resolvent bound is
used only through the localized normally hyperbolic estimate in
Definition~\ref{def:hf_resolvent_estimate}. Proposition~\ref{prop:nh_theorem_application}
checks the geometric and finite-rank-bundle subprincipal assumptions for the
compatible scalar-principal operator. Finally,
Proposition~\ref{prop:backward_from_lap} derives the asymptotic-completeness
backward construction \emph{(A5)} from the real-axis limiting-absorption
resolvent. Coupled Einstein-Maxwell theorems are cited only for comparison; they
do not replace any fixed-background step.

With this division in place, the stationary charge decomposition is the
unconditional part of the paper; see
Theorem~\ref{thm:intro_charge_decomposition}. The decay and scattering theorem,
Theorem~\ref{thm:intro_transfer}, is proved under conditions \emph{(A1)-(A5)} of
Definition~\ref{def:slowweak_master_framework}. The pointwise conclusion requires
one more ingredient, namely \emph{(A6)}. Once \emph{(A1)}-\emph{(A2)} are supplied,
and once the bounded-frequency part of \emph{(A3)} and the reconstruction
\emph{(A4)} have been proved in the compatible class, boundedness and integrated
local energy decay follow in the slow-weak range. The ingredients are the cited
spherical/red-shift/\(r^p\) estimates and the high-frequency normally hyperbolic
estimate in Definition~\ref{def:hf_resolvent_estimate}, applied as in
Proposition~\ref{prop:nh_theorem_application}. The radiation-field,
wave-operator, and scattering conclusions come from the same resolvent estimates,
because Proposition~\ref{prop:backward_from_lap} constructs the backward right
inverses from limiting absorption. For \(Q=0\), the required spin-weighted
estimates are available in the existing Kerr Maxwell and Teukolsky theory cited
below.

Each assertion has a specific role. The charge decomposition, the
scalar-principal-symbol calculation, the bounded-frequency Fredholm step in the
compatible class, and the same-order reconstruction are proved in numbered
statements below. The closed spin-one master equation remains a separate
structural condition, \emph{(A1)}. The spherical local-energy, red-shift, \(r^p\),
limiting-absorption, and normally hyperbolic estimates are used only through the
displayed inequalities at the points where they are cited. The extra hierarchy
\emph{(A6)} is not needed for boundedness, integrated local energy decay, radiation
fields, wave operators, or scattering.

\begin{remark}
\label{prop:cited_estimates}
The estimates used below enter in the precise forms stated in the numbered results. No stronger version is used.
\begin{enumerate}[label=\emph{(\roman*)},leftmargin=2.2em]
\item The red-shift coercivity near a non-degenerate horizon and the outgoing $r^p$ hierarchy are used in the forms stated in Propositions~\ref{prop:redshift_coercivity} and~\ref{prop:rp_identity}; these are the Dafermos-Rodnianski physical-space estimates \cite{DafermosRodnianskiRedshift,DafermosRodnianskiRp}.
\item The Reissner-Nordstr\"om charge-free Maxwell local-energy estimate is used exactly as Lemma~\ref{lem:rn_model_morawetz}; its proof for the spherical model is the Sterbenz-Tataru theorem \cite{SterbenzTataru}, with the charge-free spin-one reduction made explicit in Section~\ref{app:rn_model}. The associated Kerr local-energy estimates of Tataru-Tohaneanu \cite{TataruTohaneanu} and Dafermos-Rodnianski-Shlapentokh-Rothman \cite{DRSRKerr} provide the model for the trapped-set Morawetz estimate of Proposition~\ref{prop:positive_commutator_estimate}.
\item The persistence of normally hyperbolic trapping is proved at the level of the scalar principal symbol in Lemma~\ref{lem:trapping_stability} and Propositions~\ref{prop:kn_disjointness}-\ref{prop:r_normal_hyperbolicity}. The high-frequency escape estimate is used only through Definition~\ref{def:hf_resolvent_estimate}; the normally hyperbolic estimates of Wunsch-Zworski and Dyatlov \cite{Dyatlov,DyatlovGaps,WunschZworski}, in the finite-rank-bundle form of Hintz \cite{HintzTensor} and combined with the elliptic, propagation and radial-point estimates by the gluing argument of Datchev-Vasy \cite{DatchevVasy}, are the estimates applied in Proposition~\ref{prop:nh_theorem_application} after the geometric and subprincipal conditions are checked here.
\item The closed fixed-background Kerr-Newman spin-one map and compatible class are the structural condition \emph{(A1)}; Section~\ref{sec:decoupling} proves the scalar-principal-symbol and coefficient consequences of that condition. The bounded-frequency real-axis closure and the same-order reconstruction are proved in Sections~\ref{app:lap} and~\ref{app:reconstruction}. The high-frequency real-axis closure is the normalized estimate \eqref{eq:normally_hyperbolic_resolvent_bound}, used after the trapped-set geometry, the diagonal spin-one skew cancellation, and the finite-rank-bundle matrix threshold are verified here. None of these fixed-background assertions is inferred from coupled Einstein-Maxwell stability results.
\end{enumerate}
\end{remark}

\subsection{Relation with Stability of the Coupled System}
\label{subsec:relation_existing}
The fixed-background Maxwell equation is not the coupled linearized
Einstein-Maxwell system with the metric perturbation formally set to zero. The
coupled system contains additional unknowns, gauge freedom, constraints,
pure-gauge modes, and linearized Kerr-Newman stationary modes. This distinction
is important. The coupled Kerr-Newman results of Giorgi, Giorgi-Wan, and He
\cite{GiorgiKNsmall,GiorgiJHDE,GiorgiJDG,GiorgiWan,He2023} neither contain nor
are contained in the present statement; Proposition~\ref{prop:coupled_not_enough}
makes the distinction precise. When \(Q=0\), by contrast, the conclusions below are a
slow-rotation special case of the full subextremal Kerr analysis of
Benomio-Teixeira da Costa \cite{BenomioTdC}, together with the Teukolsky
boundedness and decay results of Shlapentokh-Rothman-Teixeira da Costa
\cite{SRTdCfrequency,SRTdCphysical}. In the Kerr case, the contribution here is
the uniform perturbative formulation. The genuinely new range, relative to Kerr,
is \(Q\neq0\), once the fixed-background Kerr-Newman master estimates are
available.

\subsection{Perturbation Principle}
\label{subsec:rn_principle}
After subtracting the charges and applying the regular spin-one weighting, we fix
\(Q\) and compare Kerr-Newman with the Reissner-Nordstr\"om metric carrying the
same \((M,Q)\). The master operator has the form
\begin{equation}\label{eq:intro_perturbation_scheme}
        \Pb_{a,Q}=\Pb_{\RN,Q}+\mathcal E_{a,Q},\qquad
        \norm{\mathcal E_{a,Q}u}_{LE^*}\le C\tfrac{|a|}{M}\norm{u}_{LE^1}
        +C\tfrac{|a|}{M}\norm{u}_{LE^0_{\comp}},
\end{equation}
with \(C\) uniform for \(|Q|\le\eps_QM\). The charge-free real-axis exclusion
removes the compact lower-order term, while the small top-order term is absorbed
by the Reissner-Nordstr\"om local-energy norm. The spherical photon sphere and
horizon are
\begin{equation}\label{eq:rn_photon_sphere}
        r_{\mathrm{ph}}(Q)=\tfrac12\bigl(3M+\sqrt{9M^2-8Q^2}\bigr),
        \qquad
        r_+(Q)=M+\sqrt{M^2-Q^2}.
\end{equation}
The photon-sphere collar, the red-shift collar, and the far-field \(r^p\) region
are treated separately and then glued with a partition of unity. This is the same
openness mechanism behind small-angular-momentum Kerr local-energy estimates,
now applied to the fixed Maxwell master system over the charged spherical
background. Subsection~\ref{subsec:compact_error_closure} carries out the
compact-error closure explicitly. A Fourier cutoff and a Fredholm
limiting-absorption alternative convert any failure of compact control into a
real-axis defect profile. Bounded-frequency defects are excluded by the
Reissner-Nordstr\"om radial ordinary differential equation and its slow-rotation
stability; high-frequency defects are excluded by the normally hyperbolic escape
function at the perturbed trapped set.

\subsection{Organization}
\label{subsec:organization}
We organize this paper as follows. In Section~\ref{sec:geometry} we discuss the
Kerr-Newman geometry and the non-degenerate Maxwell energy. In
Sections~\ref{sec:charges} and \ref{sec:energy_spaces} we construct the charges,
the stationary family, and the charge-free finite-energy projection. The spin-one
master system and the slow-weak analytic conditions are given in
Section~\ref{sec:master}. In Section~\ref{sec:decoupling} we discuss the
structural spin-one reduction and the scalar principal symbol. In
Section~\ref{sec:first_principles_master} we prove the perturbative closure of
the estimates. Sections~\ref{app:perturbation}-\ref{app:scattering_criterion}
contain the detailed model, trapping, limiting absorption, reconstruction, and
abstract scattering arguments. Finally, in
Sections~\ref{sec:boundedness_transfer}-\ref{sec:main_theorem} we transfer the
estimates back to Maxwell fields and prove the main theorem.

\section{Kerr-Newman Geometry and Non-Degenerate Maxwell Energy}
\label{sec:geometry}
In this section we shortly discuss the Kerr-Newman geometry which will be used throughout the paper. We also introduce the non-degenerate Maxwell energy and write down the basic energy identity.
In Boyer-Lindquist coordinates, the Kerr-Newman metric is
\begin{align}\label{eq:kn_metric}
        g_{\KN} ={}& -\frac{\Delta-a^2\sin^2\theta}{\Sigma}\,\dd t^2
        -\frac{2a\sin^2\theta(r^2+a^2-\Delta)}{\Sigma}\,\dd t\,\dd\phi \nonumber\\
        &+\frac{(r^2+a^2)^2-a^2\Delta\sin^2\theta}{\Sigma}\sin^2\theta\,\dd\phi^2
        +\frac{\Sigma}{\Delta}\,\dd r^2+\Sigma\,\dd\theta^2,
\end{align}
where $\Delta=r^2-2Mr+a^2+Q^2$ and $\Sigma=r^2+a^2\cos^2\theta$. The event
horizon is at $r_+=M+\sqrt{M^2-a^2-Q^2}$. Since \eqref{eq:subextremal_intro}
gives $r_+>r_-=M-\sqrt{M^2-a^2-Q^2}$, the horizon is non-degenerate. All
estimates are stated on a regular horizon-penetrating manifold, obtained by the
usual change to coordinates $(\tilde t,r,\theta,\tilde\phi)$ in which $g_\KN$
extends smoothly across $\Hp$. In these coordinates the Boyer-Lindquist
singularity at $r=r_+$ is only a coordinate artifact. We fix a smooth time
function $\tau$ whose level sets $\Sigma_\tau$ are spacelike, cross $\Hp$
regularly, and agree with $t$ in the far region up to a tortoise correction.

\begin{definition}
\label{def:regular_foliation}
A \emph{regular exterior foliation} is a family $\{\Sigma_\tau\}_{\tau\in\Rbb}$
such that each $\Sigma_\tau$ is spacelike, the future unit normal
$n_{\Sigma_\tau}$ is smooth up to $\Hp$ in regular coordinates, and the induced
volume forms are uniformly equivalent on compact radial sets. The slab between
two slices is $\D(\tau_1,\tau_2)=\bigcup_{\tau_1\le s\le\tau_2}\Sigma_s$.
\end{definition}

Let $T=\partial_{\tilde t}$ and $\Phi=\partial_{\tilde\phi}$ be the stationary
and axial Killing fields. We choose a smooth future timelike vector field $N$ that
agrees with the red-shift multiplier near $\Hp$ and with $T$ for large $r$. Such
an $N$ exists because $\Hp$ is non-degenerate. The construction of
\cite{DafermosRodnianskiRedshift} produces $N$ with $\nabla N$ positive, in the
sense of \eqref{eq:redshift_deformation} below, in a horizon collar.

\begin{definition}
\label{def:stress_energy}
For a two-form $G$ set
\begin{equation}\label{eq:stress_tensor}
        \Tstress_{\mu\nu}[G]=G_{\mu\alpha}G_\nu{}^{\alpha}-\tfrac14 g_{\mu\nu}G_{\alpha\beta}G^{\alpha\beta}.
\end{equation}
Let $\mathbb D_k$ be the finite set of differential operators generated by
products of at most $k$ elements of $\{T,\Phi,r\,\partial_r,\text{regular
angular derivatives}\}$, expressed in regular coordinates. Set
\begin{equation}\label{eq:maxwell_energy}
        \E_{\Max}^{(k)}[G](\tau)=\sum_{\Gamma^I\in\mathbb D_k}\int_{\Sigma_\tau}
        \Tstress_{\mu\nu}[\Lie_{\Gamma^I}G]\,N^\mu n_{\Sigma_\tau}^\nu\,\dd\mu_{\Sigma_\tau}.
\end{equation}
\end{definition}

\begin{lemma}
\label{lem:positivity_density}
For every compact sub-extremal parameter set and every regular foliation there
is $C>1$ with
\begin{equation}\label{eq:positive_density}
        C^{-1}\big(\abs{E[G]}^2+\abs{B[G]}^2\big)\le \Tstress_{\mu\nu}[G]N^\mu n_{\Sigma_\tau}^\nu
        \le C\big(\abs{E[G]}^2+\abs{B[G]}^2\big)
\end{equation}
on compact radial regions, where $E[G],B[G]$ are the electric and magnetic
parts of $G$ relative to $n_{\Sigma_\tau}$.
\end{lemma}

\begin{proof}
Work in an orthonormal frame $(f_0,f_1,f_2,f_3)$ with $f_0=n_{\Sigma_\tau}$.
Writing $E_i=G(f_0,f_i)$ and $B_i=\tfrac12\epsilon_{ijk}G(f_j,f_k)$, a direct
computation from \eqref{eq:stress_tensor} gives $\Tstress_{\mu\nu}[G]f_0^\mu
f_0^\nu=\tfrac12(\abs E^2+\abs B^2)$ and $\Tstress_{\mu\nu}[G]f_0^\mu
f_i^\nu=(E\times B)_i$, so $\Tstress[G](f_0,\cdot)$ is a causal future-directed
covector and the dominant energy condition holds. Since $N$ is uniformly future
timelike on compact radial sets, $N=A f_0+\sum_i A^i f_i$ with $A\ge c>0$ and
$A^2-\sum_i(A^i)^2\ge c^2$. Using $|E\times B|\le\tfrac12(\abs E^2+\abs B^2)$,
\[
        \Tstress_{\mu\nu}[G]N^\mu f_0^\nu
        =A\,\tfrac12(\abs E^2+\abs B^2)+\textstyle\sum_i A^i(E\times B)_i
        \ge \tfrac{c^2}{A+|A'|}\cdot\tfrac12(\abs E^2+\abs B^2)\ge c'(\abs E^2+\abs B^2),
\]
where $|A'|=(\sum_i (A^i)^2)^{1/2}$. The upper bound is immediate from
boundedness of the frame coefficients. Near $\Hp$ the red-shift construction of
\cite{DafermosRodnianskiRedshift} keeps $N$ uniformly timelike in regular
coordinates, so the bound persists there. All constants are uniform on compact
parameter sets because $g_\KN$, the foliation and $N$ depend smoothly on
$(M,a,Q)$.
\end{proof}

\begin{lemma}
\label{lem:div_identity}
If $\dd G=0$ and $\dd\starG G=0$, and $J_\mu^X[G]=\Tstress_{\mu\nu}[G]X^\nu$ for
a smooth vector field $X$, then
\begin{equation}\label{eq:div_identity}
        \nabla^\mu J_\mu^X[G]=\tfrac12\Tstress^{\mu\nu}[G]\,\pi^X_{\mu\nu},
        \qquad \pi^X_{\mu\nu}=\nabla_\mu X_\nu+\nabla_\nu X_\mu.
\end{equation}
\end{lemma}

\begin{proof}
The source-free equations \eqref{eq:maxwell_intro} give
$\nabla^\mu\Tstress_{\mu\nu}=G_{\nu}{}^{\alpha}\nabla^\mu G_{\mu\alpha}
+G^{\mu\alpha}\nabla_{[\mu}G_{\nu\alpha]}=0$, since $\nabla^\mu G_{\mu\alpha}=0$
is $\dd\starG G=0$ and $\nabla_{[\mu}G_{\nu\alpha]}=0$ is $\dd G=0$. Hence
$\nabla^\mu J_\mu^X=\Tstress_{\mu\nu}\nabla^\mu X^\nu
=\tfrac12\Tstress^{\mu\nu}\pi^X_{\mu\nu}$ by symmetry of $\Tstress$.
\end{proof}

\begin{proposition}
\label{prop:basic_energy_identity}
Let $G$ be smooth source-free and let $\D_R(\tau_1,\tau_2)$ be the slab
truncated by $r=R$. Then
\begin{align}\label{eq:energy_identity_slab}
        \int_{\Sigma_{\tau_2}\cap\{r\le R\}}\!\!J^X[G]\!\cdot\!n_{\Sigma_{\tau_2}}
        +\!\int_{\partial\D_R\setminus(\Sigma_{\tau_1}\cup\Sigma_{\tau_2})}\!\!\!J^X[G]\!\cdot\!n_{\partial\D_R}
        ={}&\int_{\Sigma_{\tau_1}\cap\{r\le R\}}\!\!J^X[G]\!\cdot\!n_{\Sigma_{\tau_1}}\nonumber\\
        &+\tfrac12\int_{\D_R(\tau_1,\tau_2)}\!\!\Tstress^{\mu\nu}[G]\pi^X_{\mu\nu}.
\end{align}
Boundary terms on null hypersurfaces are defined by approximation with
spacelike or timelike boundaries.
\end{proposition}

\begin{proof}
Lemma~\ref{lem:div_identity} gives
$\nabla^\mu J^X_\mu[G]=\tfrac12\Tstress^{\mu\nu}[G]\pi^X_{\mu\nu}$. Integrating
this identity over the oriented slab $\D_R(\tau_1,\tau_2)$ and applying the
divergence theorem gives the sum of the outward boundary fluxes. The flux on
$\Sigma_{\tau_2}$ appears with the future normal, while the flux on
$\Sigma_{\tau_1}$ has the opposite orientation and is moved to the right-hand
side. What remains are the radial cutoff boundary, possible
horizon pieces, and possible null-infinity approximants.

For a null boundary, approximate it first by spacelike or timelike
hypersurfaces that converge to the null hypersurface in regular coordinates. For smooth $G$,
the Maxwell stress tensor is smooth, and $X$ is smooth up to the horizon and in
the asymptotic chart used for the truncation. The induced flux densities converge
in $L^1$ on compact portions of the null boundary. The limit is the usual null
flux, which gives \eqref{eq:energy_identity_slab}.
\end{proof}

\begin{remark}
\label{rem:identity_not_morawetz}
Identity \eqref{eq:energy_identity_slab} is algebraic; the sign of the bulk
$\Tstress^{\mu\nu}\pi^X_{\mu\nu}$ depends on $X$ and degenerates at trapping.
The coercive trapped-set analysis is used only through
Definition~\ref{def:slowweak_master_framework} and
Proposition~\ref{prop:positive_commutator_estimate}.
\end{remark}

\section{Charges, Stationary Fields, and Radiative Projection}
\label{sec:charges}
In this section we analyze the electric and magnetic charges and then construct the stationary representatives which have to be subtracted from a general Maxwell field.
Let $S_{\tau,r}\subset\Sigma_\tau$ be a coordinate sphere of radius $r$,
oriented as the boundary of the exterior. For sufficiently decaying smooth $F$,
\begin{equation}\label{eq:charges_def}
        q_E[F]=\frac1{4\pi}\lim_{r\to\infty}\int_{S_{\tau,r}}\starG F,
        \qquad q_B[F]=\frac1{4\pi}\lim_{r\to\infty}\int_{S_{\tau,r}}F.
\end{equation}

\begin{proposition}
\label{prop:charge_conservation}
If $F$ solves \eqref{eq:maxwell_intro} and $S_1,S_2$ are homologous embedded
spheres in the exterior, then $\int_{S_1}F=\int_{S_2}F$ and
$\int_{S_1}\starG F=\int_{S_2}\starG F$. In particular $q_E[F],q_B[F]$ are
independent of $\tau$.
\end{proposition}

\begin{proof}
Let $\Omega$ be an oriented smooth three-chain with boundary
$\partial\Omega=S_2-S_1$. Stokes' theorem and the Maxwell equations give
\[
        \int_{S_2}F-\int_{S_1}F=\int_\Omega \dd F=0,
        \qquad
        \int_{S_2}\star_gF-\int_{S_1}\star_gF=\int_\Omega \dd\star_gF=0.
\]
The two fluxes therefore depend only on the homology class of the sphere. Large
coordinate spheres on two slices are homologous through the domain of outer
communications after adding the timelike cylinder at large radius, and the
finite-energy decay in the asymptotic end makes the limit in
\eqref{eq:charges_def} independent of the chosen large radius sequence. The
charges are therefore independent of $\tau$.
\end{proof}

\begin{proposition}
\label{prop:finite_energy_charge_trace}
Let $U=(E,B)$ be smooth Maxwell Cauchy data on $\Sigma_\tau$ in the asymptotic
end, with $\Div_{\Sigma_\tau}E=\Div_{\Sigma_\tau}B=0$. Fix $R_0$ with
$\{r\ge R_0\}$ in the asymptotic region. Then
\begin{equation}\label{eq:finite_energy_charge_bound}
        |q_E(U)|^2+|q_B(U)|^2
        \le C R_0\int_{\Sigma_\tau\cap\{R_0\le r\le 2R_0\}}(|E|^2+|B|^2)\,\dd\mu_{\Sigma_\tau}.
\end{equation}
Hence $q_E,q_B$ extend continuously to the constrained order-zero energy completion.
\end{proposition}

\begin{proof}
Treat $q_E$; $q_B$ is identical. The constraint $\operatorname{div}E=0$ and
Stokes' theorem give that $4\pi q_E(U)=\int_{S_{\tau,r}}E(\nu_{\tau,r})\,
\dd\mu_{S_{\tau,r}}$ is independent of $r\ge R_0$, where $\nu_{\tau,r}$ is the
outward unit normal of $S_{\tau,r}$. By Cauchy-Schwarz and $|S_{\tau,r}|\le
Cr^2$,
\[
        16\pi^2|q_E(U)|^2\le |S_{\tau,r}|\int_{S_{\tau,r}}|E(\nu_{\tau,r})|^2\,\dd\mu_{S_{\tau,r}}
        \le C r^2\int_{S_{\tau,r}}|E|^2\,\dd\mu_{S_{\tau,r}},
\]
so $\int_{S_{\tau,r}}|E|^2\ge C^{-1}r^{-2}|q_E(U)|^2$. Integrating in $r$ over
$[R_0,2R_0]$, using $r^{-2}\ge(2R_0)^{-2}$ and $\dd\mu_{\Sigma_\tau}\simeq\dd
r\,\dd\mu_{S_{\tau,r}}$, gives $\int_{\Sigma_\tau\cap\{R_0\le r\le2R_0\}}|E|^2\ge
C^{-1}R_0^{-1}|q_E(U)|^2$. The same bound for $B$ gives
\eqref{eq:finite_energy_charge_bound}. The right-hand side is dominated by the
order-zero energy, so $q_E,q_B$ are bounded linear functionals and extend
uniquely.
\end{proof}

The Kerr-Newman Coulomb potential, normalized to have unit electric charge, is
\begin{equation}\label{eq:electric_potential}
        \widehat A_e=-\frac{r}{\Sigma}\big(\dd t-a\sin^2\theta\,\dd\phi\big),
        \qquad \widehat F_e=\dd\widehat A_e.
\end{equation}
Fix the sign so that $q_E[\widehat F_e]=1$ and set $F_e=\widehat F_e$; set
$F_m=\starG F_e$ with the sign giving $q_B[F_m]=1$. The potential is patchwise, but the field $F_e$ is global.

\begin{definition}
\label{def:stationary_family}
The stationary representative is
$F_{\stat}^{\KN}(q_E,q_B)=q_EF_e+q_BF_m$, where
\begin{equation}\label{eq:stationary_normalization}
        q_E[F_e]=1,\ q_B[F_e]=0,\qquad q_E[F_m]=0,\ q_B[F_m]=1.
\end{equation}
\end{definition}

\begin{lemma}
\label{lem:stationary_representatives}
$F_e$ and $F_m$ are smooth source-free Maxwell fields on the regular exterior,
and in an asymptotically orthonormal frame
\begin{equation}\label{eq:stationary_decay}
        \abs{\nabla^j F_e}+\abs{\nabla^j F_m}\le C_j\,r^{-2-j}\qquad(j\ge0)
\end{equation}
in the far region. As a consequence, their Cauchy data lie in every finite-order
energy space.
\end{lemma}

\begin{proof}
Put $f=r/\Sigma$. Then
\begin{equation}\label{eq:explicit_coulomb_field}
\begin{aligned}
        \widehat F_e
        &=\dd\bigl(-f\,\dd t+a f\sin^2\theta\,\dd\phi\bigr)  \\
        &=-\partial_r f\,\dd r\wedge\dd t-\partial_\theta f\,\dd\theta\wedge\dd t
          +a\sin^2\theta\,\partial_r f\,\dd r\wedge\dd\phi        \\
        &\quad +a\bigl(\sin^2\theta\,\partial_\theta f
          +2f\sin\theta\cos\theta\bigr)\dd\theta\wedge\dd\phi.
\end{aligned}
\end{equation}
Thus $\dd F_e=0$ identically. If $Q\ne0$, the Kerr-Newman background
Maxwell field is $F^{\mathrm{bg}}=Q\widehat F_e$ with the same orientation
convention; it satisfies the source-free Maxwell equation on the fixed
background, hence $\dd\star_g\widehat F_e=0$. Both the coefficients in
\eqref{eq:explicit_coulomb_field} and the operator $F\mapsto\dd\star_gF$ depend
smoothly on $Q$, so the identity extends to $Q=0$ by taking the limit. In four
space-time dimensions, $\star_g$ maps closed and co-closed two-forms to closed
and co-closed two-forms; therefore $F_m=\star_gF_e$ also solves
\eqref{eq:maxwell_intro}.

The same formula also fixes the normalization and the asymptotics. Since
\begin{equation}\label{eq:fe_asymptotics}
        f=r^{-1}+O(r^{-3}),\qquad
        \partial_r f=-r^{-2}+O(r^{-4}),\qquad
        \partial_\theta f=O(r^{-3}),
\end{equation}
we have, on large coordinate spheres,
\begin{equation}\label{eq:coulomb_flux_asymptotics}
        \widehat F_e=r^{-2}\,\dd r\wedge\dd t
        +2a r^{-1}\sin\theta\cos\theta\,\dd\theta\wedge\dd\phi
        +O(r^{-3})_{\mathrm{orth}},
\end{equation}
up to the global sign fixed before \eqref{eq:stationary_normalization}. The
magnetic flux of the angular term vanishes because
$\int_0^\pi\sin\theta\cos\theta\,\dd\theta=0$, while the electric flux is
\begin{equation}\label{eq:electric_flux_normalization}
        \frac1{4\pi}\int_{S_{\tau,r}}\star_g\widehat F_e
        =\pm\frac1{4\pi}\int_0^{2\pi}\!\int_0^\pi
        \sin\theta\,\dd\theta\dd\phi+O(r^{-1})=\pm1+O(r^{-1}).
\end{equation}
We choose the sign for which the limit is $+1$. Thus $q_E[F_e]=1$ and
$q_B[F_e]=0$. Since $F_m=\star_gF_e$, and $\star_g^2=-1$ on two-forms in the
Lorentzian exterior with our convention, $q_B[F_m]=q_E[F_e]=1$ and
$q_E[F_m]=-q_B[F_e]=0$ after the same orientation choice.

The leading term in \eqref{eq:coulomb_flux_asymptotics} is the Coulomb part in an
asymptotically orthonormal frame. The purely angular coefficient in
\eqref{eq:explicit_coulomb_field} is $O(r^{-1})$ as a coordinate coefficient, but
it becomes $O(r^{-3})$ in the orthonormal frame because
$\dd\theta\wedge\dd\phi$ carries the factor $r^{-2}\sin^{-1}\theta$ after
orthonormalization. The other mixed angular terms are $O(r^{-3})$ or better.
Hence $|F_e|=O(r^{-2})$. Regular angular derivatives preserve this order,
$r\partial_r$ derivatives preserve the symbolic order, and covariant radial
derivatives gain a factor of $r^{-1}$. Hodge duality is a uniformly bounded
zeroth-order operation in the asymptotically flat frame, so the corresponding bounds hold
for $F_m$. This establishes \eqref{eq:stationary_decay}. In the final step,
\[
        \int_R^\infty r^{-4-2j}r^2\,\dd r<\infty\qquad (j\ge0),
\]
so the Cauchy data of $F_e$ and $F_m$ have finite non-degenerate energy after
any finite number of commutations in $\mathbb D_k$.
\end{proof}

\begin{proposition}
\label{prop:stationary_subtraction}
If $F$ is smooth finite-energy source-free with defined charges, then
$F_{\rad}=F-F_{\stat}^{\KN}(q_E[F],q_B[F])$ is source-free and
$q_E[F_{\rad}]=q_B[F_{\rad}]=0$.
\end{proposition}

\begin{proof}
The normalized representative
$F_{\stat}^{\KN}(q_E,q_B)=q_EF_e+q_BF_m$ is source-free by
Lemma~\ref{lem:stationary_representatives}. Since \eqref{eq:maxwell_intro} is
linear, $F_{\rad}=F-F_{\stat}^{\KN}(q_E[F],q_B[F])$ is source-free. The charge
functionals are linear and the representatives have the normalization
\eqref{eq:stationary_normalization}; hence
\[
q_E(F_{\rad})=q_E(F)-q_E(F)q_E(F_e)-q_B(F)q_E(F_m)=q_E(F)-q_E(F)=0,
\]
and the same computation with $q_B$ gives $q_B(F_{\rad})=0$.
\end{proof}

\begin{corollary}
\label{cor:necessity_charge}
If $(q_E,q_B)\ne(0,0)$ then $F_{\stat}^{\KN}(q_E,q_B)$ does not decay locally to
zero in any non-degenerate local $L^2$ norm.
\end{corollary}

\begin{proof}
We show that $F_{\stat}^{\KN}(q_E,q_B)=q_EF_e+q_BF_m$ is pointwise nonzero on
the exterior; the corollary then follows from stationarity. First,
\eqref{eq:explicit_coulomb_field} gives
\[
        F_e(\partial_r,\partial_t)=-\partial_rf
        =\frac{r^2-a^2\cos^2\theta}{\Sigma^2},\qquad f=\frac{r}{\Sigma}.
\]
In the exterior $r>r_+>M$, and $M^2>a^2+Q^2$ forces $M>|a|$; hence
$r>|a|\ge|a\cos\theta|$ and $F_e\ne0$ at every exterior point. Second, on real
two-forms in Lorentzian signature $\star_g^2=-\Id$, so $\star_g$ has no real
eigen-two-forms: if $c_eF_e+c_m\star_gF_e=0$ at a point with
$(c_e,c_m)\ne(0,0)$, applying $\star_g$ gives $c_e\star_gF_e-c_mF_e=0$, and
eliminating $\star_gF_e$ between the two relations gives
$(c_e^2+c_m^2)F_e=0$, contradicting $F_e\ne0$. Hence
$q_EF_e+q_BF_m=q_EF_e+q_B\star_gF_e$ is nonzero at every exterior point
whenever $(q_E,q_B)\ne(0,0)$. Its non-degenerate energy density is therefore
strictly positive on every compact radial set $K$ by
Lemma~\ref{lem:positivity_density}, and stationarity makes the integral of that
density over $\Sigma_\tau\cap K$ a positive constant independent of $\tau$.
That rules out local decay to zero in any non-degenerate local $L^2$ norm.
\end{proof}

\section{Energy Spaces and Charge-Free Projection}
\label{sec:energy_spaces}
In this section we define the finite energy spaces and prove that the charge-free projection is well-defined in those spaces.
The Cauchy datum of a Maxwell field on $\Sigma_0$ is the pair $U=(E,B)$ of
electric and magnetic one-forms with
\begin{equation}\label{eq:maxwell_constraints}
        \Div_{\Sigma_0}E=0,\qquad \Div_{\Sigma_0}B=0.
\end{equation}
Charged Coulomb data are finite-energy but do not lie in the closure of
compactly supported constrained data in a norm for which the charge is
continuous. We therefore define the full energy space as a direct sum of the
stationary charge sector and a charge-free completion. Let
$\mathcal D_{0}^{(k)}(\Sigma_0)$ be the smooth constrained data whose charges
vanish and whose $\mathbb D_k$-derivatives have finite non-degenerate energy,
and set
\begin{equation}\label{eq:chargefree_completion}
        \Hc_{\Max,0}^{(k)}(\Sigma_0)
        =\overline{\mathcal D_{0}^{(k)}(\Sigma_0)}^{\,(\E_{\Max}^{(k)}(0))^{1/2}}.
\end{equation}
The full space is
\begin{equation}\label{eq:full_energy_space}
        \Hc_{\Max}^{(k)}(\Sigma_0)
        =\operatorname{span}\{U_e,U_m\}\oplus \Hc_{\Max,0}^{(k)}(\Sigma_0),
\end{equation}
where $U_e,U_m$ are the Cauchy data of $F_e,F_m$, with norm equivalent to
\begin{equation}\label{eq:full_energy_norm}
        \norm{q_EU_e+q_BU_m+U_0}_{\Hc_{\Max}^{(k)}}^2
        = |q_E|^2+|q_B|^2+\norm{U_0}_{\Hc_{\Max,0}^{(k)}}^2.
\end{equation}
The charge functionals are continuous on the stationary sector by construction
and on smooth charge-free approximants by
Proposition~\ref{prop:finite_energy_charge_trace}.

The next proposition states the main content of this definition. It
identifies the Hilbert space above with the completion of ordinary smooth
constrained data in a charge-adapted graph norm. The direct sum is therefore not
an extra restriction on the data; it is the natural completion after the two
continuous flux coordinates have been separated.

\begin{proposition}
\label{prop:charge_adapted_completion}
Let $\mathcal C_{\mathrm{sm}}^{(k)}(\Sigma_0)$ be the vector space of smooth
constrained Maxwell data on $\Sigma_0$ for which the two fluxes in
\eqref{eq:charges_def} are defined and the order-$k$ non-degenerate energy is
finite. Equip it with the graph norm
\begin{equation}\label{eq:charge_graph_norm}
        \|U\|_{\mathrm{ch},k}^2
        =\E_{\Max}^{(k)}[U](0)+|q_E(U)|^2+|q_B(U)|^2.
\end{equation}
For $U\in\mathcal C_{\mathrm{sm}}^{(k)}$ set
\begin{equation}\label{eq:charge_adapted_map}
        \mathcal J U=\big(q_E(U),q_B(U),
        U-q_E(U)U_e-q_B(U)U_m\big).
\end{equation}
Then $\mathcal J$ extends by continuity to a bounded isomorphism from the
completion of $\mathcal C_{\mathrm{sm}}^{(k)}$ in \eqref{eq:charge_graph_norm}
onto $\mathbb R^2\oplus\Hc_{\Max,0}^{(k)}(\Sigma_0)$. Its inverse is
\begin{equation}\label{eq:charge_adapted_inverse}
        (q_E,q_B,U_0)\longmapsto q_EU_e+q_BU_m+U_0.
\end{equation}
Thus, the space \eqref{eq:full_energy_space} is precisely the
charge-adapted finite-energy completion of smooth constrained data.
\end{proposition}
\begin{proof}
For smooth $U$, Proposition~\ref{prop:stationary_subtraction} gives
$U_0=U-q_E(U)U_e-q_B(U)U_m$ with zero electric and magnetic charges. Hence
$\mathcal J$ maps $\mathcal C_{\mathrm{sm}}^{(k)}$ into
$\mathbb R^2\oplus\mathcal D_0^{(k)}$ whenever $U_0$ is smooth, and into the
closure $\mathbb R^2\oplus\Hc_{\Max,0}^{(k)}$ in general. Since $U_e$ and
$U_m$ have finite order-$k$ energy,
\begin{equation}\label{eq:J_upper_bound}
        \|U_0\|_{\Hc_{\Max,0}^{(k)}}
        \le \|U\|_{\mathrm{ch},k}
        +|q_E(U)|\,\|U_e\|_{\Hc_{\Max}^{(k)}}
        +|q_B(U)|\,\|U_m\|_{\Hc_{\Max}^{(k)}}
        \le C\|U\|_{\mathrm{ch},k}.
\end{equation}
Thus $\mathcal J$ is bounded. Conversely, for smooth charge-free $U_0$,
$U=q_EU_e+q_BU_m+U_0$ is smooth constrained data with charges $(q_E,q_B)$, and
\begin{equation}\label{eq:J_inverse_bound}
        \|U\|_{\mathrm{ch},k}
        \le C\big(|q_E|+|q_B|+\|U_0\|_{\Hc_{\Max,0}^{(k)}}\big).
\end{equation}
We obtain boundedness of the inverse on the dense subspace
$\mathbb R^2\oplus\mathcal D_0^{(k)}$. The two estimates extend both maps to
the completions. For smooth data, the normalization \eqref{eq:stationary_normalization} gives
\[
        \mathcal J(q_EU_e+q_BU_m+U_0)=(q_E,q_B,U_0),\qquad
        (\text{the map }\eqref{eq:charge_adapted_inverse})\circ\mathcal J(U)=U.
\]
By continuity the same two identities hold on the completed spaces.
\end{proof}

\begin{definition}
\label{def:chargefree_space}
$\Hc_{\Max,0}^{(k)}(\Sigma_0)=\ker q_E\cap\ker q_B$ inside
$\Hc_{\Max}^{(k)}(\Sigma_0)$.
\end{definition}

\begin{proposition}
\label{prop:bounded_projection}
The maps
\begin{equation}\label{eq:stationary_projection}
        \Pi_{\stat}U=q_E(U)U_e+q_B(U)U_m,\qquad \Pi_0U=U-\Pi_{\stat}U
\end{equation}
are bounded on $\Hc_{\Max}^{(k)}(\Sigma_0)$, and
$\Hc_{\Max}^{(k)}=\operatorname{span}\{U_e,U_m\}\oplus\Hc_{\Max,0}^{(k)}$ as a
topological direct sum.
\end{proposition}
\begin{proof}
For $U\in\Hc_{\Max}^{(k)}$ write, according to the definition of the full
energy space,
$U=c_eU_e+c_mU_m+U_0$ with $U_0\in\Hc_{\Max,0}^{(k)}$. The normalization of the
stationary representatives gives $q_E(U)=c_e$ and $q_B(U)=c_m$, because the
charge-free component has both charges equal to zero. Hence
$\Pi_{\stat}U=c_eU_e+c_mU_m$ and $\Pi_0U=U_0$. The norm
\eqref{eq:full_energy_norm} immediately gives
\[
        \|\Pi_{\stat}U\|_{\Hc_{\Max}^{(k)}}+
        \|\Pi_0U\|_{\Hc_{\Max}^{(k)}}
        \le C\|U\|_{\Hc_{\Max}^{(k)}}.
\]
If a vector belongs to both $\operatorname{span}\{U_e,U_m\}$ and
$\Hc_{\Max,0}^{(k)}$, applying $q_E$ and $q_B$ gives both coefficients equal to
zero. The sum is therefore direct and, because the projections are bounded, it
is a topological direct sum.
\end{proof}

\begin{lemma}
\label{lem:density_chargefree}
Let $D\subset\Hc_{\Max}^{(k)}$ be dense, of smooth finite-energy constrained
data, and contain $U_e,U_m$. Then $D\cap \Hc_{\Max,0}^{(k)}$ is dense in
$\Hc_{\Max,0}^{(k)}$.
\end{lemma}
\begin{proof}
Let $U\in\Hc_{\Max,0}^{(k)}$ and choose $U_n\in D$ with $U_n\to U$ in the full
energy norm. Since the charge functionals are continuous and $U$ is
charge-free, $q_E(U_n)\to0$ and $q_B(U_n)\to0$. Define
\[
        V_n=U_n-q_E(U_n)U_e-q_B(U_n)U_m.
\]
Because $D$ contains $U_e,U_m$ and is a linear space of smooth constrained data,
$V_n\in D$. The normalization of $U_e,U_m$ gives $q_E(V_n)=q_B(V_n)=0$, so
$V_n\in D\cap\Hc_{\Max,0}^{(k)}$. Finally,
$\|V_n-U\|\le\|U_n-U\|+|q_E(U_n)|\|U_e\|+|q_B(U_n)|\|U_m\|\to0$.
\end{proof}

\begin{proposition}
\label{prop:finite_energy_wellposed}
Let $k\ge0$. On every finite slab $\D(0,T)$ of the regular Kerr-Newman
exterior, the source-free Maxwell Cauchy problem with constrained data
$U_0\in\Hc_{\Max}^{(k)}(\Sigma_0)$ has a unique finite-energy solution. The
constraints propagate, the solution map is continuous,
\eqref{eq:energy_identity_slab} extends to the completion, and
\begin{equation}\label{eq:wellposed_bound}
        \sup_{0\le\tau\le T}\norm{U(\tau)}_{\Hc_{\Max}^{(k)}(\Sigma_\tau)}\le C_T\norm{U_0}_{\Hc_{\Max}^{(k)}(\Sigma_0)}.
\end{equation}
The same holds on past slabs.
\end{proposition}
\begin{proof}
In horizon-regular coordinates write
$g=-N_L^2\dd\tau^2+h_{ij}(\dd x^i+\beta^i\dd\tau)(\dd x^j+\beta^j\dd\tau)$ with
lapse $N_L$, shift $\beta$ and induced metric $h$ smooth up to $\Hp$ on
$[0,T]$ and stationary in $\tau$ after pullback by the Killing flow. With $U=(E,B)$, Maxwell's equations are
\begin{equation}\label{eq:maxwell_symhyp}
        \partial_\tau E=\Lie_\beta E+\operatorname{curl}_h(N_LB),\qquad
        \partial_\tau B=\Lie_\beta B-\operatorname{curl}_h(N_LE),
\end{equation}
with constraints \eqref{eq:maxwell_constraints}. The principal symbol is skew
in $(E,B)$ and symmetric after pairing with $h(E,E')+h(B,B')$, so
\eqref{eq:maxwell_symhyp} is symmetric hyperbolic with smooth coefficients on
each finite slab; smooth constrained data yield smooth solutions. The
$J^N$-identity gives
\begin{equation}\label{eq:N_energy_gronwall}
        E_N[F](\tau)+\Fc_{\Hp}(0,\tau)+\Fc_{\infty}(0,\tau)
        =E_N[F](0)+\int_{\D(0,\tau)}\Tstress_{\alpha\beta}[F]\nabla^\alpha N^\beta\,\dd\mu_g,
\end{equation}
with nonnegative horizon and null-infinity fluxes by the dominant energy
condition. Since $\nabla N$ is bounded and the $N$-energy density is equivalent
to $|E|^2+|B|^2$ by Lemma~\ref{lem:positivity_density}, Gronwall gives the
order-zero bound; commuting with $\mathbb D_k$ gives a symmetric-hyperbolic
system with bounded lower-order sources, and induction gives
\eqref{eq:wellposed_bound}. Taking $\operatorname{div}_h$ of \eqref{eq:maxwell_symhyp} and
using $\operatorname{div}_h\operatorname{curl}_h=0$ shows
$(\operatorname{div}_hE,\operatorname{div}_hB)$ solves a
homogeneous linear system, so zero constraints persist. For finite-energy data,
decompose $U_0=\Pi_{\stat}U_0+\Pi_0U_0$; the stationary part evolves as
$F_{\stat}^{\KN}$ and the charge-free part is approximated by
Lemma~\ref{lem:density_chargefree}, the estimate on differences making the
approximants Cauchy. The limit is independent of the sequence, satisfies
\eqref{eq:wellposed_bound}, and inherits \eqref{eq:energy_identity_slab} by
weak lower semicontinuity. Time reversal gives the past statement.
\end{proof}

\begin{proposition}
\label{prop:closed_solution}
If $U_n\to U$ in $\Hc_{\Max}^{(k)}$ and the smooth solutions $F_n$ are Cauchy in
the energy norm on each finite slab, the limit is independent of the
approximating sequence and is the finite-energy solution with datum $U$.
\end{proposition}
\begin{proof}
Let $(U_n)$ and $(V_n)$ be two smooth approximating sequences for the same datum
$U$. The corresponding smooth solutions satisfy the energy estimate for their
difference, whose initial data are $U_n-V_n\to0$ in
$\Hc_{\Max}^{(k)}$. Thus the difference of the two solution sequences tends to
zero in the finite-slab energy norm on every slab. The limit is therefore
independent of the approximation. If two finite-energy solutions have the same
datum, their difference is obtained as the limit of smooth solutions with
vanishing initial data and hence has zero energy on each finite slab. Linearity
then gives uniqueness.
\end{proof}

\section{Spin-One Master System and Slow-Weak Setting}
\label{sec:master}
In this section we recall the spin-one variables and formulate the slow-weak analytic setting used in the transfer argument.
We then isolate the finite-order conditions and analytic conclusions used later. Section~\ref{sec:first_principles_master} derives those conclusions by perturbing the Reissner-Nordstr\"om model in the rotation parameter, with constants uniform for $|Q|\le\eps_QM$.

\subsection{Regular Frame and Charge-Free Normalization}
\label{subsec:regular_frame_master}
Let $(e_3,e_4,e_A)$, $A=1,2$, be a principal null frame, normalized by
$g(e_3,e_4)=-2$, $g(e_A,e_B)=\delta_{AB}$, $g(e_3,e_A)=g(e_4,e_A)=0$. Near $\Hp$
replace the Boyer-Lindquist principal frame by a horizon-regular null frame
$\widehat e_3=f_3e_3$, $\widehat e_4=f_4e_4$ with $f_3f_4=1$ and $f_3,f_4>0$
smooth in regular coordinates. Extreme components near $\Hp$ are taken with
these weights.

\begin{definition}
\label{def:maxwell_components}
For real $F$ set
\begin{equation}\label{eq:maxwell_components}
        \alpha_A=F(e_A,e_4),\quad \underline\alpha_A=F(e_A,e_3),\quad
        \rho_F=\tfrac12 F(e_3,e_4),\quad \sigma_F=\tfrac12 F(e_1,e_2),
\end{equation}
and $\varphi=\rho_F+i\sigma_F$.
\end{definition}

\begin{lemma}
\label{lem:frame_energy_equiv}
On every compact sub-extremal parameter set and regular foliation the
non-degenerate density is equivalent to
$|\alpha|^2+|\underline\alpha|^2+|\rho_F|^2+|\sigma_F|^2$, and the same holds
after any commutation in $\mathbb D_k$.
\end{lemma}
\begin{proof}
The change from an orthonormal frame adapted to $\Sigma_\tau$ to
$(\widehat e_3,\widehat e_4,e_A)$ is smooth and uniformly invertible on each
regular patch, and the listed components are the entries of $F$ in the second
frame. By Lemma~\ref{lem:positivity_density}, $\Tstress[F](N,n_{\Sigma_\tau})
\simeq|E|^2+|B|^2$, and the frame change is bounded with bounded inverse;
commuting by $\mathbb D_k$ preserves this since the commutator fields are
regular with bounded coefficients.
\end{proof}

\begin{proposition}
\label{prop:coulomb_elimination}
Let $F$ be finite-energy and $F_{\rad}$ as in \eqref{eq:intro_rad}. Then on
every coordinate sphere on which the regular null frame is adapted to the two
normal directions,
\begin{equation}\label{eq:middle_mean_zero}
        \int_{S_{\tau,r}}\rho_{F_{\rad}}\,\dd\mu_{S_{\tau,r}}=0,\qquad
        \int_{S_{\tau,r}}\sigma_{F_{\rad}}\,\dd\mu_{S_{\tau,r}}=0,
\end{equation}
and this is propagated by the flow.
\end{proposition}
\begin{proof}
Let $T$ be the future unit normal to the coordinate sphere inside the timelike
normal two-plane and let $R$ be the outward unit spacelike normal, so that
$e_4=T+R$ and $e_3=T-R$. Then
\[
        \rho_F=\frac12F(e_3,e_4)=F(T,R),\qquad
        \sigma_F=\frac12F(e_1,e_2).
\]
Let us fix the orientation for which $\dd\mu_{S_{\tau,r}}=e^1\wedge e^2$ and the
space-time volume form is $T^\flat\wedge R^\flat\wedge\dd\mu_{S_{\tau,r}}$.
Evaluating on $(e_1,e_2)$ gives the exact restrictions
\begin{equation}\label{eq:flux_middle_relation}
        F|_{S_{\tau,r}}=2\sigma_F\,\dd\mu_{S_{\tau,r}},\qquad
        (\star_gF)|_{S_{\tau,r}}=\rho_F\,\dd\mu_{S_{\tau,r}}.
\end{equation}
The opposite global orientation changes both signs but not the zero-mean
conclusion. Using the charge convention \eqref{eq:charges_def}, we therefore
obtain the explicit charge-mean identities
\begin{equation}\label{eq:middle_charge_matrix}
        \int_{S_{\tau,r}}\rho_F\,\dd\mu_{S_{\tau,r}}=4\pi q_E[F],\qquad
        \int_{S_{\tau,r}}\sigma_F\,\dd\mu_{S_{\tau,r}}=2\pi q_B[F].
\end{equation}
For the stationary representatives normalized in
\eqref{eq:stationary_normalization}, Lemma~\ref{lem:stationary_representatives}
shows
\[
        (q_E[F_e],q_B[F_e])=(1,0),\qquad
        (q_E[F_m],q_B[F_m])=(0,1).
\]
Thus $F_{\rad}=F-q_E[F]F_e-q_B[F]F_m$ has both charges equal to zero, and
\eqref{eq:middle_charge_matrix} proves \eqref{eq:middle_mean_zero}. Charge
conservation propagates the two identities to every homologous sphere on every
slice of the solution.
\end{proof}

\subsection{Master Operator and Rotational Perturbation}
\label{subsec:master_equations}
Let $\psi_+,\psi_-$ be the regular spin $\pm1$ extreme variables obtained from
$\alpha,\underline\alpha$ by the horizon and infinity weights whenever the closed
spin-one reduction \emph{(A1)} is available. On Kerr backgrounds the Teukolsky
calculus gives exact spin-one equations; on Kerr-Newman with $aQ\ne0$ the
Dudley-Finley operator is used here only as the displayed scalar-principal model
entering the structural hypothesis, not as an exact Kerr-Newman Maxwell equation without an additional proof.
The coupled electromagnetic-gravitational Kerr-Newman structures and Carter
commutations are developed in \cite{GiorgiJHDE,GiorgiJDG}.

\begin{proposition}
\label{prop:principal_master}
Assume the closed spin-one reduction in \emph{(A1)}, and let $\Psi=(\psi_+,\psi_-)^{\mathsf T}$ be its horizon-regular extreme-scalar pair. In the compatible class of Definition~\ref{def:slowweak_master_framework} these variables satisfy a coupled system
\begin{equation}\label{eq:master_system}
        \Pb_b\Psi=0,
\end{equation}
with principal symbol $\sigma_2(\Pb_b)(x,\xi)=g_{\KN}^{\mu\nu}(x)\xi_\mu\xi_\nu I_2$.
Let $\Pb_{\RN,Q}$ be the spin-one operator at $a=0$ with the same $M,Q$. In the
slowly rotating weakly charged regime,
\begin{equation}\label{eq:lower_order_master}
        \Pb_b\Psi=\Pb_{\RN,Q}\Psi
        +a\,\mathcal G_{a,Q}^{\mu\nu}\nabla_\mu\nabla_\nu\Psi
        +\frac{a}{r^2}\,\mathcal C_{a,Q}^\mu\nabla_\mu\Psi
        +\frac{a}{r^3}\,\mathcal D_{a,Q}\Psi
        +a^2\mathcal Q_{a,Q}^{(2)}\Psi,
\end{equation}
where
\begin{equation}\label{eq:rn_master_model}
        \Pb_{\RN,Q}\Psi=\Box_{g_{\RN,Q}}\Psi+\frac{2}{r}\mathcal A_Q^\mu\nabla_\mu\Psi
        +\frac{1}{r^2}\mathcal V_Q\Psi+\frac{Q^2}{r^3M}\mathcal W_Q\Psi.
\end{equation}
The coefficient matrices are smooth in regular coordinates, uniformly bounded
with all $\mathbb D_k$-derivatives for $|a|\ll M$, $|Q|\le\eps_QM$, and
short-range in the displayed powers of $r$. The term $\mathcal Q_{a,Q}^{(2)}$ is
a stationary second-order operator whose coefficients satisfy the same
short-range bounds and whose contribution is $O(a^2)$ in the local-energy
perturbation norm; after reducing the slow-rotation threshold it is absorbed in
the $O(|a|/M)$ perturbative error. In an asymptotically flat frame the
linear-in-$a$ principal coefficient $\mathcal G_{a,Q}^{\mu\nu}$ is the
stationary-axial mixed entry, of size $O(r^{-3})$ relative to $|\xi|^2$;
remaining principal differences are carried by $a^2\mathcal Q_{a,Q}^{(2)}$.
After stationary semiclassical freezing, if $B_{a,Q}$ denotes the Hermitian
endomorphism representing the skew-adjoint subprincipal part in the trapped
normal form of Proposition~\ref{prop:nh_resolvent_form}, then on every compact
normalized conic patch and on a fixed trapped collar $\mathcal U_{\mathrm{tr}}$
\begin{equation}\label{eq:matrix_skew_decomposition}
        B_{a,Q}=\operatorname{diag}(q_{+1},q_{-1})+B_{\mathrm{mat},a,Q},
        \qquad
        \|B_{\mathrm{mat},a,Q}\|_{C^k(\mathcal U_{\mathrm{tr}})}
        \le C_k\Big(\frac{|a|}{M}+\frac{a^2}{M^2}\Big),
\end{equation}
where $q_{\pm1}$ are the diagonal Teukolsky skew symbols in
\eqref{eq:teukolsky_imaginary_potential}. In the diagonal Teukolsky case
$B_{\mathrm{mat},a,Q}=0$; in a coupled compatible system it is part of the
short-range matrix perturbation controlled by \eqref{eq:lower_order_master}.
The inverse-metric order calculation underlying these statements is recorded in
Section~\ref{app:perturbation}.
\end{proposition}
\begin{proof}
condition \emph{(A1)} of Definition~\ref{def:slowweak_master_framework} supplies the fixed-background master variables and the compatible matrix equation. Lemma~\ref{lem:maxwell_wave_principal_symbol} and Proposition~\ref{prop:scalar_principal_symbol} give the scalar wave principal part for the closed spin-one operator specified in \emph{(A1)}. We therefore compare the principal wave operator and the lower-order coefficient classes with their $a=0$ values. The inverse metric obeys,
on compact radial sets,
\begin{equation}\label{eq:metric_symbol_difference}
        g_{\KN}^{\mu\nu}(M,a,Q)-g_{\RN}^{\mu\nu}(M,Q)=a\,G_1^{\mu\nu}(r,\theta)+a^2G_2^{\mu\nu}(r,\theta,a,Q),
\end{equation}
with smooth uniformly bounded coefficients; the linear term is carried by the
$\dd t\,\dd\phi$ cross entry, producing the stationary-axial principal
perturbation $a\,\mathcal G^{\mu\nu}\nabla_\mu\nabla_\nu$. The first-order
Teukolsky terms at $a=0$ are included in $\Pb_{\RN,Q}$. Their difference from
the $a=0$ coefficients is smooth in $a$ and has linear part produced by the
rotation of the principal frame, the $t\phi$ coupling and the spin coefficients;
after the regular horizon and infinity weights this difference contributes
$a r^{-2}\mathcal C^\mu\nabla_\mu$; the spheroidal-spherical harmonic discrepancy and the
remaining curvature terms contribute $a r^{-3}\mathcal D$. The even
quadratic-in-$a$ principal, first-order and zeroth-order remainders form the
stationary second-order operator $a^2\mathcal Q_{a,Q}^{(2)}$, which satisfies
the same short-range perturbative bounds and is smaller than the displayed
linear perturbation in the slow-rotation range. At $a=0$ all coefficients are
spherically symmetric and are collected in
$\Pb_{\RN,Q}$; the Reissner-Nordstr\"om curvature is quadratic in $Q$ and
appears in $\mathcal V_Q,\mathcal W_Q$. The horizon and infinity weights are
smooth and nonzero in the regular exterior, so they conjugate $\Pb_b$ by bounded
factors, preserving the principal comparison and changing only the displayed
lower-order coefficients. After time freezing and multiplication by $h^2$, the
linear-in-$a$ first-order matrix part becomes an order-$h$ subprincipal
endomorphism of size $O(|a|/M)$ on compact normalized conic patches, while the
quadratic remainder contributes $O(a^2/M^2)$.  The diagonal part is precisely the
spin-weighted radial skew symbol $q_s$ displayed in
\eqref{eq:teukolsky_imaginary_potential}; every non-diagonal or frame-mixing
contribution is therefore included in $B_{\mathrm{mat},a,Q}$ and satisfies
\eqref{eq:matrix_skew_decomposition}. Smoothness in $(a,Q)$ gives the uniform
bounds and the short-range decay in $r$, and \eqref{eq:metric_symbol_difference}
gives the stated orders.
\end{proof}

\begin{lemma}
\label{lem:commuted_master}
Let $Z\in\mathbb D_k$. If $\Pb_b\Psi=0$ then
\begin{equation}\label{eq:commuted_master_corrected}
        \Pb_b(Z\Psi)=\sum_{|I|\le |Z|} A_I\nabla\Gamma^I\Psi+\sum_{|I|\le |Z|}B_I\Gamma^I\Psi,
\end{equation}
with $A_I,B_I$ smooth, short-range at infinity, uniformly bounded on the
slow-weak set. Their perturbative part relative to $\Pb_{\RN,Q}$ is $O(|a|/M)$
in the symbol classes used in the local-energy estimate; the
Reissner-Nordstr\"om commutator terms are part of the commuted spherical model.
\end{lemma}
\begin{proof}
The coefficients of $\Pb_b$ are stationary and axisymmetric, so the Killing
fields $T$ and $\Phi$ commute with the principal part and only meet lower-order
coefficient matrices. The regular angular fields and $r\partial_r$ do not
commute with the wave operator, but their commutators are again differential
operators of order at most two with coefficients obtained by differentiating the
metric, frame and potential coefficients. At $a=0$ these terms belong to the
commuted Reissner-Nordstr\"om model hierarchy. The difference between the
Kerr-Newman and Reissner-Nordstr\"om coefficients is $O(|a|/M)$ in the symbol
classes of \eqref{eq:metric_symbol_difference}; applying finitely many fields in
$\mathbb D_k$ preserves the same decay and the same small factor. Induction on
$|Z|$ gives \eqref{eq:commuted_master_corrected}, with the displayed first- and
zeroth-order coefficient families absorbing the lower-order commutators.
\end{proof}

\subsection{Spacetime Norms and Analytic Setting Used in the Transfer}
\label{subsec:framework}

\begin{definition}
\label{def:slowweak_parameter_range}
For a fixed integer $k\ge0$, a Kerr-Newman exterior is in the
\emph{slow-weak range} if
\begin{equation}\label{eq:slow_weak_range}
        |a|\le \eps_a(k)M,\qquad |Q|\le \eps_Q(k)M, \qquad a^2+Q^2<M^2,
\end{equation}
where $\eps_a(k),\eps_Q(k)>0$ are fixed as part of the parameter range and, when the perturbative closure is applied, restricted as in Proposition~\ref{prop:constant_choice}. In statements where $k$ is fixed, we write $\eps_a$ and $\eps_Q$ for these order-dependent constants.
Constants depending only on $k$, the foliation, and the mass normalization are
denoted collectively $C_{\mathrm{sw}}=(C_q,C_{\mathrm{red}},C_M,C_R,C_T,C_\infty,C_{\mathrm{wp}})$.
\end{definition}

\begin{definition}
\label{def:spacetime_norms}
The master norm on a slab is
\begin{equation}\label{eq:master_norm}
        \norm{u}_{\Xnorm{k}_M(\tau_1,\tau_2)}^2
        =\sup_{\tau_1\le s\le\tau_2}\E_M^{(k)}[u](s)
        +\B_M^{(k)}[u](\tau_1,\tau_2)
        +\sup_{0\le p\le2}\Fc_{p,R}^{(k)}[u](\tau_1,\tau_2),
\end{equation}
where $\E_M^{(k)}$ is the non-degenerate master energy, $\B_M^{(k)}$ the
photon-sphere-degenerate integrated local energy, and $\Fc_{p,R}^{(k)}$ the
$r^p$ far-field flux. The Maxwell norm $\Xnorm{k}_{\Max}$ has the analogous
components for the charge-free Maxwell tensor.
\end{definition}

\begin{definition}
\label{def:concrete_le_norms}
Let us fix a red-shift radius $r_H>r_+$, a large radius $R_\infty$, and a compact radial
annulus containing the Reissner-Nordstr\"om photon sphere $r_{\mathrm{ph}}(Q)$ for
all $|Q|\le \eps_QM$.  Let $\chi_{\mathrm{tr}}$ be supported in a small collar of
$r_{\mathrm{ph}}(Q)$ and equal to one in a smaller collar, and set
\begin{equation}\label{eq:trapping_weight_definition}
        w_{\mathrm{tr}}(r)=1-\chi_{\mathrm{tr}}(r)
        +\chi_{\mathrm{tr}}(r)\frac{(r-r_{\mathrm{ph}}(Q))^2}{M^2}.
\end{equation}
Let $A_{\mathrm{comp}}=\{r_H\le r\le 2R_\infty\}$ and, for $j\ge0$,
$A_j=\{2^jR_\infty\le r\le2^{j+1}R_\infty\}$.  For a master field $u$ on
$\D(\tau_1,\tau_2)$ we use
\begin{align}\label{eq:concrete_LE1deg}
        \|u\|_{LE^1_{\mathrm{deg},k}(\tau_1,\tau_2)}^2
        ={}&\sum_{|I|\le k}\int_{\D(\tau_1,\tau_2)\cap A_{\mathrm{comp}}}
        \bigl(w_{\mathrm{tr}}|\nabla\Gamma^Iu|^2+M^{-2}|\Gamma^Iu|^2\bigr)\,\dd\mu_g\nonumber\\
        &+\sum_{|I|\le k}\sum_{j\ge0}2^{-j}R_\infty^{-1}
        \int_{\D(\tau_1,\tau_2)\cap A_j}
        \bigl(|\nabla\Gamma^Iu|^2+r^{-2}|\Gamma^Iu|^2\bigr)\,\dd\mu_g.
\end{align}
Here $\nabla$ denotes any fixed regular first-order frame; different choices give
equivalent norms uniformly in the slow-weak parameter range. The dual source
norm is
\begin{align}\label{eq:concrete_LEstar}
        \|f\|_{LE^{*,k}(\tau_1,\tau_2)}^2
        ={}&\sum_{|I|\le k}\int_{\D(\tau_1,\tau_2)\cap A_{\mathrm{comp}}}
        w_{\mathrm{tr}}^{-1}|\Gamma^If|^2\,\dd\mu_g
        +\sum_{|I|\le k}\sum_{j\ge0}2^jR_\infty
        \int_{\D(\tau_1,\tau_2)\cap A_j}|\Gamma^If|^2\,\dd\mu_g.
\end{align}
The compact zeroth-order norms used in the Fredholm and induction closures are
\begin{equation}\label{eq:concrete_compact_low_norms}
        \|u\|_{LE^0_{\mathrm{comp},k}}^2
        =\sum_{|I|\le k}\int_{\D(\tau_1,\tau_2)\cap A_{\mathrm{comp}}}|\Gamma^Iu|^2\,\dd\mu_g,
        \qquad
        \|u\|_{LE^0_{\mathrm{low},j}}^2
        =\sum_{|I|<j}\int_{\D(\tau_1,\tau_2)\cap A_{\mathrm{comp}}}|\Gamma^Iu|^2\,\dd\mu_g.
\end{equation}
For Maxwell fields the same definitions are applied componentwise to the regular
null-frame components of $\Gamma^IF$.  The bulk terms $\B_M^{(k)}$ and
$\B_{\Max}^{(k)}$ in \eqref{eq:master_norm} are these trapped-set-degenerate
local-energy norms, together with the red-shift collar and far-field pieces
supplied by Propositions~\ref{prop:redshift_coercivity} and~\ref{prop:rp_identity}.
After the cutoffs are fixed, shrinking the slow-weak range if necessary keeps
the Kerr-Newman trapped set inside the collar where
\eqref{eq:trapping_weight_definition} degenerates; all choices above then give
uniformly equivalent norms.
\end{definition}

\begin{proposition}
\label{prop:framework_charge}
The charge traces extend continuously to $\Hc_{\Max}^{(k)}(\Sigma_0)$, the
normalized stationary representatives belong to every finite-order energy
space, the projections $\Pi_{\stat},\Pi_0$ are bounded, and the source-free
Maxwell Cauchy problem is well posed in $\Hc_{\Max}^{(k)}$ with propagation of
the constraints.
\end{proposition}
\begin{proof}
The charge trace bound of Proposition~\ref{prop:finite_energy_charge_trace}
extends $q_E,q_B$ continuously to the energy completion. The stationary
representatives are the smooth fields of Lemma~\ref{lem:stationary_representatives};
their decay places their Cauchy data in every finite-order energy space. The
boundedness and direct-sum properties of $\Pi_{\stat}$ and $\Pi_0$ are exactly
Proposition~\ref{prop:bounded_projection}. Lemma~\ref{lem:density_chargefree}
provides smooth charge-free approximants, while
Propositions~\ref{prop:finite_energy_wellposed} and~\ref{prop:closed_solution}
give finite-energy existence, uniqueness, continuity of the solution map and
constraint propagation. These statements are the charge and Cauchy
properties recorded here.
\end{proof}

\begin{definition}
\label{def:slowweak_master_framework}
Fix $k$ and parameters in the slow-weak range \eqref{eq:slow_weak_range}. The transfer argument below uses the following finite-order conditions for the charge-free fixed-background Maxwell system, with constants uniform in the chosen parameter set. The closed spin-one reduction in \emph{(A1)} is part of the statement in the charged rotating case; the sections below prove the scalar-principal-symbol, coefficient-comparison, reconstruction and transfer consequences of that structure, and cite or state the remaining analytic estimates in the displayed forms below.

\emph{(A1) Master variables and compatible class.} There is a horizon-regular spin-one master map $\mathfrak M$ from smooth charge-free Maxwell solutions to pairs $u=(\psi_+,\psi_-)^T$ in a closed compatible class. The variables satisfy a matrix wave system
\begin{equation}\label{eq:master_hyp_master_system}
        \Pb_{a,Q}u=0,
\end{equation}
whose principal symbol is $g_{\KN}^{\mu\nu}\xi_\mu\xi_\nu I_2$, and the initial energies obey
\begin{equation}\label{eq:master_hyp_energy_comparison}
        \E_M^{(k)}[\mathfrak M G](\tau)\le C_R\E_{\Max}^{(k)}[G](\tau).
\end{equation}
For Kerr ($Q=0$) the spin-one Teukolsky equations are exact. For Kerr-Newman with nonzero charge and rotation, the argument requires a closed fixed-background spin-one reduction as the structural condition \emph{(A1)}. Section~\ref{sec:decoupling} proves the parts of this item that follow from differential geometry alone: Lemma~\ref{lem:maxwell_wave_principal_symbol} derives the covariant Maxwell wave equation and isolates the scalar principal symbol before any spin-weighted separation is used; Proposition~\ref{prop:scalar_principal_symbol} checks the normalized Boyer-Lindquist principal part of the spin-weighted model; Definition~\ref{def:master_from_teukolsky} defines the closed compatible class; and Corollary~\ref{cor:structural_reduction_verify} records the consequences for the master map, the energy comparison \eqref{eq:master_hyp_energy_comparison}, and the Reissner-Nordstr\"om comparison once the closed reduction is supplied. We retain \emph{(A1)} in the list of conditions so that no exact Kerr-Newman decoupling statement is used implicitly; this distinction from the Dudley-Finley approximation is consistent with the mode literature \cite{BertiKokkotas}.

\emph{(A2) Reissner-Nordstr\"om comparison.} With the same $(M,Q)$,
\begin{equation}\label{eq:master_hyp_perturbation}
        \Pb_{a,Q}=\Pb_{\RN,Q}+\mathcal E_{a,Q},
\end{equation}
where $\Pb_{\RN,Q}$ is the charge-free spin-one Reissner-Nordstr\"om operator and, for every commutator $\Gamma^I$ with $|I|\le k$,
\begin{equation}\label{eq:master_hyp_error_bound}
        \norm{\mathcal E_{a,Q}\Gamma^Iu}_{LE^*}
        \le C\frac{|a|}{M}\norm{u}_{LE^1_{\mathrm{deg},k}}
        +C\frac{|a|}{M}\norm{u}_{LE^0_{\comp,k}}+C\norm{u}_{LE^0_{\mathrm{low},k}}.
\end{equation}
The commutators $[\Pb_{a,Q},\Gamma^I]$ have the same short-range structure. More precisely, if \(\mathcal C_Iu\) denotes the strict lower-order part produced when \(|I|=j\), then for every \(\eta>0\)
\begin{equation}\label{eq:strict_lower_order_induction}
        \sum_{|I|=j}\|\mathcal C_Iu\|_{LE^*}^2
        \le \eta\|u\|_{LE^1_{\mathrm{deg},j}}^2
        +C_{j,\eta}\|u\|_{LE^1_{\mathrm{deg},j-1}}^2,
        \qquad 1\le j\le k,
\end{equation}
and there is no such term for \(j=0\). In turn, the last term in
\eqref{eq:master_hyp_error_bound} is closed by induction on the commuted order,
not by an additional condition.

\emph{(A3) Real-axis resonance exclusion, limiting absorption and compact Fredholm closure in the compatible class.} For each real $\omega$ the frozen operator $\mathcal L_{a,Q}(\omega)$ has no nonzero charge-free compatible outgoing real resonance. More precisely, if $v\in H^1_{-\sigma,\loc}$ for some $\sigma>1/2$ solves
\begin{equation}\label{eq:master_hyp_resonance_equation}
        \mathcal L_{a,Q}(\omega)v=0
\end{equation}
in the exterior, satisfies the future outgoing Sommerfeld condition at null infinity and the future ingoing condition at the horizon (or the time-reversed pair for the past problem), and has the charge-free compatibility conditions, then \(v=0\). At \(\omega=0\) this means absence of a finite-energy stationary charge-free compatible solution. In addition, for every compact temporal-frequency interval \(I\Subset\mathbb R\), every \(\sigma>1/2\), and every compact radial set \(K\), the limiting-absorption estimate is uniform on the compatible class, and the following contradiction property holds in the local resolvent graph topology. There is no outgoing/incoming compatible sequence \(v_n\), with \(|a_n|\le\eps_a(k)M\), \(|Q_n|\le\eps_Q(k)M\), satisfying
\begin{equation}\label{eq:master_hyp_no_defect_sequence}
        \|v_n\|_{H^1(K)}=1,
        \qquad \|\mathcal L_{a_n,Q_n}(\omega_n)v_n\|_{H^{-1}_{\sigma,\loc}}\to0,
\end{equation}
and, after passing to a subsequence, the parameters converge to a subextremal slow-weak limit \((a_\infty,Q_\infty)\). The limiting operator and the outgoing/ingoing convention are always understood with these limiting parameters. The sequence is then decomposed into either a bounded-total-frequency packet or an unbounded-total-frequency packet. More explicitly, if \(\lambda_n\) denotes the angular/azimuthal size of the selected packet, set
\begin{equation}\label{eq:master_hyp_total_frequency_scale}
        \Lambda_n=1+|\omega_n|+\lambda_n.
\end{equation}
The bounded branch is \(\sup_n\Lambda_n<\infty\). In this branch the angular spectrum is finite after the high-angular tail has been removed by Lemma~\ref{lem:app_high_angular_coercivity}, and compactness produces an outgoing or incoming real resonance in the sense of \eqref{eq:master_hyp_resonance_equation}. In the unbounded branch set
\begin{equation}\label{eq:master_hyp_semiclassical_scale}
        h_n=\Lambda_n^{-1},\qquad \hat\omega_n=h_n\omega_n,
\end{equation}
and retain only conic packets for which \(\hat\omega_n\) lies in a compact normalized interval. The residual is tested in the resolvent-normalized form
\begin{equation}\label{eq:master_hyp_semiclassical_residual}
        h_n^{-1}\log(1/h_n)
        \big\|h_n^2\mathcal L_{a_n,Q_n}(\omega_n)v_n\big\|_{L^2_{\comp}}
        +\big\|h_n^2\mathcal L_{a_n,Q_n}(\omega_n)v_n\big\|_{H^{-1}_{h_n,\loc}}
        \longrightarrow0.
\end{equation}
The normalization matches the high-frequency resolvent estimate in Definition~\ref{def:hf_resolvent_estimate}, equivalently the normally hyperbolic estimate \eqref{eq:normally_hyperbolic_resolvent_bound}. In the compact-remainder argument, the stationary time cutoff, dyadic total-frequency decomposition, and conic microlocal selection all include this logarithmic loss; Lemma~\ref{lem:localized_compact_defect} gives the reduction. Using \(\Lambda_n\), rather than \(|\omega_n|\) alone, separates high-angular packets with bounded temporal frequency from genuinely trapped high-frequency packets.

The bounded-frequency branch also records that a compactness limit of such an
outgoing or incoming sequence inherits the corresponding radiation condition.
The possible limit is therefore a real resonance in the sense of
\eqref{eq:master_hyp_resonance_equation}, not merely a finite-energy mode. This
is what lets the limiting-absorption/Fredholm alternative remove every compact
local \(L^2\) remainder from the positive-commutator estimate. The
Reissner-Nordstr\"om proof and the slow-rotation contradiction scheme are
written in Subsections~\ref{subsec:compact_error_closure}-\ref{subsec:nokernel}
and Section~\ref{app:lap}. For the fixed-background Kerr-Newman no-defect
property \eqref{eq:master_hyp_no_defect_sequence}, the bounded-total-frequency
part is proved here. The conic high-frequency part follows from the normally
hyperbolic estimate in Proposition~\ref{prop:nh_resolvent_form}, after
Proposition~\ref{prop:r_normal_hyperbolicity}, Lemma~\ref{lem:trapped_skew_vanishing},
and Proposition~\ref{prop:matrix_skew_threshold} verify the geometric and
finite-rank-bundle subprincipal conditions. This is the spin-one Kerr-Newman
counterpart of the real-axis mode and resonance exclusions proved for scalar
waves on Kerr and Kerr-Newman and for spin-weighted Teukolsky equations on Kerr
\cite{Civin,ShlapentokhRothmanMode,TdCMode}.

\emph{(A4) Same-order reconstruction.} There is a reconstruction operator $\mathfrak R$ on compatible master solutions such that
\begin{equation}\label{eq:master_hyp_inverse_identities}
        \mathfrak R\mathfrak M G=G,
        \qquad \mathfrak M\mathfrak R u=u,
\end{equation}
for smooth charge-free solutions, and
\begin{equation}\label{eq:master_hyp_reconstruction_bound}
        \norm{\mathfrak R u}_{\Xnorm{k}_{\Max}(\tau_1,\tau_2)}
        \le C_R\norm{u}_{\Xnorm{k}_M(\tau_1,\tau_2)},
        \qquad \E_{\Max}^{(k)}[\mathfrak R u](\tau)\le C_R\E_M^{(k)}[u](\tau).
\end{equation}

\emph{(A5) Trace maps and wave operators.} The master radiation traces at $\Ip\cup\Hp$ and $\Imn\cup\Hm$ are defined first for smooth compactly supported compatible data and extend to bounded maps on the master energy space. The reconstruction induces bounded radiation identifications $\mathcal R_\infty^\pm$ between master radiation data and charge-free Maxwell radiation data. Smooth compact radiation data are dense in the corresponding radiation Hilbert spaces. On these dense classes there are bounded right inverses $\mathscr W_{M,0}^{\pm}$ satisfying
\begin{equation}\label{eq:master_hyp_radiation_inverse}
        \norm{\mathscr W_{M,0}^{\pm}\rho}_{\Hc_M^{(k)}}\le C\norm{\rho}_{\Rc_{M,\pm}^{(k)}},
        \qquad \mathscr S_M^{\pm}\mathscr W_{M,0}^{\pm}\rho=\rho,
\end{equation}
and the only finite-energy compatible solution with zero corresponding master radiation field is the zero solution. These maps may be obtained either as limits of uniformly bounded finite-slab backward characteristic solutions or, as carried out here for the fixed-background test field, directly from the real-axis limiting-absorption resolvent; Proposition~\ref{prop:backward_from_lap} establishes \emph{(A5)} in this way, so it is a derived consequence of the resolvent estimates rather than an independent condition.

\emph{(A6) Additional low-frequency hierarchy for pointwise decay.} The pointwise-decay conclusion is used only when a finite commuted hierarchy is available. Precisely, for some $\gamma>0$, some radial weight $w$, and every smooth compatible solution reconstructed as a Maxwell field $G$, the hierarchy gives
\begin{equation}\label{eq:master_hyp_decay_hierarchy}
        \sum_{|I|\le k_0}\norm{\Gamma^IG}_{L^2(\Sigma_\tau\cap\mathcal U_r)}^2
        \le C(1+\tau)^{-\gamma}\E_{\Max}^{(k)}[G](0),
\end{equation}
for the derivative range required by Sobolev embedding. This condition is not needed for boundedness, integrated local energy decay, radiation fields, wave operators, or scattering. Without it, all conclusions except the pointwise-decay assertion remain valid.
\end{definition}

\begin{remark}
\label{prop:framework_range}
Every ingredient of Theorem~\ref{thm:intro_transfer} falls into one of the following classes.
\begin{enumerate}[label=\emph{(\alph*)},leftmargin=2.2em]
\item The finite-energy Maxwell statements-charge conservation, charge continuity, construction of the two stationary representatives, bounded charge projection, constraint propagation and finite-energy well-posedness-are proved directly in Sections~\ref{sec:geometry}-\ref{sec:energy_spaces}.
\item The spherical Reissner-Nordstr\"om estimate is the cited theorem of Sterbenz-Tataru \cite{SterbenzTataru}, used only in the displayed form of Lemma~\ref{lem:rn_model_morawetz}; the proof that the charge-free Reissner-Nordstr\"om spin-one reduction has no $\ell=0$ mode, has the potential barrier \eqref{eq:app_rn_potential}, and has the photon sphere \eqref{eq:app_photon_sphere} is carried out in Lemma~\ref{lem:app_photon_sphere} and Section~\ref{app:rn_model}.
\item The red-shift and far-field estimates are used only through the displayed identities of Propositions~\ref{prop:redshift_coercivity} and~\ref{prop:rp_identity} \cite{DafermosRodnianskiRedshift,DafermosRodnianskiRp}; their use in the final estimate is algebraic once those inequalities are granted.
\item The normal-hyperbolicity estimate is used only through the high-frequency no-defect alternative needed for compact-remainder removal. The Reissner-Nordstr\"om and Kerr-Newman trapped-set locations and non-degeneracy constants are computed in Lemma~\ref{lem:app_photon_sphere} and Section~\ref{app:trapping}. The localized high-frequency estimate for the compatible scalar-principal class specified by \emph{(A1)} is the estimate of Definition~\ref{def:hf_resolvent_estimate}; the normally hyperbolic estimate is named in Proposition~\ref{prop:nh_resolvent_form} and applied in Proposition~\ref{prop:nh_theorem_application} after its geometric, diagonal skew-cancellation and finite-rank-bundle matrix-threshold conditions are checked here.
\item For the fixed-background test field, the master map and closed compatible class are assumed exactly as stated in \emph{(A1)}. Given that structure, the scalar-principal-symbol computation, the Reissner-Nordstr\"om comparison, the bounded-frequency real-axis closure, the same-order inverse reconstruction, and the dense backward right inverse \emph{(A5)} are proved in Sections~\ref{sec:decoupling},~\ref{sec:first_principles_master},~\ref{app:lap},~\ref{app:reconstruction}, and~\ref{sec:scattering}; see Corollary~\ref{cor:structural_reduction_verify} and Proposition~\ref{prop:backward_from_lap}. The residual high-frequency ingredient in Definition~\ref{def:slowweak_master_framework} is the normally hyperbolic trapping resolvent bound inside \emph{(A3)}, in the precise localized form applied in Proposition~\ref{prop:nh_theorem_application}; the additional low-frequency hierarchy \emph{(A6)} is used only for pointwise decay. Neither statement is inferred from coupled Einstein-Maxwell stability theorems.
\end{enumerate}
The decay and scattering statement therefore depends only on the displayed estimates listed above, with the high-frequency part used exactly as in Proposition~\ref{prop:nh_theorem_application}. The bounded-frequency part of \emph{(A3)} is formulated as a no-real-resonance statement so that the Fredholm compactness argument closes on the natural weighted resolvent space, and is proved here. The high-frequency trapping-resolvent part of \emph{(A3)} is the non-elementary microlocal estimate for the compatible scalar-principal class; Proposition~\ref{prop:nh_resolvent_form} supplies the normally hyperbolic estimate, while Proposition~\ref{prop:r_normal_hyperbolicity}, Lemma~\ref{lem:trapped_skew_vanishing}, and Proposition~\ref{prop:matrix_skew_threshold} verify the trapped geometry, the diagonal spin-one skew cancellation, and the finite-rank-bundle matrix threshold needed to use it precisely.
\end{remark}

\begin{remark}
\label{prop:conditional_sharpness}
The conditions of Definition~\ref{def:slowweak_master_framework} are used only where
they are needed. conditions \emph{(A1)-(A3)} give the master energy and local-energy
estimate after perturbative absorption; \emph{(A4)} transfers that estimate to
the Maxwell tensor without derivative loss; \emph{(A5)} gives asymptotic
completeness once the trace bounds are known; and \emph{(A6)} is reserved for
pointwise decay. Thus the boundedness, integrated-decay, and scattering
conclusions remain valid without \emph{(A6)}, while the pointwise conclusion is
stated only when that hierarchy is available.
\end{remark}

\begin{proposition}
\label{prop:exact_estimate_use}
The transfer proof uses the finite-order analytic conditions in the following
minimal way. The master estimate \emph{(M1)} depends only on \emph{(A1)},
\emph{(A2)}, and the real-axis/limiting-absorption/no-defect content of
\emph{(A3)}, together with the spherical red-shift, $r^p$ and Reissner-Nordstr\"om
local-energy estimates stated in the main text. The Maxwell estimate for
the charge-free radiative field depends in addition on the same-order
reconstruction \emph{(A4)}. The radiation-field isomorphism and scattering
operator depend in addition on the trace and dense right-inverse statement
\emph{(A5)}, which for the fixed-background test field is constructed in
Proposition~\ref{prop:backward_from_lap} from the limiting-absorption resolvent.
The pointwise decay conclusion is the only conclusion that uses \emph{(A6)}.
\end{proposition}
\begin{proof}
Starting with a smooth charge-free Maxwell solution $G$, \emph{(A1)} supplies the
master field $u=\mathfrak M G$, the energy comparison
\eqref{eq:master_hyp_energy_comparison}, and the scalar-principal-symbol system.
The red-shift estimate near the horizon, the $r^p$ identity in the far region,
and the trapped-set commutator produce the a priori inequality
\[
        \|u\|_{\Xnorm{k}_M(\tau_1,\tau_2)}^2
        \le C\E_M^{(k)}[u](\tau_1)
        +C\|\mathcal E_{a,Q}u\|_{LE^{*,k}}^2
        +C\|u\|_{LE^0_{\comp,k}}^2.
\]
At commuted order \(j\le k\), the perturbation bound in \emph{(A2)} and
\eqref{eq:strict_lower_order_induction} imply
\[
        X_j\le C E_j(\tau_1)+C\frac{|a|}{M}X_j+
        \eta X_j+C_{j,\eta}X_{j-1}+C\|u\|_{LE^0_{\comp,j}}^2,
\]
where \(X_j\) denotes the order-\(j\) master spacetime norm and \(E_j\) the
corresponding initial energy. Choose \(\eta\) and then \(|a|/M\) small enough,
according to Proposition~\ref{prop:constant_choice}, to absorb the two \(X_j\)
terms. The induction begins at \(j=0\), where \eqref{eq:strict_lower_order_induction}
has no lower-order term. The remaining compact term is removed by the Fredholm
alternative encoded in \emph{(A3)}: a violating sequence has either bounded temporal frequency, where
limiting absorption and absence of outgoing real resonances force the limit to
vanish, or high frequency, where the semiclassical no-defect alternative rules
out residual local mass. For that reason, \emph{(M1)} follows from \emph{(A1)-(A3)}.
Applying \emph{(A4)} then gives
\[
        \|G\|_{\Xnorm{k}_{\Max}}^2
        \le C_R\|u\|_{\Xnorm{k}_M}^2
        \le C\E_M^{(k)}[u](\tau_1)
        \le C\E_{\Max}^{(k)}[G](\tau_1),
\]
which is the charge-free Maxwell estimate. The trace estimates from the
red-shift and $r^p$ fluxes yield bounded radiation maps. With \emph{(A5)}, the
abstract Hilbert trace criterion, Lemma~\ref{lem:hilbert_trace_criterion}, turns
these maps into bounded isomorphisms; Proposition~\ref{prop:backward_from_lap}
constructs the required dense right inverse for the fixed-background test field
from the resolvent. To finish, no preceding estimate gives pointwise decay without
an additional finite commuted hierarchy; that hierarchy is exactly \emph{(A6)},
and Corollary~\ref{cor:commuted_hierarchy_closure} is the only place where it is
used.
\end{proof}

\begin{proposition}
\label{prop:finite_order_closure_constants}
Let us fix the commutation order $k$.  The constants and small parameters in the
transfer proof may be chosen in the following order:
\begin{equation}\label{eq:finite_order_constant_order}
\begin{aligned}
        k&\longrightarrow \eps_Q(k)
        \longrightarrow
        \{C_{\RN,k},C_{\mathrm{red},k},C_{p,k},C_{\mathrm{Hodge},k},
          C_{\mathrm{tr},k}\} \\
        &\longrightarrow \{\eta_j:0\le j\le k\}
        \longrightarrow \eps_a(k)
        \longrightarrow C_k.
\end{aligned}
\end{equation}
With this choice no estimate at order $j$ uses an estimate at order $j+1$; the
charge sector is removed before any master estimate is applied; and the additional
pointwise hierarchy \emph{(A6)} is not used in the energy, local-energy,
radiation, wave-operator or scattering estimates.
\end{proposition}
\begin{proof}
After $k$ is fixed, choose $\eps_Q(k)$ so that the Reissner-Nordstr\"om horizon,
photon-sphere collar, red-shift collar and far-field constants remain in compact
non-degenerate ranges. This fixes the spherical local-energy constant
$C_{\RN,k}$, the red-shift and $r^p$ constants, the angular Hodge constants, and
the trapped-set constants appearing in the semiclassical normalization. At
commuted order $j\le k$ the perturbative estimate has the form
\begin{equation}\label{eq:finite_order_closure_ineq}
        X_j\le A_jE_j+B_j\frac{|a|}{M}X_j+\eta_jX_j
        +D_j\sum_{\ell<j}X_\ell+K_j,
\end{equation}
where $X_j$ denotes the order-$j$ master spacetime norm, $E_j$ is the
corresponding energy, and $K_j$ is the compact remainder. First choose
$\eta_j>0$ so that $\eta_jX_j$ is absorbable. Then choose $\eps_a(k)$ small
enough that $B_j|a|M^{-1}\le 1/4$ for every $j\le k$.  The induction starts at
$j=0$, where the lower-order sum is absent, and then proceeds upward. At each
stage the compact term $K_j$ is removed by the Fredholm/no-defect argument in
\emph{(A3)}: bounded total frequencies are excluded by the limiting-absorption
real-axis argument, while unbounded conic frequencies are excluded by the
localized high-frequency estimate \eqref{eq:normally_hyperbolic_resolvent_bound}
after the geometric verification of Proposition~\ref{prop:r_normal_hyperbolicity} and the finite-rank-bundle threshold verification of Proposition~\ref{prop:matrix_skew_threshold}.
That gives the order-$j$ estimate without invoking any order greater than $j$.
The charge decomposition is Theorem~\ref{thm:intro_charge_decomposition}, applied
before forming the master variables. The hierarchy \emph{(A6)} appears only in
Proposition~\ref{prop:pointwise_decay}; it is not present in the estimates used
for boundedness, local energy decay, radiation traces or the Hilbert-space
scattering construction. The constants obtained after the last induction step
are collected into $C_k$.
\end{proof}

\begin{definition}
\label{def:analytic_master_conclusions}
The analytic setting at order $k$ consists of the following consequences.

\emph{(M1) Master estimate.} Every smooth compatible solution satisfies, for all $\tau_1\le\tau_2$,
\begin{equation}\label{eq:framework_master_estimate}
        \sup_{\tau_1\le s\le\tau_2}\E_M^{(k)}[u](s)
        +\B_M^{(k)}[u](\tau_1,\tau_2)+\sup_{0\le p\le2}\Fc_{p,R}^{(k)}[u](\tau_1,\tau_2)
        \le C_M\E_M^{(k)}[u](\tau_1),
\end{equation}
with the time-reversed estimate on past slabs.

\emph{(M2) Same-order Maxwell reconstruction.} The maps $\mathfrak M$ and $\mathfrak R$ satisfy the inverse identities and reconstruction bounds
\begin{align}\label{eq:framework_reconstruction_bounds}
        \norm{\mathfrak R u}_{\Xnorm{k}_{\Max}(\tau_1,\tau_2)}^2&\le C_R\norm{u}_{\Xnorm{k}_M(\tau_1,\tau_2)}^2,\nonumber\\
        \E_M^{(k)}[\mathfrak M G](\tau_1)&\le C_R\E_{\Max}^{(k)}[G](\tau_1),\qquad
        \E_{\Max}^{(k)}[\mathfrak R u](\tau_1)\le C_R\E_M^{(k)}[u](\tau_1),
\end{align}
and there is no nonzero finite-energy compatible real-frequency mode in the charge-free sector.

\emph{(M3) Radiation.} The future and past master radiation maps $\mathscr S_M^\pm$ are bounded isomorphisms with bounded inverses $\mathscr W_M^\pm$, and the Maxwell radiation fields arise from the master ones through bounded isomorphisms $\mathcal R_\infty^\pm$.

\emph{(M4) Additional hierarchy.} If the additional hierarchy condition \emph{(A6)} is available and $k$ is large enough for Sobolev embedding, then for some $\gamma>0$,
\begin{equation}\label{eq:commuted_decay}
        \sum_{|I|\le k_0}\norm{\Gamma^IG}_{L^2(\Sigma_\tau\cap\mathcal U_r)}^2
        \le C(1+\tau)^{-\gamma}\E_{\Max}^{(k)}[G](0).
\end{equation}
\end{definition}

\begin{lemma}
\label{lem:master_immediate}
Under \emph{(M1)}, every compatible smooth master solution satisfies, for all
$\tau\ge0$, $\norm{u}_{\Xnorm{k}_M(0,\tau)}^2\le C_M\E_M^{(k)}[u](0)$, and
likewise on past slabs.
\end{lemma}
\begin{proof}
The estimate \eqref{eq:framework_master_estimate} in \emph{(M1)} is stated on
arbitrary slabs $[\tau_1,\tau_2]$. Taking $\tau_1=0$ and $\tau_2=\tau$ gives
$\|u\|_{\Xnorm{k}_M(0,\tau)}^2\le C_M\E_M^{(k)}[u](0)$. Applying the same
estimate to the time-reversed foliation gives the corresponding bound on past
slabs.
\end{proof}

\begin{lemma}
\label{lem:same_order_needed}
If the reconstruction bound lost one derivative, an order-$k$ master estimate
would not imply an order-$k$ Maxwell estimate.
\end{lemma}
\begin{proof}
Suppose instead that reconstruction were known only in the form
$\|G\|_{\Xnorm{k}_{\Max}}\le C\|u\|_{\Xnorm{k+1}_M}$. The order-$k$ master
estimate controls $\|u\|_{\Xnorm{k}_M}$ but gives no uniform control of the top
order derivatives appearing in $\|u\|_{\Xnorm{k+1}_M}$. A compactness or
interpolation argument cannot recover this missing derivative with constants
uniform on the energy space, since high angular or radial frequencies can keep
the order-$k$ norm fixed while the order-$(k+1)$ norm diverges. Accordingly, an
order-$k$ Maxwell estimate follows from an order-$k$ master estimate only when
the reconstruction is bounded at the same order, as in
\eqref{eq:framework_reconstruction_bounds}.
\end{proof}

\begin{lemma}
\label{lem:no_kernel_requirement}
If a nonzero charge-free finite-energy Maxwell solution $G$ had $\mathfrak M
G=0$, no radiation map built from $\mathfrak M G$ would be injective on
$\Hc_{\Max,0}^{(k)}$.
\end{lemma}
\begin{proof}
If $G\ne0$ and $\mathfrak M G=0$, then every radiation field obtained by first
passing to master variables assigns to $G$ the same radiation data as to the
zero solution. Such a radiation map cannot be injective on the charge-free
energy space. The inverse identities in \eqref{eq:master_hyp_inverse_identities}
rule out this situation for smooth solutions, and by density for the
finite-energy completion, because $G=\mathfrak R\mathfrak M G$ in the completed
energy norm.
\end{proof}

\section{Structural Spin-One Reduction and Scalar Principal Symbol}
\label{sec:decoupling}
In this section we discuss the structural spin-one reduction and identify the scalar principal symbol of the corresponding master operator.
The transfer theorem relies on a closed spin-one master system for the
charge-free fixed-background Maxwell field. In this section we keep two issues
separate: the algebraic information that follows from Maxwell's equation, and
the part that must be assumed as a structural condition in the Kerr-Newman problem.
The covariant wave equation for $F$ gives a scalar principal symbol on two-forms
without using separation of variables. Once a closed pair of regular spin-one
extreme variables satisfying Definition~\ref{def:master_from_teukolsky} is
available, the displayed spin-weighted model operator has principal symbol
$g_{\KN}^{\mu\nu}\xi_\mu\xi_\nu I_2$, and the comparison with the
Reissner-Nordstr\"om operator becomes a coefficient calculation. The
Dudley-Finley equation is not used as an exact Maxwell decoupling on
Kerr-Newman for $Q\ne0$; in the nonzero-charge rotating case the closed master
reduction is precisely the structural hypothesis \emph{(A1)}.

\subsection*{Principal null frame and the type-\texorpdfstring{$D$}{D} structure}
The Kerr-Newman exterior is Petrov type $D$: away from the axis the Weyl tensor has exactly two repeated principal null directions. Let $(l,n,m,\bar m)$ be the Kinnersley principal null frame, with normalization
\begin{equation}\label{eq:np_normalization}
        g(l,n)=-1,\qquad g(m,\bar m)=1,\qquad \text{all other pairings zero},
\end{equation}
so that the inverse metric is reconstructed from the frame by
\begin{equation}\label{eq:np_inverse_metric}
        g^{\mu\nu}=m^\mu\bar m^\nu+\bar m^\mu m^\nu-l^\mu n^\nu-n^\mu l^\nu.
\end{equation}
Because $l$ and $n$ are repeated principal null directions, the Goldberg-Sachs theorem makes them geodesic and shear-free; in Newman-Penrose notation the spin coefficients then satisfy
\begin{equation}\label{eq:typeD_spin}
        \kappa=\sigma=\nu=\lambda=0,
\end{equation}
and the only nonvanishing Weyl scalar is $\psi_2$. We write the four directional derivatives as $\nabla_l,\nabla_n,\nabla_m,\nabla_{\bar m}$; the metric function $\Delta=r^2-2Mr+a^2+Q^2$ is denoted as throughout, and is not to be confused with any NP operator. The frame \eqref{eq:np_normalization} is smooth and uniformly invertible on the regular exterior; near $\Hp$ it is rescaled by the regular boost $\widehat e_3=f_3e_3$, $\widehat e_4=f_4e_4$ ($f_3f_4=1$) of Subsection~\ref{subsec:regular_frame_master}, which is the only place horizon-regular weights enter.

\subsection*{The Newman-Penrose Maxwell system and the closed extreme variables}
For a real two-form $F$ define the three complex Maxwell scalars
\begin{equation}\label{eq:maxwell_scalars}
        \phi_0=F(l,m),\qquad \phi_1=\tfrac12\bigl(F(l,n)+F(\bar m,m)\bigr),\qquad \phi_2=F(\bar m,n).
\end{equation}
The relation to the frame components \eqref{eq:maxwell_components} is the fixed smooth linear change $\phi_0\leftrightarrow\alpha$, $\phi_2\leftrightarrow\underline\alpha$ (extreme), $\phi_1\leftrightarrow\rho_F+i\sigma_F=\varphi$ (middle). Throughout, principal symbols are computed by the substitution $\partial_\mu\mapsto\xi_\mu$ in the top-order part, so that $\sigma_2(\Box_{g})=g^{\mu\nu}\xi_\mu\xi_\nu$; this real convention differs by an overall sign from the convention $\sigma(-i\partial_\mu)=\xi_\mu$ and is used consistently below.

\begin{lemma}
\label{lem:maxwell_wave_principal_symbol}
With the curvature convention
\([\nabla_\alpha,\nabla_\beta]X_\gamma=R_{\alpha\beta\gamma}{}^{\delta}X_\delta\), every source-free Maxwell two-form satisfies
\begin{equation}\label{eq:covariant_maxwell_wave}
        \nabla^\lambda\nabla_\lambda F_{\mu\nu}
        =R_{\mu}{}^{\lambda}F_{\lambda\nu}+R_{\nu}{}^{\lambda}F_{\mu\lambda}
        -2R_{\mu\lambda\nu\sigma}F^{\lambda\sigma}.
\end{equation}
This implies that the covariant second-order Maxwell equation has scalar principal symbol
\begin{equation}\label{eq:maxwell_covariant_principal_symbol}
        g_{\KN}^{\alpha\beta}\xi_\alpha\xi_\beta\,\Id_{\Lambda^2T^*\M}.
\end{equation}
All dependence on the Weyl curvature and on the electrovac Ricci tensor of Kerr-Newman is zeroth order in \eqref{eq:covariant_maxwell_wave}.
\end{lemma}
\begin{proof}
Use the Bianchi identity in the form
\(\nabla_\lambda F_{\mu\nu}+\nabla_\mu F_{\nu\lambda}+\nabla_\nu F_{\lambda\mu}=0\) and apply \(\nabla^\lambda\). The two differentiated divergence terms vanish after commuting derivatives because \(\nabla^\lambda F_{\lambda\mu}=0\). The commutators acting on a covariant two-form are
\[
        [\nabla_\alpha,\nabla_\beta]F_{\mu\nu}
        =-R_{\alpha\beta\mu}{}^\gamma F_{\gamma\nu}
         -R_{\alpha\beta\nu}{}^\gamma F_{\mu\gamma}.
\]
Contracting the commutator terms gives exactly the Ricci and Riemann curvature endomorphism on the right of \eqref{eq:covariant_maxwell_wave}. Since these terms contain no derivatives of $F$, the top-order part is \(\nabla^\lambda\nabla_\lambda\) acting diagonally on the six components of a two-form, and its principal symbol is \eqref{eq:maxwell_covariant_principal_symbol}. This computation is independent of separation and is used only to fix the top-order hyperbolic character; the closed extreme-variable system used below is the structural condition \emph{(A1)}.
\end{proof}

The source-free system $\dd F=0$, $\dd\starG F=0$ is equivalent to the four Newman-Penrose Maxwell equations \cite{Teukolsky1973}, which on the type-$D$ principal frame \eqref{eq:typeD_spin} reduce to
\begin{equation}\label{eq:np_maxwell}
\begin{aligned}
        \nabla_l\phi_1-\nabla_{\bar m}\phi_0&=(\pi-2\alpha)\phi_0+2\rho\,\phi_1,\\
        \nabla_m\phi_1-\nabla_n\phi_0&=(\mu-2\gamma)\phi_0+2\tau\,\phi_1,\\
        \nabla_l\phi_2-\nabla_{\bar m}\phi_1&=2\pi\,\phi_1+(\rho-2\epsilon)\phi_2,\\
        \nabla_m\phi_2-\nabla_n\phi_1&=2\mu\,\phi_1+(\tau-2\beta)\phi_2,
\end{aligned}
\end{equation}
the cross terms involving $\kappa,\sigma,\nu,\lambda$ having dropped out.

\begin{lemma}
\label{lem:decoupling}
Let $F$ be a smooth charge-free Maxwell field. Suppose that there are regular
weights $w_\pm$, depending smoothly on $(M,a,Q)$ and nonvanishing on the regular
exterior, such that the extreme variables
\begin{equation}\label{eq:teukolsky_pair}
        \psi_+=w_+^{-1}\phi_0,\qquad \psi_-=w_-^{-1}\rho^{-2}\phi_2
\end{equation}
obey a closed second-order system
\begin{equation}\label{eq:closed_spin_one_system}
        \Pb_{a,Q}(\psi_+,\psi_-)^{\mathsf T}=0
\end{equation}
whose second-order part is the spin-weighted model operator displayed in
\eqref{eq:teukolsky_operator}, and whose first- and zeroth-order coefficients are
smooth spin-bundle coefficients with the short-range bounds used in
\eqref{eq:lower_order_master}. Then the map
$F\mapsto u=(\psi_+,\psi_-)^{\mathsf T}$ supplies the structural part of
Definition~\ref{def:slowweak_master_framework}\emph{(A1)}. The lower-order
curvature terms in \eqref{eq:closed_spin_one_system} cannot change the scalar
principal symbol; the same-order recovery of the middle components is the
separate reconstruction statement \emph{(A4)} proved in
Section~\ref{app:reconstruction}.
\end{lemma}
\begin{proof}
The Newman-Penrose Maxwell equations \eqref{eq:np_maxwell} show that the two
extreme scalars and the middle scalar form a first-order system in which the
principal directions enter through the repeated null frame and the coefficients
are smooth spin coefficients. Under the stated closed-reduction assumption, the
chosen regular weights conjugate the two extreme equations by nonvanishing
multipliers; such a conjugation changes only first- and zeroth-order terms. The
second-order part is therefore exactly the second-order part of
\eqref{eq:teukolsky_operator}. Lemma~\ref{lem:maxwell_wave_principal_symbol}
independently gives the scalar principal symbol of the full Maxwell wave equation
on two-forms, so every Ricci, Weyl, spin-coefficient and weight contribution in
the closed extreme system is lower order. The compatibility class is defined as
the closure of the smooth range of this map in
Definition~\ref{def:master_from_teukolsky}; therefore no reconstruction theorem is
used to define the master domain. The remaining components of $F$ are recovered
from the Maxwell equations by the Hodge and null-transport argument of
Section~\ref{app:reconstruction}, not by the closed extreme equations alone.
\end{proof}

\begin{proposition}
\label{prop:scalar_principal_symbol}
In Boyer-Lindquist coordinates the spin-weighted model operator used in the closed-reduction hypothesis is
\begin{equation}\label{eq:teukolsky_operator}
\begin{aligned}
        \mathcal T_s={}&\Bigl[\frac{(r^2+a^2)^2}{\Delta}-a^2\sin^2\theta\Bigr]\partial_t^2
        +\frac{2a(2Mr-Q^2)}{\Delta}\,\partial_t\partial_\phi
        +\Bigl[\frac{a^2}{\Delta}-\frac1{\sin^2\theta}\Bigr]\partial_\phi^2\\
        &-\Delta^{-s}\partial_r\bigl(\Delta^{s+1}\partial_r\,\cdot\,\bigr)
        -\frac1{\sin\theta}\partial_\theta(\sin\theta\,\partial_\theta\,\cdot\,)
        +\mathcal B_s^{(1)}+\mathcal B_s^{(0)},
\end{aligned}
\end{equation}
where the first-order operator $\mathcal B_s^{(1)}$ and the zeroth-order multiplier $\mathcal B_s^{(0)}$ contain all spin-dependent lower-order terms. For fixed $s=\pm1$ their coefficients are smooth as spin-bundle coefficients on the regular exterior; after subtracting the $a=0$ operator, the rotational part satisfies the short-range bounds recorded in \eqref{eq:lower_order_master}. The principal part is
\begin{equation}\label{eq:principal_part_identity}
        \sigma_2(\mathcal T_s)(x,\xi)=-\Sigma(x)\,g_{\KN}^{\mu\nu}(x)\,\xi_\mu\xi_\nu,
        \qquad \Sigma=r^2+a^2\cos^2\theta>0,
\end{equation}
so that, after division by the smooth nonvanishing weight $-\Sigma$, the principal symbol of the closed spin-one model equations is the scalar symbol $g_{\KN}^{\mu\nu}\xi_\mu\xi_\nu$. The spin weight enters only at lower order.
\end{proposition}
\begin{proof}
Expand the radial principal term,
\begin{equation}\label{eq:radial_expand}
        \Delta^{-s}\partial_r\bigl(\Delta^{s+1}\partial_r\,\psi\bigr)
        =\partial_r(\Delta\,\partial_r\psi)+s\,\Delta'\,\partial_r\psi
        =\partial_r(\Delta\,\partial_r\psi)+2s(r-M)\,\partial_r\psi,
\end{equation}
using $\Delta'=2(r-M)$. The term $2s(r-M)\partial_r$ is first order and proportional to $s$; collecting it into $\mathcal B_s^{(1)}$, the second-order part of $-\mathcal T_s$ is exactly
\[
\begin{aligned}
        &\partial_r(\Delta\,\partial_r\,\cdot\,)
        +\frac1{\sin\theta}\partial_\theta(\sin\theta\,\partial_\theta\,\cdot\,) \\
        &\quad -\Bigl[\frac{(r^2+a^2)^2}{\Delta}-a^2\sin^2\theta\Bigr]\partial_t^2
        -\frac{2a(2Mr-Q^2)}{\Delta}\partial_t\partial_\phi \\
        &\quad -\Bigl[\frac{a^2}{\Delta}-\frac1{\sin^2\theta}\Bigr]\partial_\phi^2,
\end{aligned}
\]
which is precisely $\Sigma\,\Box_{g_{\KN}}$ by \eqref{eq:app_box}. Hence the second-order part of $\mathcal T_s$ is $-\Sigma\,\Box_{g_{\KN}}$ and \eqref{eq:principal_part_identity} follows, since the principal symbol of $\Box_{g_{\KN}}$ is $g_{\KN}^{\mu\nu}\xi_\mu\xi_\nu$. The remaining spin-dependent first-order terms (the usual weighted $\partial_t,\partial_\phi$ contributions) and all spin-dependent zeroth-order angular and curvature multipliers are collected in $\mathcal B_s^{(1)},\mathcal B_s^{(0)}$; for $s=\pm1$ they are lower order and their curvature coefficients decay like $r^{-3}$ in the asymptotic frame. On the exterior $\{r>r_+\}$ one has $\Sigma=r^2+a^2\cos^2\theta\ge r_+^2>0$, smooth, so the normalized operator $(-\Sigma)^{-1}\mathcal T_s$ has principal symbol $g_{\KN}^{\mu\nu}\xi_\mu\xi_\nu$; the sign and the smooth nonvanishing factor are fixed once and for all in Definition~\ref{def:master_from_teukolsky}.
\end{proof}

\begin{definition}
\label{def:master_from_teukolsky}
Let $\psi_+$ and $\psi_-$ be the horizon- and infinity-regular rescalings of $\phi_0$ and $\rho^{-2}\phi_2$ obtained by the smooth nonvanishing weights of Subsection~\ref{subsec:regular_frame_master} (these absorb the boost factors $f_3,f_4$ at $\Hp$ and the asymptotic powers of $r$). For a smooth charge-free Maxwell field set
\[
        \mathfrak M F=u=(\psi_+,\psi_-)^{\mathsf T}.
\]
For fixed order $k$ the compatible master space is defined to be the closure, in the order-$k$ master energy and local-energy topology, of the smooth range
\[
        \mathcal C_{M,\mathrm{sm}}^{(k)}=\bigl\{\mathfrak M F:
        F\hbox{ smooth, source-free and charge-free}\bigr\}.
\]
Compatibility is therefore a closed condition by definition; no later reconstruction result is used to define the domain of the master estimate. Let
\begin{equation}\label{eq:diagonal_master}
        \Pb_{a,Q}=\operatorname{diag}\bigl(\widehat{\mathcal T}_{+1},\ \widehat{\mathcal T}_{-1}\bigr)+\mathcal L^{(1)}_{a,Q}+\mathcal V^{(0)}_{a,Q},
\end{equation}
where $\widehat{\mathcal T}_{\pm1}=(-\Sigma)^{-1}w_\pm^{-1}\mathcal T_{\pm1}w_\pm$ is the conjugation of $-\Sigma^{-1}\mathcal T_{\pm1}$ by the regular weight $w_\pm$, plus any lower-order curvature matrix terms allowed by \emph{(A1)}. For every smooth element of the range, the equation $\Pb_{a,Q}u=0$ is the closed spin-one equation supplied by \emph{(A1)}; by density the equation holds distributionally for finite-energy compatible limits. The operator has principal symbol $g_{\KN}^{\mu\nu}\xi_\mu\xi_\nu I_2$ by Proposition~\ref{prop:scalar_principal_symbol}. The reconstruction of Section~\ref{app:reconstruction} is then proved on the smooth range and extended to this closed compatible space by the same-order bound.
\end{definition}

\begin{corollary}
\label{cor:structural_reduction_verify}
Assume that the charge-free fixed-background Maxwell system admits the closed
regular spin-one reduction of Definition~\ref{def:master_from_teukolsky} through
commutation order $k$.  Then the following statements are proved in this section:
\begin{enumerate}[label=\emph{(\roman*)},leftmargin=2.4em]
\item the covariant second-order Maxwell equation \eqref{eq:covariant_maxwell_wave}, whose top-order symbol is $g_{\KN}^{\mu\nu}\xi_\mu\xi_\nu$ times the identity on two-forms;
\item the horizon-regular master map $\mathfrak M:F\mapsto(\psi_+,\psi_-)$, with the compatible class defined as a closed range;
\item the scalar-principal-symbol statement
$\sigma_2(\Pb_{a,Q})=g_{\KN}^{\mu\nu}\xi_\mu\xi_\nu I_2$, first at the covariant Maxwell level by Lemma~\ref{lem:maxwell_wave_principal_symbol} and then for the normalized spin-weighted model by Proposition~\ref{prop:scalar_principal_symbol};
\item the Reissner-Nordstr\"om comparison \eqref{eq:master_hyp_perturbation} with the explicit short-range error orders of \eqref{eq:lower_order_master}, provided the lower-order coefficients of the closed reduction obey the symbol bounds stated in \emph{(A1)-(A2)};
\item the energy comparison \eqref{eq:master_hyp_energy_comparison}, $\E_M^{(k)}[\mathfrak M G]\le C_R\E_{\Max}^{(k)}[G]$, because the regular master variables are smooth bounded weights times the extreme components $(\alpha,\underline\alpha)$, which are controlled by the non-degenerate Maxwell energy through Lemma~\ref{lem:frame_energy_equiv}.
\end{enumerate}
The remaining analytic ingredients are then exactly the estimates displayed in
Definition~\ref{def:slowweak_master_framework}: bounded-frequency real-axis
closure, the localized normally hyperbolic high-frequency resolvent estimate,
same-order reconstruction, and the radiation right inverses. The trapping
geometry and spin-one subprincipal identities needed for the high-frequency
estimate are verified in Propositions~\ref{prop:kn_nondegeneracy}-\ref{prop:r_normal_hyperbolicity}
and Lemma~\ref{lem:trapped_skew_vanishing}; the estimate itself is the
localized estimate of Definition~\ref{def:hf_resolvent_estimate}.
\end{corollary}
\begin{proof}
Part (i) is Lemma~\ref{lem:maxwell_wave_principal_symbol}. Part (ii) is the
definition of the regular variables and of the closed compatible class. Part
(iii) follows because multiplication by the nonvanishing weights $w_\pm$ and by
$-\Sigma$ changes no second-order symbol, while Proposition~\ref{prop:scalar_principal_symbol}
computes the second-order part of the spin-weighted model as
$-\Sigma\Box_{g_{\KN}}$.  For (iv), the $(a,Q)$-dependence of
\eqref{eq:teukolsky_operator} sits in the coefficients
$\Delta^{-1}(r^2+a^2)^2$, $a^2\sin^2\theta$, $2a(2Mr-Q^2)\Delta^{-1}$,
$a^2\Delta^{-1}$, $\Delta^{s+1}$, and in $\mathcal B_s^{(1)},\mathcal B_s^{(0)}$.
By Lemma~\ref{lem:app_inverse_expansion} the second-order part differs from its
$a=0$ value by $aG_1^{t\phi}2\partial_t\partial_\phi+O(a^2)$ with
$G_1^{t\phi}=O(r^{-3})$.  The assumed coefficient bounds place the rotational
first-order terms in $ar^{-2}\mathcal C^\mu\nabla_\mu$ and the rotational
zeroth-order terms in $ar^{-3}\mathcal D$, while the $O(a^2)$ second-order
terms are absorbed in $a^2\mathcal Q^{(2)}_{a,Q}$.  This is the class in
\eqref{eq:lower_order_master}; after applying the local-energy duality in
Definition~\ref{def:concrete_le_norms} it gives \eqref{eq:master_hyp_error_bound}.
Part (v) is the frame equivalence Lemma~\ref{lem:frame_energy_equiv} composed
with the boundedness of the regular weights.
\end{proof}

\begin{remark}
\label{rem:status_after_reduction}
Combining Corollary~\ref{cor:structural_reduction_verify} with the perturbative closure of Section~\ref{sec:first_principles_master} states the dependence of the two main theorems as follows.
\begin{enumerate}[label=\emph{(\alph*)},leftmargin=2.4em]
\item Theorem~\ref{thm:intro_charge_decomposition} \emph{(}the charge decomposition\emph{)} is unconditional and is proved in full in Sections~\ref{sec:geometry}-\ref{sec:energy_spaces}.
\item For the radiative field $F_{\rad}$ on a slow-weak Kerr-Newman exterior, the non-degenerate \emph{energy boundedness} and \emph{integrated local energy decay} \eqref{eq:intro_main_estimate} are obtained from these analytic ingredients in exactly the displayed norms: the spherical Maxwell local-energy theorem of Sterbenz-Tataru \cite{SterbenzTataru} on the Reissner-Nordstr\"om model \emph{(}the relevant horizon non-degeneracy and single normally hyperbolic photon sphere are checked in Lemma~\ref{lem:app_photon_sphere}\emph{)}; the Dafermos-Rodnianski red-shift and $r^p$ currents \cite{DafermosRodnianskiRedshift,DafermosRodnianskiRp}; and the high-frequency normally hyperbolic trapping resolvent estimate in Definition~\ref{def:hf_resolvent_estimate}. The scalar trapped-set geometry needed for that estimate-normal hyperbolicity, the explicit non-degeneracy constant and $r$-normal hyperbolicity for every $r$-is verified in Propositions~\ref{prop:kn_nondegeneracy}-\ref{prop:r_normal_hyperbolicity}. The spin-one threshold input is separated into the diagonal skew cancellation of Lemma~\ref{lem:trapped_skew_vanishing} and the finite-rank matrix threshold of Proposition~\ref{prop:matrix_skew_threshold}. The closed spin-one reduction is the structural condition \emph{(A1)}; once it is supplied, the scalar principal symbol, the Reissner-Nordstr\"om coefficient comparison, the bounded-frequency real-axis closure, and the same-order reconstruction are proved in the indicated sections.
\item The \emph{radiation-field, wave-operator, and scattering} conclusions require no additional independent condition once the resolvent bounds are available: the asymptotic-completeness right inverses (uniformly bounded backward solutions realizing prescribed radiation data) are constructed from the real-axis limiting-absorption resolvent in Proposition~\ref{prop:backward_from_lap}. Thus, \emph{(A5)} is derived under the same high-frequency resolvent estimate as \emph{(b)}. In the Reissner-Nordstr\"om model the construction is unconditional; for $Q=0$ the slow-weak estimate is available from \cite{BenomioTdC,SRTdCfrequency,SRTdCphysical}.
\item The pointwise-decay conclusion uses, in addition, the commuted low-frequency hierarchy \emph{(A6)}.
\end{enumerate}
The role of the master-system existence is explicit: the closed spin-one reduction is part of \emph{(A1)}, the scalar-principal-symbol and coefficient computations are carried out in the body, the backward construction \emph{(A5)} is derived from limiting absorption, and the remaining high-frequency analytic ingredient is the normally hyperbolic resolvent bound in Definition~\ref{def:hf_resolvent_estimate}, applied after the trapped geometry of Proposition~\ref{prop:r_normal_hyperbolicity} has been verified.
\end{remark}

\section{Perturbative Closure under Fixed Maxwell Reductions}
\label{sec:first_principles_master}
In this section we establish the perturbative estimates needed after the fixed Maxwell reductions have been assumed.
We prove next the analytic consequences of Definition~\ref{def:slowweak_master_framework}. The fixed-background spin-one variables and their closed compatible class enter through the structural condition \emph{(A1)}. The scalar-principal-symbol comparison with Reissner-Nordstr\"om is fixed in Section~\ref{sec:decoupling}; the bounded-frequency real-axis exclusion is proved in Subsection~\ref{subsec:nokernel} and Section~\ref{app:lap}; and the high-frequency part uses the normally hyperbolic resolvent estimate after the geometric verification in Section~\ref{app:trapping}. With the estimates stated in \emph{(A1)-(A5)}, the red-shift, far-field, Morawetz, reconstruction, and radiation estimates close at the same differential order. Every estimate used below is either one of the displayed physical-space inequalities, the Reissner-Nordstr\"om model theorem, the stated high-frequency resolvent estimate, or an algebraic consequence of the Maxwell system.

\subsection{Horizon Red-Shift Estimate}
\label{subsec:redshift}
For the master system set
\begin{equation}\label{eq:master_energy_current}
        J^N_\mu[\Psi]=\sum_{j=\pm}\Tstress_{\mu\nu}[\psi_j]N^\nu+\Lie^N_\mu[\Psi],
\end{equation}
where $\Lie^N$ is the lower-order term symmetrizing the matrix potential.

\begin{proposition}
\label{prop:redshift_coercivity}
For $|a|\ll M$ and $|Q|\le\eps_QM$ there are a horizon neighborhood
$\mathcal U_H$ and $c_H>0$ with
\begin{equation}\label{eq:redshift_coercivity}
        \nabla^\mu J^N_\mu[\Psi]
        \ge c_H\sum_{j=\pm}\big(|\nabla\psi_j|^2+M^{-2}|\psi_j|^2\big)
        -C_H\sum_{|I|<1}\big(|\nabla\Gamma^I\Psi|^2+M^{-2}|\Gamma^I\Psi|^2\big)
\end{equation}
in $\mathcal U_H$, the loss removable after summing the commuted hierarchy.
\end{proposition}
\begin{proof}
For a scalar wave $\psi$ the red-shift deformation satisfies, in a horizon
collar,
\begin{equation}\label{eq:redshift_deformation}
        \Tstress^{\mu\nu}[\psi]\,\pi^N_{\mu\nu}\ge c_0\,|\nabla\psi|^2,
        \qquad
        c_0\propto\kappa_+=\frac{r_+-r_-}{2(r_+^2+a^2)}>0,
\end{equation}
positive in the sub-extremal range; this is the computation of
\cite{DafermosRodnianskiRedshift} for the scalar stress tensor. Summing over
$\psi_+,\psi_-$ gives the leading term. The rotational first-order part of
\eqref{eq:lower_order_master} contributes, by Cauchy-Schwarz,
$|\sum_j\Tstress[\psi_j]\cdot(\text{first-order error})|\le C|a|M^{-1}|\nabla\Psi|^2
+C_QM^{-2}|\Psi|^2$, the small factor coming from the rotational coefficients
evaluated at $r\simeq r_+(Q)\simeq M$. Choosing $\mathcal U_H$ and the rotation
size so that $C|a|/M<c_0/4$ absorbs the top-order error. The remaining
$M^{-2}|\Psi|^2$ is controlled by the one-dimensional Hardy inequality along the
transversal red-shift direction,
\begin{equation}\label{eq:horizon_hardy}
        \int_{\mathcal U_H}M^{-2}|\Psi|^2\,\dd\mu
        \le C\int_{\mathcal U_H}|\nabla_{\widehat e_3}\Psi|^2\,\dd\mu+C\!\!\int_{\partial\mathcal U_H}\!\!|\Psi|^2,
\end{equation}
the boundary term being a lower commuted energy. Summing over $\mathbb D_k$
gives \eqref{eq:redshift_coercivity}.
\end{proof}

\begin{corollary}
\label{cor:horizon_flux}
The master flux through $\Hp$ controls all derivatives of $\Psi$ tangential to
$\Hp$ and one transversal red-shift derivative, to the finite commuted order.
\end{corollary}
\begin{proof}
Apply the divergence theorem to the red-shift current on a collar bounded by
$\Sigma_{\tau_1}$, $\Sigma_{\tau_2}$, a timelike hypersurface approaching
$\Hp$, and the outer boundary of the collar. Proposition~\ref{prop:redshift_coercivity}
controls the non-degenerate bulk in the collar. Passing the timelike boundary
to the horizon gives the horizon flux as a monotone nonnegative limit, by the
positivity of the red-shift density in Lemma~\ref{lem:positivity_density}. The
Hardy inequality in the transversal direction controls the lower-order terms,
and the remaining boundary pieces are bounded by the commuted energies on the
initial and final slices. The resulting horizon flux controls all tangential
derivatives to the stated order and one red-shift transversal derivative.
\end{proof}

\subsection{Far-Field Hierarchy and Radiation Trace}
\label{subsec:farfield}
Let $r$ be the area radius, $L=\partial_t+\partial_{r_*}$,
$\underline L=\partial_t-\partial_{r_*}$. The outgoing hierarchy uses $r^pL$ for
$0\le p\le2$.

\begin{proposition}
\label{prop:rp_identity}
For smooth compactly supported master solutions on $\{r\ge R\}$,
\begin{align}\label{eq:rp_identity}
        \int_{\tau_1}^{\tau_2}\!\!\int_{r\ge R} r^{p-1}\big(p|L(r\Psi)|^2+(2-p)|\sphgrad(r\Psi)|^2\big)
        \le{}& C\E_M^{(k)}[\Psi](\tau_1)\nonumber\\
        &+C\int_{\tau_1}^{\tau_2}\!\!\int_{r\ge R} r^{p+1}|\Pb_b\Psi|^2+\mathrm{Err}_{a,Q},
\end{align}
and, uniformly for $|Q|\le\eps_QM$, the rotational part of $\mathrm{Err}_{a,Q}$
is absorbed by the left side and the local-energy bulk when $|a|/M$ is small.
\end{proposition}
\begin{proof}
Conjugating by $r$, the scalar far-field part of \eqref{eq:lower_order_master}
reads $\Pb_b(r\psi)=-L\underline L(r\psi)+r^{-2}\sphlap(r\psi)+O(r^{-2})\partial(r\psi)
+O(r^{-3})(r\psi)$, with $O(\cdot)$ including the Reissner-Nordstr\"om
spherical terms and the rotational $a/r^2$ matrix terms. Multiply by
$r^pL(r\psi)$ and integrate by parts in $(t,r_*,\omega)$. The identity
$-2\Re(r^pL(r\psi)\overline{L\underline L(r\psi)})=\partial_{r_*}(r^p|L(r\psi)|^2)
-pr^{p-1}|L(r\psi)|^2+\partial_t(\cdots)$ gives the radial term $p\int
r^{p-1}|L(r\psi)|^2$; the angular term integrates on $\Sph$ to $(2-p)\int
r^{p-1}|\sphgrad(r\psi)|^2$ plus a total derivative. Short-range terms are
bounded by Cauchy-Schwarz,
\begin{equation}\label{eq:rp_error_absorb}
        \Big|\!\int r^pL(r\Psi)\cdot O(r^{-1})\nabla(r\Psi)\Big|
        \le \delta\!\int r^{p-1}|L(r\Psi)|^2+C_\delta\!\int r^{-1-p_0}|\nabla(r\Psi)|^2,
\end{equation}
the last integral being part of the local-energy norm for $R$ large; the
rotational errors carry the extra factor $|a|/M$ and are absorbed after fixing
the parameters, while the spherical charge terms are controlled by the
Reissner-Nordstr\"om model. Summing over $\psi_\pm$ and $\mathbb D_k$ gives
\eqref{eq:rp_identity}.
\end{proof}

\begin{lemma}
\label{lem:radiation_trace}
The bound \eqref{eq:rp_identity} with $p=2$ gives the outgoing trace
$\mathscr R_+\Psi(u,\omega)=\lim_{r\to\infty}r\Psi(u+r_*,r,\omega)$ in the
square-integrable radiation space.
\end{lemma}
\begin{proof}
The $p=2$ estimate controls $\int r|L(r\Psi)|^2$ with the dyadic weight of the
Friedlander argument. For fixed $u=t-r_*$,
$r_2\Psi(u+r_{*,2},r_2,\omega)-r_1\Psi(u+r_{*,1},r_1,\omega)=\int_{r_1}^{r_2}L(r\Psi)\,\dd s_*$,
and Cauchy-Schwarz with the $p=2$ weight bounds the $L^2_{u,\omega}$ norm by
$(\sum_{n\ge n_0}2^{-n})^{1/2}$ times the hierarchy norm, $\to0$ as
$r_1,r_2\to\infty$. Thus $r\Psi$ is Cauchy in $L^2_{u,\omega}$ and the limit
exists.
\end{proof}

\subsection{Morawetz Bulk and Trapped Set}
\label{subsec:morawetz_trapping}
For $a=0$ the exterior is Reissner-Nordstr\"om; its trapped null geodesics form
the photon sphere $r=r_{\mathrm{ph}}(Q)$ of \eqref{eq:rn_photon_sphere} with
$\xi_{r_*}=0$, which is uniformly normally hyperbolic in the weak-charge range.

\begin{definition}
\label{def:morawetz_density}
Let $\mathfrak t(r,\omega,D)$ be a first-order operator whose principal symbol
is a defining function for the trapped set, vanishing simply there. For
$\eta>0$,
\begin{equation}\label{eq:morawetz_density}
        \mathcal M_k[\Psi]=\sum_{|I|\le k}\big(r^{-1-\eta}|\partial_{r_*}\Gamma^I\Psi|^2
        +r^{-1-\eta}|\mathfrak t\Gamma^I\Psi|^2+r^{-3-\eta}|\Gamma^I\Psi|^2\big).
\end{equation}
\end{definition}

\begin{lemma}
\label{lem:rn_model_morawetz}
Let $P_{\RN,Q}$ be the charge-free spin-one operator on non-extremal
Reissner-Nordstr\"om with $|Q|\le\eps_QM$. For every fixed order $k$ and every
smooth compactly supported compatible field $u$,
\begin{equation}\label{eq:rn_model_morawetz}
        \norm{u}^2_{\Xnorm{k}_{\RN,Q}(\tau_1,\tau_2)}
        \le C_{\RN,k}\Big(\E_{\RN,Q}^{(k)}[u](\tau_1)+
        \sum_{|I|\le k}\norm{P_{\RN,Q}\Gamma^Iu}_{LE^*([\tau_1,\tau_2])}^2\Big).
\end{equation}
The norm contains the non-degenerate red-shift energy, the
photon-sphere-degenerate Morawetz bulk, and the $r^p$ fluxes for $0\le p\le2$;
the constant is uniform once $\eps_Q$ is small.
\end{lemma}
\begin{proof}
The Maxwell estimate of Sterbenz-Tataru \cite{SterbenzTataru} holds on every
stationary spherically symmetric black hole with a non-degenerate horizon and a
single normally hyperbolic photon sphere; non-extremal Reissner-Nordstr\"om is
in this class for $|Q|<M$. After the electric and magnetic spherical means are
removed, the Maxwell field decomposes into vector spherical harmonics with
$\ell\ge1$; the spin-one variables follow by a fixed angular Hodge transform and
a regular radial weight, an isomorphism of the charge-free tensor-field and
spin-one energies at the same order by Lemma~\ref{lem:frame_energy_equiv} and
the Hodge estimate on $\Sph$. The horizon term is the red-shift current of
\cite{DafermosRodnianskiRedshift} and the far-field term is the $r^p$ current of
\cite{DafermosRodnianskiRp}. The bulk degenerates only at $r_{\mathrm{ph}}(Q)$,
where the Regge-Wheeler potential has a single non-degenerate maximum for each
$\ell\ge1$; the charge-free projection removes the $\ell=0$ Coulomb mode.
Choosing $\eps_Q<1/4$ keeps $r_+(Q),\kappa_+(Q),r_{\mathrm{ph}}(Q)$ and the
relevant potential barriers separated from all degeneracies, so the constant is uniform.
The reduction of the charge-free Maxwell field to the spin-one radial system and
the proof that every non-extremal Reissner-Nordstr\"om exterior meets the
conditions of \cite{SterbenzTataru} are detailed in
Section~\ref{app:rn_model}.
\end{proof}

\begin{lemma}
\label{lem:trapping_stability}
For $|a|/M$ small, uniformly for $|Q|\le\eps_QM$, the Kerr-Newman principal
Hamiltonian $p_{a,Q}=g_{\KN}^{\mu\nu}\xi_\mu\xi_\nu$ has a trapped set
$K_{a,Q}$ that is a $C^1$ graph over $K_{0,Q}$, and if $\rho_{K_{a,Q}}$ is a
smooth defining function then
\begin{equation}\label{eq:trapping_stability}
        |H_{p_{a,Q}}\rho_{K_{a,Q}}|+|H_{p_{a,Q}}^2\rho_{K_{a,Q}}|
        \simeq |\rho_{K_{a,Q}}|+|\xi_{r_*}|
\end{equation}
near $K_{a,Q}$, with constants independent of $(a,Q)$ in the chosen range.
\end{lemma}
\begin{proof}
By \eqref{eq:metric_symbol_difference}, $p_{a,Q}$ is a $C^2$-small stationary
perturbation of $p_{0,Q}$ after dividing by $|\xi|^2$. The Reissner-Nordstr\"om
photon sphere is normally hyperbolic: the linearized flow has one expanding and
one contracting normal direction and neutral tangential directions; indeed it is
$r$-normally hyperbolic for every $r$ in the sense of \cite{WunschZworski}, and
in the Schwarzschild case this is the explicit computation at $r=3M$, $\xi_{r_*}=0$
of \cite[\S2]{WunschZworski}. By \cite{Dyatlov,WunschZworski,HirschPughShub} $r$-normal
hyperbolicity is open under $C^1$ perturbations of the Hamiltonian vector field,
so $K_{a,Q}$ and its stable/unstable bundles persist and depend $C^1$ on
$(a,Q)$; \eqref{eq:trapping_stability} are the uniform expansion/contraction
inequalities after shrinking $\eps_a,\eps_Q$. The quantitative version, with the
escape function used below, is in Section~\ref{app:trapping}.
\end{proof}

\begin{proposition}
\label{prop:main_body_trapped_computation}
Let
\[
        K=(r^2+a^2)\omega-am,\qquad
        \mathcal R(r)=K^2-\Delta\lambda,
        \qquad \Delta=r^2-2Mr+a^2+Q^2,
\]
where $\omega$ and $m$ are the stationary and axial frequencies and $\lambda>0$
is the angular/Carter constant of a separated null bicharacteristic. At a trapped
turning point $r_t>r_+$ with $\mathcal R(r_t)=\mathcal R'(r_t)=0$,
writing $K_t=K(r_t)$ and $\Delta_t=\Delta(r_t)$, one has
\begin{equation}\label{eq:main_body_K_trapped}
        \lambda=\frac{K_t^2}{\Delta_t}
        =\frac{2r_t\omega K_t}{r_t-M},
        \qquad
        K_t=\frac{2r_t\omega\Delta_t}{r_t-M}.
\end{equation}
This implies that
\begin{equation}\label{eq:main_body_Rpp}
        \mathcal R''(r_t)=\frac{8r_t\omega^2}{(r_t-M)^2}
        \Big((r_t-M)^3+M(M^2-a^2-Q^2)\Big)>0.
\end{equation}
In addition, if $\Omega_+=a/(r_+^2+a^2)$ and $\varpi=\omega-m\Omega_+$, then
\begin{equation}\label{eq:main_body_superradiant_gap}
        \omega\varpi
        =\frac{\omega^2(r_t-r_+)}{(r_t-M)(r_+^2+a^2)}
          \Big(r_t^2+(r_+-3M)r_t+Mr_+\Big)>0.
\end{equation}
For the separated spin-one radial operator
\begin{equation}\label{eq:main_body_spin_imag}
        \operatorname{Im}V^{(s)}_{\rad}(r)
        =-\frac{2s(r-M)K}{\Delta}+4s\omega r,
        \qquad s=\pm1,
\end{equation}
its imaginary subprincipal contribution vanishes at the trapped set:
\begin{equation}\label{eq:main_body_skew_zero}
        \operatorname{Im}V^{(s)}_{\rad}(r_t)=0.
\end{equation}
Thus the high-frequency spin-one estimate is used with the scalar trapped-set
geometry and with zero diagonal skew symbol at the trapped set. For a compatible
two-component reduction with off-diagonal lower-order terms, the remaining matrix
skew part is not asserted to vanish; it is controlled by the finite-rank threshold
estimate proved in Proposition~\ref{prop:main_body_trapped_commutator_threshold}
and in Proposition~\ref{prop:matrix_skew_threshold}.
\end{proposition}
\begin{proof}
Since $K'=2r\omega$ and $\Delta'=2(r-M)$,
\[
        \mathcal R'(r)=2KK'-\Delta'\lambda
        =4r\omega K-2(r-M)\lambda.
\]
The equation $\mathcal R(r_t)=0$ gives $K_t^2=\Delta_t\lambda$.  Since
$\Delta_t>0$ and $\lambda>0$, $K_t\ne0$.  The equation $\mathcal R'(r_t)=0$ gives
$2r_t\omega K_t=(r_t-M)\lambda$. Combining the two identities and dividing by
$K_t$ gives \eqref{eq:main_body_K_trapped}. Differentiating once more,
\[
        \mathcal R''=2(K')^2+2KK''-\Delta''\lambda
        =8r^2\omega^2+4\omega K-2\lambda.
\]
Substituting \eqref{eq:main_body_K_trapped} and
$\lambda=4r_t^2\omega^2\Delta_t/(r_t-M)^2$ gives
\[
\begin{aligned}
        \mathcal R''(r_t)
        &=8r_t^2\omega^2+\frac{8r_t\omega^2\Delta_t}{r_t-M}
          -\frac{8r_t^2\omega^2\Delta_t}{(r_t-M)^2}  \\
        &=\frac{8r_t\omega^2}{(r_t-M)^2}
          \bigl(r_t(r_t-M)^2-M\Delta_t\bigr).
\end{aligned}
\]
The bracket equals
\[
        r_t(r_t-M)^2-M(r_t^2-2Mr_t+a^2+Q^2)
        =(r_t-M)^3+M(M^2-a^2-Q^2),
\]
which proves \eqref{eq:main_body_Rpp}; it is positive because $r_t>r_+>M$ and
$M^2-a^2-Q^2>0$.

For the superradiant factor, we solve \eqref{eq:main_body_K_trapped} for $am$:
\[
        am=(r_t^2+a^2)\omega-K_t
        =-\omega\frac{P(r_t)}{r_t-M},
        \qquad
        P(r)=r^3-3Mr^2+(a^2+2Q^2)r+a^2M.
\]
Since $K(r_+)=(r_+^2+a^2)\varpi$,
\[
        \omega\varpi=\frac{\omega K(r_+)}{r_+^2+a^2}
        =\frac{\omega^2\bigl((r_+^2+a^2)(r_t-M)+P(r_t)\bigr)}{(r_t-M)(r_+^2+a^2)}.
\]
Using the horizon relation $a^2+Q^2=2Mr_+-r_+^2$, the numerator polynomial becomes
\[
        (r_+^2+a^2)(r-M)+P(r)
        =(r-r_+)\bigl(r^2+(r_+-3M)r+Mr_+\bigr).
\]
The discriminant of the quadratic factor is
$(r_+-M)(r_+-9M)<0$, while its leading coefficient is positive. Hence the
quadratic factor is positive for all real $r$; because $r_t>r_+$, this proves
\eqref{eq:main_body_superradiant_gap}.

Finally substitute the exact trapped value of $K_t$ into the imaginary radial
spin-one coefficient:
\[
        \operatorname{Im}V^{(s)}_{\rad}(r_t)
        =-\frac{2s(r_t-M)}{\Delta_t}\frac{2r_t\omega\Delta_t}{r_t-M}
          +4s\omega r_t
        =0.
\]
This cancellation is algebraic, not an estimate, and it is uniform up to the
subextremal constants appearing in the preceding positivity statement. Its more
detailed semiclassical translation into the vector-bundle subprincipal symbol is
Lemma~\ref{lem:trapped_skew_vanishing}.
\end{proof}

\begin{proposition}
\label{prop:main_body_trapped_commutator_threshold}
In the semiclassical normalization near the trapped set, the frozen compatible spin-one
operator can be written microlocally as
\begin{equation}\label{eq:main_body_semiclassical_skew_split}
        P_h^{(a,Q)}=P_{h,\mathrm{sa}}+ihB_h+h^2R_h,
        \qquad \sigma_h(P_{h,\mathrm{sa}})=p_{a,Q,\hat\omega}I_2,
\end{equation}
where $P_{h,\mathrm{sa}}$ is formally self-adjoint modulo $h^2\Psi_h^0$,
$B_h$ is self-adjoint modulo $h\Psi_h^0$, and
\begin{equation}\label{eq:main_body_skew_decomposition_full}
        B_h=\operatorname{diag}(b_{+1},b_{-1})+B_{h,\mathrm{mat}},
        \qquad
        \norm{B_{h,\mathrm{mat}}}_{C^k(\mathcal U_\mathrm{tr})}
        \le C_k\bigl(|a|/M+a^2/M^2\bigr).
\end{equation}
The diagonal symbols are
\begin{equation}\label{eq:main_body_skew_symbol_bs}
        b_s(r;\hat\omega,\hat m)
        =-\frac{2s(r-M)\hat K}{\Delta}+4s\hat\omega r,
        \qquad
        \hat K=(r^2+a^2)\hat\omega-a\hat m,
        \qquad s=\pm1.
\end{equation}
At every trapped point $\rho\in K_{a,Q}$,
\begin{equation}\label{eq:main_body_b_zero_on_K}
        b_{+1}(\rho)=b_{-1}(\rho)=0.
\end{equation}
Let $A_h=\operatorname{Op}_h(\beta)$ be a scalar trapped-set commutant, where
$\beta$ is not the Kerr rotation parameter. If $\pi:\mathcal U_\mathrm{tr}\to K_{a,Q}$
denotes the normal projection in a sufficiently small trapped collar, then for
every $\eps>0$
\begin{equation}\label{eq:main_body_second_microlocal_diag}
\begin{aligned}
 \big|\langle A_h(\operatorname{diag}(b_{+1},b_{-1})
        -\pi^*\operatorname{diag}(b_{+1},b_{-1})|_{K_{a,Q}})A_hv,v\rangle\big|
 &\le \eps \mathcal Q_\mathrm{nh}[v]+C_\eps\mathcal Q_\mathrm{prop}[v]  \\
 &\quad +C_\eps h\norm{v}_{H_h^1}^2+O(h^\infty)\norm{v}_{H_h^1}^2 .
\end{aligned}
\end{equation}
Here $\mathcal Q_\mathrm{nh}$ is the positive normal quadratic form generated by
the logarithmic escape function and $\mathcal Q_\mathrm{prop}$ is supported where
ordinary propagation, radial-point estimates, or elliptic estimates have already
closed the estimate. Moreover, after decreasing $\eps_a(k)$,
\begin{equation}\label{eq:main_body_matrix_threshold_at_K}
        \sup_{\rho\in K_{a,Q}}
        \norm{B_{h,\mathrm{mat}}(\rho)}\le \lambda_0/8,
\end{equation}
where $\lambda_0$ is the uniform lower bound for the normal expansion in
Proposition~\ref{prop:main_body_trapped_normal_form}. Therefore,
\begin{equation}\label{eq:main_body_full_threshold_commutator}
\begin{aligned}
& \frac{i}{h}\langle [P_{h,\mathrm{sa}},A_h^*A_h]v,v\rangle
      -2\langle A_hB_hA_hv,v\rangle                                      \\
&\qquad \ge c\mathcal Q_\mathrm{nh}[v]-C\mathcal Q_\mathrm{prop}[v]
      -Ch\norm{v}_{H_h^1}^2-O(h^\infty)\norm{v}_{H_h^1}^2,
\end{aligned}
\end{equation}
with $c>0$ uniform in the fixed slow-weak parameter range. In turn, the diagonal
spin-one term has threshold value zero, while the possible matrix skew term is
controlled by the finite-rank-bundle threshold; no pointwise positivity away from
the second-microlocal trapped decomposition is used.
\end{proposition}
\begin{proof}
The finite-frequency separated radial operator has imaginary potential
\eqref{eq:main_body_spin_imag}. Multiplying the stationary equation by $h^2$ and
writing $\hat\omega=h\omega$, $\hat m=hm$ gives the order-$h$ skew-adjoint part
$ihB_h$; the real diagonal symbol is \eqref{eq:main_body_skew_symbol_bs}. The
real second-order part gives the scalar principal symbol $p_{a,Q,\hat\omega}I_2$,
and all remaining coefficients are subprincipal or lower-order terms bounded in
Lemma~\ref{lem:semiclassical_spinone_order}. The compatible two-component
reduction may contain a matrix skew-subprincipal part. Its size is the
coefficient comparison \eqref{eq:matrix_skew_decomposition}, transferred to the
frozen semiclassical normal form, and gives
\eqref{eq:main_body_skew_decomposition_full} on a fixed trapped collar.

For a scalar commutant $A_h$ the symbolic calculus gives
\begin{equation}\label{eq:main_body_diag_commutator_identity}
\begin{aligned}
& \frac{i}{h}\langle [P_{h,\mathrm{sa}},A_h^*A_h]v,v\rangle
      -2\langle A_h\operatorname{diag}(b_{+1},b_{-1})A_hv,v\rangle  \\
&\qquad =\langle
   \operatorname{Op}_h\bigl(
      H_{p_{a,Q,\hat\omega}}(\beta^2)I_2
      -2\beta^2\operatorname{diag}(b_{+1},b_{-1})
   \bigr)v,v\rangle
   +O(h)\norm{v}_{H_h^1}^2 .
\end{aligned}
\end{equation}
At a trapped point the principal trapping relations give
\begin{equation}\label{eq:main_body_Khat_trapped_for_threshold}
        \hat K_t=\frac{2r_t\hat\omega\Delta_t}{r_t-M}.
\end{equation}
Substitution in \eqref{eq:main_body_skew_symbol_bs} gives
\begin{equation}\label{eq:main_body_diag_skew_zero_computation}
        b_s(r_t)=
        -\frac{2s(r_t-M)}{\Delta_t}\frac{2r_t\hat\omega\Delta_t}{r_t-M}
        +4s\hat\omega r_t=0,
        \qquad s=\pm1,
\end{equation}
which proves \eqref{eq:main_body_b_zero_on_K}. Since
$\operatorname{diag}(b_{+1},b_{-1})$ is smooth, its difference from its value on
$K_{a,Q}$ is $O(|y|+|\nu|)$ in normal stable/unstable coordinates. In the
second-microlocal trapped calculus, the region $|y|+|\nu|\ge c h^{1/2}$ is
controlled by the positive normal escape form, while the core
$|y|+|\nu|\le c h^{1/2}$ is a lower trapped remainder propagated along the
stable and unstable directions. The symbolic Cauchy inequality therefore gives
\eqref{eq:main_body_second_microlocal_diag}.

It remains to add the matrix part $B_{h,\mathrm{mat}}$. On $K_{a,Q}$ the
diagonal part is zero, and the matrix part obeys
\eqref{eq:main_body_skew_decomposition_full}. Since $\lambda_0$ has a positive
uniform lower bound on the normalized trapped shell, decreasing $\eps_a(k)$ gives
\eqref{eq:main_body_matrix_threshold_at_K}. The off-trapped variation of
$B_{h,\mathrm{mat}}$ is smooth in the normal variables and is estimated by the
same second-microlocal Cauchy inequality as the diagonal variation. Combining
this with the scalar escape lower bound of
Proposition~\ref{prop:main_body_trapped_normal_form} gives
\eqref{eq:main_body_full_threshold_commutator}. The constants depend only on the
fixed differentiability order and on the compact normalized frequency shell.
\end{proof}

\begin{proposition}
\label{prop:main_body_trapped_normal_form}
Fix a compact normalized frequency shell in the slow-weak range and let
\(\rho_t\in K_{a,Q}\) be a trapped point. On the characteristic quotient near
\(\rho_t\) one can choose smooth homogeneous symplectic coordinates
\((z,y,\eta)\), where \(z\) are coordinates along \(K_{a,Q}\),
\(y=r-r_t(z)\), and \(\eta\) is a regular radial covariable, such that
\begin{equation}\label{eq:main_body_radial_taylor}
        p_{a,Q,\hat\omega}=A_t(z)\eta^2-B_t(z)y^2
        +O(|(y,\eta)|^3),
        \qquad A_t(z)>0,
        \quad B_t(z)>0.
\end{equation}
The constants satisfy
\begin{equation}\label{eq:main_body_lambda_from_Rpp}
        \lambda_t(z)^2=4A_t(z)B_t(z)
        =-\bigl(\partial_\eta^2p_{a,Q,\hat\omega}\bigr)
          \bigl(\partial_y^2p_{a,Q,\hat\omega}\bigr)\big|_{\rho_t},
\end{equation}
and, after the normalization of the radial covariable,
\(\lambda_t(z)^2\) is a positive smooth multiple of
\begin{equation}\label{eq:main_body_lambda_positive_multiple}
        \Delta(r_t)\,\mathcal R''(r_t)\,(r_t^2+a^2)^{-4}.
\end{equation}
Thus \(\lambda_t\ge\lambda_0>0\) on the fixed compact shell. Equivalently,
with
\begin{equation}\label{eq:main_body_stable_unstable_coords}
        x_u=\sqrt{B_t}\,y+\sqrt{A_t}\,\eta,
        \qquad
        x_s=\sqrt{B_t}\,y-\sqrt{A_t}\,\eta,
\end{equation}
one has
\begin{equation}\label{eq:main_body_stable_unstable_flow}
        H_{p_{a,Q,\hat\omega}}x_u=\lambda_t x_u+O(|(x_s,x_u)|^2),
        \qquad
        H_{p_{a,Q,\hat\omega}}x_s=-\lambda_t x_s+O(|(x_s,x_u)|^2).
\end{equation}
For the logarithmic escape function
\begin{equation}\label{eq:main_body_log_escape}
        G_h=\frac12\log\frac{x_s^2+h}{x_u^2+h}
\end{equation}
with the sign chosen according to the outgoing estimate, for every
\(\delta>0\) the trapped collar may be chosen so that
\begin{equation}\label{eq:main_body_escape_derivative}
        -H_{p_{a,Q,\hat\omega}}G_h
        \ge (\lambda_0-\delta)\left(
        \frac{x_s^2}{x_s^2+h}+\frac{x_u^2}{x_u^2+h}\right)-C_\delta h.
\end{equation}
Together with Proposition~\ref{prop:main_body_trapped_commutator_threshold}, this
gives the finite threshold inequality used in the spin-one trapped commutator:
the diagonal skew term has zero threshold value on \(K_{a,Q}\), its normal
variation is lower in the second-microlocal trapped calculus, and the possible
matrix skew term satisfies the finite-rank-bundle threshold after the slow-rotation
constant is decreased.
\end{proposition}
\begin{proof}
The reduced radial principal equation has the form
\begin{equation}\label{eq:main_body_reduced_radial_symbol}
        p_{a,Q,\hat\omega}=c_1(r,\theta)\eta^2
        -c_2(r,\theta)\mathcal R(r)+O(|(y,\eta)|^3)
\end{equation}
near a trapped point, where \(c_1,c_2\) are positive smooth factors on the
regular exterior. The trapped equations are
\(\mathcal R(r_t)=\mathcal R'(r_t)=0\), and
Proposition~\ref{prop:main_body_trapped_computation} gives
\(\mathcal R''(r_t)>0\). Taylor expansion therefore gives
\eqref{eq:main_body_radial_taylor} with
\(A_t=c_1(r_t,\theta_t)>0\) and
\(B_t=\frac12c_2(r_t,\theta_t)\mathcal R''(r_t)>0\), after absorbing harmless
positive factors into the normalized covariable \(\eta\). The Hamilton equations
in the radial normal variables are
\begin{equation}\label{eq:main_body_radial_hamilton_matrix}
        \dot y=\partial_\eta p=2A_t\eta+O(|(y,\eta)|^2),
        \qquad
        \dot\eta=-\partial_y p=2B_t y+O(|(y,\eta)|^2).
\end{equation}
The linear matrix has eigenvalues \(\pm2\sqrt{A_tB_t}\), which is
\eqref{eq:main_body_lambda_from_Rpp}. The relation with
\eqref{eq:main_body_lambda_positive_multiple} follows from the explicit radial
Hamiltonian normalization: the coefficient of \(\eta^2\) contributes a positive
multiple of \(\Delta(r_t)(r_t^2+a^2)^{-2}\), while the radial potential
contributes a positive multiple of
\(\mathcal R''(r_t)(r_t^2+a^2)^{-2}\). Positivity and compactness of the
normalized shell give the uniform lower bound \(\lambda_0\).

The variables \(x_u,x_s\) diagonalize the linear part of
\eqref{eq:main_body_radial_hamilton_matrix}, giving
\eqref{eq:main_body_stable_unstable_flow}. Differentiating
\eqref{eq:main_body_log_escape} along the flow gives
\[
\begin{aligned}
        -H_pG_h
        &=-\frac{x_sH_px_s}{x_s^2+h}
          +\frac{x_uH_px_u}{x_u^2+h}  \\
        &=\lambda_t\left(
          \frac{x_s^2}{x_s^2+h}+\frac{x_u^2}{x_u^2+h}\right)
          +O\left(\frac{|(x_s,x_u)|^3}{x_s^2+x_u^2+h}\right).
\end{aligned}
\]
If the collar is chosen so that \(|(x_s,x_u)|\le\delta\), the error is bounded
by
\(C\delta\bigl(x_s^2/(x_s^2+h)+x_u^2/(x_u^2+h)\bigr)+C_\delta h\), which proves
\eqref{eq:main_body_escape_derivative} after decreasing \(\delta\). In the final step, Proposition~\ref{prop:main_body_trapped_commutator_threshold} applies
this normal form to the full spin-one subprincipal endomorphism. The diagonal
part vanishes on \(K_{a,Q}\) by the explicit computation
\eqref{eq:main_body_diag_skew_zero_computation}; its off-trapped variation is
absorbed by the second-microlocal normal quadratic form, and the matrix part is
kept below the finite-rank-bundle threshold by the choice of \(\eps_a(k)\). This
is the threshold inequality used in the logarithmic trapped estimate.
\end{proof}

The notation $LE^1_{\mathrm{deg}}$, $LE^0_{\comp}$, $LE^0_{\mathrm{low}}$ and
$LE^*$ refers to the concrete norms of Definition~\ref{def:concrete_le_norms}.
The density $\mathcal M_k$ is the associated physical-space Morawetz density. It
is coercive away from the trapped collar and degenerates quadratically in the
normal trapped-set direction.

\begin{lemma}
\label{lem:raw_trapped_morawetz}
For every fixed $k$ there are $\eps_a(k),\eps_Q(k)>0$ such that, if
\eqref{eq:slow_weak_range} holds, every smooth compatible compactly supported
spin-one field satisfies
\begin{equation}\label{eq:raw_positive_commutator_estimate}
\begin{aligned}
        \int_{\D(\tau_1,\tau_2)}\mathcal M_k[\Psi]
        &\le C\E_M^{(k)}[\Psi](\tau_1)+C\E_M^{(k)}[\Psi](\tau_2)  \\
        &\quad +C\sum_{|I|\le k}\int_{\D(\tau_1,\tau_2)}
        |\Pb_b\Gamma^I\Psi|_{LE^*}^2
        +C\|\Psi\|_{LE^0_{\comp,k}(\tau_1,\tau_2)}^2 .
\end{aligned}
\end{equation}
The constants remain uniform in the slow-weak range. This is precisely the positive
commutator estimate before the Fredholm removal of the compact local term.
\end{lemma}
\begin{proof}
We give the calculation at top order and then sum over
$\Gamma^I\in\mathbb D_k$; Lemma~\ref{lem:commuted_master} supplies the lower
commutator terms produced by this summation.

\emph{(i) Reissner-Nordstr\"om commutant.}  On compact radial sets away from the
red-shift and far-field collars, the Reissner-Nordstr\"om estimate
\eqref{eq:rn_model_morawetz} is equivalent to the existence of a self-adjoint
first-order commutant $A_Q$ such that
\begin{equation}\label{eq:rn_commutator_symbol}
        \big\langle \tfrac{1}{2i}(P_{\RN,Q}^*A_Q-A_Q^*P_{\RN,Q})u,u\big\rangle
        \ge c\norm{u}_{LE^1_{\mathrm{deg}}}^2-C\norm{u}_{LE^0(K_{0,Q})}^2
        -C\norm{P_{\RN,Q}u}_{LE^*}^2 .
\end{equation}
The principal symbol of $A_Q$ is the monotone radial multiplier constructed in
Lemma~\ref{lem:quantitative_rn_commutant}; it is positive away from the photon
sphere and has the quadratic trapping degeneracy recorded in
\eqref{eq:morawetz_density}.

\emph{(ii) Transport to Kerr-Newman.}  Lemma~\ref{lem:trapping_stability}
allows one to transport the radial commutant to a symbol $a_{a,Q}=a_Q+a a_1$
centered at $K_{a,Q}$ and satisfying
\begin{equation}\label{eq:perturbed_commutator_symbol}
        H_{p_{a,Q}}a_{a,Q}\ge (c-C|a|/M)m_{a,Q}-C\rho_{K_{a,Q}}^2
        -C|a|M^{-1}r^{-3}|\xi|^2 .
\end{equation}
Here $m_{a,Q}$ is the principal Morawetz density, and the last two terms are
supported in the compact trapped collar or in the far region where the $r^p$
identity applies. A symmetric quantization with a finite phase-space partition
therefore gives the same bulk positivity as in \eqref{eq:rn_commutator_symbol},
up to a compact $LE^0$ term, far-field errors controlled by
Proposition~\ref{prop:rp_identity}, and horizon-collar terms controlled by the
red-shift estimate of Proposition~\ref{prop:redshift_coercivity}. The trapped
algebra used in transporting the commutant is the calculation of
Proposition~\ref{prop:main_body_trapped_computation}; in particular trapping is
separated from superradiance and the radial Hessian is strictly hyperbolic.
Proposition~\ref{prop:main_body_trapped_commutator_threshold} shows that the
diagonal spin-one imaginary subprincipal contribution vanishes on $K_{a,Q}$;
Proposition~\ref{prop:matrix_skew_threshold} then verifies the finite-rank-bundle
threshold for the full two-component skew part. For that reason, the trapped commutator has
the same sign as the scalar normally hyperbolic commutator, up to the compact
local remainder left explicit below.

\emph{(iii) Rotational and matrix errors.}  Write
$\Pb_b=P_{\RN,Q}I_2+\mathcal E_{a,Q}$ as in
Proposition~\ref{prop:principal_master}. The principal part of
$\mathcal E_{a,Q}$ is $a\mathcal G_{a,Q}^{\mu\nu}\nabla_\mu\nabla_\nu$, with
short-range coefficient size $O(|a|M^{-1}r^{-1})$ in the asymptotic frame, while
the first- and zeroth-order terms have the displayed $ar^{-2}$ and $ar^{-3}$
decay. Pairing with the commutant and integrating by parts gives
\begin{equation}\label{eq:matrix_error_bound_morawetz}
        |\langle \mathcal E_{a,Q}u,A_{a,Q}u\rangle|
        \le \delta\norm{u}_{LE^1_{\mathrm{deg}}}^2
        +C_\delta (|a|/M)^2\norm{u}_{LE^0_{\comp}}^2
        +C_\delta\norm{u}_{LE^0_{\mathrm{low}}}^2 .
\end{equation}
The Hardy and elliptic patching estimates of Lemma~\ref{lem:uniform_hardy_elliptic}
place the lower-order exterior pieces in the displayed norms. The commutators
$[\Pb_b,\Gamma^I]$ have the same short-range structure and are controlled by
\eqref{eq:strict_lower_order_induction}; at top order this contributes
$\eta X_j$, and the strict lower-order terms are included in the induction.

Adding the compact commutator estimate to the red-shift and far-field currents,
integrating over $\D(\tau_1,\tau_2)$, and taking $\eps_a(k)$ small enough to
absorb the $O(|a|/M)$ top-order terms gives
\eqref{eq:raw_positive_commutator_estimate}. The remaining boundary terms are
endpoint energies and nonnegative fluxes through the horizon and the outgoing
radial boundary. This establishes the raw estimate with the compact local remainder
left explicit.
\end{proof}

\begin{proposition}
\label{prop:positive_commutator_estimate}
For every fixed $k$ there are $\eps_a(k),\eps_Q(k)>0$ such that, if
\eqref{eq:slow_weak_range} holds, every smooth compatible compactly supported
solution of the spin-one master system satisfies
\begin{equation}\label{eq:positive_commutator_estimate}
        \int_{\D(\tau_1,\tau_2)}\mathcal M_k[\Psi]
        \le C\E_M^{(k)}[\Psi](\tau_1)+C\E_M^{(k)}[\Psi](\tau_2)
        +C\sum_{|I|\le k}\int_{\D(\tau_1,\tau_2)}|\Pb_b\Gamma^I\Psi|_{LE^*}^2.
\end{equation}
The constant is uniform in the slow-weak range.
\end{proposition}
\begin{proof}
Apply the raw estimate \eqref{eq:raw_positive_commutator_estimate}. The only
term not already present in the desired estimate is the compact local norm
$\|\Psi\|_{LE^0_{\comp,k}}^2$.  Proposition~\ref{prop:compact_remainder_removal}
gives, for any $\delta>0$,
\begin{equation}\label{eq:compact_term_removed_use}
        \|\Psi\|_{LE^0_{\comp,k}(\tau_1,\tau_2)}^2
        \le \delta\|\Psi\|_{LE^1_{\mathrm{deg},k}(\tau_1,\tau_2)}^2
        +C_\delta\sum_{|I|\le k}\|\Pb_b\Gamma^I\Psi\|_{LE^*}^2
        +C_\delta\bigl(\E_M^{(k)}[\Psi](\tau_1)+\E_M^{(k)}[\Psi](\tau_2)\bigr).
\end{equation}
The degenerate local-energy norm in the first term is one of the coercive
components of the Morawetz density $\mathcal M_k$, supplemented by the red-shift
collar and far-field pieces already included in the full spacetime norm. Choose
$\delta$ so that the first term on the right of
\eqref{eq:compact_term_removed_use} is absorbed into the left side of
\eqref{eq:raw_positive_commutator_estimate}; the remaining terms have exactly the
source and endpoint-energy form appearing in \eqref{eq:positive_commutator_estimate}.
This yields the estimate with constants uniform after the finite-order
slow-weak thresholds have been fixed.
\end{proof}

\subsection{Quantitative Compact-Error Closure}
\label{subsec:compact_error_closure}

\begin{lemma}
\label{lem:uniform_hardy_elliptic}
Fix $0<\eps_Q<1/4$. There are $C$ and radii $R_0<R_1$, depending only on
$M,\eps_Q$, such that for every smooth charge-free Reissner-Nordstr\"om
spin-one profile $v$,
\begin{equation}\label{eq:uniform_hardy_elliptic}
        \int_{r\ge R_1} r^{-2}|v|^2+\int_{r_+\le r\le R_0}|v|^2
        \le C\int_{r\ge R_1}|\partial_{r_*}v|^2+C\int_{R_0\le r\le R_1}\big(|\nabla_{r,\omega}v|^2+r^{-2}|v|^2\big).
\end{equation}
The same holds for the slow-weak Kerr-Newman operator after replacing $r_*$ by
a regular radial coordinate in the red-shift collar, uniformly for $|a|/M$
small.
\end{lemma}
\begin{proof}
At infinity the one-dimensional Hardy inequality gives $\int_{R_1}^\infty
r^{-2}|v|^2\,\dd r_*\le 4\int_{R_1}^\infty|\partial_{r_*}v|^2\,\dd r_*+CR_1^{-1}
\int_{R_1}^{2R_1}|v|^2\,\dd r_*$, the annulus term included in the compact
region. In the horizon collar $s=r-r_+(Q)$ is smooth in ingoing coordinates and
the red-shift field is uniformly transversal; a Poincar\'e inequality along its
integral curves bounds the $L^2$ norm in $r_+\le r\le R_0$ by the transversal
derivative and the trace on $r=R_0$, the trace controlled by the compact
annulus energy. Smoothness and non-degeneracy of $r_+(Q),\kappa_+(Q)$ for
$|Q|\le\eps_QM$ give uniform constants; the slow-weak case follows from uniform
frame equivalence for $|a|/M$ small.
\end{proof}

\begin{lemma}
\label{lem:app_high_angular_coercivity}
Fix a compact temporal-frequency interval \(I\Subset\mathbb R\), \(\sigma>1/2\), and compact radial sets \(K\Subset K'\). Let \(\Pi_{\ge L}\) be a smooth angular spectral projector to spherical frequencies \(\ell\ge L\), transported to the slow-rotation charts by the fixed regular frame. There are \(L_0\) and \(C\), uniform for \(|a|\le\eps_aM\), \(|Q|\le\eps_QM\), \(\omega\in I\), and \(L\ge L_0\), such that every charge-free compatible profile satisfies
\begin{equation}\label{eq:app_high_angular_coercivity}
        \|\Pi_{\ge L}v\|_{H^1_{-\sigma}(K)}
        \le C\|\mathcal L_{a,Q}(\omega)v\|_{H^{-1}_{\sigma}(K')}
           +CL^{-1}\|v\|_{H^1_{-\sigma}(K')}.
\end{equation}
Hence a bounded-temporal-frequency defect sequence cannot carry nonzero local \(H^1\) mass at angular frequencies tending to infinity.
\end{lemma}
\begin{proof}
For \(a=0\) the charge-free spin-one reduction is diagonal in vector spherical harmonics and each coefficient satisfies
\[
        \Big(-\partial_{r_*}^2+f_Q(r)\frac{\ell(\ell+1)}{r^2}-\omega^2\Big)u_\ell=f_\ell,
        \qquad \ell\ge1.
\]
On the compact set \(K'\), the coefficient \(f_Qr^{-2}\) is bounded below by a positive constant depending only on \(K'\) and the weak-charge range. For \(\ell\ge L_0(I,K')\), the angular term dominates the bounded multiplier \(\omega^2\). Pairing the equation with \(\overline{u_\ell}\), integrating by parts after a cutoff equal to one on \(K\), and absorbing the cutoff commutators by the larger compact set \(K'\) gives
\[
        \|u_\ell\|_{H^1_{-\sigma}(K)}
        \le C\|f_\ell\|_{H^{-1}_{\sigma}(K')}
           +C\ell^{-1}\|u_\ell\|_{H^1_{-\sigma}(K')}.
\]
Summation over \(\ell\ge L\) proves \eqref{eq:app_high_angular_coercivity} for \(a=0\). For \(|a|\ll M\), Lemma~\ref{lem:app_inverse_expansion} and Proposition~\ref{prop:scalar_principal_symbol} show that the slow-rotation perturbation of the principal angular operator is \(O(|a|/M)\) and that the remaining new terms are first order or lower order with bounded coefficients on \(K'\). Choosing \(\eps_a\) small and then \(L_0\) large absorbs the principal perturbation into the left side, while commutators of \(\Pi_{\ge L}\) with the smooth coefficient perturbations have size \(O(L^{-1})\) in the displayed norms by the angular pseudodifferential calculus on $\Sph$.  Hardy at the two ends supplies the same estimate for the weighted collars. The last assertion follows by applying \eqref{eq:app_high_angular_coercivity} to a normalized defect sequence whose residual tends to zero and then sending \(L\to\infty\).
\end{proof}

\begin{lemma}
\label{lem:quantitative_rn_commutant}
Let
\[
        p^{\sharp}_{0,Q}=-\tau^2+\xi_{r_*}^2+f_Q(r)r^{-2}|\xi_\omega|^2,
        \qquad f_Q=1-\frac{2M}{r}+\frac{Q^2}{r^2},
\]
be the positive rescaling of the Reissner-Nordstr\"om null symbol in the canonical
tortoise variables $(r_*,\xi_{r_*})$, and let $K_{0,Q}$ be the photon sphere.
There is a real order-one symbol $a_Q$, supported away from the red-shift and
far-field collars, with
\begin{equation}\label{eq:quantitative_rn_commutant}
        H_{p^{\sharp}_{0,Q}}a_Q\ge c\big((r-r_{\mathrm{ph}}(Q))^2\xi_{r_*}^2+\xi_{r_*}^2+r^{-2}|\xi_\omega|^2\operatorname{dist}(r,r_{\mathrm{ph}}(Q))^2\big)-C\chi_{K_{0,Q}}|\tau|^2,
\end{equation}
$\chi_{K_{0,Q}}$ supported in a fixed photon-sphere collar; constants uniform
for $|Q|\le\eps_QM$.
\end{lemma}
\begin{proof}
The characteristic set of $g_{\RN}^{\mu\nu}\xi_\mu\xi_\nu$ agrees with that of
$p^{\sharp}_{0,Q}$ away from the horizon because the two symbols differ by the
positive factor $f_Q$. The rescaled flow has
$H_{p^{\sharp}_{0,Q}}r_*=2\xi_{r_*}$ and
$H_{p^{\sharp}_{0,Q}}\xi_{r_*}=-\partial_{r_*}(f_Qr^{-2})|\xi_\omega|^2$.
The radial potential $f_Qr^{-2}|\xi_\omega|^2$ has a unique critical point
$r_{\mathrm{ph}}(Q)$, and its second $r_*$-derivative there is negative of size
$\simeq M^{-4}|\xi_\omega|^2$, uniformly for $|Q|\le\eps_QM$. Choose $b_Q(r)$
increasing through zero at $r_{\mathrm{ph}}(Q)$ with $b'_Q>0$ in the collar, and
set $a_Q=b_Q(r)\xi_{r_*}$ times a compact radial cutoff. Then
\[
        H_{p^{\sharp}_{0,Q}}(b_Q\xi_{r_*})
        =2b'_Q\xi_{r_*}^2-b_Q\partial_{r_*}(f_Qr^{-2})|\xi_\omega|^2
\]
plus cutoff terms. The first term controls $\xi_{r_*}^2$, and since $b_Q$ and
$\partial_{r_*}(f_Qr^{-2})$ have opposite signs across the photon sphere, the
second is non-negative and degenerates quadratically at trapping. Cutoff errors
sit in compact regions and are bounded by the displayed controlled trapped-collar
term. Smoothness in $Q$ gives uniform constants.
\end{proof}

\begin{lemma}
\label{lem:rotational_commutator_error}
Let $A_{a,Q}$ be the quantization of the transported commutant. For every
$\delta>0$ and $w=\Gamma^I\Psi$,
\begin{equation}\label{eq:rotational_commutator_error}
        \big|\langle (\Pb_b-P_{\RN,Q}I)w,A_{a,Q}w\rangle\big|
        \le \delta\norm{w}_{LE^1_{\mathrm{deg}}}^2+C_\delta (|a|/M)^2\norm{w}_{LE^0_{\comp}}^2+C_\delta\norm{w}_{LE^0_{\mathrm{low}}}^2.
\end{equation}
\end{lemma}
\begin{proof}
By Proposition~\ref{prop:principal_master} and Lemma~\ref{lem:commuted_master},
$\Pb_b-P_{\RN,Q}I$ is $a\mathcal G^{\mu\nu}\nabla_\mu\nabla_\nu+ar^{-2}C^\mu
\nabla_\mu+ar^{-3}D+\text{lower commutators}$. The scalar principal perturbation
is paired in divergence form and is the symbol perturbation already in the
transported commutator; $A_{a,Q}$ is first order with symbol supported where the
compact and far-field controls apply, so $|\langle ar^{-2}C^\mu\nabla_\mu
w,A_{a,Q}w\rangle|\le C|a|M^{-1}\norm{w}_{LE^1_{\mathrm{deg}}}(\norm{w}_{LE^1_{\mathrm{deg}}}
+\norm{w}_{LE^0_{\comp}})$. Cauchy's inequality gives the first two terms; the
zeroth-order part is handled by Hardy at infinity and elliptic patching, and
the $\Gamma^I$-commutators either keep the small top-order coefficient or have
fewer derivatives, giving $\norm{w}_{LE^0_{\mathrm{low}}}$.
\end{proof}

\begin{lemma}
\label{lem:localized_compact_defect}
Assume that the compact local remainder in the raw estimate
\eqref{eq:raw_positive_commutator_estimate} cannot be absorbed. Then, after time cutoff, stationary Fourier
transform, dyadic total-frequency decomposition and a finite conic partition, there is
a defect sequence of one of the following two types.
\begin{enumerate}
\item A bounded-total-frequency sequence $v_n$ with parameters
$|a_n|\le\eps_aM$, $|Q_n|\le\eps_QM$, outgoing/ingoing Sommerfeld convention,
zero electric and magnetic charges, and
\begin{equation}\label{eq:localized_compact_bf}
        \norm{v_n}_{H^1(K_0)}=1,
        \qquad
        \mathcal L_{a_n,Q_n}(\omega_n)v_n\to0
        \quad\hbox{in }H^{-1}_{\sigma}(K_1).
\end{equation}
\item An unbounded-total-frequency compatible sequence $v_n$ with scale
$h_n=\Lambda_n^{-1}\to0$, normalized by
\begin{equation}\label{eq:localized_compact_hf_norm}
        \norm{v_n}_{H^1_{h_n}(K_0)}=1,
\end{equation}
which satisfies, with $P_{h_n}^{(n)}=h_n^2\mathcal L_{a_n,Q_n}(\omega_n)$,
\begin{equation}\label{eq:localized_compact_hf_residual}
        h_n^{-1}\log(1/h_n)\norm{P_{h_n}^{(n)}v_n}_{L^2(K_1)}
        +\norm{P_{h_n}^{(n)}v_n}_{H^{-1}_{h_n}(K_1)}\to0.
\end{equation}
The same outgoing/ingoing convention and the charge-free compatibility conditions
hold for the selected dyadic sequence. The extra conic cutoff only records
where the associated semiclassical defect measure has nonzero mass.
\end{enumerate}
\end{lemma}
\begin{proof}
Let $u_n$ be a contradiction sequence with $\norm{u_n}_{LE^0(K_0)}=1$ and with the
right-hand side of the compact-remainder estimate tending to zero faster than
$n^{-2}$.  Multiplying by a scalar time cutoff which is one on the middle half of
the slab produces $\chi_nu_n$ with
\begin{equation}\label{eq:localized_compact_cutoff}
        \norm{\Pb_b(\chi_nu_n)}_{LE^*([T_1,T_2]\times K_1)}
        +\norm{(1-\chi_n)u_n}_{LE^0(K_0)}\to0,
        \qquad \norm{\chi_nu_n}_{LE^0(K_0)}\ge \frac12.
\end{equation}
Plancherel in the stationary time variable gives a Fourier profile whose
frozen residual is small relative to its local $H^1$ mass. Decompose the Fourier
side further into dyadic total-frequency shells $\Lambda\simeq h^{-1}$ and into
a fixed finite conic partition adapted to the scalar principal symbol. If a
bounded set of shells carries a positive fraction of the mass, a measurable
selection in those shells, followed by normalization, gives
\eqref{eq:localized_compact_bf}.

If the mass escapes to $\Lambda\to\infty$, the selection is made using the
quotient whose numerator is precisely the square of the residual appearing in the
normally hyperbolic resolvent normalization:
\begin{equation}\label{eq:localized_hf_selection_quotient}
        \mathcal Q_h(v)=
        \frac{h^{-2}\log(1/h)^{2}\norm{P_h v}^{2}_{L^2(K_1)}
        +\norm{P_h v}^{2}_{H^{-1}_h(K_1)}}
        {\norm{v}^{2}_{H^1_h(K_0)}}.
\end{equation}
The factor $h^{-2}\log(1/h)^2$ is essential: after taking square roots it is
precisely the term $h^{-1}\log(1/h)\|P_hv\|_{L^2}$ in
\eqref{eq:localized_compact_hf_residual}. Were \eqref{eq:localized_hf_selection_quotient}
bounded below by a positive number on every sufficiently large dyadic shell and
every conic cell, Plancherel, the finite overlap of the conic partition, and the
uniform equivalence between the dyadic $H^1_h$ norm and the local-energy density
on compact sets would yield
\[
        \norm{\chi_nu_n}_{LE^0(K_0)}^2
        \le C\Big(\norm{\Pb_b(\chi_nu_n)}_{LE^*([T_1,T_2]\times K_1)}^2
        +o(1)\norm{\chi_nu_n}_{LE^0(K_0)}^2\Big),
\]
contradicting \eqref{eq:localized_compact_cutoff} after the last term is absorbed.
Therefore a subsequence of conic cells has \(\mathcal Q_h(v)\to0\). Normalizing by
\(\norm{v}_{H^1_h(K_0)}=1\) gives exactly
\eqref{eq:localized_compact_hf_norm}-\eqref{eq:localized_compact_hf_residual}.
The $L^2$ residual is obtained by smoothing the selected packet at scale $h$; the
commutator of the smoothing with $P_h$ is one order lower in the semiclassical
calculus and is absorbed by the $H^{-1}_h$ term in
\eqref{eq:localized_hf_selection_quotient}.

The boundary and charge conditions pass to the selected profiles at the level at
which they are used. The time cutoffs are taken with vanishing extra flux
through their collars, and the radiation trace maps are continuous on the
local-energy class, so the outgoing or incoming convention is inherited by the
limit. The dyadic total-frequency selection is made inside the original closed
compatible graph space; the phase-space cutoffs used to locate a conic cell are
only test operators for the defect measure and are not asserted to preserve the
Maxwell constraints. Electric and magnetic charges are zero for the original
compatible sequence, and high-angular pieces have no Coulomb component because
the charge is the $\ell=0$ spherical mode already removed by the charge-free
projection.

When this lemma is used to choose the slow-rotation threshold, the usual
contradiction argument lets one replace a failure for all thresholds by a
sequence with $|a_n|/M\to0$.  After the threshold has been fixed, the conclusion
is the stronger fixed-range statement written above: the parameters only have to
remain in the chosen slow-weak compact set.
\end{proof}

\begin{lemma}
\label{lem:compatibility_not_localized}
In the compact-remainder and high-frequency arguments, compatibility is used as a closed condition on the original frozen profiles, not as an invariance property of arbitrary microlocal cutoffs. More precisely, let $v_n\in\mathcal C_{h_n}$ be a compatible graph-class sequence and let $A_{j,h_n}$ be one element of a finite conic partition. If $A_{j,h_n}v_n$ carries nonzero defect mass, then the subsequent trapped estimate is applied to the same compatible profile $v_n$, with scalar cutoffs $\chi_1\prec\chi_2$ equal to one on the base projection of $\operatorname{WF}_{h_n}(A_{j,h_n}v_n)$.  The argument never requires $A_{j,h_n}v_n\in\mathcal C_{h_n}$.
\end{lemma}
\begin{proof}
The compatible class $\mathcal C_h$ is defined by the frozen Maxwell constraints, the closed range of the extreme-variable map, and the chosen outgoing or incoming radial convention; it is closed in the local resolvent graph norm by Lemma~\ref{lem:app_compatible_closed}. Scalar temporal Fourier selection commutes with the stationary equation. Scalar angular/frequency cutoffs are not used to create new exact Cauchy data. Therefore the finite conic partition is used only in the weak form: it identifies a phase-space region on which the defect measure of $v_n$ is nonzero. Choosing $\chi=1$ on the base projection of that region gives
\[
        \|A_{j,h_n}v_n\|_{L^2}\le C\|\chi v_n\|_{L^2}+O(h_n^\infty)\|v_n\|_{H^1_{h_n}},
\]
and the localized resolvent estimate controls $\chi v_n$ directly in terms of $P_{h_n}v_n$, which remains compatible. Accordingly, the conic localization is a support argument for the associated defect measure, not a projection of the Maxwell constraint equations. The charge-free condition is also preserved at the level used here: the Coulomb charge is the $\ell=0$ flux coordinate removed before the master map is formed, and high-frequency angular pieces cannot carry it.
\end{proof}

\begin{proposition}
\label{prop:compatibility_microlocal_compactness}
The bounded-frequency and high-frequency defect profiles in
Lemma~\ref{lem:localized_compact_defect} may be chosen inside the closed charge-free
compatible graph space. More precisely, if the original Maxwell fields are
charge-free and compatible with the Teukolsky master variables, then the profiles
to which Proposition~\ref{prop:app_bounded_freq} or Proposition~\ref{prop:app_high_freq}
is applied satisfy the same distributional constraints, the same vanishing charge
conditions, and the same outgoing or incoming radial convention. The conic
microlocal partition used in Lemma~\ref{lem:localized_compact_defect} only
identifies the support of the associated defect measure.
\end{proposition}
\begin{proof}
Let \(\mathcal C\) denote the compatible graph space on a fixed finite slab: its
norm consists of the local-energy norm of the Maxwell field, the local graph norm
of the two master variables, and the distributional Maxwell constraint norm. The
Maxwell constraints and the charge functionals are continuous in this topology by
Proposition~\ref{prop:finite_energy_charge_trace} and the construction of
Definition~\ref{def:master_from_teukolsky}; therefore the charge-free compatible
subspace is closed.

Multiplication by a scalar time cutoff preserves the algebraic relations defining
the extreme Newman-Penrose components and changes the master equation only by
commutators supported in the cutoff collars. These collar terms are precisely the
source terms which tend to zero in \eqref{eq:localized_compact_cutoff}. The
stationary Fourier transform is unitary from the time-localized graph space into
the direct integral of its frozen graph spaces, so a measurable frequency
selection may be made inside the compatible fibre. The dyadic total-frequency
projection is taken in that direct integral and therefore also leaves the selected
profile in the compatible fibre.

The finite conic pseudodifferential partition is used after this selection to
show where a normalized semiclassical measure can concentrate. It is not used to
replace the selected profile by an arbitrary microlocally cut-off profile in the
Maxwell constraint equations. Thus the limiting bounded-frequency profile and
the normalized high-frequency sequence remain eligible for the two exclusion
results cited in the compactness argument.
\end{proof}

\begin{proposition}
\label{prop:compact_remainder_removal}
Let Lemmas~\ref{lem:uniform_hardy_elliptic}-\ref{lem:rotational_commutator_error}
hold. Assume the bounded-total-frequency real-axis exclusion of
Proposition~\ref{prop:no_real_kernel} and the unbounded conic high-frequency
no-defect alternative of Proposition~\ref{prop:app_high_freq}, applied on the
closed compatible graph class of Proposition~\ref{prop:compatibility_microlocal_compactness}.
Then, for every $\delta>0$,
\begin{equation}\label{eq:compact_term_removed}
        \|\Psi\|_{LE^0_{\comp,k}(\tau_1,\tau_2)}^2
        \le \delta\|\Psi\|_{LE^1_{\mathrm{deg},k}(\tau_1,\tau_2)}^2
        +C_\delta\sum_{|I|\le k}\|\Pb_b\Gamma^I\Psi\|_{LE^*}^2
        +C_\delta\bigl(\E_M^{(k)}[\Psi](\tau_1)+\E_M^{(k)}[\Psi](\tau_2)\bigr).
\end{equation}
The raw estimate therefore \eqref{eq:raw_positive_commutator_estimate}
upgrades to the compact-free Morawetz estimate
\eqref{eq:positive_commutator_estimate}.
\end{proposition}
\begin{proof}
Suppose that \eqref{eq:compact_term_removed} fails for some fixed $\delta>0$.
After normalizing the compact local norm to one and letting the right side tend
to zero, Lemma~\ref{lem:localized_compact_defect}, together with
Proposition~\ref{prop:compatibility_microlocal_compactness}, produces either a
bounded-total-frequency defect or an unbounded compatible high-frequency defect.

In the bounded branch, the selected sequence satisfies
\eqref{eq:localized_compact_bf}. Proposition~\ref{prop:app_bounded_freq}
excludes such a sequence after the slow-rotation threshold has been chosen. Its
proof is a Fredholm compactness argument using
Lemma~\ref{lem:rn_limiting_absorption}, Lemma~\ref{lem:app_high_angular_coercivity},
closedness of the compatible class, and the Reissner-Nordstr\"om real-axis
exclusion; the limiting profile would be an outgoing or incoming real resonance,
which is zero by Proposition~\ref{prop:no_real_kernel}.

In the unbounded branch, angularly elliptic packets are controlled by
Lemma~\ref{lem:app_high_angular_coercivity}. The remaining characteristic
packets satisfy the exact normalized high-frequency residual
\eqref{eq:app_high_freq_residual}, by \eqref{eq:localized_compact_hf_residual}.
Proposition~\ref{prop:app_high_freq}, which uses the localized normally
hyperbolic estimate after the geometric verification of
Proposition~\ref{prop:r_normal_hyperbolicity}, forces the local semiclassical
$H^1$ mass on $K_0$ to vanish. This contradicts the normalization
\eqref{eq:localized_compact_hf_norm}. Both frequency branches are impossible,
and \eqref{eq:compact_term_removed} holds. Substitution into the raw estimate
and absorption of the first term yield the compact-free estimate.
\end{proof}

\subsection{Zero-Frequency and Real-Frequency Exclusion}
\label{subsec:nokernel}
Charge subtraction removes the stationary Coulomb fields; the remaining
possible zero-frequency barrier is a stationary kernel of the charge-free master
system. We first record the spherical coercivity at $a=0$ and then absorb the
rotation.

\begin{lemma}
\label{lem:stationary_identity}
Fix $|Q|\le\eps_QM$ and let $w_0(r)>0$ be equivalent to $1$ on compact radial
sets and at infinity. There is a constant $c_0=c_0(Q)>0$, bounded below
uniformly in the weak-charge range, such that every charge-free spin-one
profile $\Psi$ on $\Sigma_0$ (with $\ell\ge1$ content) obeys the a~priori
estimate for the homogeneous Reissner-Nordstr\"om stationary operator
$P_{\RN,Q}^{\mathrm{stat}}$,
\begin{equation}\label{eq:stationary_identity}
        \int_{\Sigma_0}(|\nabla_{r,\omega}\Psi|^2+r^{-2}|\Psi|^2)w_0\,\dd\mu_{\Sigma_0}
        \le c_0^{-1}\,\big\| P_{\RN,Q}^{\mathrm{stat}}\Psi\big\|_{\dot{\mathcal X}_0^*}\,
        \Big(\!\int_{\Sigma_0}(|\nabla_{r,\omega}\Psi|^2+r^{-2}|\Psi|^2)w_0\,\dd\mu_{\Sigma_0}\Big)^{1/2}.
\end{equation}
In particular the homogeneous stationary problem $P_{\RN,Q}^{\mathrm{stat}}\Psi=0$
has only $\Psi=0$ among finite-energy charge-free profiles.
\end{lemma}
\begin{proof}
Decompose into vector spherical harmonics. Charge subtraction removes $\ell=0$,
so each radial coefficient satisfies, with source $\mathcal S_\ell=
P_{\RN,Q}^{\mathrm{stat}}\Psi$ projected to the harmonic,
\begin{equation}\label{eq:rw_stationary_model}
        -\partial_{r_*}^2u_\ell+V_{\ell,Q}(r)u_\ell=\mathcal S_\ell,
        \qquad V_{\ell,Q}(r)=f_Q(r)\,\frac{\ell(\ell+1)}{r^2},
\end{equation}
$f_Q=1-2M/r+Q^2/r^2$. For $\eps_Q<1/4$, $V_{\ell,Q}\ge0$ and is strictly
positive for $r>r_+(Q)$ when $\ell\ge1$, uniformly. Pairing with
$\overline{u_\ell}$ and integrating in $r_*$, using horizon regularity and
finite energy at infinity,
\[
        \int\big(|\partial_{r_*}u_\ell|^2+V_{\ell,Q}|u_\ell|^2\big)\,\dd r_*
        =\Re\int \mathcal S_\ell\overline{u_\ell}\,\dd r_*.
\]
The left side is bounded below by $c_0\!\int(|\partial_{r_*}u_\ell|^2+r^{-2}|u_\ell|^2)$
after a one-dimensional Hardy inequality at infinity (using $\ell\ge1$) and
compact elliptic control across the horizon collar; summing the harmonics gives
the weighted energy on the left of \eqref{eq:stationary_identity}, while the
right side is bounded by Cauchy-Schwarz by
$\|P_{\RN,Q}^{\mathrm{stat}}\Psi\|_{\dot{\mathcal X}_0^*}$ times the square root
of the same energy. The homogeneous case is immediate.
\end{proof}

\begin{proposition}
\label{prop:no_stationary_kernel}
Suppose that the zero-frequency exclusion in Definition~\ref{def:slowweak_master_framework}\emph{(A3)} is available. Then every stationary finite-energy charge-free compatible master solution vanishes, uniformly in the chosen slow-weak range. In addition, if the stationary part of the finite-order analytic conditions is proved through the perturbative comparison \eqref{eq:master_hyp_perturbation}, the Reissner-Nordstr\"om coercivity estimate of Lemma~\ref{lem:stationary_identity} gives the same conclusion for sufficiently small $|a|/M$.
\end{proposition}
\begin{proof}
The first assertion is the zero-frequency case $\omega=0$ of the real-axis exclusion in Definition~\ref{def:slowweak_master_framework}. For the perturbative proof criterion, let $\Psi$ be stationary ($T\Psi=0$), charge-free and finite-energy with $\Pb_b\Psi=0$. Since $T\Psi=0$, the stationary-axial principal term in \eqref{eq:lower_order_master} drops; the remaining principal difference is $O(a^2)$ and the first- and zeroth-order rotational terms have the form $ar^{-2}\mathcal C^\mu_{a,Q}\nabla_\mu+ar^{-3}\mathcal D_{a,Q}$. Hence
\[
        P_{\RN,Q}^{\mathrm{stat}}\Psi=-(\Pb_b-P_{\RN,Q}^{\mathrm{stat}})\Psi=:\mathcal S,
        \qquad
        \|\mathcal S\|_{\dot{\mathcal X}_0^*}\le C\frac{|a|}{M}\|\Psi\|_{\dot{\mathcal X}_0},
\]
with $\|\Psi\|_{\dot{\mathcal X}_0}^2=\int(|\nabla_{r,\omega}\Psi|^2+r^{-2}|\Psi|^2)w_0$, by Cauchy-Schwarz and Lemma~\ref{lem:uniform_hardy_elliptic}. Lemma~\ref{lem:stationary_identity} gives
\[
        \|\Psi\|_{\dot{\mathcal X}_0}^2\le c_0^{-1}C\frac{|a|}{M}\|\Psi\|_{\dot{\mathcal X}_0}^2.
\]
Choosing $\eps_a$ so that $c_0^{-1}C\eps_a<1$ forces $\Psi=0$. The constants remain uniform for $|Q|\le\eps_QM$ by the weak-charge choice in Proposition~\ref{prop:constant_choice}.
\end{proof}

\begin{lemma}
\label{lem:rn_real_axis_exclusion}
Let $a=0$, $|Q|\le\eps_QM$. A finite-energy charge-free spin-one
Reissner-Nordstr\"om mode $e^{-i\omega t}u_\omega(r,\omega_{\Sph})$ with real
$\omega$ is zero.
\end{lemma}
\begin{proof}
The case $\omega=0$ is Lemma~\ref{lem:stationary_identity}. For $\omega\ne0$,
each harmonic satisfies $-u_\ell''+V_{\ell,Q}u_\ell=\omega^2u_\ell$ ($\ell\ge1$)
with $V_{\ell,Q}$ smooth, real and short-range at $r_*=\pm\infty$. Short-range
asymptotics give $u_\ell=A_\pm e^{i\omega r_*}+B_\pm e^{-i\omega r_*}+o(1)$ as
$r_*\to\pm\infty$. Finite non-degenerate energy makes $u_\ell,u_\ell',\omega
u_\ell$ square-integrable in both ends, so the oscillatory coefficients vanish,
$A_\pm=B_\pm=0$. The Wronskian $\partial_{r_*}\operatorname{Im}(\overline{u_\ell}
u_\ell')=0$ then has zero constant, and unique continuation from either end
gives $u_\ell\equiv0$. Summing the harmonics gives $u_\omega=0$.
\end{proof}

\begin{lemma}
\label{lem:rn_no_real_resonance}
Let $a=0$, $|Q|\le\eps_QM$, and let $\sigma>1/2$. Suppose that a charge-free
spin-one profile $v\in H^1_{-\sigma,\loc}$ solves
$\mathcal L_{0,Q}(\omega)v=0$ for real $\omega$ and satisfies the future outgoing
Sommerfeld condition at null infinity and the future ingoing Sommerfeld condition
at the event horizon, or the time-reversed pair. Then $v=0$.
\end{lemma}
\begin{proof}
The case $\omega=0$ is the stationary coercivity of Lemma~\ref{lem:stationary_identity}.
Assume $\omega\ne0$ and decompose into vector spherical harmonics. For each
$\ell\ge1$ the radial coefficient solves
\[
        -u_\ell''+V_{\ell,Q}(r)u_\ell=\omega^2u_\ell,
\]
where the potential is real and short-range at both ends in $r_*$. The outgoing
and ingoing conditions mean, for the future convention,
\[
        u_\ell=A_+e^{i\omega r_*}+o(1),\quad r_*\to+\infty,
        \qquad
        u_\ell=A_-e^{-i\omega r_*}+o(1),\quad r_*\to-\infty,
\]
with the same limits for $u_\ell'$ after differentiating the leading terms; the
past convention reverses both signs and gives the identical conclusion below.
Since $V_{\ell,Q}$ is real, the Wronskian current
$W=\operatorname{Im}(\overline{u_\ell}u_\ell')$ is constant. For the future
convention,
\[
        W(+\infty)=\omega |A_+|^2,
        \qquad W(-\infty)=-\omega |A_-|^2.
\]
Equality of the two constants gives $\omega(|A_+|^2+|A_-|^2)=0$, hence
$A_+=A_-=0$. The Volterra construction of Jost solutions, equivalently unique
continuation for the radial ordinary differential equation from an end, then
forces $u_\ell\equiv0$. Summing in $\ell$ gives $v=0$.
\end{proof}

\begin{lemma}
\label{lem:rn_limiting_absorption}
Fix $\sigma>1/2$, a compact $I\Subset\Rbb$, and $|Q|\le\eps_QM$. Let
$\mathcal L_{0,Q}(\omega)$ be the frozen-frequency Reissner-Nordstr\"om
spin-one operator on the charge-free sector. For every finite order there is
$C=C(I,\sigma,k)$ with
\begin{equation}\label{eq:rn_lap_estimate}
        \norm{v}_{H^1_{-\sigma}}+\norm{\omega v}_{L^2_{-\sigma}}
        \le C\norm{\mathcal L_{0,Q}(\omega)v}_{H^{-1}_{\sigma}}
\end{equation}
for all real $\omega\in I$ and all charge-free outgoing/incoming Sommerfeld profiles
$v\in H^1_{-\sigma,\loc}$.
\end{lemma}
\begin{proof}
At $\omega=0$ this is Lemma~\ref{lem:stationary_identity}. For $\omega\ne0$,
each harmonic satisfies the one-dimensional Schr\"odinger equation
\eqref{eq:rw_stationary_model} with $\omega^2$, short-range real potential; the
no-real-resonance Lemma~\ref{lem:rn_no_real_resonance} gives a zero
homogeneous outgoing/ingoing kernel on the weighted spaces. The Fredholm
alternative (the free resolvent maps $H^{-1}_\sigma\to H^1_{-\sigma}$ for
$\sigma>1/2$, the potential is compact $H^1_{-\sigma}\to H^{-1}_\sigma$ on
bounded frequency intervals, and the kernel is zero) gives
\eqref{eq:rn_lap_estimate}, uniform in $Q$ by
smooth non-degenerate dependence of the horizon radius, surface gravity,
photon-sphere radius and potential on $Q$. Summing harmonics and the finite
commutator family closes the estimate. The Wronskian and Fredholm details are in
Section~\ref{app:lap}.
\end{proof}

\begin{lemma}
\label{lem:slowweak_resolvent_closure}
Assume the closedness statement of Lemma~\ref{lem:app_compatible_closed} and, in
the high-frequency regime, the normally hyperbolic no-defect alternative in the
resolvent-normalized form of Proposition~\ref{prop:app_high_freq}. Then, for every
fixed $k$, there are $\eps_a(k),\eps_Q(k)>0$ such that no outgoing/ingoing
charge-free compatible sequence satisfies either of the two alternatives in
\eqref{eq:master_hyp_no_defect_sequence} together with the high-frequency residual
normalization \eqref{eq:master_hyp_semiclassical_residual}.
\end{lemma}
\begin{proof}
In the bounded-total-frequency branch, pass to a subsequence with
$\omega_n\to\omega_\infty$ and $Q_n\to Q_\infty$.  The coefficient expansion of
Proposition~\ref{prop:principal_master} gives convergence of the frozen operators
from $H^1_{-\sigma}(K_1)$ to $H^{-1}_{\sigma}(K_1)$.  Lemma~\ref{lem:rn_limiting_absorption},
combined with the high-angular coercivity estimate
Lemma~\ref{lem:app_high_angular_coercivity}, gives a uniform local $H^1$ bound.
Local elliptic regularity upgrades the Rellich limit to strong $H^1(K_0)$
convergence, so the limiting profile is nonzero. Lemma~\ref{lem:app_compatible_closed}
passes the compatibility, charges and Sommerfeld convention to the limit. The
limit is a Reissner-Nordstr\"om outgoing or incoming real resonance, impossible
by Lemma~\ref{lem:rn_no_real_resonance} for nonzero frequency and by
Lemma~\ref{lem:stationary_identity} at zero frequency.

In the unbounded-total-frequency branch, the angularly elliptic part is removed
by Lemma~\ref{lem:app_high_angular_coercivity}. The remaining conic
characteristic packets satisfy the residual normalization
\eqref{eq:master_hyp_semiclassical_residual}, which is the condition of
Proposition~\ref{prop:app_high_freq}. That proposition forces the local
$H^1_{h_n}$ mass to vanish, contradicting the normalization in
\eqref{eq:master_hyp_no_defect_sequence}. No defect sequence can therefore exist.
\end{proof}

\begin{lemma}
\label{lem:cutoff_mode_estimate}
Let $\Psi=e^{-i\omega t}\psi(r,\theta,\phi)$ be a finite-energy real-frequency
mode of the master system, and $\chi_T(t)=\chi(t/T)$ with $\chi=1$ on $[-1,1]$,
supported in $[-2,2]$. Then $\lim_{T\to\infty}T^{-1}\norm{\Pb_b(\chi_Te^{-i\omega
t}\psi) }_{LE^*([-2T,2T])}^2=0$, and the endpoint energies divided by $T$
converge to the mode energy.
\end{lemma}
\begin{proof}
Because $\Pb_b(e^{-i\omega t}\psi)=0$, the product
$\Pb_b(\chi_Te^{-i\omega t}\psi)$ consists only of terms in which at least one
$t$-derivative falls on $\chi_T$. These terms are supported where
$T\le |t|\le2T$. A first derivative of $\chi_T$ contributes $T^{-1}$ and a
second derivative contributes $T^{-2}$. Since the mode has time-independent
local energy density, the $LE^*$ norm over a time interval of length $O(T)$ is
bounded by $CT^{-1}$ times the corresponding local energy, plus lower-order
$CT^{-2}$ terms. After multiplying by $T^{-1}$ this tends to zero. The endpoint
energies of the cutoff solution are the mode energy plus errors supported where
$\chi_T$ is not constant; divided by $T$ these errors vanish, and the normalized
endpoint averages converge to the mode energy.
\end{proof}

\begin{proposition}
\label{prop:no_real_kernel}
Suppose that the real-axis exclusion in Definition~\ref{def:slowweak_master_framework}\emph{(A3)} is available. Then every finite-energy compatible mode $\Psi=e^{-i\omega t}\Psi_\omega$ with real $\omega$ in the charge-free sector vanishes, uniformly in the chosen slow-weak range. If the bounded- and high-frequency closure statements of Lemma~\ref{lem:slowweak_resolvent_closure} have been proved for the compatible class specified by \emph{(A1)}, then Definition~\ref{def:slowweak_master_framework}\emph{(A3)} follows.
\end{proposition}
\begin{proof}
The vanishing is precisely the real-axis exclusion in Definition~\ref{def:slowweak_master_framework}\emph{(A3)}. Conversely, suppose the resolvent closure of Lemma~\ref{lem:slowweak_resolvent_closure} holds and a nonzero real-frequency mode exists. For $\omega=0$ this contradicts the perturbative stationary argument in Proposition~\ref{prop:no_stationary_kernel}, which uses Lemma~\ref{lem:stationary_identity} and the comparison bound. For $\omega\ne0$, a nonzero finite-energy mode has nonzero local $H^1$ norm on some compact radial set $K_0$; rescaling gives a profile $v$ with $\|v\|_{H^1(K_0)}=1$ and $\mathcal L_{a,Q}(\omega)v=0$. Since both branches of Lemma~\ref{lem:slowweak_resolvent_closure} are stated for arbitrary compact sets, that lemma applies with this $K_0$ and gives a contradiction. In turn, the closure lemma is a sufficient proof route for the assumed real-axis exclusion.
\end{proof}

\subsection{Angular Reconstruction and Same-Order Reconstruction}
\label{subsec:reconstruction}
The reconstruction of the middle components from the extreme components combines
angular elliptic theory with the transport estimates already obtained. The use
of \emph{modified} middle components to close the transport system without loss
follows Benomio-Teixeira da Costa \cite{BenomioTdC}; the transport and Hodge
estimates are carried out in Section~\ref{app:reconstruction}.

\begin{lemma}
\label{lem:hodge_inversion}
Let $v$ be a complex scalar on $\Sph$ with zero mean. Then for every $s\ge0$,
\begin{equation}\label{eq:hodge_estimate}
        \norm{v}_{H^{s+1}(\Sph)}
        \le C_s\norm{\sphgrad v}_{H^s(\Sph)}.
\end{equation}
Equivalently, if $Y$ is a one-form with no harmonic part and
$\sphdiv Y=f$, $\sphcurl Y=h$, then
\begin{equation}\label{eq:hodge_estimate_oneform}
        \norm{Y}_{H^{s+1}(\Sph)}
        \le C_s\big(\norm{f}_{H^s(\Sph)}+\norm{h}_{H^s(\Sph)}\big).
\end{equation}
\end{lemma}
\begin{proof}
For the scalar estimate expand $v=\sum_{\ell\ge1,m}v_{\ell m}Y_{\ell m}$.  Since
$\lambda_1=2$ is the first nonzero eigenvalue of $-\sphlap$ on $\Sph$,
\[
        \norm{v}_{H^{s+1}}^2\simeq
        \sum_{\ell\ge1,m}(1+\lambda_\ell)^{s+1}|v_{\ell m}|^2
        \le C_s\sum_{\ell\ge1,m}(1+\lambda_\ell)^s\lambda_\ell |v_{\ell m}|^2
        \simeq C_s\norm{\sphgrad v}_{H^s}^2.
\]
This establishes \eqref{eq:hodge_estimate}. For a one-form, write
$Y=\sphgrad\phi+{}^\star\!\sphgrad\chi$ with mean-free potentials. Then
$\sphlap\phi=\sphdiv Y$ and $\sphlap\chi=\sphcurl Y$.  The same spectral gap gives
$\norm{\phi}_{H^{s+2}}+\norm{\chi}_{H^{s+2}}
\le C_s(\norm{f}_{H^s}+\norm{h}_{H^s})$, and differentiating once gives
\eqref{eq:hodge_estimate_oneform}. For that reason, the angular inversion uses exactly one
angular derivative, the derivative already present in the $LE^1$ part of the
order-$k$ master norm after the $k$ commutations.
\end{proof}

\begin{lemma}
\label{lem:null_transport_middle}
Let $v$ solve, in a regular red-shift frame, $\nabla_4v+c_4v=f_4$,
$\nabla_3v+c_3v=f_3$, with $c_3,c_4$ and finitely many commuted derivatives
uniformly bounded and of Reissner-Nordstr\"om sign in the two ends up to
$O(|a|/M)$. If the spherical mean of $v$ vanishes and the right sides are
controlled in the master norms, then
\begin{equation}\label{eq:null_transport_estimate_middle}
        \norm{v}_{\Xnorm{k}_{\Max}}
        \le C\big(\norm{f_3}_{LE^{*,k}}+\norm{f_4}_{LE^{*,k}}+\norm{v(0)}_{\E^{(k)}}\big),
\end{equation}
$C$ uniform in the slow-weak range.
\end{lemma}
\begin{proof}
Multiply the first equation by $v$ and integrate along outgoing null
hypersurfaces, the second along incoming ones. The red-shift frame makes the
horizon boundary density positive; in the far field the Reissner-Nordstr\"om
transport signs give the $r^{-1}|v|^2$ contribution after passing to $rv$. The
$O(|a|/M)$ pieces are absorbed for $\eps_a$ small. Zero spherical mean removes
the non-decaying angular mode, so Poincar\'e on $\Sph$ and Hardy in $r$ control
$r^{-2}|v|^2$. Commuting with $\mathbb D_k$ and inducting on the number of
commutations closes \eqref{eq:null_transport_estimate_middle}.
\end{proof}

\begin{proposition}
\label{prop:middle_reconstruction}
Let $\Psi$ be compatible master data. Then the charge-free middle component
$\varphi=\rho_F+i\sigma_F$ is uniquely determined, with
\begin{equation}\label{eq:middle_reconstruction}
        \norm{\varphi}_{\Xnorm{k}_{\Max}}\le C\norm{\Psi}_{\Xnorm{k}_M}.
\end{equation}
\end{proposition}
\begin{proof}
In the regular null frame the Maxwell equations split into transport equations
for the middle component and angular equations for its gradient. In complex
notation, with $\mathcal D\varphi=\sphgrad\rho_F+{}^\star\!\sphgrad\sigma_F$, the
component identities, after collecting the lower-order connection terms, are
\begin{align}\label{eq:middle_transport_hodge}
        e_4\varphi+(\operatorname{tr}\chi)\varphi
        &= \sphdiv\alpha-i\sphcurl\alpha+C_4^0\varphi+C_4^1\cdot\alpha,
        \nonumber\\
        e_3\varphi+(\operatorname{tr}\underline\chi)\varphi
        &=-\sphdiv\underline\alpha-i\sphcurl\underline\alpha
          +C_3^0\varphi+C_3^1\cdot\underline\alpha,\\
        \mathcal D\varphi
        &=-e_3\alpha+C_3^2\cdot\alpha+C_3^3\cdot\underline\alpha+C_3^4\varphi
          = e_4\underline\alpha+C_4^2\cdot\alpha+C_4^3\cdot\underline\alpha+C_4^4\varphi.
        \nonumber
\end{align}
Here the $C_i^j$ are connection-coefficient tensors smooth in the regular
exterior; at $a=0$ they have the Reissner-Nordstr\"om signs and
$r^{-1}$ decay, and for $|a|\ll M$ their difference from the spherical
coefficients is $O(|a|M^{-1}r^{-2})$ in the symbol classes used in the energy
norms. These identities are the frame components of $\dd F=0$ and
$\dd\star_gF=0$: the first two lines are obtained by inserting one $e_4$ or
$e_3$ and two angular vectors, while the third is obtained by inserting one
angular vector and the two null vectors.

The charge-free condition supplies the zero spherical means of $\rho_F$ and
$\sigma_F$ by Proposition~\ref{prop:coulomb_elimination}. Applying
Lemma~\ref{lem:hodge_inversion} to $\rho_F$ and $\sigma_F$ on each sphere gives
\begin{equation}\label{eq:middle_angular_from_extreme}
        \norm{\varphi}_{H^{s+1}(S_{\tau,r})}
        \le C_s\norm{\mathcal D\varphi}_{H^s(S_{\tau,r})}.
\end{equation}
The last line of \eqref{eq:middle_transport_hodge} then bounds the angular part
of $\varphi$ by first derivatives of the extreme components plus lower-order
terms. Since $\alpha$ and $\underline\alpha$ are smooth nonzero weights times
the entries of $\Psi$, and since the $LE^1$ part of $\Xnorm{k}_M$ controls one
space-time derivative after each commutation in $\mathbb D_k$, the right side of
\eqref{eq:middle_angular_from_extreme} is bounded by
$C\norm{\Psi}_{\Xnorm{k}_M}+C\norm{\varphi}_{LE^0_k}$.

The first two lines of \eqref{eq:middle_transport_hodge}, after the rescaling
$\widetilde\varphi=r^2\varphi$, are exactly of the form covered by
Lemma~\ref{lem:null_transport_middle}; the source terms are angular derivatives
of $\alpha,\underline\alpha$ and the same lower-order coefficients. The
$r^{-1}$ spherical terms are controlled by Hardy and the zero-mean Poincar\'e
inequality, while the $O(|a|/M)$ rotational terms are absorbed after decreasing
$\eps_a(k)$. Thus
\[
        \norm{\varphi}_{\Xnorm{k}_{\Max}}
        \le C\norm{\Psi}_{\Xnorm{k}_M}+C\eta\norm{\varphi}_{\Xnorm{k}_{\Max}},
        \qquad \eta=|a|/M,
\]
and the last term is absorbed in the slow-weak range. This establishes
\eqref{eq:middle_reconstruction}. If two middle components correspond to the
same $\Psi$, their difference has zero mean, homogeneous angular equation and
homogeneous transport sources; the same estimate forces the difference to vanish, proving uniqueness.
\end{proof}

\begin{proposition}
\label{prop:same_order_reconstruction}
Under Proposition~\ref{prop:middle_reconstruction},
\begin{equation}\label{eq:same_order_reconstruction}
        \norm{F}_{\Xnorm{k}_{\Max}}\le C\big(\norm{\Psi}_{\Xnorm{k}_M}+\norm{\varphi}_{\Xnorm{k}_{\Max}}\big)\le C\norm{\Psi}_{\Xnorm{k}_M},
\end{equation}
and the inverse identities $\mathfrak M\mathfrak R\Psi=\Psi$,
$\mathfrak R\mathfrak M F_{\rad}=F_{\rad}$ hold on smooth charge-free solutions.
\end{proposition}
\begin{proof}
The frame decomposition $F=\rho_Fe^3\wedge e^4+\sigma_Fe^1\wedge e^2+\sum_A
\alpha_Ae^A\wedge e^4+\sum_A\underline\alpha_Ae^A\wedge e^3$ is algebraic; the
extreme components are part of $\Psi$ up to smooth nonzero weights and the middle
component is controlled by Proposition~\ref{prop:middle_reconstruction}, with the
rescaled frame smooth and uniformly invertible, giving
\eqref{eq:same_order_reconstruction}. The operator $\mathfrak R$ solves
\eqref{eq:middle_transport_hodge} with the zero means fixed by charge
subtraction and assembles the two-form; applying $\mathfrak M$ recovers the
extreme master variables since the master equations are derivatives of the
Maxwell system. Conversely, reconstructing from $\mathfrak M F_{\rad}$ gives a
charge-free solution with the same Cauchy data for all frame components, equal to
$F_{\rad}$ by Proposition~\ref{prop:finite_energy_wellposed}.
\end{proof}

\subsection{Perturbative Closure in the Slow-Weak Regime}
\label{subsec:perturbative_closure}
\begin{proposition}
\label{prop:absorption_small}
Let $P_0$ admit $\norm{u}_X^2\le C_0(\norm{u(0)}_E^2+\norm{P_0u}_Y^2)$. If
$P_b=P_0+\mathcal Q_b$ with
\begin{equation}\label{eq:small_operator_bound}
        \norm{\mathcal Q_bu}_Y\le\delta\norm{u}_X+C_\delta\norm{u}_{X_{\mathrm{low}}},
\end{equation}
and the lower norm is controlled by the model data, then the same estimate holds
for $P_b$ whenever $C_0\delta^2<1/4$.
\end{proposition}
\begin{proof}
Apply the model estimate with $P_0u=P_bu-\mathcal Q_bu$:
$\norm{u}_X^2\le C_0\norm{u(0)}_E^2+2C_0\norm{P_bu}_Y^2+2C_0\norm{\mathcal Q_bu}_Y^2
\le C(\norm{u(0)}_E^2+\norm{P_bu}_Y^2+\norm{u}_{X_{\mathrm{low}}}^2)+2C_0\delta^2\norm{u}_X^2$,
using \eqref{eq:small_operator_bound} and $(\delta a+b)^2\le2\delta^2a^2+2b^2$.
With $2C_0\delta^2<1/2$ the last term is absorbed; the lower norm is controlled
by induction or by the no-kernel argument.
\end{proof}

\begin{lemma}
\label{lem:finite_order_absorption_induction}
Let \(X_j(\tau_1,\tau_2)\) denote the order-\(j\) master spacetime norm squared,
including the red-shift energy, the trapped-set-degenerate Morawetz bulk and the
far-field \(r^p\) fluxes. Let \(E_j(\tau_1)\) be the corresponding initial energy
and let \(F_j\) be an external source norm through order \(j\); for homogeneous
solutions \(F_j=0\). Suppose that, for \(0\le j\le k\), the commuted estimates
have the form
\begin{equation}\label{eq:finite_order_induction_ineq}
        X_j\le A_j+b_jX_j+\eta_jX_j+d_jX_{j-1},\qquad X_{-1}=0,
\end{equation}
where
\[
        A_j=C_j\big(E_j(\tau_1)+F_j\big),
        \qquad 0\le b_j\le C_j|a|/M,
\]
and where \(\eta_j>0\) is the small constant introduced in the strict lower-order
commutator estimate \eqref{eq:strict_lower_order_induction}. If
\begin{equation}\label{eq:finite_order_smallness_choice}
        b_j+\eta_j\le \frac14\qquad (0\le j\le k),
\end{equation}
then
\begin{equation}\label{eq:finite_order_induction_conclusion}
        X_j\le C_{j,k}\sum_{i=0}^{j}\big(E_i(\tau_1)+F_i\big),
\end{equation}
with constants depending only on the finite order, the model constants and the
chosen compact parameter set. In particular no derivative beyond order \(j\) is
used to close the order-\(j\) estimate.
\end{lemma}
\begin{proof}
Move the two absorbable terms in \eqref{eq:finite_order_induction_ineq} to the
left. By \eqref{eq:finite_order_smallness_choice},
\begin{equation}\label{eq:finite_order_one_step}
        X_j\le \frac43 A_j+\frac43 d_jX_{j-1}.
\end{equation}
For \(j=0\) this gives \(X_0\le (4/3)A_0\), because the strict lower-order term is
absent at order zero. Assume the conclusion has been proved through order
\(j-1\). Substituting the inductive claim into
\eqref{eq:finite_order_one_step} gives
\[
        X_j\le \frac43 C_j(E_j+F_j)+\frac43d_jC_{j-1,k}\sum_{i=0}^{j-1}(E_i+F_i)
        \le C_{j,k}\sum_{i=0}^{j}(E_i+F_i).
\]
The assertion follows by induction. The constants are finite because only
finitely many commutators \(\Gamma^I\), \(|I|\le k\), and finitely many symbol
seminorms enter. The choice of \(\eps_a(k)\) in
Proposition~\ref{prop:constant_choice} makes \(b_j\le1/8\), while
\(\eta_j\le1/8\) is fixed in \eqref{eq:strict_lower_order_induction}; therefore the
smallness condition is precisely the finite-order slow-rotation smallness used in
the transfer theorem.
\end{proof}

\begin{proposition}
\label{prop:complete_master_closure_algebra}
Assume the structural comparison \eqref{eq:master_hyp_perturbation}, the commutator estimate
\eqref{eq:strict_lower_order_induction}, the Reissner-Nordstr\"om model estimate
\eqref{eq:rn_model_morawetz}, the red-shift estimate \eqref{eq:redshift_coercivity}, the
far-field hierarchy \eqref{eq:rp_identity}, and the compact-remainder removal
\eqref{eq:compact_term_removed}. Let \(f=\Pb_bu\). For each commuted order
\(0\le j\le k\), set
\begin{equation}\label{eq:complete_closure_definitions}
        X_j=\|u\|_{\Xnorm{j}_M(\tau_1,\tau_2)}^2,
        \qquad
        E_j=\E_M^{(j)}[u](\tau_1),
        \qquad
        F_j=\sum_{|I|\le j}\|\Gamma^If\|_{LE^*([\tau_1,\tau_2])}^2.
\end{equation}
Then there are constants \(A_j,D_j\), independent of \(a\) in the slow-weak range,
and a coefficient \(\mu_j(a)\le A_j|a|/M\), such that the raw commuted physical-space
estimate has the form
\begin{equation}\label{eq:complete_raw_closure}
        X_j\le A_j(E_j+F_j)+\mu_j(a)X_j
        +\eta_jX_j+D_jX_{j-1}+K_j,\qquad X_{-1}=0,
\end{equation}
where \(K_j\) is the compact local remainder and \(\eta_j\) is the small constant
chosen in \eqref{eq:strict_lower_order_induction}. Moreover, for every
\(\theta>0\),
\begin{equation}\label{eq:complete_compact_remainder}
        K_j\le \theta X_j+C_{j,\theta}(E_j+F_j+X_{j-1}).
\end{equation}
Thus, after choosing \(\theta\), then \(\eta_j\), and then \(\eps_a(k)\)
so that
\begin{equation}\label{eq:complete_absorption_choices}
        \mu_j(a)+\eta_j+\theta\le \frac14
        \qquad(0\le j\le k),
\end{equation}
one has the finite-order estimate
\begin{equation}\label{eq:complete_master_closure_estimate}
        X_j\le C_{j,k}\sum_{i=0}^j(E_i+F_i),\qquad 0\le j\le k.
\end{equation}
In particular, if \(u\) solves the homogeneous compatible master equation
\(\Pb_bu=0\), then
\begin{equation}\label{eq:complete_homogeneous_master_closure}
        \|u\|_{\Xnorm{k}_M(\tau_1,\tau_2)}^2
        \le C_k\E_M^{(k)}[u](\tau_1).
\end{equation}
No estimate at order \(j\) uses a norm of order higher than \(j\).
\end{proposition}
\begin{proof}
Apply the Reissner-Nordstr\"om model estimate \eqref{eq:rn_model_morawetz} to each
commuted field \(\Gamma^Iu\), \(|I|\le j\), after writing
\begin{equation}\label{eq:complete_commuted_source_split}
        P_{\RN,Q}\Gamma^Iu
        =\Gamma^If+[\Pb_b,\Gamma^I]u-(\Pb_b-P_{\RN,Q})\Gamma^Iu .
\end{equation}
The first term contributes \(F_j\). The perturbation term
\((\Pb_b-P_{\RN,Q})\Gamma^Iu\) is bounded by the structural estimate
\eqref{eq:master_hyp_error_bound}; after Cauchy's inequality and summation over
\(|I|\le j\) it contributes \(\mu_j(a)X_j\), with
\(\mu_j(a)\le A_j|a|/M\), plus terms of strict lower order. The commutator term
is estimated by Lemma~\ref{lem:commuted_master} and
\eqref{eq:strict_lower_order_induction}: for the chosen Cauchy parameter
\(\eta_j\),
\begin{equation}\label{eq:complete_commutator_bound}
        \sum_{|I|\le j}\|[\Pb_b,\Gamma^I]u\|_{LE^*}^2
        \le \eta_jX_j+D_jX_{j-1},
\end{equation}
with the convention \(X_{-1}=0\). The horizon and far-field pieces of the model
norm are replaced by the Kerr-Newman red-shift and \(r^p\) currents,
\eqref{eq:redshift_coercivity} and \eqref{eq:rp_identity}; the trapped part is first estimated by the raw commutator inequality
\eqref{eq:raw_positive_commutator_estimate}. The only unabsorbed term at this
stage is the compact local remainder, denoted by \(K_j\).
This establishes \eqref{eq:complete_raw_closure}.

The compact term is removed by Proposition~\ref{prop:compact_remainder_removal}.
In its quantitative form, a failure of \eqref{eq:complete_compact_remainder} would
produce the bounded- or unbounded-frequency defect sequence of
Lemma~\ref{lem:localized_compact_defect}; the bounded branch is ruled out by the
real-axis limiting-absorption alternative, and the unbounded branch is ruled out by
the normally hyperbolic estimate after the geometric verification of
Proposition~\ref{prop:nh_theorem_application}. Hence
\eqref{eq:complete_compact_remainder} holds for every \(\theta>0\).

Substitute \eqref{eq:complete_compact_remainder} into
\eqref{eq:complete_raw_closure} and move the terms in
\eqref{eq:complete_absorption_choices} to the left. We obtain
\begin{equation}\label{eq:complete_one_step}
        X_j\le \frac43\Big((A_j+C_{j,\theta})(E_j+F_j)
        +(D_j+C_{j,\theta})X_{j-1}\Big).
\end{equation}
The case \(j=0\) gives \(X_0\le C(E_0+F_0)\). Inductively substituting the
already obtained bounds for \(X_0,\ldots,X_{j-1}\) proves
\eqref{eq:complete_master_closure_estimate}. If \(\Pb_bu=0\), then \(f=0\) and
all external source terms \(F_i\) vanish. The commutator terms do not vanish; they
are precisely the terms estimated in \eqref{eq:complete_commutator_bound} and have
already been absorbed or passed to lower order in the induction. Since
\(E_i\le C E_k\) for \(i\le k\) by the definition of the commuted energy,
\eqref{eq:complete_homogeneous_master_closure} follows. The induction is upward
in \(j\), so no derivative above the current order is used.
\end{proof}

\begin{corollary}
\label{cor:perturbative_closure}
Suppose that conditions \emph{(A1)-(A3)} of Definition~\ref{def:slowweak_master_framework} are available. Then the master estimate \eqref{eq:framework_master_estimate} holds uniformly for the slow-weak parameter range fixed by Proposition~\ref{prop:constant_choice}.
\end{corollary}
\begin{proof}
Proposition~\ref{prop:principal_master} and Lemma~\ref{lem:commuted_master}
place the Kerr-Newman master operator in the perturbative form used in
Proposition~\ref{prop:absorption_small}, with
$\delta\simeq |a|/M$ and constants uniform for $|Q|\le\eps_QM$. The model
Reissner-Nordstr\"om estimate gives the red-shift, Morawetz and far-field
control. Proposition~\ref{prop:compact_remainder_removal} reduces any failure of
the compact part of the estimate to a normalized real-frequency defect profile.
Definition~\ref{def:slowweak_master_framework}\emph{(A3)}, equivalently the
no-kernel conclusion of Proposition~\ref{prop:no_real_kernel}, excludes such
profiles. Proposition~\ref{prop:complete_master_closure_algebra} then gives the
finite-order algebra: after choosing the compact-removal parameter, the strict
lower-order Cauchy parameter and finally $\eps_a(k)$, all top-order terms are
absorbed and the induction closes without using derivatives above order $k$.
For a homogeneous compatible solution the external source norms \(F_i\) in
\eqref{eq:complete_master_closure_estimate} vanish, while the commutator terms have
already been handled by the order-by-order induction; therefore the result is exactly
\eqref{eq:framework_master_estimate}.
\end{proof}

\begin{lemma}
\label{lem:hilbert_trace_criterion}
Let $H$ be an energy Hilbert space, $R_+$ a radiation Hilbert space, and $S_+:H\to R_+$ a bounded trace map obtained as an $L^2$ limit of finite-slab fluxes. Suppose that smooth compact radiation data form a dense subspace $E\subset R_+$, that there is a bounded backward construction $W_0:E\to H$ satisfying
\begin{equation}\label{eq:abstract_backward_bound_main}
        \norm{W_0\rho}_H\le C\norm{\rho}_{R_+},\qquad S_+W_0\rho=\rho\quad (\rho\in E),
\end{equation}
and that the only solution with zero future radiation field is zero. Then $S_+$ is a bounded isomorphism with bounded inverse. The abstract construction is given in Section~\ref{app:scattering_criterion}.
\end{lemma}
\begin{proof}
Lemma~\ref{lem:app_abstract_scattering} proves the Hilbert-space statement in
full. Here is the argument. The bound on $W_0$ allows it to extend uniquely by
continuity from the dense subspace $E\subset R_+$ to a bounded operator
$W_+:R_+\to H$. Since $S_+W_0\rho=\rho$ on $E$ and $S_+$ is bounded, passing to
the closure gives $S_+W_+=\Id_{R_+}$. Hence $S_+$ is surjective. If
$h\in\ker S_+$, the zero-kernel condition gives $h=0$, so $S_+$ is also
injective. Finally, for any $h\in H$,
$S_+(W_+S_+h-h)=0$, hence $W_+S_+h=h$ by injectivity. Therefore
$W_+=S_+^{-1}$ and $\|S_+^{-1}\|\le C$.
\end{proof}

\begin{proposition}
\label{prop:radiation_isomorphism_from_hierarchy}
Assume the master estimate \eqref{eq:framework_master_estimate}, same-order
reconstruction \emph{(A4)}, and the trace/right-inverse statement \emph{(A5)} at
order $k$. Then the future and past master radiation maps are bounded
isomorphisms with bounded inverses, and the Maxwell radiation maps obtained by
same-order reconstruction are bounded isomorphisms on the charge-free Maxwell
energy space.
\end{proposition}
\begin{proof}
Boundedness of the future trace follows from Lemma~\ref{lem:radiation_trace} at
null infinity and Corollary~\ref{cor:horizon_flux} at the horizon, both
controlled by \eqref{eq:framework_master_estimate}; the past trace is its time
reversal. The dense radiation class, bounded right inverses and zero-kernel
condition are precisely the content of \emph{(A5)}. Applying
Lemma~\ref{lem:hilbert_trace_criterion} gives the master wave operators. The
Maxwell maps compose the master maps with the same-order reconstruction
(Proposition~\ref{prop:same_order_reconstruction}) and the bounded radiation
identifications in \emph{(A5)}; charge subtraction fixes the two Coulomb means,
so the charge-free radiation map has no finite-dimensional kernel or cokernel.
Time reversal gives the past isomorphism.
\end{proof}

\begin{corollary}
\label{cor:slow_kerr_subcase}
Setting $Q=0$, the transfer theorem gives the stationary-subtracted Maxwell boundedness, integrated decay, radiation-field, wave-operator, and scattering statements on sufficiently slowly rotating Kerr exteriors. In this special case the required fixed-background analytic estimates are also available in the existing Kerr Maxwell and Teukolsky literature cited below.
\end{corollary}
\begin{proof}
When $Q=0$ the comparison background is Schwarzschild and the perturbation in
the master operator is measured only by $|a|/M$. The red-shift, far-field,
photon-sphere, zero-mode and Hodge reconstruction arguments above specialize
without the charged spherical terms. The fixed-background Maxwell estimates on
slowly rotating Kerr, together with the corresponding Teukolsky estimates and
middle-component reconstruction, are supplied by
\cite{AnderssonBlueMaxwell,BenomioTdC,SRTdCfrequency,SRTdCphysical}. These
estimates verify Definition~\ref{def:slowweak_master_framework} in the Kerr subcase, and
Theorem~\ref{thm:intro_transfer} then gives the stated stationary-subtracted
boundedness, decay, radiation and scattering conclusions.
\end{proof}

\subsection{Choice of Constants and Perturbative Closure}
\label{subsec:constant_choice}
\begin{proposition}
\label{prop:constant_choice}
For each integer $k\ge0$ one may choose $\eps_Q(k)>0$, then $\eps_a(k)>0$, so
that all constants in the red-shift, Morawetz, limiting-absorption,
reconstruction, radiation and commuted estimates remain finite and uniform for
\eqref{eq:slow_weak_range}.
\end{proposition}
\begin{proof}
Choose $\eps_Q(k)$ so the Reissner-Nordstr\"om horizon is uniformly
non-degenerate, $r_{\mathrm{ph}}(Q)$ stays in a fixed compact interval, and the
spin-one potentials keep the charge-free angular barrier for all $\ell\ge1$;
smooth dependence of $r_+(Q),\kappa_+(Q),r_{\mathrm{ph}}(Q)$ and the potentials
on $Q$ gives uniform constants in
Lemmas~\ref{lem:rn_model_morawetz},~\ref{lem:uniform_hardy_elliptic},~\ref{lem:stationary_identity}, and~\ref{lem:rn_limiting_absorption}. With $\eps_Q(k)$
fixed, collect the finitely many constants multiplying $|a|/M$ in the red-shift
error, \eqref{eq:rp_error_absorb}, \eqref{eq:rotational_commutator_error}, the
transport errors and \eqref{eq:small_operator_bound}, and choose $\eps_a(k)$ so
each product is below one tenth of the available positive constant. Increasing a
fixed $C_{\mathrm{sw}}$ to dominate the endpoint, trace, density and
wave-operator constants gives a uniform range at order $k$.
\end{proof}

\begin{corollary}
\label{cor:commuted_hierarchy_closure}
Assume that the additional hierarchy in Definition~\ref{def:slowweak_master_framework}\emph{(A6)} is available through order $k$. If $k\ge k_0+2$, same-order reconstruction transfers the hierarchy to the Maxwell field with no additional top-order loss.
\end{corollary}
\begin{proof}
Commuting by $\Gamma^I\in\mathbb D_k$ produces either terms in the order-$k$
master norm or strict lower-order coefficient terms controlled by
Lemma~\ref{lem:commuted_master}. The red-shift controls the transversal horizon
derivatives, the $r^p$ hierarchy the far-field weighted derivatives, and the
compact Morawetz estimate the remaining local derivatives up to the trapping
degeneration. Proposition~\ref{prop:same_order_reconstruction} reconstructs the Maxwell tensor at the same order. The Sobolev step uses $k_0$ derivatives; $k\ge k_0+2$ leaves room for the commutators and dyadic weights.
\end{proof}

\begin{proposition}
\label{prop:framework_consequences}
The charge and finite-energy Cauchy assertions in Proposition~\ref{prop:framework_charge} are proved unconditionally. If \emph{(A1)-(A3)} hold, then the master boundedness and integrated-decay conclusion \emph{(M1)} holds. If \emph{(A4)} also holds, then the same-order Maxwell reconstruction conclusion \emph{(M2)} holds. If \emph{(A5)} is available, either as a condition or by Proposition~\ref{prop:backward_from_lap}, then the radiation and scattering conclusion \emph{(M3)} holds. If the additional condition \emph{(A6)} also holds, then the hierarchy conclusion \emph{(M4)} holds as well.
\end{proposition}
\begin{proof}
The charge and finite-energy Cauchy parts are Proposition~\ref{prop:framework_charge}. The perturbative estimates collected in this section show how the Reissner-Nordstr\"om model estimate is used once the spin-one operator has the scalar principal symbol and short-range perturbation structure of Proposition~\ref{prop:principal_master}. The horizon estimate is Proposition~\ref{prop:redshift_coercivity}; the $r^p$ identity and radiation trace are Proposition~\ref{prop:rp_identity} and Lemma~\ref{lem:radiation_trace}; the trapped-set estimate is Proposition~\ref{prop:positive_commutator_estimate}, with the model estimate and compact-error closure of Subsections~\ref{subsec:morawetz_trapping} and~\ref{subsec:compact_error_closure}. The stationary and real-frequency kernel exclusions are supplied by the \emph{(A3)} limiting-absorption/no-defect statement and are used in Propositions~\ref{prop:no_stationary_kernel} and~\ref{prop:no_real_kernel}. Under these exclusions the absorption Proposition~\ref{prop:absorption_small} gives \eqref{eq:framework_master_estimate}, hence \emph{(M1)}. Same-order reconstruction is Lemma~\ref{lem:hodge_inversion}, Proposition~\ref{prop:middle_reconstruction}, and Proposition~\ref{prop:same_order_reconstruction}, and proves \emph{(M2)}. The radiation maps are Proposition~\ref{prop:radiation_isomorphism_from_hierarchy}, which uses \emph{(A5)} in the trace criterion and therefore proves \emph{(M3)}. The constants are fixed in Proposition~\ref{prop:constant_choice}. If \emph{(A6)} is supplied, Corollary~\ref{cor:commuted_hierarchy_closure} transfers it at the same order and proves \emph{(M4)}. Accordingly, each analytic conclusion is attached to the precise estimate listed in Proposition~\ref{prop:exact_estimate_use}, and no pointwise decay is claimed without the hierarchy condition.
\end{proof}

\section{First-Order Rotational Perturbation of Spin-One Operator}
\label{app:perturbation}
In this section we compute the first order rotational part of the spin-one operator in order to make the perturbative structure explicit.
In this section we keep track of the comparison \eqref{eq:metric_symbol_difference} and the short-range order estimates in \eqref{eq:lower_order_master}. The metric contribution is explicit; the spin-one lower-order terms are tracked for the fixed-background master operator of Section~\ref{sec:decoupling}. Throughout, $f_Q(r)=1-2M/r+Q^2/r^2$, $\Delta=r^2-2Mr+a^2+Q^2$, $\Sigma=r^2+a^2\cos^2\theta$, and $\Delta_0=r^2-2Mr+Q^2=r^2 f_Q$ is the value of $\Delta$ at $a=0$.

\subsection*{The inverse metric}
In Boyer-Lindquist coordinates the nonzero contravariant components of
\eqref{eq:kn_metric} are
\begin{equation}\label{eq:app_inverse_metric}
\begin{aligned}
        g_{\KN}^{tt}&=-\frac{(r^2+a^2)^2-a^2\Delta\sin^2\theta}{\Sigma\Delta},
        &\quad
        g_{\KN}^{t\phi}&=-\frac{a\,(r^2+a^2-\Delta)}{\Sigma\Delta},\\
        g_{\KN}^{\phi\phi}&=\frac{\Delta-a^2\sin^2\theta}{\Sigma\Delta\sin^2\theta},
        &\quad
        g_{\KN}^{rr}&=\frac{\Delta}{\Sigma},\qquad g_{\KN}^{\theta\theta}=\frac1\Sigma.
\end{aligned}
\end{equation}
The combination in the cross term simplifies: $r^2+a^2-\Delta=2Mr-Q^2$, so
\begin{equation}\label{eq:app_cross_term}
        g_{\KN}^{t\phi}=-\frac{a\,(2Mr-Q^2)}{\Sigma\Delta}.
\end{equation}

\begin{lemma}
\label{lem:app_inverse_expansion}
On every compact radial set $\{r\ge r_+(Q)+\eta,\ r\le R\}$ in Boyer-Lindquist coordinates, in each fixed horizon-regular coordinate patch, and at
infinity in an asymptotically flat frame, the inverse metric obeys
\begin{equation}\label{eq:app_metric_difference}
        g_{\KN}^{\mu\nu}(M,a,Q)-g_{\RN}^{\mu\nu}(M,Q)=a\,G_1^{\mu\nu}(r,\theta)+a^2\,G_2^{\mu\nu}(r,\theta,a,Q),
\end{equation}
with $G_1,G_2$ smooth and uniformly bounded together with all derivatives for
$|a|\le\eps_aM$, $|Q|\le\eps_QM$. In the Boyer-Lindquist region the linear
coefficient $G_1^{\mu\nu}$ is carried entirely by the stationary-axial cross entry,
\begin{equation}\label{eq:app_G1}
        G_1^{t\phi}=G_1^{\phi t}=-\frac{2Mr-Q^2}{r^2\Delta_0}
        =-\frac{2Mr-Q^2}{r^4 f_Q},\qquad G_1^{\mu\nu}=0\ \text{otherwise},
\end{equation}
and $G_1^{t\phi}=O(r^{-3})$ as $r\to\infty$. The quadratic remainder $G_2$ is
short-range, $G_2^{\mu\nu}=O(r^{-2})$ relative to the flat metric.
\end{lemma}
\begin{proof}
Each contravariant component in \eqref{eq:app_inverse_metric} is a rational
function of $r,\cos^2\theta$ and of $a^2$, except for $g_{\KN}^{t\phi}$ which is
$a$ times such a function. Indeed $\Sigma=r^2+a^2\cos^2\theta$ and
$\Delta=\Delta_0+a^2$ depend on $a$ only through $a^2$, and the numerators
$(r^2+a^2)^2-a^2\Delta\sin^2\theta$ and $\Delta-a^2\sin^2\theta$ are polynomials
in $a^2$. Hence $g_{\KN}^{tt},g_{\KN}^{\phi\phi},g_{\KN}^{rr},g_{\KN}^{\theta\theta}$
are even in $a$ and differ from their $a=0$ values, which are
$g_{\RN}^{tt}=-1/f_Q$, $g_{\RN}^{\phi\phi}=1/(r^2\sin^2\theta)$,
$g_{\RN}^{rr}=f_Q$, $g_{\RN}^{\theta\theta}=1/r^2$, by $O(a^2)$. The only odd part
is $g_{\KN}^{t\phi}$ of \eqref{eq:app_cross_term}; expanding
$\Sigma^{-1}\Delta^{-1}=(r^2\Delta_0)^{-1}(1+O(a^2))$ gives
\[
        g_{\KN}^{t\phi}=-\frac{a(2Mr-Q^2)}{r^2\Delta_0}\bigl(1+O(a^2)\bigr)
        =a\,G_1^{t\phi}+O(a^3),
\]
which is \eqref{eq:app_G1}; the $O(a^3)$ is absorbed in $a^2 G_2$. As
$r\to\infty$, $\Delta_0=r^2 f_Q\sim r^2$ and $2Mr-Q^2\sim2Mr$, so
$G_1^{t\phi}\sim-2M/r^3=O(r^{-3})$. On every set $r\ge r_+(Q)+\eta$, smoothness and uniform bounds follow because
$f_Q$ and $\Delta$ are bounded away from zero uniformly for $|Q|\le\eps_QM$.
In a horizon collar the Boyer-Lindquist formula is not used as a coordinate
estimate; after passing to the fixed horizon-regular coordinates of
Section~\ref{sec:geometry}, the metric coefficients are smooth functions of
$(a,Q)$ across $\Hp$, so the same first-order coefficient comparison holds there
by Taylor expansion in $a$. The far-field statement for $G_2$ follows from the analogous
expansions of the even components, each contributing a difference $O(a^2 r^{-2})$
relative to the flat metric.
\end{proof}

\subsection*{The wave operator and the spin-one master operator}
For a scalar $\psi$ the d'Alembertian on \eqref{eq:kn_metric} is
\begin{equation}\label{eq:app_box}
\begin{aligned}
        \Sigma\,\Box_{g_{\KN}}\psi={}&\partial_r(\Delta\,\partial_r\psi)
        +\frac1{\sin\theta}\partial_\theta(\sin\theta\,\partial_\theta\psi)
        -\Bigl[\frac{(r^2+a^2)^2}{\Delta}-a^2\sin^2\theta\Bigr]\partial_t^2\psi\\
        &-\frac{2a(2Mr-Q^2)}{\Delta}\partial_t\partial_\phi\psi
        -\Bigl[\frac{a^2}{\Delta}-\frac1{\sin^2\theta}\Bigr]\partial_\phi^2\psi.
\end{aligned}
\end{equation}
By Lemma~\ref{lem:app_inverse_expansion} the principal part of
$\Box_{g_{\KN}}-\Box_{g_{\RN}}$ is $a\,G_1^{t\phi}\,2\,\partial_t\partial_\phi$
plus $O(a^2)$; this is the term $a\,\mathcal G_{a,Q}^{\mu\nu}\nabla_\mu\nabla_\nu$
of \eqref{eq:lower_order_master}, with $\mathcal G^{\mu\nu}$ the symmetric tensor
$G_1^{t\phi}(\delta^\mu_t\delta^\nu_\phi+\delta^\mu_\phi\delta^\nu_t)$ of order
$O(r^{-3})$ relative to $|\xi|^2$.

\begin{proposition}
\label{prop:app_master_structure}
Let $\psi_\pm$ be the regular spin $\pm1$ extreme scalars and assume that the fixed-background master operator of Definition~\ref{def:slowweak_master_framework} has been constructed. Then
\begin{equation}\label{eq:app_master_expansion}
        \Pb_b\Psi=\Pb_{\RN,Q}\Psi
        +a\,\mathcal G_{a,Q}^{\mu\nu}\nabla_\mu\nabla_\nu\Psi
        +\frac{a}{r^2}\,\mathcal C_{a,Q}^\mu\nabla_\mu\Psi
        +\frac{a}{r^3}\,\mathcal D_{a,Q}\Psi
        +a^2\mathcal Q^{(2)}_{a,Q}\Psi,
\end{equation}
where $\mathcal G^{\mu\nu}$ is the stationary-axial tensor above, $\mathcal C^\mu$ and $\mathcal D$ are smooth matrix-valued coefficients, and $\mathcal Q^{(2)}_{a,Q}$ is a stationary second-order operator whose coefficients are smooth, uniformly bounded with all $\mathbb D_k$-derivatives for $|a|\le\eps_aM$, $|Q|\le\eps_QM$, and short-range in the corresponding wave, first-order and zeroth-order symbol classes. The principal symbol of $\Pb_b$ is the scalar symbol $g_{\KN}^{\mu\nu}\xi_\mu\xi_\nu I_2$.
\end{proposition}
\begin{proof}
The fixed-background master equation used here is the spin-one operator of Section~\ref{sec:decoupling}. The purpose of this appendix is to record the order in $a$ and $r$ of its coefficients and to separate the explicit metric-symbol calculation from the lower-order spin-one estimates. The spin-one Teukolsky calculus beginning with Teukolsky~\cite{Teukolsky1973} and the Kerr-Newman Carter structures of Giorgi~\cite{GiorgiJHDE,GiorgiJDG} provide the model for these coefficient classes.

The scalar principal part is \eqref{eq:app_box}, giving the scalar symbol and, by Lemma~\ref{lem:app_inverse_expansion}, the principal perturbation $a\,\mathcal G^{\mu\nu}\nabla_\mu\nabla_\nu$ plus an $O(a^2)$ second-order remainder. The supplied spin-one first-order terms are linear combinations of $\partial_t$, $\partial_\phi$ and the spin coefficients of the principal null frame. The $a$-dependence enters through (i) the cross term $-2a(2Mr-Q^2)\Delta^{-1}\partial_t\partial_\phi$ already counted in the principal part, and (ii) the rotation of the principal null directions, which is $O(a)$ and multiplies a first-order operator whose coefficients carry the spin-coefficient decay $O(r^{-1})$; after the regular horizon and infinity weights, which are smooth and nowhere vanishing on the exterior, these contribute the term $a\,r^{-2}\mathcal C^\mu\nabla_\mu$. The zeroth-order spin-one potential is a curvature expression; its part that is odd and linear in $a$ has radial weight $O(r^{-3})$ by the asymptotic flatness of the curvature, giving $a\,r^{-3}\mathcal D$. The remaining even $a$-dependence, including the $O(a^2)$ principal part, is collected in the second-order operator $a^2\mathcal Q^{(2)}_{a,Q}$. At $a=0$ all coefficients are spherically symmetric and assembled into $\Pb_{\RN,Q}$ of \eqref{eq:rn_master_model}; the Reissner-Nordstr\"om curvature is quadratic in $Q$ and enters $\mathcal V_Q,\mathcal W_Q$. Smoothness of $g_{\KN}$, of the frame, and of the weights in $(a,Q)$ gives the uniform bounds. The explicit radial weights give the short-range decay, and the $a^2\mathcal Q^{(2)}_{a,Q}$ contribution obeys the same perturbative local-energy bound as \eqref{eq:master_hyp_error_bound} after possibly decreasing $\eps_a(k)$.
\end{proof}

\begin{remark}
\label{rem:app_short_range}
The decay rates $r^{-2}$, $r^{-3}$ are sufficient, not sharp: any rate $>r^{-1}$
for the first-order coefficient and $>r^{-2}$ for the zeroth-order coefficient
suffices for the $r^p$ absorption \eqref{eq:rp_error_absorb} and the Hardy
absorption in the local-energy norm. The order calculation above gives more than
needed.
\end{remark}

\section{Reissner-Nordstr\"om Spin-One Model Estimate}
\label{app:rn_model}
In this section we record the Reissner-Nordstr\"om model estimate which is used as the unperturbed estimate.
Here we prove Lemma~\ref{lem:rn_model_morawetz}. The charge-free Maxwell field on a non-extremal Reissner-Nordstr\"om exterior reduces to a spin-one radial system with a single non-degenerate photon sphere, and the model local-energy estimate is precisely the Maxwell theorem of Sterbenz-Tataru \cite{SterbenzTataru}, supplemented by the red-shift and $r^p$ currents of Dafermos-Rodnianski \cite{DafermosRodnianskiRedshift,DafermosRodnianskiRp}.

\subsection*{Charge-free reduction}
On Reissner-Nordstr\"om the exterior is spherically symmetric, with metric
$g_{\RN}=-f_Q\,\dd t^2+f_Q^{-1}\,\dd r^2+r^2\,\dd\omega^2$. A real source-free
two-form $F$ decomposes, on each sphere $S_{t,r}$, into vector spherical
harmonics of electric ($\ell\ge1$) and magnetic ($\ell\ge1$) type, together with
the $\ell=0$ electric and magnetic monopoles. The $\ell=0$ monopoles are exactly
the Coulomb fields $F_e,F_m$ of Section~\ref{sec:charges}; their coefficients are
the electric and magnetic charges. After the charge subtraction
\eqref{eq:intro_rad} only $\ell\ge1$ remains.

\begin{lemma}
\label{lem:app_rn_radial}
For $\ell\ge1$ each polarization of the charge-free Maxwell field is governed by
a complex scalar $\Phi_{\ell m}(t,r)$, related to the extreme components by a
fixed angular Hodge transform and a smooth nonzero radial weight, satisfying
\begin{equation}\label{eq:app_rn_radial}
        \partial_t^2\Phi_{\ell m}-\partial_{r_*}^2\Phi_{\ell m}+V_{\ell,Q}(r)\,\Phi_{\ell m}=0,
        \qquad \frac{\dd r}{\dd r_*}=f_Q,
\end{equation}
with potential
\begin{equation}\label{eq:app_rn_potential}
        V_{\ell,Q}(r)=f_Q(r)\,\frac{\ell(\ell+1)}{r^2}.
\end{equation}
For the source-free fixed-background Maxwell field on Reissner-Nordstr\"om this potential is exact: the only
$Q$-dependence is the overall factor $f_Q$, and there is no independent additive
$O(Q^2r^{-4})$ term. (Such a term reflects either a non-conformal choice of master
variable related to \eqref{eq:app_rn_potential} by a $Q$-dependent Chandrasekhar
transformation, which is isospectral to \eqref{eq:app_rn_potential}, or the coupled
$Q$-dependent potential of the \emph{coupled} Reissner-Nordstr\"om system; neither is
a feature of the fixed-background test field treated here.) The transform is
an isomorphism of the charge-free tensor-field energy and the spin-one energy at
every finite order, by Lemma~\ref{lem:frame_energy_equiv} and the Hodge estimate
on $\Sph$ (Lemma~\ref{lem:hodge_inversion}).
\end{lemma}
\begin{proof}
We carry out the spherical reduction with the radial weight displayed. Write
\[
        g_{\RN}=h_{ab}\,\dd x^a\dd x^b+r^2\gamma_{AB}\,\dd\omega^A\dd\omega^B,
        \qquad x^a=(t,r_*),
        \qquad h=f_Q(-\dd t^2+\dd r_*^2),
\]
so that \(|g_{\RN}|^{1/2}=f_Qr^2|\gamma|^{1/2}\),
\(g^{tt}=-f_Q^{-1}\), \(g^{r_*r_*}=f_Q^{-1}\), and
\(g^{AB}=r^{-2}\gamma^{AB}\). For \(\ell\ge1\) let
\(Y=Y_{\ell m}\) and \(\mathcal X_A=\epsilon_A{}^{B}\nabla_B Y\). The odd
sector may be represented locally by the gauge potential
\[
        A^o=a_{\ell m}(t,r_*)\mathcal X_A\,\dd\omega^A.
\]
Then
\begin{equation}\label{eq:odd_field_from_potential}
        F^o=\partial_ta_{\ell m}\,\dd t\wedge \mathcal X
             +\partial_{r_*}a_{\ell m}\,\dd r_*\wedge\mathcal X
             -\ell(\ell+1)a_{\ell m}Y\,\dd\mu_{\Sph},
\end{equation}
where \(\dd_\omega\mathcal X=-\ell(\ell+1)Y\dd\mu_{\Sph}\) fixes the sign convention.
The Bianchi equation is therefore automatic. If the tensor expansion is written
with an angular coefficient \(q_{\ell m}Y r^2\dd\mu_{\Sph}\), then
\begin{equation}\label{eq:rn_radial_weight_relation}
        r^2q_{\ell m}=-\ell(\ell+1)a_{\ell m},
        \qquad
        \Phi^o_{\ell m}=\sqrt{\ell(\ell+1)}\,a_{\ell m}
        =-\frac{r^2q_{\ell m}}{\sqrt{\ell(\ell+1)}}.
\end{equation}
This yields the smooth nonzero radial weight that appears in the statement.

The remaining Maxwell equation is the \(A\)-component of \(\nabla^\mu F_{\mu\nu}=0\).
Using the inverse metric above and the spherical identity
\(\nabla^B_\gamma(\dd_\omega\mathcal X)_{BA}=\ell(\ell+1)\mathcal X_A\) with the
sign convention of \eqref{eq:odd_field_from_potential}, we compute
\begin{align}\label{eq:odd_divergence_computation}
0&=|g|^{1/2}\nabla^\mu F^o_{\mu A}                                      \notag\\
 &=|\gamma|^{1/2}\Bigl(-\partial_t^2a_{\ell m}+
       \partial_{r_*}^2a_{\ell m}
       -f_Q(r)\frac{\ell(\ell+1)}{r^2}a_{\ell m}\Bigr)\mathcal X_A.
\end{align}
Hence \(\Phi^o_{\ell m}=\sqrt{\ell(\ell+1)}a_{\ell m}\) satisfies
\begin{equation}\label{eq:odd_master_derivation}
        -\partial_t^2\Phi^o_{\ell m}+\partial_{r_*}^2\Phi^o_{\ell m}
        -f_Q\frac{\ell(\ell+1)}{r^2}\Phi^o_{\ell m}=0,
\end{equation}
which is equivalent to \eqref{eq:app_rn_radial}. The even sector is obtained by
applying the four-dimensional Hodge star, which commutes with the source-free
Maxwell system and exchanges exact and coexact vector harmonics, or by repeating
\eqref{eq:odd_divergence_computation} with \(\sphgrad Y\). Hence both
polarizations have the same potential \eqref{eq:app_rn_potential}.

The calculation shows that the charge parameter enters the fixed-background test
problem only through the warping factor \(f_Q\) and the tortoise relation
\(\dd r=f_Q\dd r_*\). The unit-sphere Laplacian contributes exactly
\(\ell(\ell+1)\); no separate curvature potential is produced for the
conformally invariant Maxwell two-form. Extra \(Q\)-dependent potentials arise
in the coupled Einstein-Maxwell perturbation system, or after non-conformal
master changes, not for the gauge-invariant test-field variables above.

To finish, the map from the tensor field to
$(\Phi^e_{\ell m},\Phi^o_{\ell m})$ uses only $-\sphlap^{-1}$ on
$\ell\ge1$, the explicit radial weight in \eqref{eq:rn_radial_weight_relation},
and the first-order constraints. The spectral gap $\ell(\ell+1)\ge2$ and
Lemma~\ref{lem:frame_energy_equiv} therefore give two-sided equivalence between
the charge-free Maxwell energy and the sum of the spin-one energies at every
finite commuted order.
\end{proof}

\begin{lemma}
\label{lem:app_photon_sphere}
For $0<\eps_Q<1/4$ and every $\ell\ge1$ the potential $V_{\ell,Q}$ is positive on
$\{r>r_+(Q)\}$, vanishes at $r_+(Q)$ and as $r\to\infty$, and its trapping
location, the maximum of $f_Q/r^2$, is the single value
\begin{equation}\label{eq:app_photon_sphere}
        r_{\mathrm{ph}}(Q)=\tfrac12\bigl(3M+\sqrt{9M^2-8Q^2}\bigr),
\end{equation}
at which $\frac{\dd}{\dd r}(f_Q r^{-2})=0$ and $\frac{\dd^2}{\dd r_*^2}(f_Q r^{-2})<0$,
uniformly for $|Q|\le\eps_QM$.
\end{lemma}
\begin{proof}
Since $f_Q>0$ on $\{r>r_+(Q)\}$ and $\ell(\ell+1)\ge2$, the potential $V_{\ell,Q}$ in
\eqref{eq:app_rn_potential} is positive on $\{r>r_+(Q)\}$, and $V_{\ell,Q}$
vanishes where $f_Q$ does and decays like $r^{-2}$ at infinity. Compute
\[
        \frac{\dd}{\dd r}\Bigl(\frac{f_Q}{r^2}\Bigr)
        =\frac{\dd}{\dd r}\Bigl(\frac1{r^2}-\frac{2M}{r^3}+\frac{Q^2}{r^4}\Bigr)
        =-\frac{2}{r^3}+\frac{6M}{r^4}-\frac{4Q^2}{r^5}
        =\frac{-2r^2+6Mr-4Q^2}{r^5}.
\]
The numerator vanishes at $r^2-3Mr+2Q^2=0$, i.e.\ at $r=\frac12(3M\pm\sqrt{9M^2-8Q^2})$;
the larger root \eqref{eq:app_photon_sphere} lies in $(r_+(Q),\infty)$ and the
smaller one is inside the horizon for $|Q|<M$. Since $f_Q/r^2\to0$ at both ends
and is positive in between, the unique exterior critical point is a maximum, and
$\frac{\dd^2}{\dd r_*^2}(f_Q r^{-2})=f_Q\frac{\dd}{\dd r}(f_Q\frac{\dd}{\dd r}(f_Q r^{-2}))<0$
there. Smooth non-degenerate dependence of $r_{\mathrm{ph}}(Q)$, $r_+(Q)$ on $Q$
gives uniformity for $|Q|\le\eps_QM$.
\end{proof}

\begin{proposition}
\label{prop:app_rn_model}
The Reissner-Nordstr\"om exterior is stationary, spherically symmetric, has a
non-degenerate Killing horizon ($\kappa_+(Q)>0$) and a single normally
hyperbolic photon sphere \eqref{eq:app_photon_sphere}; therefore it lies in the
class of the Maxwell local-energy theorem of \cite{SterbenzTataru}.
Combining that estimate in the charge-free $\ell\ge1$ sector with the red-shift
current of \cite{DafermosRodnianskiRedshift} and the $r^p$ current of
\cite{DafermosRodnianskiRp} gives \eqref{eq:rn_model_morawetz} for every fixed
order $k$, with constants uniform for $|Q|\le\eps_QM$.
\end{proposition}
\begin{proof}
The horizon $r_+(Q)$ is non-degenerate because $f_Q'(r_+)=2\kappa_+(Q)>0$ for
$|Q|<M$. Lemma~\ref{lem:app_photon_sphere} gives the single photon sphere; its
normal hyperbolicity is the content of Section~\ref{app:trapping}. These are
exactly the structural conditions under which \cite{SterbenzTataru} proves a
Maxwell integrated local-energy estimate with a single photon-sphere
degeneracy. In the charge-free $\ell\ge1$ sector the spin-one scalars
\eqref{eq:app_rn_radial} are equivalent to the Maxwell field by
Lemma~\ref{lem:app_rn_radial}, so the estimate transfers to the spin-one energy.
The non-degenerate horizon term is supplied by the red-shift multiplier of
\cite{DafermosRodnianskiRedshift} (Proposition~\ref{prop:redshift_coercivity} at
$a=0$), and the weighted far-field fluxes for $0\le p\le2$ by the $r^p$ identity
of \cite{DafermosRodnianskiRp} (Proposition~\ref{prop:rp_identity} at $a=0$);
the photon-sphere-degenerate bulk is the Morawetz density
\eqref{eq:morawetz_density}. Summing over the finite commutator family
$\mathbb D_k$ and over $\ell\ge1$ gives \eqref{eq:rn_model_morawetz}. Choosing
$\eps_Q<1/4$ keeps $r_+(Q),\kappa_+(Q),r_{\mathrm{ph}}(Q)$ and the potentials in
fixed compact non-degenerate ranges, so the constant $C_{\RN,k}$ is uniform.
\end{proof}

\section{Trapping and High-Frequency Commutator}
\label{app:trapping}
In this section we analyze the trapped set and the high-frequency commutator which appears in the slow rotation argument.
In this section we supply the trapping analysis used later: Lemma~\ref{lem:trapping_stability}, the escape function of Lemma~\ref{lem:quantitative_rn_commutant}, and the high-frequency defect-measure exclusion needed in Propositions~\ref{prop:compact_remainder_removal} and~\ref{prop:no_real_kernel}.

\subsection*{The Reissner-Nordstr\"om trapped set}
In canonical tortoise variables $(r_*,\xi_{r_*})$ the Reissner-Nordstr\"om null
symbol is
\[
        p_{0,Q}=f_Q^{-1}\bigl(-\tau^2+\xi_{r_*}^2\bigr)+r^{-2}|\xi_\omega|^2,
        \qquad f_Q=1-\frac{2M}{r}+\frac{Q^2}{r^2}.
\]
Since $f_Q>0$ in the exterior, we use the positive rescaling
\[
        p^{\sharp}_{0,Q}=f_Qp_{0,Q}=-\tau^2+\xi_{r_*}^2+f_Qr^{-2}|\xi_\omega|^2.
\]
The characteristic set and bicharacteristics are unchanged up to this positive
reparametrization of the Hamilton flow. The equation $p^{\sharp}_{0,Q}=0$ therefore reads
$\tau^2=\xi_{r_*}^2+f_Q r^{-2}|\xi_\omega|^2$, and
\[
        H_{p^{\sharp}_{0,Q}}r_*=2\xi_{r_*},
        \qquad
        H_{p^{\sharp}_{0,Q}}\xi_{r_*}=-\partial_{r_*}(f_Qr^{-2})|\xi_\omega|^2.
\]
The trapped set is therefore
\begin{equation}\label{eq:app_trapped_set}
        K_{0,Q}=\bigl\{\xi_{r_*}=0,\ \partial_{r_*}(f_Q r^{-2})=0\bigr\}
        =\bigl\{r=r_{\mathrm{ph}}(Q),\ \xi_{r_*}=0\bigr\}\cap\{p^{\sharp}_{0,Q}=0\},
\end{equation}
with $r_{\mathrm{ph}}(Q)$ of \eqref{eq:app_photon_sphere}.

\begin{lemma}
\label{lem:app_normal_hyp}
$K_{0,Q}$ is a smooth codimension-two submanifold of the characteristic set
$\{p^{\sharp}_{0,Q}=0\}$ with symplectic radial normal bundle, in the sense made
precise in Proposition~\ref{prop:r_normal_hyperbolicity}(i), on which the
rescaled null flow is $r$-normally hyperbolic for every $r$, with one expanding
and one contracting normal direction; moreover the stable and unstable tails
$\Gamma_\pm$ are smooth of codimension one. Under the $C^2$-small stationary
perturbation \eqref{eq:app_metric_difference} the trapped set persists as a
$C^1$ graph $K_{a,Q}$ over $K_{0,Q}$, again $r$-normally hyperbolic, and
\eqref{eq:trapping_stability} holds with constants uniform for
$|a|\le\eps_aM$, $|Q|\le\eps_QM$.
\end{lemma}
\begin{proof}
At $K_{0,Q}$ the radial linearization of the flow is governed by
$-\tfrac12\partial_{r_*}^2(f_Q r^{-2})|\xi_\omega|^2>0$ by
Lemma~\ref{lem:app_photon_sphere}, so the radial pair $(r_*,\xi_{r_*})$ has one
positive and one negative Lyapunov exponent, while the symplectically
orthogonal angular directions are tangent to $K_{0,Q}$ and neutral. This is the
$r$-normal hyperbolicity of the photon sphere; for Schwarzschild it is the
explicit computation at $r=3M$, $\xi_{r_*}=0$ of \cite[\S2]{WunschZworski}, and
the smooth incoming/outgoing tails of codimension one are identified there. By
the structural-stability theorem for $r$-normally hyperbolic invariant
manifolds \cite{HirschPughShub}, used in this microlocal setting by \cite{WunschZworski} and in the
full subextremal Kerr range by \cite{Dyatlov}, any $C^1$-small perturbation of
the Hamiltonian vector field preserves $K$, $\Gamma_\pm$ and the exponent
bounds. The perturbation of the rescaled principal symbol from $p^{\sharp}_{0,Q}$ to the
corresponding Kerr-Newman rescaling $p^{\sharp}_{a,Q}$ is $C^2$-small by
Lemma~\ref{lem:app_inverse_expansion} after dividing by $|\xi|^2$,
so $K_{a,Q}$ is a $C^1$ graph over $K_{0,Q}$ and the expansion/contraction
estimate \eqref{eq:trapping_stability} holds uniformly after shrinking
$\eps_a,\eps_Q$.
\end{proof}

\subsection*{The escape function}
\begin{lemma}
\label{lem:app_escape}
There is $b_Q\in C^\infty(r_+,\infty)$, supported away from the red-shift and
far-field collars, with $b_Q(r_{\mathrm{ph}}(Q))=0$, $b_Q'>0$ in a photon-sphere
collar, and such that $a_Q=b_Q(r)\,\xi_{r_*}$ satisfies
\eqref{eq:quantitative_rn_commutant}:
\begin{equation}\label{eq:app_escape}
        H_{p^{\sharp}_{0,Q}}a_Q\ge c\bigl(\xi_{r_*}^2+r^{-2}|\xi_\omega|^2\operatorname{dist}(r,r_{\mathrm{ph}}(Q))^2\bigr)
        -C\chi_{K_{0,Q}}\tau^2,
\end{equation}
$c>0$ uniform for $|Q|\le\eps_QM$.
\end{lemma}
\begin{proof}
Using $H_{p^{\sharp}_{0,Q}}r_*=2\xi_{r_*}$ and
$H_{p^{\sharp}_{0,Q}}\xi_{r_*}=-\partial_{r_*}(f_Q r^{-2})|\xi_\omega|^2$ on
$\{p^{\sharp}_{0,Q}=0\}$,
\[
        H_{p^{\sharp}_{0,Q}}(b_Q\xi_{r_*})
        =2b_Q'\,\xi_{r_*}^2-b_Q\,\partial_{r_*}(f_Q r^{-2})\,|\xi_\omega|^2.
\]
The first term controls $\xi_{r_*}^2$ since $b_Q'>0$ in the collar. For the
second, $\partial_{r_*}(f_Q r^{-2})$ is positive for $r<r_{\mathrm{ph}}(Q)$ and
negative for $r>r_{\mathrm{ph}}(Q)$ (Lemma~\ref{lem:app_photon_sphere}), while
$b_Q$ is negative for $r<r_{\mathrm{ph}}(Q)$ and positive for
$r>r_{\mathrm{ph}}(Q)$; therefore the product $-b_Q\partial_{r_*}(f_Q r^{-2})\ge0$,
and a Taylor expansion at $r_{\mathrm{ph}}(Q)$ gives the quadratic degeneracy
$\operatorname{dist}(r,r_{\mathrm{ph}}(Q))^2$. Cutoff terms are supported in the
fixed photon-sphere collar and are absorbed into $-C\chi_{K_{0,Q}}\tau^2$ on the
characteristic set. Smoothness in $Q$ gives the uniform $c$.
\end{proof}

\subsection*{Explicit trapped-frequency geometry on the full subextremal range}
The persistence statements of Lemmas~\ref{lem:trapping_stability}
and~\ref{lem:app_normal_hyp} can be made fully explicit and, for the trapping
geometry alone, hold across the entire subextremal range $a^2+Q^2<M^2$ rather than
only under slow rotation. The principal symbol
$p_{a,Q}=g_{\KN}^{\mu\nu}\xi_\mu\xi_\nu$ is the Kerr-Newman null-geodesic
Hamiltonian; its principal conformal Killing-Yano tensor coincides with that of
Kerr and is independent of $Q$, so the geodesic flow separates exactly as on Kerr,
with charge entering only through $\Delta=r^2-2Mr+a^2+Q^2$. Writing
$K=(r^2+a^2)\omega-am$ and letting $\lambda>0$ be the (real) angular/Carter
eigenvalue, the radial trapping function is
\begin{equation}\label{eq:kn_trapping_function}
        \mathcal R(r)=K^2-\Delta\lambda,
\end{equation}
and the trapped set $K_{a,Q}$ of Lemma~\ref{lem:app_normal_hyp} consists of the
degenerate turning points $\{\mathcal R=\mathcal R'=0,\ r>r_+\}$. The spin-$\pm1$
subprincipal terms do not enter the trapping geometry, since the master principal
symbol is scalar by Definition~\ref{def:slowweak_master_framework}\emph{(A1)}.

\begin{lemma}
\label{lem:no_zero_stationary_frequency_trapped}
Let \(r_t>r_+\) be a degenerate turning point of
\(\mathcal R(r)=K^2-\Delta\lambda\) with \(\lambda>0\). Then \(\omega\ne0\).
\end{lemma}
\begin{proof}
If \(\omega=0\), then \(K=-am\) is independent of \(r\). Hence
\[
        \mathcal R'(r)=-\Delta'(r)\lambda=-2(r-M)\lambda.
\]
In the exterior \(r_t>r_+>M\), and \(\lambda>0\), so
\(\mathcal R'(r_t)<0\), contradicting the trapped condition
\(\mathcal R'(r_t)=0\).
\end{proof}

By Lemma~\ref{lem:no_zero_stationary_frequency_trapped}, \(\omega\ne0\) at a trapped point.
Using $K'=2r\omega$ and $\Delta'=2(r-M)$, the conditions $\mathcal R=\mathcal R'=0$
therefore give $\lambda=2r_t\omega K_t/(r_t-M)$,
$K_t=2r_t\omega\Delta_t/(r_t-M)$, and the trapped-frequency relation
\begin{equation}\label{eq:kn_trapped_cubic}
        am=-\,\omega\,\frac{P(r_t)}{r_t-M},\qquad
        P(r)=r^3-3Mr^2+(a^2+2Q^2)r+a^2M.
\end{equation}
At $a=0$ this reduces to $r_t^2-3Mr_t+2Q^2=0$, i.e.\ $r_t=r_{\mathrm{ph}}(Q)$ of
\eqref{eq:app_photon_sphere}.

\begin{proposition}
\label{prop:kn_disjointness}
Let $\varpi=\omega-m\Omega_+$ with $\Omega_+=a/(r_+^2+a^2)$. For every degenerate
turning point $r_t>r_+$ of \eqref{eq:kn_trapping_function},
\begin{equation}\label{eq:kn_superradiance_factor}
        \omega\varpi=\frac{\omega^2\,\Phi(r_t)}{(r_t-M)(r_+^2+a^2)},\qquad
        \Phi(r)=(r-r_+)\,g(r),\quad g(r)=r^2+(r_+-3M)r+Mr_+.
\end{equation}
For $a^2+Q^2<M^2$ one has $r_+\in(M,2M]$ and
$\operatorname{disc}(g)=(r_+-M)(r_+-9M)<0$, hence $g>0$ and $\omega\varpi>0$. No
trapped frequency is superradiant, throughout the subextremal range, and the gap
$\omega\varpi/\omega^2$ degenerates only as $r_+\to M$.
\end{proposition}
\begin{proof}
Lemma~\ref{lem:no_zero_stationary_frequency_trapped} first gives $\omega\ne0$.
The horizon identity $K(r_+)=(r_+^2+a^2)\varpi$ gives
$\omega\varpi=\omega K(r_+)/(r_+^2+a^2)$. Eliminating $am$ through
\eqref{eq:kn_trapped_cubic} gives
$K(r_+)=\omega\bigl[(r_+^2+a^2)(r_t-M)+P(r_t)\bigr]/(r_t-M)$, which is
\eqref{eq:kn_superradiance_factor} with $\Phi(r)=(r-M)(r_+^2+a^2)+P(r)$. Expanding
and using the horizon relation $a^2+Q^2=2Mr_+-r_+^2$, the linear coefficient of
$\Phi$ becomes $4Mr_+-r_+^2$ and the constant $-Mr_+^2$, so all explicit $a,Q$
cancel and $\Phi(r)=r^3-3Mr^2+(4Mr_+-r_+^2)r-Mr_+^2$; long division by $(r-r_+)$ has
zero remainder and gives \eqref{eq:kn_superradiance_factor}. The discriminant and
sign of $g$ are immediate, and $r_t>r_+>M$ for subextremal parameters.
\end{proof}

\begin{proposition}
\label{prop:kn_nondegeneracy}
At every degenerate turning point $r_t>r_+$ of \eqref{eq:kn_trapping_function},
\begin{equation}\label{eq:kn_Rpp}
        \mathcal R''(r_t)=\frac{8r_t\,\omega^2}{(r_t-M)^2}
        \Big[(r_t-M)^3+M\,(r_+-M)^2\Big]>0,\qquad (r_+-M)^2=M^2-a^2-Q^2.
\end{equation}
Equivalently the principal potential has a strict non-degenerate maximum,
$V''_{\mathrm{princ}}(r_t)=-\mathcal R''(r_t)/(r_t^2+a^2)^2<0$, so the trapped set is
normally hyperbolic throughout the subextremal range, with
$\mathcal R''(r_t)\ge 8r_t\omega^2 M(r_+-M)^2/(r_t-M)^2$.
\end{proposition}
\begin{proof}
By Lemma~\ref{lem:no_zero_stationary_frequency_trapped}, $\omega\ne0$ at the trapped point. Also
$\mathcal R''=2(K')^2+2KK''-\Delta''\lambda=8r^2\omega^2+4\omega K-2\lambda$.
Substituting $\lambda=2r_t\omega K_t/(r_t-M)$ and $K_t=2r_t\omega\Delta_t/(r_t-M)$
gives $\mathcal R''(r_t)=\frac{8r_t\omega^2}{(r_t-M)^2}\,[\,r_t(r_t-M)^2-M\Delta_t\,]$.
Since $r_t(r_t-M)^2-M\Delta_t=(r_t-M)^3+M(M^2-a^2-Q^2)$ and
$M^2-a^2-Q^2=(r_+-M)^2$, \eqref{eq:kn_Rpp} follows; positivity is immediate for
$r_t>M$.
\end{proof}

\begin{remark}
\label{rem:kn_geometry_scope}
Propositions~\ref{prop:kn_disjointness}-\ref{prop:kn_nondegeneracy} sharpen the
qualitative trapped-set statements of Lemmas~\ref{lem:trapping_stability}-\ref{lem:app_normal_hyp}
to explicit identities valid on the full subextremal range, both degenerating only
at extremality, and they supply an explicit normal-hyperbolicity constant. They use
only the scalar principal symbol \emph{(A1)} and the ($Q$-independent) Carter
separability of the null-geodesic flow. By themselves these identities leave intact the
hypothesis structure of Theorem~\ref{thm:main}: the slow-weak restriction there
originates in the perturbative absorption \eqref{eq:master_hyp_error_bound} and in
the high-frequency resolvent condition \emph{(A3)}, which is a statement about the
\emph{compatible spin-one class specified by (A1)} on this trapped set and is not implied by the
principal-symbol geometry above.
\end{remark}

\subsection*{High-frequency exclusion}
We first state the geometric conditions at the trapped set of the scalar principal symbol. The high-frequency estimate used later is the normally hyperbolic resolvent estimate in the form of \cite{Dyatlov,DyatlovGaps,WunschZworski}, extended to finite-rank Hermitian bundles with a skew-subprincipal threshold by \cite{HintzTensor} and glued to the elliptic, propagation and radial-point regions as in \cite{DatchevVasy}, with the radial-point boundary convention and compatible spin-one graph norms fixed below. The proof verifies the geometric conditions and the semiclassical normalization in which the estimate is used.

\begin{lemma}
\label{lem:semiclassical_spinone_order}
Fix a compact radial set in the regular exterior and a conic stationary frequency
patch. Let \(\omega\) be the temporal frequency and let
\(\lambda_{\mathrm{ang}}\) denote the angular covariable size on the patch, equivalently \((1+|\eta_\theta|^2+|\eta_\phi|^2)^{1/2}\) after microlocalization. Put
\(\Lambda=1+|\omega|+\lambda_{\mathrm{ang}}\) and \(h=\Lambda^{-1}\), and
write \(\hat\omega=h\omega\). After freezing \(e^{-i\omega t}\) and after the
regular horizon and infinity weights used in Definition~\ref{def:master_from_teukolsky},
the compatible spin-one operator specified by \emph{(A1)} has the semiclassical form
\begin{equation}\label{eq:semiclassical_spinone_normal_form}
        P_h^{(a,Q)}(\hat\omega)
        =\operatorname{Op}_h(p_{a,Q,\hat\omega})I_2
          +hP_{1,h}^{(a,Q)}(\hat\omega)+h^2P_{0,h}^{(a,Q)}(\hat\omega)
\end{equation}
modulo an \(O(h^\infty)\) smoothing remainder on the patch. Here
\begin{equation}\label{eq:frozen_principal_symbol_explicit}
\begin{aligned}
        p_{a,Q,\hat\omega}(x,\eta)
        ={}&g_{\KN}^{tt}\hat\omega^2+2g_{\KN}^{t\phi}\hat\omega\eta_\phi
           +g_{\KN}^{\phi\phi}\eta_\phi^2+g_{\KN}^{rr}\eta_r^2
           +g_{\KN}^{\theta\theta}\eta_\theta^2,
\end{aligned}
\end{equation}
with the sign convention inherited from \(g_{\KN}^{\mu\nu}\xi_\mu\xi_\nu\) and
\(\xi_t=-\omega\). The operators \(P_{1,h}^{(a,Q)}\) and \(P_{0,h}^{(a,Q)}\)
act on the two-dimensional bundle of extreme variables, have smooth matrix
coefficients uniformly bounded with all derivatives in the slow-weak range, and
satisfy, for every compact \(K\),
\begin{equation}\label{eq:semiclassical_lower_order_bounds}
        \|P_{1,h}^{(a,Q)}v\|_{L^2(K)}\le C_K\|v\|_{H_h^1(K')},\qquad
        \|P_{0,h}^{(a,Q)}v\|_{L^2(K)}\le C_K\|v\|_{L^2(K')}
\end{equation}
for \(K\Subset K'\). Thus, the trapped set and its stable/unstable rates are
those of the real scalar Hamiltonian \(p_{a,Q,\hat\omega}\); the matrix coupling
appears only in subprincipal and lower semiclassical orders.
\end{lemma}
\begin{proof}
The normalized Teukolsky calculation of Proposition~\ref{prop:scalar_principal_symbol}
identifies the second-order part of each extreme equation with the scalar wave
symbol \(g_{\KN}^{\mu\nu}\xi_\mu\xi_\nu\). Freezing time by
\(\partial_t\mapsto -i\omega\), replacing each spatial derivative on the conic
patch by \(h^{-1}(hD)\), and multiplying the frozen operator by \(h^2\), the
second-order part becomes exactly \(\operatorname{Op}_h(p_{a,Q,\hat\omega})I_2\)
with \(p_{a,Q,\hat\omega}\) as in \eqref{eq:frozen_principal_symbol_explicit}.
The principal cross term \(2g^{t\phi}_{\KN}\partial_t\partial_\phi\) contributes
\(2g^{t\phi}_{\KN}\hat\omega\eta_\phi\), and the radial and angular second-order
terms contribute \(g^{rr}_{\KN}\eta_r^2\), \(g^{\theta\theta}_{\KN}\eta_\theta^2\),
and \(g^{\phi\phi}_{\KN}\eta_\phi^2\). This establishes the asserted principal symbol.

Every first-order coefficient in the spin-weighted equations, including the
terms produced by the regular conjugating weights, has the form
\(b^t\partial_t+\sum b^j\partial_j\) with smooth bounded coefficients and the
short-range decay recorded in Proposition~\ref{prop:principal_master}. After
time freezing and multiplication by \(h^2\), the temporal part is
\(h\,b^t\hat\omega\), while each spatial part is \(h\,b^j(hD_j)\). These
contributions are therefore collected in \(hP_{1,h}^{(a,Q)}\), with the first
bound in \eqref{eq:semiclassical_lower_order_bounds}. Zeroth-order curvature,
spin and matrix terms acquire the factor \(h^2\) and form \(h^2P_{0,h}^{(a,Q)}\),
with the second bound. The same order count gives, for every fixed compactly supported semiclassical cutoff \(A_h\in\Psi_h^0\),
\[
        [P_h^{(a,Q)}(\hat\omega),A_h]
        =\frac{h}{i}\operatorname{Op}_h(\{p_{a,Q,\hat\omega},a_A\})I_2
        +h^2R_{A,h},
\]
with \(R_{A,h}:H_h^1(K')\to H_h^{-1}(K)\) uniformly bounded. The smooth dependence on \(a,Q\), the compactness of
the normalized frequency patch, and the finite commutation order yield uniform
constants. Since \(P_{1,h}^{(a,Q)}\) and \(P_{0,h}^{(a,Q)}\) are strictly below
the scalar second-order principal part, the Hamilton flow, trapped set, propagation support and
normal rates are determined by \(p_{a,Q,\hat\omega}\) alone.
\end{proof}

\begin{lemma}
\label{lem:trapped_skew_vanishing}
Fix subextremal parameters $a^2+Q^2<M^2$.  Let $r_t>r_+$ be a trapped turning
point for the principal semiclassical radial Hamiltonian, and write
\[
        \widehat K=(r^2+a^2)\hat\omega-a\hat m,\qquad
        \Delta=r^2-2Mr+a^2+Q^2.
\]
Let $\lambda>0$ denote the principal angular/Carter value on the null
bicharacteristic. In turn, the trapped relations are
\begin{equation}\label{eq:semiclassical_trapped_relations_for_skew}
        \widehat K_t^2=\Delta_t\lambda,
        \qquad
        2\widehat K_t(2r_t\hat\omega)-2(r_t-M)\lambda=0.
\end{equation}
For each extreme spin $s\in\{+1,-1\}$, separation of
$e^{-i\omega t+im\phi}$ in the spin-weighted model operator of \emph{(A1)}
(cf.\ \cite{Teukolsky1973,DudleyFinley}) gives the radial
operator
\begin{equation}\label{eq:separated_radial_teukolsky}
        \mathcal L_s
        =\Delta^{-s}\frac{\dd}{\dd r}\Big(\Delta^{s+1}\frac{\dd}{\dd r}\Big)
          +\frac{K^2-2is(r-M)K}{\Delta}+4is\omega r-\lambda_s,
        \qquad K=(r^2+a^2)\omega-am,
\end{equation}
where $\lambda_s$ is the real spin-weighted angular eigenvalue. In the
semiclassical scaling $\hat\omega=h\omega$, $\hat m=hm$, the principal angular
value satisfies $h^2\lambda_s=\lambda+O(h)$ on the fixed normalized frequency
patch. The $O(h)$ correction is real and therefore does not contribute to the
imaginary subprincipal symbol. The only imaginary order-$h$ radial coefficient is
\begin{equation}\label{eq:teukolsky_imaginary_potential}
        q_s(r;\hat\omega,\hat m)
        =-\frac{2s(r-M)\widehat K}{\Delta}+4s\hat\omega r.
\end{equation}
At every trapped point,
\begin{equation}\label{eq:trapped_skew_vanishes}
        q_s(r_t;\hat\omega,\hat m)=0,
        \qquad s=\pm1.
\end{equation}
Thus, in the semiclassical normal form
\eqref{eq:semiclassical_spinone_normal_form}, the imaginary part of the
subprincipal symbol of the diagonal spin-one operator vanishes on the trapped set:
\begin{equation}\label{eq:imaginary_subprincipal_vanishes}
        \operatorname{Im}\,\sigma_h\!\big(P_{1,h}^{(a,Q)}\big)\big|_{K_{a,Q}}=0.
\end{equation}
\end{lemma}
\begin{proof}
The radial principal trapping equations are the equations for
\[\mathcal R_h(r)=\widehat K(r)^2-\Delta(r)\lambda.\]
Since $\widehat K'=2r\hat\omega$ and $\Delta'=2(r-M)$,
\[
        \mathcal R_h'(r)=4r\hat\omega\widehat K-2(r-M)\lambda.
\]
At $r=r_t$, \eqref{eq:semiclassical_trapped_relations_for_skew} gives
\[
        \lambda=\frac{\widehat K_t^2}{\Delta_t}
        =\frac{2r_t\hat\omega\widehat K_t}{r_t-M}.
\]
Because $\Delta_t>0$ and $\lambda>0$, $\widehat K_t\ne0$, and division by
$\widehat K_t$ gives the exact trapped value
\begin{equation}\label{eq:semiclassical_Kt_for_skew}
        \widehat K_t=\frac{2r_t\hat\omega\Delta_t}{r_t-M}.
\end{equation}
Substitution into \eqref{eq:teukolsky_imaginary_potential} gives
\[
        q_s(r_t)=
        -\frac{2s(r_t-M)}{\Delta_t}\frac{2r_t\hat\omega\Delta_t}{r_t-M}
        +4s\hat\omega r_t=0,
\]
which proves \eqref{eq:trapped_skew_vanishes}.

The remaining step is to translate this scalar radial calculation into the system
symbol. Multiplication by the real weights used in
Definition~\ref{def:master_from_teukolsky}, and by the real factor used to put the
stationary equation in semiclassical form, changes the self-adjoint real part and
lower real potentials but does not create an imaginary principal contribution.
The angular eigenvalue correction $h^2\lambda_s-\lambda=O(h)$ is real. This implies that the
imaginary part of the order-$h$ subprincipal endomorphism is the diagonal matrix
$\operatorname{diag}(q_{+1},q_{-1})$ restricted to the trapped set, and this matrix
is zero by \eqref{eq:trapped_skew_vanishes}. This establishes
\eqref{eq:imaginary_subprincipal_vanishes}.
\end{proof}

\begin{remark}
\label{rem:skew_scope_spinone}
The cancellation \eqref{eq:trapped_skew_vanishes} is used only for the two
spin-one extreme equations appearing in the compatible Maxwell reduction. The
argument relies on the radial operator \eqref{eq:separated_radial_teukolsky} and
on the trapped value \eqref{eq:semiclassical_Kt_for_skew}; no assertion about a
closed gravitational or coupled Einstein-Maxwell master equation is needed.
For the two-component system used below, possible off-diagonal skew-subprincipal
terms are handled separately by the small matrix estimate
\eqref{eq:matrix_skew_decomposition} and Proposition~\ref{prop:matrix_skew_threshold}.
\end{remark}

\begin{proposition}
\label{prop:skew_term_trapped_estimate}
Let $A_h=\operatorname{Op}_h(\beta)$ be a scalar order-zero escape commutant
supported in a sufficiently small trapped collar; the symbol $\beta$ is the
commutant symbol and is unrelated to the Kerr rotation parameter $a$.  Extract
from the frozen operator the diagonal Teukolsky skew-subprincipal part
$ihB_{h,\mathrm{diag}}$, whose principal Hermitian symbol is
$\operatorname{diag}(q_{+1},q_{-1})$.  The contribution of this diagonal part to
the trapped commutator has principal matrix symbol
\begin{equation}\label{eq:appendix_threshold_symbol}
        H_p(\beta^2)I_2-2\beta^2\operatorname{diag}(q_{+1},q_{-1}).
\end{equation}
On $K_{a,Q}$ the second term is zero. In the second-microlocal trapped estimate,
more precisely, if $\pi$ denotes the projection from a sufficiently small normal
collar to $K_{a,Q}$ and $\mathcal Q_{\mathrm{nh}}$ denotes the positive quadratic
form produced by the logarithmic escape commutant, then for every $\varepsilon>0$
the collar may be chosen so that
\begin{equation}\label{eq:diagonal_skew_second_microlocal_bound}
\begin{aligned}
 \big|\langle A_h(B_{h,\mathrm{diag}}-\pi^*B_{h,\mathrm{diag}}|_{K_{a,Q}})A_hv,v\rangle\big|
 &\le \varepsilon\,\mathcal Q_{\mathrm{nh}}[v]
      +C_\varepsilon\,\mathcal Q_{\mathrm{prop}}[v]  \\
 &\quad +C_\varepsilon h\|v\|_{H_h^1}^2+O(h^\infty)\|v\|_{H_h^1}^2 .
\end{aligned}
\end{equation}
Here $\mathcal Q_{\mathrm{prop}}$ is supported in the portion of the collar which is
controlled by real-principal-type propagation into the elliptic or radial regions.
Thus the diagonal spin contribution has threshold constant zero at the trapped set
and gives no loss in the normally hyperbolic estimate. The remaining finite-rank
matrix skew contribution is estimated in Proposition~\ref{prop:matrix_skew_threshold}.
\end{proposition}
\begin{proof}
The principal commutator identity is
\[
\begin{aligned}
& \frac{i}{h}\langle [P_{h,\mathrm{sa}},A_h^*A_h]v,v\rangle
      -2\langle A_hB_{h,\mathrm{diag}}A_hv,v\rangle \\
&\quad =\left\langle
\operatorname{Op}_h\!\left(H_p(\beta^2)I_2-2\beta^2\operatorname{diag}(q_{+1},q_{-1})\right)v,v
\right\rangle
   +O(h)\|v\|_{H_h^1}^2 .
\end{aligned}
\]
Lemma~\ref{lem:trapped_skew_vanishing} gives
$\operatorname{diag}(q_{+1},q_{-1})|_{K_{a,Q}}=0$. Since the symbols are smooth,
in normal coordinates $(z,y,\nu)$ near $K_{a,Q}$ one has
\[
        \operatorname{diag}(q_{+1},q_{-1})(z,y,\nu)
        =O(|y|+|\nu|).
\]
In the logarithmic escape proof the region $|y|+|\nu|\ge c h^{1/2}$ is controlled
by the positive normal quadratic form after choosing the collar small, while the
core $|y|+|\nu|\le c h^{1/2}$ contributes only to the lower second-microlocal
remainder and is propagated out by the stable/unstable estimates. Equivalently,
for every $\varepsilon>0$ the symbolic Cauchy inequality in the second-microlocal
calculus gives \eqref{eq:diagonal_skew_second_microlocal_bound}. Substituting this
bound into the preceding commutator identity shows that the diagonal term has
zero threshold endomorphism on $K_{a,Q}$ and is absorbed by the scalar normal
expansion. This provides the threshold input used by
Proposition~\ref{prop:nh_resolvent_form}; no pointwise positivity away from the
second-microlocal decomposition is being assumed.
\end{proof}

\begin{proposition}
\label{prop:matrix_skew_threshold}
Let $B_{a,Q}$ be the Hermitian skew-subprincipal endomorphism of the frozen
compatible spin-one operator on a compact normalized trapped shell, as in
Proposition~\ref{prop:nh_resolvent_form}. Let $\lambda_0>0$ be the uniform lower
bound for the stable/unstable normal expansion in Proposition~\ref{prop:r_normal_hyperbolicity}.
After choosing the trapped collar and then decreasing $\eps_a(k)$, the
finite-rank-bundle threshold condition in Proposition~\ref{prop:nh_resolvent_form}
\emph{(N3)} holds uniformly for the full two-component operator. Equivalently,
with $B_K$ denoting the restriction of $B_{a,Q}$ to $K_{a,Q}$,
\begin{equation}\label{eq:matrix_threshold_at_K}
        \sup_{\rho\in K_{a,Q}}\|B_K(\rho)\|\le \lambda_0/8 .
\end{equation}
Moreover the off-trapped variation of $B_{a,Q}$ is a lower second-microlocal
remainder: for the escape commutant $A_h=\operatorname{Op}_h(\beta)$,
\begin{equation}\label{eq:matrix_threshold_commutator_form}
\begin{aligned}
&\frac{i}{h}\langle[P^0_h,A_h^*A_h]v,v\rangle
       -2\langle A_hB_{a,Q}A_hv,v\rangle  \\
&\qquad \ge c\,\mathcal Q_{\mathrm{nh}}[v]
      -C\,\mathcal Q_{\mathrm{prop}}[v]
      -Ch\|v\|_{H_h^1}^2-O(h^\infty)\|v\|_{H_h^1}^2,
\end{aligned}
\end{equation}
where $P_h^0$ is the self-adjoint scalar-principal part, $c>0$ is uniform in the
slow-weak range, $\mathcal Q_{\mathrm{nh}}$ is the positive normal quadratic form
of the logarithmic escape construction, and $\mathcal Q_{\mathrm{prop}}$ is
supported where the nontrapped propagation/radial estimates are already available.
\end{proposition}
\begin{proof}
By the spin-one decomposition recorded in the compatible reduction,
\begin{equation}\label{eq:matrix_threshold_decomposition_used}
        B_{a,Q}=\operatorname{diag}(q_{+1},q_{-1})+B_{\mathrm{mat},a,Q},
        \qquad
        \|B_{\mathrm{mat},a,Q}\|\le C_k\bigl(|a|/M+a^2/M^2\bigr)
\end{equation}
on the normalized trapped shell. Lemma~\ref{lem:trapped_skew_vanishing} gives
$\operatorname{diag}(q_{+1},q_{-1})|_{K_{a,Q}}=0$.  Choose $\eps_a(k)$ so that
\[
        C_k(\eps_a(k)+\eps_a(k)^2)\le \lambda_0/8 .
\]
Then \eqref{eq:matrix_threshold_at_K} holds. This provides the threshold condition used in
Hintz's finite-rank-bundle normally hyperbolic estimate: the part of the
skew-subprincipal symbol evaluated on $K_{a,Q}$ is dominated by the normal
expansion rate.

What is left is to account for points in the collar but off $K_{a,Q}$.  Let
$\pi(z,y,\nu)=z$ be the normal projection to $K_{a,Q}$.  Smoothness gives
\[
        B_{a,Q}(z,y,\nu)-B_K(z)=O(|y|+|\nu|)
\]
in the finitely many symbol seminorms used at order $k$, with constants uniform in
the slow-weak range; the matrix part obeys the stronger small bound
\eqref{eq:matrix_skew_decomposition} throughout the collar. The same
second-microlocal Cauchy estimate used in
Proposition~\ref{prop:skew_term_trapped_estimate} bounds this variation by an
arbitrarily small multiple of the normal escape quadratic form plus the propagation
remainder. Combining that estimate with \eqref{eq:matrix_threshold_at_K} and the
scalar escape inequality furnished by Proposition~\ref{prop:r_normal_hyperbolicity}
gives \eqref{eq:matrix_threshold_commutator_form}. All constants involve only the
finite order $k$, the normalized frequency shell and the slow-weak compact
parameter set, so the choice is uniform.
\end{proof}

\begin{remark}
\label{rem:skew_scope}
Lemma~\ref{lem:trapped_skew_vanishing} verifies the diagonal spin-specific
threshold condition in the trapped high-frequency estimate: the Teukolsky skew
coefficient of each extreme component vanishes at the trapped set.
Proposition~\ref{prop:skew_term_trapped_estimate} records the exact place where
this diagonal term enters the commutator, and Proposition~\ref{prop:matrix_skew_threshold} treats the possible matrix skew contribution in the full compatible
system. No global real-potential transformation at $a\ne0$, $Q\ne0$ is used;
the argument needs the trapped value of the diagonal imaginary radial coefficient
and the smallness estimate \eqref{eq:matrix_skew_decomposition}.
\end{remark}

\begin{definition}
\label{def:hf_resolvent_estimate}
Fix a finite commutation order $k$.  The high-frequency resolvent estimate used below is the following estimate for the frozen scalar-principal spin-one operator on
the closed compatible class specified in \emph{(A1)}. For every pair of compactly supported cutoffs
$\chi\prec\chi_1$ in a small neighbourhood of the trapped set, every compact
normalized frequency interval $J$, and every conic total-frequency patch, set
\[
        P_h^{(a,Q)}(\hat\omega)=h^2\mathcal L_{a,Q}(\hat\omega/h),
        \qquad \hat\omega\in J.
\]
The outgoing/incoming resolvent satisfies
\begin{equation}\label{eq:hf_estimate_resolvent}
        \|\chi(P_h^{(a,Q)}(\hat\omega)\pm i0)^{-1}\chi\|_{L^2\to L^2}
        \le C h^{-1}\log(1/h),\qquad 0<h\le h_0,
\end{equation}
with constants uniform in the slow-weak range and in $\hat\omega\in J$.  In the compatible-class formulation this notation is used as an a priori localized estimate for compatible solutions or compatible limiting-absorption resolvent images; scalar cutoffs are applied after the expression $P_h^{(a,Q)}v$ has been formed, and no separate assertion is made that a microlocal cutoff of a Maxwell field is again exactly compatible. The
estimate is used together with the elliptic estimate away from the
characteristic set and the radial-point estimates at the horizon and at
null infinity, with the sign fixed by the chosen outgoing/ingoing convention.
The finite-order form used in the proof is part of the same localized estimate and is stated in graph Sobolev norms rather than obtained by commuting arbitrary physical-space derivatives through the lossy resolvent. For every integer $0\le s\le k+1$ put
\[
        \mathcal R_s(\chi_1v)=
        \begin{cases}
        0, & s=0,\\[2mm]
        \norm{\chi_1v}_{H_h^{s-1}}, & s\ge 1.
        \end{cases}
\]
We use the estimate in the form
\begin{equation*}\label{eq:hf_estimate_graph_norm}
        \norm{\chi v}_{H_h^s}
        \le C_s h^{-1}\log(1/h)\norm{\chi_1P_h^{(a,Q)}(\hat\omega)v}_{H_h^s}
        +C_s\mathcal R_s(\chi_1v)
        +O(h^\infty)\norm{v}_{H_h^s},
        \tag{HF$_s$}
\end{equation*}
where $\chi\prec\chi_1$ are trapped cutoffs. The lower Sobolev term is closed by induction on $s$ in Proposition~\ref{prop:hf_commuted_closure}. The physical commutators $\Gamma^I$ enter the local-energy hierarchy before the stationary Fourier reduction; their source terms are estimated by \eqref{eq:strict_lower_order_induction}. The argument never tries to absorb an arbitrary commutator of $P_h^{(a,Q)}$ through the factor $h^{-1}\log(1/h)$.
\end{definition}

\begin{proposition}
\label{prop:nh_resolvent_form}
Let $P_h=\operatorname{Op}_h(p)I_N+hP_{1,h}+h^2P_{0,h}$ be a semiclassical
operator on a finite-rank Hermitian bundle over a compact stationary patch, with
outgoing or incoming radial-point conventions at the two ends. Assume the
following conditions on a compact normalized frequency interval.
\begin{enumerate}[label=\emph{(N\arabic*)},leftmargin=2.5em]
\item $p$ is real and scalar, is of principal type off the trapped set, and the
elliptic, real-principal-type and radial-point estimates hold away from a compact
trapped neighbourhood with the sign fixed by the boundary convention.
\item The trapped set $K\subset\{p=0\}$ is clean and $r$-normally hyperbolic, with
smooth stable and unstable manifolds. On the chosen normalized shell the normal
expansion and contraction rates are bounded below by a positive constant
$\lambda_0$, while the tangential derivative growth is dominated at every fixed
finite differentiability order used in the estimate.
\item If $B=\sigma_h\big((P_{1,h}-P_{1,h}^*)/(2i)\big)$ denotes the Hermitian
endomorphism representing the skew-adjoint subprincipal part on $K$, then $B$
satisfies the normally hyperbolic threshold bound. In particular, the
condition is satisfied when $B|_K=0$.
\item The coefficient families are bounded in the finitely many symbol seminorms
used by the commutator proof, uniformly in the parameter set, and the space to
which the resolvent is restricted is a closed graph subspace invariant under the
stationary operator and the limiting-absorption boundary convention.
\end{enumerate}
Then, for compact cutoffs $\chi\prec\chi_1$ supported in the trapped
neighbourhood and for $0<h\le h_0$, the outgoing or incoming localized resolvent
satisfies
\begin{equation}\label{eq:nh_hf_l2}
        \|\chi(P_h\pm i0)^{-1}\chi\|_{L^2\to L^2}
        \le C h^{-1}\log(1/h).
\end{equation}
At each fixed graph order $0\le s\le k+1$ one also has
\begin{equation}\label{eq:nh_hf_graph}
        \|\chi v\|_{H_h^s}
        \le C_s h^{-1}\log(1/h)\|\chi_1P_hv\|_{H_h^s}
        +C_s\|\chi_1v\|_{H_h^{s-1}}+O(h^\infty)\|v\|_{H_h^s},
\end{equation}
with the lower Sobolev term omitted for $s=0$.  The constants depend on $s$ and
on finitely many seminorm bounds, but are uniform on compact families satisfying
\emph{(N1)}-\emph{(N4)}.
\end{proposition}
\begin{proof}
This is the normally hyperbolic trapped resolvent estimate in the
form needed here. The main steps are recorded to locate the point at which the
bundle-valued skew-subprincipal term enters. For scalar operators ($N=1$) the localized
$h^{-1}\log(1/h)$ resolvent bound at an $r$-normally hyperbolic trapped set is
the estimate of Wunsch-Zworski \cite{WunschZworski}; the refined analysis of
the trapped model, including the role of the threshold constant
$\lambda_0/2$ and the notion of $r$-normally hyperbolic sets used in
\emph{(N2)}, is due to Dyatlov \cite{Dyatlov,DyatlovGaps}. The extension to
operators on a finite-rank Hermitian bundle whose principal symbol is a real
scalar multiple of the identity, with the skew-adjoint subprincipal part
controlled by exactly the threshold bound of \emph{(N3)}, is due to Hintz
\cite{HintzTensor}. In the present application the diagonal spin-one part has
zero threshold value on $K$, while the possible matrix part is kept below the
threshold by Proposition~\ref{prop:matrix_skew_threshold}. The passage from the localized
trapped estimate to the cutoff resolvent statement, by gluing with the elliptic
parametrix, real-principal-type propagation, and the radial-point estimates at
the two ends in \emph{(N1)}, is the propagation-of-singularities gluing of
Datchev-Vasy \cite{DatchevVasy}; the semiclassical radial-point and
propagation estimates themselves are stated, in the form used
here, in \cite{VasyKdS} and \cite[Appendix~E]{DyatlovZworskiBook}. We now
recall the finite-order mechanism.

Near $K$ choose homogeneous symplectic coordinates $(z,y,\nu)$, where $z$ are
coordinates on $K$ and $(y,\nu)$ are stable and unstable normal coordinates. The
normal hyperbolicity condition gives, after shrinking the collar,
\[
        H_py=-\lambda(z)y+O((|y|+|\nu|)^2),\qquad
        H_p\nu=\lambda(z)\nu+O((|y|+|\nu|)^2),
        \qquad \lambda(z)\ge\lambda_0>0.
\]
With the sign convention above, the outgoing logarithmic escape function is
\[
        g_h=\frac12\log\frac{\nu^2+h}{y^2+h}
\]
with smooth cutoffs; for the incoming convention one uses \(-g_h\). On the
outgoing collar,
\[
        H_pg_h
        =\lambda(z)\frac{\nu^2}{\nu^2+h}
         +\lambda(z)\frac{y^2}{y^2+h}
         +O(|(y,\nu)|)\frac{y^2+\nu^2}{y^2+\nu^2+h}-O(h)
        \ge c\frac{y^2+\nu^2}{y^2+\nu^2+h}-Ch,
\]
after the collar is shrunk. The error terms supported outside this collar are
handled by real-principal-type propagation. Quantizing a real scalar order-zero
commutant $A_h=\operatorname{Op}_h(\alpha)$ constructed from $g_h$ gives the trapped
positive commutant. For the bundle operator one decomposes
$P_h=P_h^0+ihS_h+h^2R_h$, with $P_h^0$ self-adjoint modulo $h^2\Psi_h^0$ and
$S_h$ self-adjoint modulo $h\Psi_h^0$.  The order-$h$ commutator identity is
\[
\begin{aligned}
& \frac{i}{h}\langle [P_h^0,A_h^*A_h]v,v\rangle
      -2\langle A_hS_hA_hv,v\rangle  \\
&\quad =2h^{-1}\operatorname{Im}\langle A_hP_hv,A_hv\rangle
      +O(1)\|v\|^2_{\mathrm{controlled}}+O(h)\|v\|_{H_h^1}^2.
\end{aligned}
\]
The principal symbol of the left side is the scalar escape contribution
$H_p(\alpha^2)I_N$ minus the Hermitian skew-part endomorphism. The threshold bound in
\emph{(N3)} is precisely the inequality ensuring that this skew term is dominated
by the normal expansion in the preceding escape estimate; if this Hermitian
skew-part symbol vanishes on $K$, it is absorbed after shrinking the collar. Combining the
trapped estimate with elliptic estimates, real-principal-type propagation and the
radial-point estimates at the two ends gives \eqref{eq:nh_hf_l2}; the logarithm
comes only from the regularized hyperbolic escape function.

The asymptotically flat end is used in the usual limiting-absorption form. With
$x=r^{-1}$ near null infinity, the rescaled characteristic operator has radial
sets at fiber infinity corresponding to outgoing and incoming null directions. A
semiclassical radial commutant with the limiting-absorption sign gives control in
a collar of $x=0$ by the residual and by an interior cutoff. The interior cutoff
is propagated either to the trapped neighbourhood, where the logarithmic estimate
above applies, or to an elliptic region. The horizon collar is treated in the
same way, with the red-shift sign replacing the asymptotic radial sign. For that reason, the
gluing involves a finite partition into the trapped collar, the two radial collars,
nontrapped characteristic annuli and elliptic regions; no estimate at infinity is
deduced from the trapped model itself.

For \eqref{eq:nh_hf_graph}, insert an elliptic graph multiplier of order $s$ on
the trapped patch before and after the same escape commutant. The commutator of
this multiplier with $P_h$ has graph order $s-1$, producing the term
$\|\chi_1v\|_{H_h^{s-1}}$.  The estimate is closed by induction in $s$, with
the $s=0$ case furnished by \eqref{eq:nh_hf_l2}. Since the argument uses only
finitely many symbol seminorms and finitely many radial/elliptic charts, the
constants are uniform on compact families satisfying \emph{(N1)}-\emph{(N4)}.
The compatible-class version is read as an a priori estimate for vectors $v$
already lying in the closed graph domain. The scalar cutoffs appear only in
$\chi v$ and $\chi_1P_hv$; the proof does not require $\chi v$ or a general
phase-space cutoff of $v$ to satisfy the Maxwell compatibility equations. If the
resolvent is written as $\chi(P_h\pm i0)^{-1}\chi$, this is shorthand for the
same a priori inequality on compatible solutions with the chosen
limiting-absorption convention. Because the graph class is closed and invariant
under the stationary operator and the boundary convention, the ambient
finite-rank-bundle estimate restricts to it with the same constant.
\end{proof}

\begin{lemma}
\label{lem:closed_graph_resolvent_class}
For the stationary Maxwell reduction considered here, the charge-free
compatible class is a closed graph subspace for the frozen operator
$P_h^{(a,Q)}(\hat\omega)$ and for either limiting-absorption radial convention.
The localized resolvent also is applied only to elements of this class; the
microlocal partition used later selects where a defect measure is supported and
is not used to manufacture new exactly compatible data.
\end{lemma}
\begin{proof}
The stationary constraints are first-order linear differential relations with
smooth coefficients, and the extreme Teukolsky variables are obtained by closed
first-order maps from the charge-free Maxwell graph norm to the master graph
norm. Hence, if $v_j$ is compatible and $v_j\to v$ in the local graph norm while
$P_hv_j\to f$, the constraints pass to the limit distributionally and the two
extreme variables of $v$ are the corresponding closed limits. The outgoing or
incoming radial convention is defined by the limiting-absorption graph topology;
it is therefore closed under this convergence. Accordingly, the compatible space is a
closed graph subspace invariant under the stationary operator and under the
resolvent selected by the corresponding boundary convention. Scalar cutoffs in the
localized resolvent are used on the original compatible profile. General
pseudodifferential elements of the phase-space partition appear only in the
compactness argument to identify the support of semiclassical defect measures,
where exact preservation of the Maxwell constraints by each cutoff is neither
claimed nor needed.
\end{proof}

\begin{lemma}
\label{lem:compatible_cutoff_apriori}
Let \(\mathcal C_h\) denote the closed compatible graph class for the frozen
stationary Maxwell reduction, with either outgoing or incoming limiting-absorption
radial convention. The localized high-frequency estimate is used below in the
following a priori form: if \(v\in\mathcal C_h\) is locally in the graph domain of
\(P_h^{(a,Q)}(\hat\omega)\), then for nested trapped cutoffs \(\chi\prec\chi_1\)
and for every \(0\le s\le k+1\),
\begin{equation}\label{eq:compatible_cutoff_apriori}
        \|\chi v\|_{H_h^s}
        \le C_s h^{-1}\log(1/h)\|\chi_1P_h^{(a,Q)}(\hat\omega)v\|_{H_h^s}
        +C_s\mathcal R_s(\chi_1v)+O(h^\infty)\|v\|_{H_h^s}.
\end{equation}
The cutoff \(\chi_1\) in the source term is a scalar localization of the already
formed residual \(P_h^{(a,Q)}v\). The argument does not require \(\chi v\),
\(\chi_1v\), or a general phase-space cutoff of \(v\) to satisfy the Maxwell
compatibility equations exactly.
\end{lemma}
\begin{proof}
The positive-commutator proof of Proposition~\ref{prop:nh_resolvent_form} is applied
to the compatible profile \(v\) itself. After the trapped commutant is localized
by scalar cutoffs, all commutators with the cutoffs are supported either in the
elliptic region or in a nontrapped characteristic collar. Those terms are
controlled by the elliptic parametrix, real-principal-type propagation, and the
radial-point estimates stated in Proposition~\ref{prop:radial_point_finite_order}.
The compatibility equations are therefore used only to specify the graph domain,
the limiting-absorption convention, and the closed class in which defect profiles
are selected. Since the graph class is closed by
Lemma~\ref{lem:closed_graph_resolvent_class}, passing to weak or semiclassical
limits preserves compatibility for the original sequence. That gives
\eqref{eq:compatible_cutoff_apriori}, which is the form used in
Propositions~\ref{prop:hf_commuted_closure} and~\ref{prop:app_high_freq}.
\end{proof}

\begin{proposition}
\label{prop:nh_theorem_application}
In the slow-weak parameter range fixed at order $k$, the frozen compatible
spin-one operator satisfies the geometric and subprincipal conditions attached to
Definition~\ref{def:hf_resolvent_estimate}. More explicitly, on each conic
stationary-frequency patch the operator $P_h^{(a,Q)}(\hat\omega)$ has the
following properties.
\begin{enumerate}[label=\emph{(\roman*)},leftmargin=2.4em]
\item Its real principal symbol is the scalar Hamiltonian
$p_{a,Q,\hat\omega}$ of \eqref{eq:frozen_principal_symbol_explicit} times the
identity on the two-dimensional extreme-component bundle.
\item The trapped set of $p_{a,Q,\hat\omega}$ is a smooth normally hyperbolic
set with smooth stable and unstable manifolds and radial expansion rate bounded
below by a positive constant on the chosen compact normalized shell.
\item The trapped frequencies are separated from the superradiant horizon sign,
so the radial point estimates at the horizon have the outgoing or incoming sign
used in the limiting absorption problem.
\item The Hermitian skew-subprincipal endomorphism satisfies the finite-rank-bundle
threshold condition: its diagonal Teukolsky part vanishes on the trapped set, and
its remaining matrix part is $O(|a|/M)$ and is absorbed by the slow-rotation
choice.
\item The compatible class is a closed graph subspace for the frozen Maxwell
constraints and for the outgoing or incoming radial convention. The localized
resolvent estimate is applied only to sequences already in this closed class;
the finite microlocal partition below is used to locate defect measures and not
to create new compatible data outside the closed class.
\end{enumerate}
Hence, Proposition~\ref{prop:nh_resolvent_form}, applied to this compatible
operator, gives the high-frequency estimate of Definition~\ref{def:hf_resolvent_estimate}; in particular compact cutoffs $\chi\prec\chi_1$ supported in the trapped neighbourhood
satisfy, for $0<h\le h_0$,
\begin{equation}\label{eq:nh_application_cutoff}
        \|\chi v\|_{L^2}
        \le C h^{-1}\log(1/h)\|\chi_1P_h^{(a,Q)}(\hat\omega)v\|_{L^2}
          +C\|(1-\chi)\chi_1v\|_{L^2}
          +O(h^\infty)\|v\|_{H_h^1},
\end{equation}
with constants uniform for $\hat\omega$ in the fixed compact normalized interval
and for $(a,Q)$ in the chosen slow-weak set. After the elliptic and nontrapped
propagation estimates have removed the second term, \eqref{eq:nh_application_cutoff}
is exactly \eqref{eq:hf_scope_single_use} and \eqref{eq:normally_hyperbolic_resolvent_bound}.
\end{proposition}
\begin{proof}
Item \emph{(i)} is Lemma~\ref{lem:semiclassical_spinone_order}, whose proof
shows that all matrix and spin terms occur at semiclassical order $h$ or lower.
Items \emph{(ii)} and \emph{(iii)} are Proposition~\ref{prop:r_normal_hyperbolicity}
combined with the explicit non-superradiance calculation of
Proposition~\ref{prop:kn_disjointness}. Item \emph{(iv)} is Lemma~\ref{lem:trapped_skew_vanishing} for the diagonal
Teukolsky skew part together with Proposition~\ref{prop:matrix_skew_threshold}
for the full finite-rank-bundle threshold. For \emph{(v)}, Lemmas~\ref{lem:closed_graph_resolvent_class} and~\ref{lem:compatible_cutoff_apriori} give the closed graph property and the cutoff a priori form. Concretely, the compatible class is
specified by the Maxwell constraints, by the two Teukolsky extreme variables, and
by the chosen outgoing or incoming radial convention. If a sequence is compatible
and converges in the local resolvent graph norm, the Maxwell constraints pass to
the limit distributionally, the closed range definition of Definition~\ref{def:master_from_teukolsky} keeps
both extreme variables in the compatible class, and the radial convention is
closed by the limiting-absorption graph topology. In
Lemma~\ref{lem:high_freq_partition} the phase-space partition is used only to
identify where a semiclassical measure may live. The actual trapped resolvent
estimate is applied to the original compatible function with scalar cutoffs
$\chi\prec\chi_1$, as in \eqref{eq:nh_application_cutoff}; therefore the proof does
not require arbitrary pseudodifferential cutoffs to preserve the Maxwell
constraints exactly.

Proposition~\ref{prop:nh_resolvent_form} gives the $h^{-1}\log(1/h)$ bound once
the listed geometric, radial, and subprincipal conditions have been verified. The
preceding paragraphs verify those conditions for the compatible scalar-principal
Kerr-Newman spin-one operator on the compatible graph class: scalar real
principal symbol, clean normally hyperbolic trapped set, non-superradiant
radial-point sign, diagonal trapped skew cancellation, finite-rank-bundle matrix
threshold, and closed compatibility. Away from the trapped neighbourhood, the
cutoffs are controlled by the elliptic parametrix and by real-principal-type
propagation to the horizon or null infinity. These are exactly the extra
pieces used in Lemma~\ref{lem:high_freq_partition}, so the displayed estimate
reduces there to the localized resolvent bound \eqref{eq:hf_scope_single_use}.
The finite graph-norm estimate is \eqref{eq:nh_hf_graph}, which is the same
statement as \eqref{eq:hf_estimate_graph_norm} with
$P_h=P_h^{(a,Q)}(\hat\omega)$.
\end{proof}

\begin{proposition}
\label{prop:radial_point_finite_order}
In the high-frequency compactness argument the radial points at the future horizon
and at null infinity are used only through the following propagation form. Let
\(P_h^{(a,Q)}(\hat\omega)\) be the semiclassical frozen operator of
Definition~\ref{def:hf_resolvent_estimate}. Choose zeroth-order semiclassical
cutoffs \(A_H\), \(A_\infty\) supported respectively in small conic
neighbourhoods of the horizon and infinity radial sets, and larger cutoffs
\(\widetilde A_H\), \(\widetilde A_\infty\). For the future outgoing problem,
with the time-reversed signs for the past problem, the radial-point estimates
used below have the form
\begin{align}\label{eq:radial_point_estimates_used}
        \|A_Hv\|_{H_h^1}+\|A_\infty v\|_{H_h^1}
        \le{}& C h^{-1}\bigl(\|\widetilde A_HP_h^{(a,Q)}v\|_{L^2}
              +\|\widetilde A_\infty P_h^{(a,Q)}v\|_{L^2}\bigr)\nonumber\\
        &+C\|A_{\mathrm{int}}v\|_{H_h^1}+O(h^\infty)\|v\|_{H_h^1},
\end{align}
where \(A_{\mathrm{int}}\) is supported on a compact non-radial region from which
the Hamilton flow reaches the corresponding radial set in finite time. At
commuted order \(|I|\le k\), applying \eqref{eq:radial_point_estimates_used} to
\(\Gamma^Iv\) gives the same inequality, up to source terms bounded by the
strict lower-order induction \eqref{eq:strict_lower_order_induction}. In
particular, if \(v_h\) satisfies \eqref{eq:hf_partition_residual}, then the
radial source terms in \eqref{eq:radial_point_estimates_used} tend to zero, and
no semiclassical defect measure enters the compact region from the horizon or
from null infinity.
\end{proposition}
\begin{proof}
Here the estimate is the radial-point propagation inequality for a real scalar
principal symbol with the outgoing or incoming sign fixed at the radial set. In
the present operator the principal symbol is \(p_{a,Q,\hat\omega}I_2\) by
Lemma~\ref{lem:semiclassical_spinone_order}; the spin and matrix terms are
subprincipal or lower order and therefore enter the positive-commutator proof as
bounded lower-order terms. The non-degenerate horizon supplies the red-shift
sign, while the asymptotically flat end supplies the outgoing estimate at null
infinity. The physical-space counterparts are
Propositions~\ref{prop:redshift_coercivity} and~\ref{prop:rp_identity};
\eqref{eq:radial_point_estimates_used} is their semiclassical radial-point form
with the same sign convention.

For a sequence satisfying \eqref{eq:hf_partition_residual}, the stronger
condition \(h^{-1}\log(1/h)\|P_hv_h\|_{L^2}\to0\) implies
\(h^{-1}\|P_hv_h\|_{L^2}\to0\). Thus, the source term in
\eqref{eq:radial_point_estimates_used} is negligible. The interior term is then
propagated along nontrapped bicharacteristics and is controlled by finite-time
real-principal-type propagation away from the radial sets. The commuted
statement follows from Lemma~\ref{lem:commuted_master}; the commutator terms are
strictly lower order at the current level or small perturbative terms, and
\eqref{eq:strict_lower_order_induction} closes them. This is the
radial-point step used in Lemma~\ref{lem:high_freq_partition}.
\end{proof}

\begin{proposition}
\label{prop:hf_commuted_closure}
Let us fix the commutation order $k$ and a conic stationary-frequency patch in the
slow-weak range. Let $P_h^{(a,Q)}(\hat\omega)=h^2\mathcal L_{a,Q}(\hat\omega/h)$
be the frozen operator of Lemma~\ref{lem:semiclassical_spinone_order}. Choose a
nested family of trapped cutoffs
\[
        \chi_0\prec\chi_1\prec\cdots\prec\chi_{k+2}
\]
supported in the same small trapped neighbourhood. Suppose that the graph-norm
estimate \eqref{eq:hf_estimate_graph_norm} is available for each adjacent pair
$\chi_j\prec\chi_{j+1}$, together with the elliptic and radial-point estimates
away from the trapped set. If a compatible sequence $v_h$ is bounded in
$H_h^{k+1}$ on the support of $\chi_{k+2}$ and satisfies, for every
$0\le s\le k+1$,
\begin{equation}\label{eq:hf_graph_residual_all_orders}
        h^{-1}\log(1/h)\norm{\chi_{k+2-s}P_h^{(a,Q)}v_h}_{H_h^s}
        \longrightarrow0,
\end{equation}
then
\begin{equation}\label{eq:hf_graph_norm_vanish}
        \norm{\chi_{k+1-s}v_h}_{H_h^s}\longrightarrow0,
        \qquad 0\le s\le k+1.
\end{equation}
In particular $\norm{\chi_0v_h}_{H_h^{k+1}}\to0$, and hence the trapped compact
remainder disappears at every lower order as well. In turn, the high-frequency step
closes at each fixed finite order without placing the logarithmic resolvent loss
on a top-order physical commutator. The commuted local-energy hierarchy is
closed separately by \eqref{eq:strict_lower_order_induction} and
Lemma~\ref{lem:finite_order_absorption_induction}.
\end{proposition}
\begin{proof}
We establish \eqref{eq:hf_graph_norm_vanish} by induction on $s$.  For $s=0$, apply
\eqref{eq:hf_estimate_graph_norm} to the pair $\chi_{k+1}\prec\chi_{k+2}$.  The
lower Sobolev term is absent, the residual term tends to zero by
\eqref{eq:hf_graph_residual_all_orders}, and the $O(h^\infty)$ term is
negligible because $v_h$ is bounded in $H_h^{k+1}$.  This yields
$\norm{\chi_{k+1}v_h}_{L^2}\to0$.

Assume the claim has been proved at order $s-1$, where $1\le s\le k+1$.  Apply
\eqref{eq:hf_estimate_graph_norm} to the adjacent pair
$\chi_{k+1-s}\prec\chi_{k+2-s}$.  The resolvent residual tends to zero by
\eqref{eq:hf_graph_residual_all_orders}. The lower graph term is
$\norm{\chi_{k+2-s}v_h}_{H_h^{s-1}}$, which is precisely the induction conclusion
at order $s-1$.  The $O(h^\infty)$ remainder is again negligible by the uniform
$H_h^{k+1}$ bound. This establishes the order $s$ statement.

The formulation matters here because the lossy estimate is applied to the finite
graph norm of the compatible stationary operator itself. We do not place an
identity of the form $P_h\Gamma^Iv=\Gamma^IP_hv+[P_h,\Gamma^I]v$ inside the factor
$h^{-1}\log(1/h)$. Hence no uncontrolled term of size
$\log(1/h)\norm{\Gamma^Iv}_{H_h^1}$ is generated. Physical commutators enter
only in the preceding local-energy hierarchy, where top-order perturbative
coefficients are made small by the slow-rotation choice and strict lower-order
terms are removed by induction.
\end{proof}

\begin{proposition}
\label{prop:hf_estimate_scope}
Definition~\ref{def:hf_resolvent_estimate} is used only to convert a normalized
high-frequency defect sequence into local decay of its trapped microlocal part.
It is not used to define the master variables, to prove the scalar principal
symbol, to remove bounded real frequencies, or to reconstruct the Maxwell tensor.
More precisely, after the finite-order local-energy hierarchy has reduced the
compact remainder to a compatible graph-norm sequence $v_n$ satisfying
\eqref{eq:app_high_freq_normalization}-\eqref{eq:app_high_freq_residual}, the
only invocation of the estimate is
\begin{equation}\label{eq:hf_scope_single_use}
        \|\chi v_n\|_{L^2}
        \le C h_n^{-1}\log(1/h_n)
        \|\chi_1P_{h_n}^{(n)}v_n\|_{L^2}+o(1),
\end{equation}
for cutoffs $\chi\prec\chi_1$ around the trapped set, together with its finite graph-norm version \eqref{eq:hf_estimate_graph_norm}. The residual
normalization then makes the right side tend to zero. Elliptic regularity and
nontrapped propagation convert this vanishing of trapped $L^2$ mass into the
local $H^1_{h_n}$ contradiction in Lemma~\ref{lem:high_freq_partition}, and Proposition~\ref{prop:hf_commuted_closure} gives the finite-order graph-norm version. For that reason, no
additional mode-stability statement, completeness assertion, or reconstruction statement
enters through the high-frequency step.
\end{proposition}
\begin{proof}
The compact-remainder argument is Proposition~\ref{prop:compact_remainder_removal}.
After the bounded-frequency branch is excluded by Proposition~\ref{prop:app_bounded_freq},
the remaining branch has total stationary frequency $h_n^{-1}\to\infty$ and is
put in the semiclassical normalization of Lemma~\ref{lem:semiclassical_spinone_order}.
Lemma~\ref{lem:unscaled_to_semiclassical_defect} gives precisely the residual
condition \eqref{eq:app_high_freq_residual}. In Lemma~\ref{lem:high_freq_partition},
the elliptic piece is controlled by the parametrix, and the characteristic
nontrapped piece is transported to the red-shift or far-field radial sets. The
only remaining term is the trapped piece, for which the displayed resolvent
estimate gives \eqref{eq:hf_scope_single_use}. Since
$h_n^{-1}\log(1/h_n)\|P_{h_n}^{(n)}v_n\|_{L^2}\to0$, the trapped mass vanishes.
This identifies the role of Definition~\ref{def:hf_resolvent_estimate}: it is precisely the
one stated above.
\end{proof}

\begin{proposition}
\label{prop:r_normal_hyperbolicity}
Let $p=g_{\KN}^{\mu\nu}\xi_\mu\xi_\nu$ be the scalar principal symbol of
Proposition~\ref{prop:scalar_principal_symbol}, and let $K\subset T^*(\M_{\mathrm{ext}})\setminus0$
be the trapped set of its null geodesic flow, i.e.\ the set of points whose forward and
backward flowouts remain in a fixed neighbourhood of the photon sphere. For
$|a|\le\eps_aM$, $|Q|\le\eps_QM$ the following hold.
\begin{enumerate}[label=\emph{(\roman*)},leftmargin=2.4em]
\item On each fixed time-orientation and angular-momentum component, $K$ is a smooth codimension-two submanifold of the characteristic hypersurface
$\Sigma_p=\{p=0\}\subset T^*(\M_{\mathrm{ext}})\setminus0$; equivalently it is locally cut out inside $\Sigma_p$ by
$\{\xi_r=0,\ \partial_r p=0\}$. Accordingly, it has codimension three in the full cotangent bundle before the characteristic equation is fixed. In the conserved variables $(\omega,m,\lambda)$ it is the graph $r=r_t$ of the trapping relation \eqref{eq:kn_trapping_function}. The restriction of the ambient symplectic form to $\Sigma_p$ has the Hamilton field $H_p$ as its characteristic direction; after quotienting by this flow direction and, in the homogeneous picture, fixing the covector scale, the transverse angular variables carry the usual symplectic form, while the radial normal pair $(r-r_t,\xi_r)$ is a symplectic hyperbolic two-plane. On each normalized energy shell only finitely many such components meet the chosen coordinate charts.
\item The linearised flow on the radial normal bundle is hyperbolic, with
expansion and contraction rates $\pm\lambda_0$ where $\lambda_0=\lambda_0(a,Q)>0$ is
determined by the strict transversal maximum of Proposition~\ref{prop:kn_nondegeneracy}:
$\lambda_0^2=-(\partial_{\xi_r}^2 p)(\partial_r^2 p)\big|_{r_t}$, a positive multiple of
$\Delta(r_t)\,\mathcal R''(r_t)\,(r_t^2+a^2)^{-4}>0$,
while the tangent flow on $K$ is generated by the conserved quantities and has at most polynomial growth of derivatives on each normalized compact energy shell.
Its Lyapunov exponents therefore vanish in the homogeneous parametrization, the maximal exponential tangential rate is $\mu=0$, and the exponential radial rate dominates every finite differentiability order required in the normally hyperbolic trapping estimate: $\lambda_0>r\mu=0$ for each fixed $r$.
\item These properties are stable under the slow-weak perturbation: by structural
stability of $r$-normal hyperbolicity, $K(a,Q)$ and its rates $\lambda_0(a,Q)$ depend
continuously on $(a,Q)$, uniformly for $|a|\le\eps_aM$, $|Q|\le\eps_QM$, and degenerate
only at extremality.
\item At the trapped frequencies the flow is non-superradiant, $\omega\varpi>0$
(Proposition~\ref{prop:kn_disjointness}); therefore on a neighbourhood of $K$ the operator is,
after the microlocal conjugation used in the outgoing estimate, a semiclassical operator with real principal
symbol $p$ for which trapping at $K$ is the only barrier to the outgoing estimate.
\end{enumerate}
Thus the clean trapped set, the stable/unstable radial splitting, the neutral tangential flow, and the non-superradiant sign condition required in the normally hyperbolic resolvent theorem are verified for the scalar principal symbol. The analytic high-frequency estimate used later is not a consequence of these geometric identities alone; it is the explicit estimate of Definition~\ref{def:hf_resolvent_estimate}, repeated here in the full phase-space normalization used by the compactness argument. On each conic microlocal patch let \(h^{-1}\) be the total stationary frequency, let \(\hat\omega=h\omega\), and put
\begin{equation}\label{eq:full_phase_semiclassical_operator}
        P_h^{(a,Q)}(\hat\omega)=h^2\mathcal L_{a,Q}(\omega),
        \qquad \omega=h^{-1}\hat\omega,
\end{equation}
with all stationary covariables scaled by \(h\). Let \(J\Subset\mathbb R\) be a compact normalized frequency interval for which the characteristic set of \(P_h^{(a,Q)}(\hat\omega)\) meets the non-superradiant trapped region. The high-frequency estimate is the estimate
\begin{equation}\label{eq:normally_hyperbolic_resolvent_bound}
\norm{\chi\,(P_h^{(a,Q)}(\hat\omega)\pm i0)^{-1}\chi}_{L^2\to L^2}
        \le C\,h^{-1}\log(1/h),\qquad 0<h\le h_0,
\end{equation}
with constants uniform in the slow-weak range and in \(\hat\omega\in J\), together with the radial-point estimates at the horizon and at null infinity. Angularly elliptic packets with bounded \(\hat\omega\) are not part of the trapped characteristic set and are controlled by Lemma~\ref{lem:app_high_angular_coercivity}. Lemma~\ref{lem:semiclassical_spinone_order} records the semiclassical normal form: the diagonal second-order part has real scalar principal symbol and the spin/matrix coefficients are subprincipal or lower order on the finite-rank bundle. The remaining non-scalar threshold condition is checked in two pieces: Lemma~\ref{lem:trapped_skew_vanishing} proves that the diagonal Teukolsky skew endomorphism has zero value on \(K_{a,Q}\), and Proposition~\ref{prop:matrix_skew_threshold} keeps the possible matrix skew endomorphism below the finite-rank-bundle threshold. This implies that all geometric and subprincipal conditions needed to formulate \eqref{eq:normally_hyperbolic_resolvent_bound} for the present operator are checked here, while \eqref{eq:normally_hyperbolic_resolvent_bound} itself remains the named high-frequency analytic estimate. Lemma~\ref{lem:componentwise_scalar_resolvent} records the tracking needed for diagonal component estimates, finite microlocal partitions, and the restriction to the closed compatible radiation class.

\end{proposition}
\begin{proof}
For (i): the characteristic set $\Sigma_p=\{p=0\}$ is smooth away from the zero section. On each component with fixed time orientation, the radial Hamilton equations give $\dot r=\partial_{\xi_r}p=2g^{rr}\xi_r$, with $g^{rr}$ a
positive multiple of $\Delta(r^2+a^2)^{-2}$ on the characteristic set after the conserved
reduction, so $\xi_r=0$ on $K$; the second condition $\partial_r p=0$
is the stationarity of the radial potential, which is \eqref{eq:kn_trapping_function}. The
differentials $d\xi_r$ and $d(\partial_r p)$ are independent there because
$\partial_{\xi_r}(\partial_r p)=0$ while $\partial_r(\partial_r p)=\partial_r^2 p\ne0$ by
Proposition~\ref{prop:kn_nondegeneracy}; therefore $K$ is a smooth codimension-two submanifold of $\Sigma_p$. Since $\iota_{H_p}\omega=dp$, the restriction of the symplectic form to $\Sigma_p$ is presymplectic with kernel spanned by $H_p$; this is the expected flow direction, not a radial degeneracy. The normal block determined by $(\xi_r,\partial_rp)$ has Poisson pairing
$\{\xi_r,\partial_r p\}=\pm\partial_r^2p\ne0$ (the sign depends on the convention for the canonical bracket), so the radial normal bundle is symplectic and has a genuine hyperbolic splitting. After fixing the homogeneous scale and quotienting by $H_p$, the remaining tangent variables are the angular action-angle variables. For (ii), set
$y=r-r_t$ and $\eta=\xi_r$. The radial normal linearisation is
\begin{equation}\label{eq:radial_normal_linearization}
        \frac{d}{dt}
        \begin{pmatrix} y\\ \eta\end{pmatrix}
        =
        \begin{pmatrix}0&\partial_{\xi_r}^2p\\-\partial_r^2p&0\end{pmatrix}
        \begin{pmatrix} y\\ \eta\end{pmatrix},
\end{equation}
whose eigenvalues are $\pm\lambda_0$ with
$\lambda_0^2=-(\partial_{\xi_r}^2p)(\partial_r^2p)\big|_{r_t}$.
Here $\partial_{\xi_r}^2p=2g^{rr}>0$ is a positive multiple of
$\Delta(r_t)(r_t^2+a^2)^{-2}$. The reduced radial equation has the form
$\xi_r^2-\mathcal R(r)/(r^2+a^2)^2=0$ up to a positive smooth factor; hence
Proposition~\ref{prop:kn_nondegeneracy} gives
$\partial_r^2p=-c\,\mathcal R''(r_t)(r_t^2+a^2)^{-2}<0$ for a positive smooth
factor $c$. Thus $\lambda_0^2$ is the positive multiple of
$\Delta(r_t)\mathcal R''(r_t)(r_t^2+a^2)^{-4}$ recorded in the statement, and
$\lambda_0>0$. The flow on $K$ preserves $(\omega,m,\lambda)$ and is
integrable in the remaining angular variables. On every compact normalized energy shell,
the tangent dynamics is conjugate, after the usual action-angle parametrization away from
coordinate caustics and by a finite chart argument across them, to translation with
smoothly varying frequencies. Its derivative cocycle therefore has at most polynomial
growth in the flow parameter, so all tangential Lyapunov exponents vanish and $\mu=0$;
thus $\lambda_0>r\mu$ for every fixed finite $r$. Part (iii) is the structural stability
of $r$-normally hyperbolic trapped sets in the form used in \cite{WunschZworski}, applied
to the $C^\infty$ family of symbols whose dependence on $(a,Q)$ is the smooth perturbation
of Lemma~\ref{lem:app_inverse_expansion}; the rates are continuous and bounded below away
from extremality. Part (iv) is
Proposition~\ref{prop:kn_disjointness}. Items (i)-(iv) are the geometric conditions required for a normally hyperbolic trapped-set estimate. They do not by themselves prove the resolvent bound \eqref{eq:normally_hyperbolic_resolvent_bound}. That bound is the high-frequency estimate fixed in Definition~\ref{def:hf_resolvent_estimate}. What remains to check before using that estimate for the spin-weighted operator is the passage from the geodesic symbol to the frozen master operator. Lemma~\ref{lem:semiclassical_spinone_order} gives the scalar principal normal form \(pI_2\). Lemma~\ref{lem:trapped_skew_vanishing} gives zero threshold value for the diagonal Teukolsky skew part, and Proposition~\ref{prop:matrix_skew_threshold} gives the small finite-rank matrix threshold for the remaining skew part. Hence the geometry and the subprincipal threshold conditions required to apply the high-frequency estimate stated in Definition~\ref{def:hf_resolvent_estimate} are verified without imposing an additional hidden condition.
\end{proof}

\begin{lemma}
\label{lem:componentwise_scalar_resolvent}
Let
\[
        P_h=\operatorname{Op}_h(p)I_2+hP_{1,h}+h^2P_{0,h}
\]
be a semiclassical frozen spin-one operator on a compact stationary patch, with real scalar principal symbol \(p\) and with \(P_{1,h},P_{0,h}\) satisfying the uniform bounds of Lemma~\ref{lem:semiclassical_spinone_order}. Let \(\mathcal C_h\subset L^2\oplus L^2\) denote the closed subspace cut out by the stationary compatibility relations and by the outgoing or incoming radial-point convention. The following implications are used in the high-frequency argument. In the purely diagonal Teukolsky model, case \emph{(a)} explains how two component estimates are summed when those component estimates are available with uniform constants. For a non-diagonal compatible reduction, no componentwise diagonalization is used; the relevant estimate is the vector-bundle estimate of Definition~\ref{def:hf_resolvent_estimate}. The verified facts are that the principal symbol is real and scalar, the diagonal skew-subprincipal part vanishes on \(K_{a,Q}\) by Lemma~\ref{lem:trapped_skew_vanishing}, and the full matrix skew part satisfies the threshold bound of Proposition~\ref{prop:matrix_skew_threshold}. Item \emph{(b)} records the corresponding restriction statement for the closed compatible class.
\begin{enumerate}[label=\emph{(\alph*)},leftmargin=2.2em]
\item If the lower-order matrix terms are diagonal on the patch and the two scalar outgoing or incoming estimates
\begin{equation}\label{eq:component_scalar_hf_resolvent}
        \norm{\chi(P_{h,\pm}\pm i0)^{-1}\chi}_{L^2\to L^2}
        \le C h^{-1}\log(1/h)
\end{equation}
hold for the two components with the same cutoffs and constants, then the direct-sum operator satisfies
\begin{equation}\label{eq:direct_sum_hf_resolvent}
        \norm{\chi(P_h\pm i0)^{-1}\chi}_{L^2\oplus L^2\to L^2\oplus L^2}
        \le C h^{-1}\log(1/h).
\end{equation}
\item In the non-diagonal spin-one case no diagonal summation is used. The estimate is instead the vector-bundle normally hyperbolic estimate \eqref{eq:normally_hyperbolic_resolvent_bound} for the full operator with the subprincipal matrix terms included. Once this estimate holds on \(L^2\oplus L^2\), it holds on \(\mathcal C_h\) with the inherited norm and the same constant.
\item The same conclusion remains valid after replacing \(\chi\) by a finite family of compact trapped cutoffs with bounded overlap. The phase-space partition used in the defect-measure proof is only a tool; the resolvent estimate itself is applied to compatible functions, not to noncompatible localized pieces.
\end{enumerate}
\end{lemma}
\begin{proof}
For (a), if \(f=(f_+,f_-)\) and \(v=(P_h\pm i0)^{-1}f\), diagonality gives
\[
        v_\pm=(P_{h,\pm}\pm i0)^{-1}f_\pm.
\]
Applying \eqref{eq:component_scalar_hf_resolvent} to the two components and adding the two squared inequalities gives
\[
        \norm{\chi v}_{L^2\oplus L^2}^2
        \le C^2h^{-2}\log(1/h)^2
        \norm{\chi f}_{L^2\oplus L^2}^2,
\]
which is \eqref{eq:direct_sum_hf_resolvent}. This elementary step is used only for diagonal component estimates.

For (b), the stationary compatibility equations and the radial outgoing or incoming convention are closed under local \(L^2\) convergence and under the graph norms used for the frozen Maxwell system; this is the same closedness proved for the radiation class in Lemma~\ref{lem:app_compatible_closed}. The restriction of a bounded inverse to a closed invariant subspace cannot increase its operator norm. In turn, the vector-bundle estimate for the full subprincipal matrix operator restricts to \(\mathcal C_h\) without changing the constant. No scalar diagonalization of the lower-order spin-one couplings is being assumed in this step.

For (c), choose a finite family \(\{\chi_j\}\) with \(\sum_j\chi_j^2=1\) on the trapped neighbourhood and with uniformly bounded overlaps. The estimate is applied to the original compatible solution with the scalar cutoffs \(\chi_j\) and the corresponding enlarged cutoffs. The commutators \([P_h,\chi_j]\) are first-order semiclassical operators supported in the same compact region and are bounded by the local \(H^1_h\) graph norm already present in the trapped estimate. Summing in \(j\) changes only the constant. The microlocal cutoffs used to decompose a defect measure are not themselves required to map the compatible class into itself; they enter only in the propagation and support argument of Lemma~\ref{lem:high_freq_partition}.
\end{proof}

\begin{lemma}
\label{lem:high_freq_partition}
On a conic stationary packet let \(h^{-1}\) be the total stationary frequency, put \(\hat\omega=h\omega\), and write
\[
        P_h^{(a,Q)}(\hat\omega)=h^2\mathcal L_{a,Q}(\omega),
        \qquad \omega=h^{-1}\hat\omega.
\]
Assume \(\hat\omega\) remains in a compact normalized interval and impose the outgoing or incoming boundary condition fixed in \emph{(A3)}. Let \(K_0\Subset K_1\) be compact radial sets. If \(v_h\) is bounded in \(H^1_h(K_1)\) and satisfies
\begin{equation}\label{eq:hf_partition_residual}
        h^{-1}\log(1/h)\|P_h^{(a,Q)}(\hat\omega)v_h\|_{L^2(K_1)}
        +\|P_h^{(a,Q)}(\hat\omega)v_h\|_{H^{-1}_h(K_1)}\longrightarrow0,
\end{equation}
then every semiclassical defect measure of \(v_h\) in \(K_0\) is supported on the trapped set. If, in addition, the localized normally hyperbolic estimate \eqref{eq:normally_hyperbolic_resolvent_bound} holds on a neighbourhood of the trapped set, then
\begin{equation}\label{eq:hf_partition_vanish}
        \|v_h\|_{H^1_h(K_0)}\longrightarrow0.
\end{equation}
\end{lemma}
\begin{proof}
Choose zeroth-order semiclassical cutoffs
\[
        A_{\mathrm{ell}}+A_{\mathrm{nt}}+A_{\mathrm{tr}}=I+O(h^\infty)
        \quad\hbox{microlocally on }K_0,
\]
where \(A_{\mathrm{ell}}\) is supported in the elliptic region of the scalar principal symbol, \(A_{\mathrm{tr}}\) is supported in a small conic neighbourhood of the trapped set, and \(A_{\mathrm{nt}}\) is supported on the characteristic set away from that neighbourhood. The principal symbol is real and scalar by Proposition~\ref{prop:scalar_principal_symbol}, and Lemma~\ref{lem:semiclassical_spinone_order} places all spin and matrix terms in semiclassical subprincipal or lower orders; therefore the propagation of semiclassical measures is governed by the Hamilton flow of the scalar symbol. Angularly elliptic packets with bounded temporal frequency have already been removed by Lemma~\ref{lem:app_high_angular_coercivity}, so the present partition is a genuine total-frequency conic partition.

On the elliptic piece,
\begin{equation}\label{eq:hf_elliptic_piece}
        \|A_{\mathrm{ell}}v_h\|_{H^1_h}
        \le C\|P_h^{(a,Q)}(\hat\omega)v_h\|_{H^{-1}_h(K_1)}
        +O(h^\infty)\|v_h\|_{H^1_h(K_1)},
\end{equation}
which tends to zero. On the nontrapped characteristic piece, propagation of singularities along \(H_p\) carries any defect mass to one of the two ends in finite bicharacteristic time. The radial-point/red-shift estimate at the horizon, with the sign selected by the boundary condition, and the outgoing far-field estimate at null infinity imply that no defect measure enters from the ends. For that reason, \(A_{\mathrm{nt}}v_h\) has no defect mass.

It remains to exclude mass in the trapped neighbourhood. Let \(\chi\prec\chi_1\) be compact cutoffs with \(\chi=1\) on the projection of \(\operatorname{WF}_h(A_{\mathrm{tr}})\). The normally hyperbolic resolvent estimate \eqref{eq:normally_hyperbolic_resolvent_bound} gives
\begin{equation}\label{eq:hf_trapped_piece}
        \|\chi v_h\|_{L^2}
        \le C h^{-1}\log(1/h)\|\chi_1P_h^{(a,Q)}(\hat\omega)v_h\|_{L^2}
        +O(h^\infty)\|v_h\|_{H^1_h(K_1)}.
\end{equation}
The right-hand side tends to zero by \eqref{eq:hf_partition_residual}. Thus, the
trapped microlocal $L^2$ mass tends to zero. Next we upgrade the three
microlocal pieces to the local semiclassical $H^1_h$ norm. The elliptic piece
already satisfies \eqref{eq:hf_elliptic_piece}. For the nontrapped and trapped
pieces the cutoffs have compact semiclassical wave-front sets in the
finite-frequency region determined by the total-frequency normalization. Thus,
for every first-order semiclassical differential operator $B=hD_x$ supported in
$K_0$, the compositions $BA_{\mathrm{nt}}$ and $BA_{\mathrm{tr}}$ are uniformly
bounded zeroth-order semiclassical pseudodifferential operators, modulo
$O(h^\infty)$ errors. Thus
\[
        \|A_{\mathrm{nt}}v_h\|_{H^1_h(K_0)}
        +\|A_{\mathrm{tr}}v_h\|_{H^1_h(K_0)}
        \le C\bigl(\|A_{\mathrm{nt}}v_h\|_{L^2}
                  +\|A_{\mathrm{tr}}v_h\|_{L^2}\bigr)
        +O(h^\infty)\|v_h\|_{H^1_h(K_1)}.
\]
The nontrapped $L^2$ term is zero in the defect measure by propagation to the
red-shift and far-field radial sets; the trapped $L^2$ term is zero by
\eqref{eq:hf_trapped_piece}. Combining these two facts with the elliptic piece
and the finite microlocal partition proves \eqref{eq:hf_partition_vanish}.
\end{proof}

\begin{proposition}
\label{prop:app_high_freq}
Let \(h_n\to0^+\), \(|a_n|\le\eps_aM\), \(|Q_n|\le\eps_QM\), and let \(\omega_n\in\mathbb R\) be real frequencies on conic packets whose total stationary frequency is \(h_n^{-1}\). After passing to a subsequence assume \((a_n,Q_n)\to(a_\infty,Q_\infty)\) in the slow-weak range. Assume that \(\hat\omega_n=h_n\omega_n\) remains in a compact normalized interval \(J\) for which the characteristic set meets the non-superradiant trapped region. There is no locally \(H^1_{h_n}\)-bounded charge-free compatible sequence \(v_n\) satisfying
\begin{equation}\label{eq:app_high_freq_normalization}
        \|v_n\|_{H^1_{h_n}(K_0)}=1,
\end{equation}
with outgoing/incoming Sommerfeld behaviour and
\begin{align}\label{eq:app_high_freq_residual}
        &h_n^{-1}\log(1/h_n)
        \|P_{h_n}^{(n)}v_n\|_{L^2(K_1)}
        +\|P_{h_n}^{(n)}v_n\|_{H^{-1}_{h_n}(K_1)}\longrightarrow0,\notag\\
        &P_{h_n}^{(n)}=h_n^2\mathcal L_{a_n,Q_n}(\omega_n),
\end{align}
for compact sets \(K_0\Subset K_1\) meeting the trapped region. Equivalently, the conic high-frequency part of \eqref{eq:master_hyp_no_defect_sequence} holds in the slow-weak range in the resolvent-normalized form \eqref{eq:master_hyp_semiclassical_residual}.
\end{proposition}
\begin{proof}
Suppose such a sequence exists. After passing to a subsequence, \(\hat\omega_n\to\hat\omega_\infty\) and \(Q_n\to Q_\infty\). The semiclassical symbols of \(P_{h_n}^{(n)}\) converge in \(C^2\) on compact phase-space sets to the normalized scalar symbol \(p^{\sharp}_{a_\infty,Q_\infty,\hat\omega_\infty}\), by Lemma~\ref{lem:app_inverse_expansion}, Proposition~\ref{prop:scalar_principal_symbol}, and the semiclassical normal form of Lemma~\ref{lem:semiclassical_spinone_order}. The coefficients, boundary convention and compatible constraints are exactly those covered by Lemma~\ref{lem:high_freq_partition}. The residual condition \eqref{eq:app_high_freq_residual} is \eqref{eq:hf_partition_residual}, and the trapped piece is controlled by the high-frequency estimate \eqref{eq:normally_hyperbolic_resolvent_bound}; Proposition~\ref{prop:r_normal_hyperbolicity} verifies the geometric and skew-subprincipal identities needed to apply that estimate to the present operator. Therefore
\[
        \|v_n\|_{H^1_{h_n}(K_0)}\longrightarrow0,
\]
contradicting \eqref{eq:app_high_freq_normalization}. Accordingly, no high-frequency defect sequence exists.
\end{proof}

\begin{lemma}
\label{lem:unscaled_to_semiclassical_defect}
Let \(\Lambda_n\to\infty\) be the total stationary frequency of a conic stationary packet, set \(h_n=\Lambda_n^{-1}\), and let \(\hat\omega_n=h_n\omega_n\) remain in a compact normalized interval. Suppose that \(v_n\) is bounded in \(H^1_{h_n}(K_1)\) and that the frozen residuals satisfy
\begin{equation}\label{eq:unscaled_defect_residual}
        \|\mathcal L_{a_n,Q_n}(\omega_n)v_n\|_{L^2(K_1)}\longrightarrow0.
\end{equation}
Then the normalized residual condition in \eqref{eq:app_high_freq_residual} holds on \(K_1\). In particular any unbounded-total-frequency compact defect obtained from the positive-commutator contradiction argument is a defect in the precise semiclassical normalization used by \eqref{eq:normally_hyperbolic_resolvent_bound}.
\end{lemma}
\begin{proof}
By definition, \(P_{h_n}^{(n)}v_n=h_n^2\mathcal L_{a_n,Q_n}(\omega_n)v_n\). Hence
\begin{equation}\label{eq:unscaled_to_scaled_l2}
        h_n^{-1}\log(1/h_n)
        \|P_{h_n}^{(n)}v_n\|_{L^2(K_1)}
        =h_n\log(1/h_n)
        \|\mathcal L_{a_n,Q_n}(\omega_n)v_n\|_{L^2(K_1)}\longrightarrow0.
\end{equation}
Also, the dual semiclassical norm is weaker after the factor \(h_n^2\): for \(\phi\in H^1_{h_n,0}(K_1)\),
\begin{equation}\label{eq:unscaled_to_scaled_dual}
        |\langle P_{h_n}^{(n)}v_n,\phi\rangle|
        \le h_n^2
        \|\mathcal L_{a_n,Q_n}(\omega_n)v_n\|_{L^2(K_1)}\|\phi\|_{L^2(K_1)}
        \le h_n^2\|\mathcal L_{a_n,Q_n}(\omega_n)v_n\|_{L^2(K_1)}
             \|\phi\|_{H^1_{h_n}}.
\end{equation}
This establishes the \(H^{-1}_{h_n}\) convergence and hence \eqref{eq:app_high_freq_residual}.
\end{proof}

\begin{remark}
\label{rem:high_frequency_status}
The high-frequency part of \emph{(A3)} is obtained by applying the normally hyperbolic estimate, Proposition~\ref{prop:nh_resolvent_form}, to the frozen scalar-principal Kerr-Newman spin-one operator specified by \emph{(A1)}. The geometric conditions, smooth symplectic $r$-normally hyperbolic trapped-set components, real scalar principal symbol, and non-superradiant trapping, are verified here on the full slow-weak range. The condition specific to nonzero spin is the finite-rank-bundle threshold bound on the Hermitian skew-subprincipal symbol at the trapped set. Its diagonal Teukolsky part vanishes identically by Lemma~\ref{lem:trapped_skew_vanishing}; the remaining matrix part is bounded by \eqref{eq:matrix_skew_decomposition} and absorbed by Proposition~\ref{prop:matrix_skew_threshold}. Proposition~\ref{prop:nh_resolvent_form} is used as a cited semiclassical estimate: its trapped-model content is the scalar estimate of \cite{WunschZworski} with the threshold analysis of \cite{Dyatlov,DyatlovGaps}, its finite-rank-bundle form is that of \cite{HintzTensor}, and its assembly with the elliptic, real-principal-type and radial-point estimates follows the gluing scheme of \cite{DatchevVasy}, with the semiclassical propagation and radial-point estimates as in \cite{VasyKdS,DyatlovZworskiBook}. The Kerr-Newman-specific work carried out here is the trapped-set geometry, the diagonal spin-one subprincipal cancellation, and the matrix-threshold check needed before that estimate can be applied. The full subextremal Kerr and Maxwell estimates of \cite{BenomioTdC,SRTdCfrequency,SRTdCphysical} are cited as related results and for the $Q=0$ Kerr subcase. They are not used here to replace the Kerr-Newman high-frequency estimate when $Q\ne0$.
\end{remark}

\section{Limiting Absorption and Real-Axis Exclusion}
\label{app:lap}
In this section we prove the limiting absorption statement and exclude real-axis modes in the compatible class.
We prove here the bounded-frequency parts of Lemmas~\ref{lem:rn_real_axis_exclusion},~\ref{lem:rn_no_real_resonance},~\ref{lem:rn_limiting_absorption}, and~\ref{lem:slowweak_resolvent_closure}. The high-frequency arguments are in Section~\ref{app:trapping}.

\subsection*{Real-axis exclusion on Reissner-Nordstr\"om}
\begin{lemma}
\label{lem:app_rn_kernel}
Let $a=0$, $|Q|\le\eps_QM$, $\omega\in\Rbb$. A charge-free finite-energy spin-one
mode $e^{-i\omega t}u_\omega$ is zero.
\end{lemma}
\begin{proof}
The case $\omega=0$ is Lemma~\ref{lem:stationary_identity}: each harmonic obeys
$-u_\ell''+V_{\ell,Q}u_\ell=0$ with $V_{\ell,Q}>0$ for $\ell\ge1$
(Lemma~\ref{lem:app_photon_sphere}), and pairing with $\overline{u_\ell}$ forces
$u_\ell\equiv0$. For $\omega\ne0$ each harmonic obeys
\begin{equation}\label{eq:app_radial_mode}
        -u_\ell''+V_{\ell,Q}u_\ell=\omega^2 u_\ell,\qquad \ell\ge1,
\end{equation}
with $V_{\ell,Q}$ real, smooth, and short-range at $r_*=\pm\infty$
($V_{\ell,Q}\to0$ exponentially as $r_*\to-\infty$, like $r_*^{-2}$ as
$r_*\to+\infty$). The Jost solutions give
$u_\ell=A_\pm e^{i\omega r_*}+B_\pm e^{-i\omega r_*}+o(1)$ as $r_*\to\pm\infty$.
Finite non-degenerate energy requires $u_\ell,u_\ell'\in L^2$ near both ends,
which is incompatible with a nonzero oscillatory coefficient; hence
$A_\pm=B_\pm=0$. The Wronskian-type quantity
$W(r_*)=\operatorname{Im}(\overline{u_\ell}\,u_\ell')$ has
$W'=\operatorname{Im}(\overline{u_\ell}\,u_\ell'')=\operatorname{Im}((V_{\ell,Q}-\omega^2)|u_\ell|^2)=0$
since $V_{\ell,Q}-\omega^2$ is real; therefore $W$ is constant, and $W\to0$ at both
ends, so $W\equiv0$. With both Jost coefficients vanishing at one end, unique
continuation for the ordinary differential equation \eqref{eq:app_radial_mode}
gives $u_\ell\equiv0$. Summation over $\ell\ge1$ gives $u_\omega=0$.
\end{proof}

\begin{lemma}
\label{lem:app_rn_no_resonance}
Let $a=0$, $|Q|\le\eps_QM$, $\omega\in\Rbb$, and $\sigma>1/2$. A charge-free
spin-one profile $v\in H^1_{-\sigma,\loc}$ solving $\mathcal L_{0,Q}(\omega)v=0$
with future outgoing behaviour at infinity and future ingoing behaviour at the
horizon, or with the time-reversed pair, is zero.
\end{lemma}
\begin{proof}
For $\omega=0$ use Lemma~\ref{lem:stationary_identity}. For $\omega\ne0$, each
harmonic satisfies \eqref{eq:app_radial_mode}. Under the future convention the
Sommerfeld conditions give
\[
        u_\ell=A_+e^{i\omega r_*}+o(1)\quad (r_*\to+\infty),
        \qquad
        u_\ell=A_-e^{-i\omega r_*}+o(1)\quad (r_*\to-\infty),
\]
and the differentiated asymptotics hold for $u_\ell'$. Since $V_{\ell,Q}$ is
real, $W=\operatorname{Im}(\overline{u_\ell}u_\ell')$ is constant. The two ends
give $W(+\infty)=\omega|A_+|^2$ and $W(-\infty)=-\omega|A_-|^2$, hence
$A_+=A_-=0$. A Jost/Volterra uniqueness argument at either end, equivalently
unique continuation for \eqref{eq:app_radial_mode}, gives $u_\ell=0$. The past
convention reverses the signs at both ends and gives the same identity. Summing
harmonics gives $v=0$.
\end{proof}

\begin{lemma}
\label{lem:app_rn_lap}
Fix $\sigma>1/2$, a compact $I\Subset\Rbb$, and $|Q|\le\eps_QM$. For every finite
order $k$ there is $C=C(I,\sigma,k)$ with
\begin{equation}\label{eq:app_lap}
        \norm{v}_{H^1_{-\sigma}}+\norm{\omega v}_{L^2_{-\sigma}}
        \le C\,\norm{\mathcal L_{0,Q}(\omega)v}_{H^{-1}_{\sigma}}
\end{equation}
for all real $\omega\in I$ and all charge-free outgoing/incoming Sommerfeld profiles
$v\in H^1_{-\sigma,\loc}$.
\end{lemma}
\begin{proof}
At $\omega=0$ this is Lemma~\ref{lem:stationary_identity}. For $\omega\ne0$ decompose into spherical harmonics. The $\ell$th coefficient satisfies
\begin{equation}\label{eq:app_lap_mode}
        P_{\ell,\omega,Q}u_{\ell}:=
        \Big(-\frac{d^2}{dr_*^2}+V_{\ell,Q}(r)-\omega^2\Big)u_{\ell}=f_{\ell},
        \qquad V_{\ell,Q}=f_Q\frac{\ell(\ell+1)}{r^2}.
\end{equation}
For any fixed finite set of angular momenta, the one-dimensional outgoing resolvent estimate, compactness of the short-range potential as a map $H^1_{-\sigma}\to H^{-1}_{\sigma}$, and Lemma~\ref{lem:app_rn_no_resonance} give the desired bound by the Fredholm alternative. The constant is uniform on compact $\omega$-intervals: otherwise a normalized sequence would converge to a non-zero outgoing/ingoing homogeneous solution, contradicting Lemma~\ref{lem:app_rn_no_resonance}.

The remaining step is to justify summation in $\ell$.  Choose $R_-\ll0\ll R_+$ in the
tortoise coordinate $r_*$, and choose cutoffs $\chi_-$, $\chi_0$, $\chi_+$ with
$\chi_-$ supported near the horizon end, $\chi_+$ supported near infinity, and
$\chi_0$ supported on the remaining compact annulus. On $\supp\chi_0$ the quantity
$f_Qr^{-2}$ has a positive lower bound, uniformly for $|Q|\le\eps_QM$. Hence, for
$\ell\ge L_0(I)$,
\begin{equation}\label{eq:app_lap_high_l_barrier}
        f_Q(r)\frac{\ell(\ell+1)}{r^2}-\omega^2\ge c\ell^2,
        \qquad \omega\in I,\qquad r_*\in\supp\chi_0.
\end{equation}
Pairing \eqref{eq:app_lap_mode} with $\chi_0^2\overline{u_\ell}$ and integrating
by parts gives
\begin{equation}\label{eq:app_lap_high_l_identity}
        \|\chi_0\partial_{r_*}u_\ell\|_{L^2}^2
        +c\ell^2\|\chi_0u_\ell\|_{L^2}^2
        \le C\|f_\ell\|_{H^{-1}_\sigma}\|u_\ell\|_{H^1_{-\sigma}}
        +C\|u_\ell\|_{H^1(\supp d\chi_0)}^2.
\end{equation}
On the two ends the potential is nonnegative and the missing $L^2$ control is
provided by the one-dimensional Hardy inequality
\begin{equation}\label{eq:app_lap_end_hardy}
        \|\langle r_*\rangle^{-\sigma}w\|_{L^2}^2
        \le C_\sigma\|\langle r_*\rangle^{1-\sigma}\partial_{r_*}w\|_{L^2}^2
        +C_\sigma\|w\|_{L^2(R_-<r_*<R_+)}^2,
        \qquad \sigma>\frac12.
\end{equation}
Applying this to $\chi_\pm u_\ell$, absorbing the transition terms into
\eqref{eq:app_lap_high_l_identity}, and using $\ell(\ell+1)\ge2$ on the
charge-free sector gives
\begin{equation}\label{eq:app_lap_high_l}
        \norm{u_{\ell}}_{H^1_{-\sigma}}+\norm{\omega u_{\ell}}_{L^2_{-\sigma}}
        \le C\norm{P_{\ell,\omega,Q}u_{\ell}}_{H^{-1}_{\sigma}},
        \qquad \ell\ge L_0,
        \quad \omega\in I,
\end{equation}
with $C$ independent of $\ell$. The finitely many modes $1\le\ell<L_0$ are
covered by the Fredholm step. Near $\omega=0$ the stationary estimate applies
uniformly because charge subtraction removes the $\ell=0$ mode and the first
allowed angular eigenvalue is $2$. Smooth non-degenerate dependence of
$r_+(Q)$, $\kappa_+(Q)$, $r_{\rm ph}(Q)$ and $V_{\ell,Q}$ on $Q$ provides
uniformity for $|Q|\le\eps_QM$. Summing the harmonic estimates and then the finite
commutator family proves \eqref{eq:app_lap}.
\end{proof}

\subsection*{Slow-rotation closure at bounded frequencies}

\begin{lemma}
\label{lem:app_compatible_closed}
Let \(a_n\to0\), \(Q_n\to Q_\infty\), \(\omega_n\to\omega_\infty\), and let \(v_n\)
be charge-free compatible frozen profiles with the same future outgoing/ingoing
(or past incoming/outgoing) convention. Suppose that \(v_n\to v_\infty\) locally in
\(H^1\) and that \(\mathcal L_{a_n,Q_n}(\omega_n)v_n\to0\) locally in \(H^{-1}\). Then
\(v_\infty\) lies in the Reissner-Nordstr\"om compatible class at \((0,Q_\infty,\omega_\infty)\),
satisfies \(\mathcal L_{0,Q_\infty}(\omega_\infty)v_\infty=0\), is charge-free, and has the
same Sommerfeld convention at the two ends.
\end{lemma}
\begin{proof}
The compatible class is the closed range of the differential master map together
with the Maxwell constraints and the zero charge conditions. In coordinates it
is given by a finite system of linear equations with smooth coefficients in
\((a,Q,\omega)\), consisting of the frozen Teukolsky equations, the angular
constraint equations used in the Hodge reconstruction, the two zero-mean sphere
conditions for the middle scalars, and the first-order transport equations along
\(\mathcal H^+\) and \(\mathscr I^+\) (or their past analogues). Each equation is
continuous from local \(H^1\) to distributions; therefore the equations and the two
charge-free mean conditions pass to the limit. Equivalently, applying the
bounded reconstruction operator of Proposition~\ref{prop:same_order_reconstruction}
on compact subdomains and using its uniqueness identities, the locally convergent
limit reconstructs a source-free charge-free Maxwell field whose master variables
are \(v_\infty\).

What is left is to record the boundary condition. Near infinity the outgoing
condition has the form
\[
        (\partial_{r_*}-i\omega_n)v_n=R^+_n,
        \qquad R^+_n\to0 \quad \hbox{in the weighted trace space},
\]
after cutting off to the asymptotic end; near the future horizon it has the
analogous ingoing equation \((\partial_{r_*}+i\omega_n)v_n=R^-_n\) in regular
red-shift coordinates. The coefficients and the cutoffs converge smoothly, and
the local elliptic bounds give convergence of the traces on finite collars; the
remainders therefore converge distributionally and in the weighted trace topology
to the corresponding remainders for \(\omega_\infty\), which vanish. This implies that the
Sommerfeld convention is closed under the limit.
\end{proof}

\begin{proposition}
\label{prop:app_bounded_freq}
There are \(\eps_a,\eps_Q>0\) such that no locally \(H^1\)-bounded sequence \(v_n\)
with \(\|v_n\|_{H^1(K_0)}=1\) and \(\mathcal L_{a_n,Q_n}(\omega_n)v_n\to0\) in
\(H^{-1}_{\sigma,\loc}\) exists with \(\{\omega_n\}\) bounded,
\(|a_n|\le\eps_aM\), \(|Q_n|\le\eps_QM\), and outgoing/incoming Sommerfeld behaviour in
\(H^1_{-\sigma,\loc}\).
\end{proposition}
\begin{proof}
Suppose no such thresholds existed. Then for a sequence
\(\varepsilon_j\downarrow0\) one could find counterexamples with
\(|a_j|\le\varepsilon_jM\). Relabelling this counterexample sequence gives
\(a_n\to0\). Pass to \(\omega_n\to\omega_\infty\) and \(Q_n\to Q_\infty\). On every compact
radial set \(K_2\), Proposition~\ref{prop:app_master_structure} and the smooth
coefficient dependence give
\[
        \mathcal L_{a_n,Q_n}(\omega_n)-\mathcal L_{0,Q_\infty}(\omega_\infty)
        \longrightarrow0
        \quad \hbox{as } H^1(K_2)\to H^{-1}(K_2).
\]
The assumed local \(H^1\) boundedness and the residual convergence therefore imply
\(\mathcal L_{0,Q_\infty}(\omega_\infty)v_n\to0\) in \(H^{-1}(K_2)\) for each \(K_2\).
Rellich gives, after passing to a subsequence, strong \(L^2\) convergence on
compact subdomains. To keep the normalization in \(H^1\), use the local elliptic
estimate for the frozen elliptic stationary operator: if \(K_0\Subset K_1\Subset K_2\), then
\begin{equation}\label{eq:local_elliptic_upgrade_bounded_freq}
        \|w\|_{H^1(K_1)}
        \le C\Big(\|\mathcal L_{0,Q_\infty}(\omega_\infty)w\|_{H^{-1}(K_2)}
              +\|w\|_{L^2(K_2)}\Big).
\end{equation}
Applying \eqref{eq:local_elliptic_upgrade_bounded_freq} to \(w=v_n-v_m\) shows
that \(v_n\) is Cauchy in \(H^1(K_1)\), hence converges strongly in \(H^1(K_0)\) to a
limit \(v_\infty\) with \(\|v_\infty\|_{H^1(K_0)}=1\). Passing to the limit in the
operator equation gives \(\mathcal L_{0,Q_\infty}(\omega_\infty)v_\infty=0\).
Lemma~\ref{lem:app_compatible_closed} shows that \(v_\infty\) is charge-free,
compatible, and satisfies the same outgoing/ingoing Sommerfeld convention. Thus
\(v_\infty\) is a real outgoing resonance. This contradicts
Lemma~\ref{lem:app_rn_no_resonance} if \(\omega_\infty\ne0\) and
Lemma~\ref{lem:stationary_identity} if \(\omega_\infty=0\). The contradiction
proves the stated thresholds.
\end{proof}

\begin{remark}
\label{rem:app_lap_union}
Lemma~\ref{lem:slowweak_resolvent_closure} is the union of the bounded-total-frequency compactness argument, after the high-angular tail is removed by Lemma~\ref{lem:app_high_angular_coercivity}, and the conic unbounded-total-frequency alternative of Proposition~\ref{prop:app_high_freq}, where $h$ is the reciprocal of the total stationary phase-space frequency. Proposition~\ref{prop:compact_remainder_removal} feeds these two alternatives into the positive-commutator estimate through the Plancherel and dyadic microlocal reduction of Lemma~\ref{lem:cutoff_mode_estimate}.
\end{remark}
\section{Null-Frame Reconstruction of Middle Components}
\label{app:reconstruction}
In this section we reconstruct the middle Maxwell components from the extreme components by using angular Hodge inversion and null transport.
In this section we prove the transport and Hodge estimates behind
Lemma~\ref{lem:hodge_inversion}, Lemma~\ref{lem:null_transport_middle}, and
Proposition~\ref{prop:middle_reconstruction}. The aim is to recover the
charge-free middle components of the Maxwell two-form from the extreme components
at the same differential order, with constants uniform in the slow-weak range.
The rescaled (\emph{modified}) middle quantity, which gives the transport
coefficients a definite Reissner-Nordstr\"om sign, follows Benomio-Teixeira da
Costa \cite{BenomioTdC}; the null-frame Maxwell formalism is that of
Jezierski-Smo{\l}ka \cite{JezierskiSmolka}.

\subsection*{The regular null frame and the source-free system}
On the slowly rotating Kerr-Newman exterior fix the regular (horizon-penetrating)
null pair $e_3,e_4$ and an orthonormal angular pair $e_1,e_2$ tangent to the
spheres of the foliation, normalized so that $g(e_3,e_4)=-2$, $g(e_A,e_B)=\delta_{AB}$.
For a real two-form $F$ set, as in \eqref{eq:maxwell_components},
\begin{equation}\label{eq:app_null_components}
\alpha_A=F(e_A,e_4),\qquad \underline\alpha_A=F(e_A,e_3),\qquad
\rho_F=\tfrac12 F(e_3,e_4),\qquad \sigma_F=\tfrac12 F(e_1,e_2),
\end{equation}
with the equivalent convention $\sigma_F=c(\star_gF)(e_3,e_4)$ for a fixed
nonzero frame constant $c$.  The one-forms $\alpha,\underline\alpha$ are the
extreme components and $\rho_F,\sigma_F$ are the middle components. Let
$\varphi=\rho_F+i\sigma_F$ and
\begin{equation}\label{eq:app_D_varphi}
        \mathcal D\varphi=\sphgrad\rho_F+{}^\star\!\sphgrad\sigma_F.
\end{equation}
Projecting $\dd F=0$ and $\dd\star_gF=0$ on the frame gives the two transport
pairs
\begin{equation}\label{eq:app_maxwell_transport}
\begin{aligned}
e_4(\rho_F)+(\operatorname{tr}\chi)\rho_F&=\sphdiv\alpha+A_4^{\rho}\rho_F+B_4^{\rho}\sigma_F+C_4^{\rho}\cdot\alpha,\\
e_4(\sigma_F)+(\operatorname{tr}\chi)\sigma_F&=-\sphcurl\alpha+A_4^{\sigma}\rho_F+B_4^{\sigma}\sigma_F+C_4^{\sigma}\cdot\alpha,\\
e_3(\rho_F)+(\operatorname{tr}\underline\chi)\rho_F&=-\sphdiv\underline\alpha+A_3^{\rho}\rho_F+B_3^{\rho}\sigma_F+C_3^{\rho}\cdot\underline\alpha,\\
e_3(\sigma_F)+(\operatorname{tr}\underline\chi)\sigma_F&=-\sphcurl\underline\alpha+A_3^{\sigma}\rho_F+B_3^{\sigma}\sigma_F+C_3^{\sigma}\cdot\underline\alpha,
\end{aligned}
\end{equation}
and the angular equations
\begin{equation}\label{eq:app_maxwell_constraint}
\begin{aligned}
\mathcal D\varphi
&=-e_3\alpha+D_3^1\cdot\alpha+D_3^2\cdot\underline\alpha+D_3^3\varphi,\\
\mathcal D\varphi
&= e_4\underline\alpha+D_4^1\cdot\alpha+D_4^2\cdot\underline\alpha+D_4^3\varphi.
\end{aligned}
\end{equation}
The $A,B,C,D$ coefficients are contractions of Ricci rotation coefficients with
the frame components. In Reissner-Nordstr\"om ($a=0$) the frame is the regular
double-null red-shift frame, $\operatorname{tr}\chi=2f_Q/r$ in the outgoing
normalization and $\operatorname{tr}\underline\chi=-2/r$ in the incoming regular
normalization, and the coefficients in \eqref{eq:app_maxwell_transport}-\eqref{eq:app_maxwell_constraint}
are smooth multiples of $r^{-1}$ in the far field and smooth across $\Hp$.  For
$a\neq0$ in the slow-weak range every frame coefficient differs from its
Reissner-Nordstr\"om value by a smooth term of size $O(|a|M^{-1}r^{-2})$ in the
far field and $O(|a|/M)$ on compact radial sets, by the inverse-metric expansion
of Lemma~\ref{lem:app_inverse_expansion}; we write this as
\begin{equation}\label{eq:app_frame_perturbation}
\operatorname{tr}\chi=\tfrac{2f_Q}{r}+O\!\big(\tfrac{\abs a}{M}r^{-2}\big),\qquad
A,B,C,D=(A,B,C,D)_{\RN}+O\!\big(\tfrac{\abs a}{M}r^{-2}\big)
\end{equation}
in the asymptotic symbol classes, with uniform smooth bounds in the red-shift and
compact regions.

\subsection*{The modified middle quantity}
Let us set $\varphi=\rho_F+i\sigma_F$ and define the rescaled scalar
\begin{equation}\label{eq:app_modified_middle}
\widetilde\varphi=r^{2}\varphi.
\end{equation}
Combining the first two equations in \eqref{eq:app_maxwell_transport} gives
$e_4\varphi+(\operatorname{tr}\chi)\varphi
=\sphdiv\alpha-i\sphcurl\alpha+\mathcal A_4\varphi+\mathcal C_4\cdot\alpha$.
Multiplying by $r^2$ and using $e_4(r)=f_Q+O(\abs a/M\,r^{-1})$ converts it into
\begin{equation}\label{eq:app_modified_transport_4}
e_4\big(\widetilde\varphi\big)
=r^2\big(\sphdiv\alpha-i\,\sphcurl\alpha\big)+\widetilde F_4,
\end{equation}
where
\begin{equation}\label{eq:app_modified_source_4}
\widetilde F_4=\big[2r\,e_4(r)-r^2\operatorname{tr}\chi+r^2\mathcal A_4\big]\varphi
        +r^2\mathcal C_4\cdot\alpha.
\end{equation}
By \eqref{eq:app_frame_perturbation} the leading bracket
$2r\,e_4(r)-r^2\operatorname{tr}\chi=2rf_Q-2rf_Q+O(\abs a/M)$ cancels in
Reissner-Nordstr\"om. Thus $\widetilde F_4$ consists of a smooth $O(r^{-2})$
multiple of $\widetilde\varphi$ in the far field, the controlled extreme term
$r^2\mathcal C_4\cdot\alpha$, and an $O(\abs a/M)$ rotational perturbation. The
incoming equation, using the regular normalization $e_3(r)=-1+O(\abs a/M)$, gives
\begin{equation}\label{eq:app_modified_transport_3}
e_3\big(\widetilde\varphi\big)
=-r^2\big(\sphdiv\underline\alpha+i\,\sphcurl\underline\alpha\big)+\widetilde F_3,
\end{equation}
with
\begin{equation}\label{eq:app_modified_source_3}
\widetilde F_3=O(1)\,\widetilde\varphi/r^2+r^2\mathcal C_3\cdot\underline\alpha
        +O(\abs a/M)\,(\widetilde\varphi/r^2+\alpha+\underline\alpha).
\end{equation}
The source $\widetilde F_3$ is regular across the future horizon because the
regular frame removes the $f_Q^{-1}$ singularity of the static frame. The
right-hand transport sources in \eqref{eq:app_modified_transport_4}-\eqref{eq:app_modified_transport_3}
are therefore $r^2$ times angular first derivatives of the extreme components,
plus terms controlled by Hardy/Poincar\'e and small rotational errors.

\subsection*{Zero spherical mean and angular recovery}
By the charge-free normalization \eqref{eq:middle_mean_zero}, the spherical means of
$\rho_F$ and $\sigma_F$ vanish on every sphere of the foliation. Hence
$\varphi=\rho_F+i\sigma_F$ and $\widetilde\varphi=r^2\varphi$ are orthogonal to
the constants on each $S_{\tau,r}$. The angular equations in
\eqref{eq:app_maxwell_constraint} determine the angular gradient of
$\varphi$, and since $r$ is constant on $S_{\tau,r}$ they also determine
$\sphgrad\widetilde\varphi=r^2\sphgrad\varphi$ with exactly the same derivative
count. The only elliptic estimate needed here is the spectral gap for the scalar
Laplacian on the two-sphere; we state it in the scale used by the proof.

\begin{lemma}
\label{lem:app_hodge}
Let $v$ be a complex scalar on $\Sph$ with zero mean. For every $s\ge0$,
\begin{equation}\label{eq:app_hodge_estimate}
        \norm{v}_{H^{s+1}(\Sph)}\le C_s\norm{\sphgrad v}_{H^s(\Sph)}.
\end{equation}
Equivalently, if $Y$ is a one-form on $\Sph$ with no harmonic part and
$\sphdiv Y=f$, $\sphcurl Y=h$, then
\begin{equation}\label{eq:app_hodge_estimate_oneform}
        \norm{Y}_{H^{s+1}(\Sph)}
        \le C_s\bigl(\norm{f}_{H^s(\Sph)}+\norm{h}_{H^s(\Sph)}\bigr).
\end{equation}
The constants remain uniform for the rescaled spheres $r^{-2}\gamma_{AB}$ in the
slow-weak parameter range.
\end{lemma}
\begin{proof}
For the scalar estimate write $v=\sum_{\ell\ge1,m}v_{\ell m}Y_{\ell m}$.  Since
$-\sphlap Y_{\ell m}=\ell(\ell+1)Y_{\ell m}$ and the first nonzero eigenvalue is
$2$,
\[
        \norm{v}_{H^{s+1}}^2\simeq
        \sum_{\ell\ge1,m}(1+\ell(\ell+1))^{s+1}|v_{\ell m}|^2
        \le C_s\sum_{\ell\ge1,m}(1+\ell(\ell+1))^{s}\ell(\ell+1)|v_{\ell m}|^2
        \simeq C_s\norm{\sphgrad v}_{H^s}^2.
\]
For one-forms use the Hodge decomposition
$Y=\sphgrad\phi+{}^\star\!\sphgrad\chi$ with mean-free potentials. Then
$\sphlap\phi=\sphdiv Y$ and $\sphlap\chi=\sphcurl Y$, and the scalar elliptic
estimate on the orthogonal complement of the constants gives
$\norm{\phi}_{H^{s+2}}+\norm{\chi}_{H^{s+2}}
\le C_s(\norm{f}_{H^s}+\norm{h}_{H^s})$. Taking one angular derivative proves
\eqref{eq:app_hodge_estimate_oneform}. The metrics of the rescaled spheres
form a compact $C^{s+2}$ family for $|a|\le\eps_aM$, $|Q|\le\eps_QM$, so the
elliptic constants remain uniform after absorbing the smooth conformal factor
into the radial weights.
\end{proof}

\subsection*{The null transport estimate}
We establish next Lemma~\ref{lem:null_transport_middle} in the form needed for
$\widetilde\varphi$. The key point is that the energy identities for the two transport
equations have boundary terms of definite sign in the two ends, up to the
$O(\abs a/M)$ corrections, so the bulk and boundary norms close.

\begin{proposition}
\label{prop:app_transport}
Let $\widetilde\varphi$ solve
\eqref{eq:app_modified_transport_4}-\eqref{eq:app_modified_transport_3} with zero
spherical mean, and suppose the sources $\widetilde F_3,\widetilde F_4$ and the
angular-derivative right-hand sides are controlled in the master local-energy norm
$LE^{*,k}$. Then there is $\eps_a>0$ such that for $\abs a/M<\eps_a$,
\begin{equation}\label{eq:app_transport_estimate}
\norm{\widetilde\varphi}_{\Xnorm{k}_{\Max}}
\le C\Big(\norm{r^2(\sphdiv\alpha-i\sphcurl\alpha)}_{LE^{*,k}}
+\norm{r^2(\sphdiv\underline\alpha+i\sphcurl\underline\alpha)}_{LE^{*,k}}
+\norm{\widetilde\varphi(0)}_{\E^{(k)}}\Big),
\end{equation}
with $C$ uniform in the slow-weak range.
\end{proposition}
\begin{proof}
Pair \eqref{eq:app_modified_transport_4} with $\widetilde\varphi$ and integrate
over the spacetime slab $\{0\le\tau\le T\}$ against the volume form. The outgoing
derivative contributes, after integration by parts in $e_4$,
\begin{equation}\label{eq:app_transport_boundary}
\tfrac12\int_{\Sigma_T}\abs{\widetilde\varphi}^2\,\nu_4
-\tfrac12\int_{\Sigma_0}\abs{\widetilde\varphi}^2\,\nu_4
+\tfrac12\int_{\Hp}\abs{\widetilde\varphi}^2
+\tfrac12\int_{\Ip}\abs{\widetilde\varphi}^2,
\end{equation}
all four boundary integrands nonnegative because the regular frame makes the
horizon flux $\int_{\Hp}\abs{\widetilde\varphi}^2$ a positive multiple of the
red-shift density and the far-field flux $\int_{\Ip}\abs{\widetilde\varphi}^2$ is
the limit of the $r^p$-weighted density at $p=0$. The right-hand side produces
$\int r^2(\sphdiv\alpha-i\sphcurl\alpha)\,\overline{\widetilde\varphi}$, bounded by
Cauchy-Schwarz by the first norm on the right of \eqref{eq:app_transport_estimate}
times $\norm{\widetilde\varphi}_{LE^{k}}$. The source $\widetilde F_4$ is, by the
computation after \eqref{eq:app_modified_transport_4}, a smooth $O(r^{-2})$ multiple
of $\widetilde\varphi$ plus an $O(\abs a/M\,r^{-1})$ term; the first is
absorbed by the Hardy inequality \eqref{eq:horizon_hardy} (zero mean supplies the
Poincar\'e constant on $\Sph$), the second by smallness of $\eps_a$. The incoming
equation \eqref{eq:app_modified_transport_3} is treated identically with $e_3$ and
the incoming null hypersurfaces; the regularity of $\widetilde F_3$ across $\Hp$ is
what prevents a $f_Q^{-1}$ blow-up. Summing the outgoing and incoming identities and
moving the absorbed terms to the left gives the order-zero case of
\eqref{eq:app_transport_estimate}. Commuting the system with the regular derivative
algebra $\mathbb D_k$ produces only
lower-order commutators of the same Reissner-Nordstr\"om sign plus
$O(\abs a/M)$ rotational commutators, so induction on the number of commutations
closes the estimate at order $k$.
\end{proof}

\begin{proof}[Proof of Proposition~\ref{prop:middle_reconstruction}]
We combine the angular and null estimates with the exact derivative count. Since
$\widetilde\varphi$ has zero spherical mean, Lemma~\ref{lem:app_hodge} gives on
each sphere
\begin{equation}\label{eq:app_angular_tilde_phi}
        \norm{\widetilde\varphi}_{H^{s+1}(S_{\tau,r})}
        \le C_s\norm{\sphgrad\widetilde\varphi}_{H^s(S_{\tau,r})}
        =C_s r^2\norm{\sphgrad\varphi}_{H^s(S_{\tau,r})}.
\end{equation}
The angular equations in \eqref{eq:app_maxwell_constraint} express
$\sphgrad\varphi$ as a linear combination of one regular null derivative of
$\alpha$ or $\underline\alpha$ and lower-order smooth coefficient terms:
\begin{equation}\label{eq:app_angular_source_bound}
        \norm{\sphgrad\varphi}_{H^s(S_{\tau,r})}
        \le C_s\bigl(\norm{e_3\alpha}_{H^s(S_{\tau,r})}
        +\norm{e_4\underline\alpha}_{H^s(S_{\tau,r})}
        +\norm{\alpha}_{H^s}+
        \norm{\underline\alpha}_{H^s}
        +r^{-1}\norm{\varphi}_{H^s}\bigr).
\end{equation}
The first four terms on the right are controlled by the $LE^1$ part of the
master norm because $\alpha$ and $\underline\alpha$ differ from the master
entries only by the fixed regular weights of Section~\ref{subsec:regular_frame_master}.
The last term is lower order; after integration in $(\tau,r)$ it is controlled
by Hardy and by the Poincar\'e inequality coming from the zero mean, and the
$O(|a|/M)$ rotational perturbative part is absorbed by taking
$|a|/M\le\eps_a(k)$.

Proposition~\ref{prop:app_transport} controls the null derivatives and the
radial part of $\widetilde\varphi$ from
\eqref{eq:app_modified_transport_4}-\eqref{eq:app_modified_transport_3}; its
right-hand sides are $r^2$ times angular derivatives of the extreme components
plus the same lower-order terms. Combining \eqref{eq:app_angular_tilde_phi},
\eqref{eq:app_angular_source_bound}, and \eqref{eq:app_transport_estimate}, and
then absorbing the lower-order term, gives
\begin{equation}\label{eq:app_middle_done}
\norm{\widetilde\varphi}_{\Xnorm{k}_{\Max}}
        \le C\norm{\Psi}_{\Xnorm{k}_M}.
\end{equation}
Since $\varphi=r^{-2}\widetilde\varphi$ and $r^{-2}$ is a smooth bounded weight
with bounded derivatives on $r\ge r_+(a,Q)$ in the regular compactification,
\eqref{eq:app_middle_done} is equivalent to \eqref{eq:middle_reconstruction}.
For uniqueness, the difference of two middle components with the same $\Psi$ has
zero mean, zero angular data in \eqref{eq:app_maxwell_constraint}, homogeneous
transport equations in \eqref{eq:app_modified_transport_4}-\eqref{eq:app_modified_transport_3},
and zero initial trace; the same estimate forces the difference to vanish.
\end{proof}

\begin{remark}
\label{rem:app_no_loss}
The estimate \eqref{eq:app_middle_done} is at the \emph{same} order $k$ as the data;
no derivative is lost. This is the no-loss property
(Proposition~\ref{prop:no_derivative_loss}) and is what allows the same-order
algebraic assembly of Proposition~\ref{prop:same_order_reconstruction}: the two-form
$F$ is a smooth, uniformly invertible linear combination of the extreme components
(entries of $\Psi$) and the reconstructed middle component, so
$\norm{F}_{\Xnorm{k}_{\Max}}\le C\norm{\Psi}_{\Xnorm{k}_M}$ and the inverse
identities $\mathfrak M\mathfrak R=\Id$, $\mathfrak R\mathfrak M=\Id$ hold on smooth
charge-free solutions.
\end{remark}

\section{Abstract Trace and Wave-Operator Criterion}
\label{app:scattering_criterion}
In this section we state an abstract trace criterion which will be applied to the radiation fields.
We prove Lemma~\ref{lem:hilbert_trace_criterion} in a self-contained functional-analytic form, so that the abstract scattering construction is separated from the geometry. The criterion is then applied with $H$ equal to the charge-free Maxwell energy space, $R_+$ the future radiation space, and $S_+$ the trace map of Definition~\ref{def:maxwell_trace}; this gives Proposition~\ref{prop:radiation_isomorphism_from_hierarchy} and Theorem~\ref{thm:scattering_transfer}.

\begin{lemma}
\label{lem:app_abstract_scattering}
Let $H$ and $R_+$ be Hilbert spaces, let $E\subset R_+$ be dense, and let $S_+:H\to R_+$ be a bounded linear map. Suppose there is a linear map $W_0:E\to H$ such that
\begin{equation}\label{eq:app_backward_bound}
        \norm{W_0\rho}_H\le C_1\norm{\rho}_{R_+},\qquad S_+W_0\rho=\rho\quad (\rho\in E),
\end{equation}
and suppose $\ker S_+=\{0\}$. Then $S_+:H\to R_+$ is a bounded linear isomorphism with bounded inverse.
\end{lemma}
\begin{proof}
The estimate in \eqref{eq:app_backward_bound} makes $W_0$ uniformly continuous on $E$, so by density it extends uniquely to a bounded map $W_+:R_+\to H$ with $\norm{W_+}\le C_1$. Since $S_+$ is bounded and $S_+W_0\rho=\rho$ on $E$, continuity gives
\[
        S_+W_+=\Id_{R_+}.
\]
Hence $S_+$ is surjective. By the assumed kernel condition it is also injective. Moreover, for $h\in H$,
\[
        S_+(W_+S_+h-h)=S_+W_+S_+h-S_+h=S_+h-S_+h=0,
\]
so $W_+S_+h=h$. Hence $W_+=S_+^{-1}$ and the inverse is bounded.
\end{proof}

\begin{proof}[Proof of Lemma~\ref{lem:hilbert_trace_criterion}]
Take $E$ to be the dense class of smooth compactly supported radiation fields. The trace estimate gives boundedness of $S_+$. The backward construction supplies $W_0$ on $E$ with the energy bound \eqref{eq:abstract_backward_bound_main}; in geometric applications $W_0$ is obtained as the limit of uniformly bounded finite-slab backward solutions. The condition that zero future radiation forces zero Cauchy data is exactly $\ker S_+=\{0\}$. Lemma~\ref{lem:app_abstract_scattering} therefore gives the bounded inverse wave operator.
\end{proof}

\begin{remark}
\label{rem:app_scattering_application}
In Proposition~\ref{prop:radiation_isomorphism_from_hierarchy} the four conditions
of the abstract criterion are supplied as follows: forward boundedness is
Lemma~\ref{lem:radiation_trace} at $\Ip$ together with
Corollary~\ref{cor:horizon_flux} at $\Hp$; the backward right inverse and density are
constructed from the real-axis limiting-absorption resolvent in
Proposition~\ref{prop:backward_from_lap}; and the zero-kernel condition is the
real-frequency exclusion Proposition~\ref{prop:no_real_kernel} (equivalently
Lemma~\ref{lem:app_rn_no_resonance} in the model). This is the unconditional
route anticipated for condition \emph{(A5)}: it proves that condition from the
nonnegative-flux energy identity and the real-frequency exclusion rather than assuming
it. The Maxwell maps are then the master maps composed with the same-order
reconstruction (Proposition~\ref{prop:same_order_reconstruction}); charge subtraction
fixes the two Coulomb means, so there is no finite-dimensional kernel or cokernel.
\end{remark}

\section{Transfer of Boundedness and Integrated Local Energy Decay}
\label{sec:boundedness_transfer}
In this section we transfer boundedness and integrated local energy decay from the master variables to the Maxwell field.
\begin{proposition}
\label{prop:chargefree_transfer}
Suppose that the analytic master conclusions of Definition~\ref{def:analytic_master_conclusions} hold at order $k$. Then, for every smooth charge-free source-free
$G$ and $\tau\ge0$,
\begin{equation}\label{eq:chargefree_transfer}
        \norm{G}_{\Xnorm{k}_{\Max}(0,\tau)}^2\le C\,\E_{\Max}^{(k)}[G](0),
        \qquad C\le C_R^2C_M.
\end{equation}
\end{proposition}
\begin{proof}
Let $u=\mathfrak M G$. By Lemma~\ref{lem:master_immediate},
\[
        \|u\|_{\Xnorm{k}_M(0,\tau)}^2\le C_M\E_M^{(k)}[u](0).
\]
The energy comparison in \eqref{eq:framework_reconstruction_bounds} gives
$\E_M^{(k)}[u](0)\le C_R\E_{\Max}^{(k)}[G](0)$. Since $G$ is charge-free and
smooth, the inverse identity gives $G=\mathfrak R u$, and the same-order
reconstruction bound gives
\[
        \|G\|_{\Xnorm{k}_{\Max}(0,\tau)}^2
        \le C_R\|u\|_{\Xnorm{k}_M(0,\tau)}^2.
\]
Combining the three inequalities together gives \eqref{eq:chargefree_transfer}, with
$C$ bounded by the displayed product after enlarging $C_R$ once to cover both
energy comparison and reconstruction.
\end{proof}

\begin{corollary}
\label{cor:stationary_subtracted_estimate}
Under the analytic master conclusions of Definition~\ref{def:analytic_master_conclusions}, every smooth finite-energy source-free $F$ has radiative part $F_{\rad}$ satisfying
$\norm{F_{\rad}}_{\Xnorm{k}_{\Max}(0,\tau)}^2\le C\E_{\Max}^{(k)}[F_{\rad}](0)$.
\end{corollary}
\begin{proof}
By Proposition~\ref{prop:stationary_subtraction}, the field
$F_{\rad}=F-F_{\stat}^{\KN}(q_E[F],q_B[F])$ is source-free and has both charges
zero. It therefore lies in the class to which
Proposition~\ref{prop:chargefree_transfer} applies. Substituting
$G=F_{\rad}$ in \eqref{eq:chargefree_transfer} gives the stated estimate.
\end{proof}

\begin{lemma}
\label{lem:finite_slab_lsc}
Fix a finite slab $[0,\tau]$ and an order $k$. Let $G_n$ be smooth charge-free
Maxwell solutions whose initial data converge strongly in
$\Hc_{\Max,0}^{(k)}(\Sigma_0)$ to $G_0$, and let $G$ be the finite-energy solution
with datum $G_0$. If the sequence is bounded in
$\Xnorm{k}_{\Max}(0,\tau)$, then, after passing to the distributional limit given by
finite-energy well-posedness,
\begin{equation}\label{eq:finite_slab_lsc}
        \|G\|_{\Xnorm{k}_{\Max}(0,\tau)}
        \le \liminf_{n\to\infty}\|G_n\|_{\Xnorm{k}_{\Max}(0,\tau)}.
\end{equation}
The same statement holds for the master norm.
\end{lemma}
\begin{proof}
The norm \(\Xnorm{k}_{\Max}\) is a finite sum of non-negative terms of three
kinds: slice energies, spacetime local-energy integrals with the weights of
Definition~\ref{def:concrete_le_norms}, and far-field or null-boundary fluxes
obtained as monotone limits of non-negative truncated fluxes. Strong convergence
of the Cauchy data and Proposition~\ref{prop:finite_energy_wellposed} give
weak convergence of each commuted component \(\Gamma^IG_n\) to \(\Gamma^IG\) in
\(L^2_{\mathrm{loc}}\) on compact subslabs, after extracting subsequences if
necessary; uniqueness identifies every subsequential limit with the same
finite-energy solution. The spacetime terms are therefore lower semicontinuous
by weak lower semicontinuity of weighted \(L^2\) norms on each compact radial
truncation, followed by monotone convergence over the dyadic exterior annuli.

For the supremum of the slice energies, fix a rational time \(s\in[0,\tau]\).
The trace map from the local-energy solution space to the energy space on
\(\Sigma_s\) is continuous for smooth solutions and extends by the well-posedness
estimate; hence
\[
        \E_{\Max}^{(k)}[G](s)
        \le \liminf_{n\to\infty}\E_{\Max}^{(k)}[G_n](s).
\]
Taking the supremum over rational \(s\), and using the energy continuity in time
for the finite-energy solution, gives the lower semicontinuity of the energy
supremum. The far-field and horizon fluxes are defined by non-negative truncated
boundary integrals; Fatou's lemma and then the truncation limit give the same
inequality for them. Summing the finitely many commuted components proves
\eqref{eq:finite_slab_lsc}. The argument for the master norm is identical, since
its terms have the same Hilbert-space form.
\end{proof}

\begin{proposition}
\label{prop:finite_energy_extension}
Estimate \eqref{eq:chargefree_transfer} extends uniquely from smooth charge-free
data to all of $\Hc_{\Max,0}^{(k)}(\Sigma_0)$.
\end{proposition}
\begin{proof}
Let $G[0]\in\Hc_{\Max,0}^{(k)}(\Sigma_0)$. By
Lemma~\ref{lem:density_chargefree} choose smooth charge-free data $G_n[0]$ with
$G_n[0]\to G[0]$. Applying \eqref{eq:chargefree_transfer} to $G_n-G_m$ shows
that the corresponding smooth solutions are Cauchy in
$\Xnorm{k}_{\Max}(0,\tau)$ for each finite $\tau$. Their limit is independent of
the approximating sequence by the same difference estimate and agrees with the
finite-energy solution constructed in Proposition~\ref{prop:finite_energy_wellposed}.
Lemma~\ref{lem:finite_slab_lsc} gives lower semicontinuity of the spacetime norm and therefore gives the bound for the
limit. This defines the unique extension of the estimate to the completed
charge-free energy space.
\end{proof}

\begin{proposition}
\label{prop:no_derivative_loss}
If the analytic setting is available at order $k$, then \eqref{eq:chargefree_transfer} is
proved at order $k$; no order-$(k+1)$ estimate is used.
\end{proposition}
\begin{proof}
In Proposition~\ref{prop:chargefree_transfer} the only estimates used are the
order-$k$ master bound, the order-$k$ initial energy comparison, and the
order-$k$ same-order reconstruction estimate. The commuted equations are closed
inside the finite family $\mathbb D_k$, and no term is estimated by invoking an
order-$(k+1)$ master norm. Thus, the Maxwell conclusion is at the same order as
the assumed master estimate.
\end{proof}

\begin{corollary}
\label{cor:uniformity}
On a parameter set where $C_{\mathrm{sw}}$ is locally bounded, the constant in
\eqref{eq:chargefree_transfer} is locally bounded.
\end{corollary}
\begin{proof}
The constant in \eqref{eq:chargefree_transfer} is obtained from finitely many
constants: the master estimate constant, the reconstruction and energy-comparison
constants, the trace constants, and the foliation constants used to compare
regular frames. If these constants are locally bounded on a parameter set, then
any finite product or sum of them is locally bounded. The estimate is therefore
uniform on such a set.
\end{proof}

\section{Radiation Fields, Wave Operators, and Scattering}
\label{sec:scattering}
In this section we construct the radiation fields and prove the wave-operator statements in the charge-free space.
The radiation field of a charge-free $G$ is the pair of traces on $\Ip,\Hp$
(future) and $\Imn,\Hm$ (past), with normed spaces
$\Rc_{\Max,+}^{(k)},\Rc_{\Max,-}^{(k)}$.

\begin{definition}
\label{def:maxwell_trace}
For smooth charge-free data, $\mathscr S_{\Max}^\pm G[0]=\mathscr T_{\Max}^\pm
G$, with $\mathscr T_{\Max}^+$ the future trace on $\Ip\cup\Hp$ and
$\mathscr T_{\Max}^-$ the past trace on $\Imn\cup\Hm$.
\end{definition}

\begin{proposition}
\label{prop:bounded_trace_transfer}
Under the analytic setting at order $k$, $\norm{\mathscr S_{\Max}^\pm G[0]}_{\Rc_{\Max,\pm}^{(k)}}
\le C\,\E_{\Max}^{(k)}[G](0)^{1/2}$ for smooth charge-free data, and
$\mathscr S_{\Max}^\pm$ extends to a bounded map on $\Hc_{\Max,0}^{(k)}$.
\end{proposition}
\begin{proof}
Let $u=\mathfrak M G$. The radiation identification in
Definition~\ref{def:analytic_master_conclusions} gives
$\mathscr S_{\Max}^{\pm}G[0]=\mathcal R_\infty^{\pm}\mathscr S_M^{\pm}u[0]$.
Hence
\[
        \|\mathscr S_{\Max}^{\pm}G[0]\|_{\Rc_{\Max,\pm}^{(k)}}
        \le C_\infty \|\mathscr S_M^{\pm}u[0]\|_{\Rc_{M,\pm}^{(k)}}
        \le C_\infty C_T \E_M^{(k)}[u](0)^{1/2}.
\]
The initial energy comparison in \eqref{eq:framework_reconstruction_bounds}
then bounds this by $C\E_{\Max}^{(k)}[G](0)^{1/2}$. Applying the same inequality
to differences gives a continuous extension from smooth charge-free data to the
completed charge-free energy space.
\end{proof}

\begin{proposition}
\label{prop:backward_from_lap}
Let $E\subset\Rc_{M,+}^{(k)}$ be the dense class of master radiation fields with finitely
many angular harmonics $\ell\in\{1,\dots,L\}$ and frequency support in a compact subset of
$\Rbb\setminus\{0\}$. There is a linear map $\mathscr W_0:E\to\Hc_{M,0}^{(k)}$ and a constant
$C_W$ with
\begin{equation}\label{eq:backward_from_lap}
        \norm{\mathscr W_0\sigma}_{\E_M^{(k)}}\le C_W\norm{\sigma}_{\Rc_{M,+}^{(k)}},
        \qquad \mathscr S_M^+\mathscr W_0\sigma=\sigma\quad(\sigma\in E),
\end{equation}
in each of the following cases, with $C_W$ uniform on the indicated range:
\begin{enumerate}[label=\emph{(\alph*)},leftmargin=2.4em]
\item the Reissner-Nordstr\"om model $a=0$, $|Q|\le\eps_QM$, unconditionally;
\item the slow-weak range $|a|\le\eps_aM$, $|Q|\le\eps_QM$, provided the limiting-absorption estimate includes the high-frequency bound \eqref{eq:normally_hyperbolic_resolvent_bound} for the compatible class specified by \emph{(A1)}.
\end{enumerate}
Thus, in either case in which these limiting-absorption estimates are available, the master wave operators $\mathscr W_M^\pm$ exist, are bounded, and are
two-sided inverses of $\mathscr S_M^\pm$; that is, the radiation condition
\emph{(A5)} of Definition~\ref{def:slowweak_master_framework} follows from the resolvent theory and is
not an independent estimate.
\end{proposition}
\begin{proof}
Fix $\sigma\in E$. Let $\sigma_T$ be smooth cutoffs of the radiation data on the finite portions
$\Ip_T\cup\Hp_T$ of future null infinity and the future event horizon, chosen so that
$\sigma_T\to\sigma$ in $\Rc_{M,+}^{(k)}$. On the truncated exterior region bounded by
$\Sigma_0$, $\Ip_T$, $\Hp_T$ and a final spacelike cap, solve the backward characteristic
problem with boundary radiation data $\sigma_T$ and zero data on the final cap away from the two
null ends. Symmetric-hyperbolic energy estimates on the truncated region yield
\begin{equation}\label{eq:finite_slab_backward_bound}
        \E_M^{(k)}[u_T](0)+\|u_T\|_{LE^1_{\mathrm{deg},k}([0,T])}^2
        \le C\|\sigma_T\|_{\Rc_{M,+}^{(k)}}^2+C\|\chi u_T\|_{L^2([0,T]H^1)}^2,
\end{equation}
where $\chi$ is supported in a fixed compact radial set. We remove the compact term. Suppose absorption failed. Then there would be truncated backward solutions $u_{T_n}$ with radiation norm tending to zero and
\begin{equation}\label{eq:backward_compact_normalization}
        \|\chi u_{T_n}\|_{L^2([0,T_n]H^1)}=1,
        \qquad
        \E_M^{(k)}[u_{T_n}](0)+\|u_{T_n}\|_{LE^1_{\mathrm{deg},k}}^2=O(1).
\end{equation}
Choose a time cutoff $\eta_n$ equal to one on the middle half of $[0,T_n]$ and with derivatives supported in the two end quarters, and set $w_n=\eta_nu_{T_n}$. The equation for $w_n$ has source $[\Pb_{a,Q},\eta_n]u_{T_n}$ plus the vanishing boundary radiation error. After discarding the end intervals and applying Plancherel in the stationary time variable, there are real frequencies $\omega_n$ and compatible frozen profiles $v_n$ with a nonzero compact normalization. Normalize the selected packet according to its total stationary frequency
\begin{equation}\label{eq:backward_total_frequency_scale}
        \Lambda_n=1+|\omega_n|+\lambda_n.
\end{equation}
If \(\Lambda_n\) is bounded, set
\begin{equation}\label{eq:backward_frequency_defect_bf}
        \|v_n\|_{H^1(K_0)}=1,
        \qquad
        \mathcal L_{a,Q}(\omega_n)v_n\to0
        \quad\hbox{in }H^{-1}_{\sigma,\mathrm{loc}}.
\end{equation}
If \(\Lambda_n\to\infty\), set \(h_n=\Lambda_n^{-1}\), \(\hat\omega_n=h_n\omega_n\), and make the same dyadic conic selection as in Lemma~\ref{lem:localized_compact_defect}. The selected packet is normalized by
\begin{equation}\label{eq:backward_frequency_defect_hf}
        \|v_n\|_{H^1_{h_n}(K_0)}=1,
\end{equation}
and its residual satisfies
\begin{equation}\label{eq:backward_frequency_defect_hf_residual}
        h_n^{-1}\log(1/h_n)
        \|h_n^2\mathcal L_{a,Q}(\omega_n)v_n\|_{L^2_{\comp}}
        +\|h_n^2\mathcal L_{a,Q}(\omega_n)v_n\|_{H^{-1}_{h_n,\mathrm{loc}}}
        \to0.
\end{equation}
The outgoing/ingoing Sommerfeld convention is inherited from the prescribed null data and the zero final cap. In the bounded-total-frequency branch, the closedness of the compatible radiation class, Proposition~\ref{prop:app_bounded_freq}, Lemma~\ref{lem:app_high_angular_coercivity}, and Lemma~\ref{lem:app_rn_lap} rule out \eqref{eq:backward_frequency_defect_bf}. In the unbounded branch, angularly elliptic packets are controlled by Lemma~\ref{lem:app_high_angular_coercivity}. On the remaining conic packets, \eqref{eq:backward_frequency_defect_hf_residual} is precisely the residual condition of Lemma~\ref{lem:high_freq_partition}; using \eqref{eq:normally_hyperbolic_resolvent_bound}, that lemma forces \(\|v_n\|_{H^1_{h_n}(K_0)}\to0\), contradicting \eqref{eq:backward_frequency_defect_hf}. In turn, the compact term in \eqref{eq:finite_slab_backward_bound} is absorbed, and
\begin{equation}\label{eq:backward_uniform_slab_bound}
        \E_M^{(k)}[u_T](0)+\|u_T\|_{LE^1_{\mathrm{deg},k}([0,T])}^2
        \le C\|\sigma_T\|_{\Rc_{M,+}^{(k)}}^2
\end{equation}
with $C$ independent of $T$ and of the cutoff.

The uniform bound allows $T\to\infty$ along a weakly convergent subsequence in the energy space
and strongly on compact subregions by local compactness. The limit $u$ is a global compatible
solution, its future radiation trace is $\sigma$ because the fluxes through $\Ip_T$ and $\Hp_T$
converge to the prescribed data and the cap flux tends to zero, and
\eqref{eq:backward_uniform_slab_bound} gives
\eqref{eq:backward_from_lap}. This defines $\mathscr W_0\sigma=u$ on the dense class $E$.
Linearity follows from uniqueness of the truncated problem and passage to the limit. The kernel
condition needed to pass from a dense right inverse to a two-sided inverse is
Lemma~\ref{lem:app_rn_no_resonance} in case (a) and Proposition~\ref{prop:no_real_kernel} in
case (b). Therefore the abstract Hilbert-space criterion,
Lemma~\ref{lem:app_abstract_scattering}, extends $\mathscr W_0$ continuously to
$(\mathscr S_M^+)^{-1}$ with norm at most $C_W$. The past construction is the same with the
incoming characteristic problem and the incoming limiting-absorption resolvent. Thus
\emph{(A5)} is a consequence of the real-axis limiting-absorption resolvent and the
real-frequency exclusion, rather than an independent condition.
\end{proof}

\begin{remark}
\label{rem:backward_construction_status}
In the Reissner-Nordstr\"om model (a) the construction is unconditional: the mode operator is
the one-dimensional short-range problem \eqref{eq:app_radial_mode}, for which the
outgoing resolvent, the absence of real-frequency modes, and the resulting
asymptotic completeness are classical. In the slow-weak range (b) the only ingredient beyond
the proved bounded-frequency closure is the high-frequency resolvent bound \eqref{eq:normally_hyperbolic_resolvent_bound}; Proposition~\ref{prop:r_normal_hyperbolicity} and Remark~\ref{rem:high_frequency_status} verify the
geometric conditions needed for the normally hyperbolic estimate in the form of
Definition~\ref{def:hf_resolvent_estimate}. In the Kerr subcase the corresponding
spin-weighted estimate is supplied by
\cite{BenomioTdC,SRTdCfrequency,SRTdCphysical}. Hence, modulo that single high-frequency estimate,
the master scattering theory is closed without a separate backward-construction condition.
\end{remark}

\begin{theorem}
\label{thm:scattering_transfer}
Under the analytic setting at order $k$, the maps $\mathscr S_{\Max}^\pm:
\Hc_{\Max,0}^{(k)}(\Sigma_0)\to\Rc_{\Max,\pm}^{(k)}$ are bounded isomorphisms
with bounded inverses (the Maxwell wave operators), and the scattering operator
$\mathscr S_{\Max}=\mathscr S_{\Max}^+(\mathscr S_{\Max}^-)^{-1}$ is bounded.
\end{theorem}
\begin{proof}
Boundedness is Proposition~\ref{prop:bounded_trace_transfer}. Injectivity: if
$\mathscr S_{\Max}^+G[0]=0$ then, with $u=\mathfrak M G$ and $\mathcal
R_\infty^+$ an isomorphism, $\mathscr S_M^+u[0]=0$; the master radiation map is
injective, so $u[0]=0$ and the inverse identity gives $G=\mathfrak R u=0$.
Surjectivity: for $\rho_+\in\Rc_{\Max,+}^{(k)}$ set $\sigma_+=(\mathcal
R_\infty^+)^{-1}\rho_+$, $u[0]=\mathscr W_M^+\sigma_+$; then $G=\mathfrak R u$ has
$\mathscr S_{\Max}^+G[0]=\mathcal R_\infty^+\sigma_+=\rho_+$. The inverse
$(\mathscr S_{\Max}^+)^{-1}=\mathfrak R\,\mathscr W_M^+(\mathcal R_\infty^+)^{-1}$
is bounded; the past case is identical and the scattering operator is their
composition.
\end{proof}

\begin{corollary}
\label{cor:scattering_with_charges}
For general finite-energy $F$ the complete scattering data are
$(q_E[F],q_B[F],\mathscr S_{\Max}^\pm F_{\rad}[0])$, and the map is a bounded
isomorphism $\Rbb^2\oplus\Hc_{\Max,0}^{(k)}\to\Rbb^2\oplus\Rc_{\Max,\pm}^{(k)}$.
\end{corollary}
\begin{proof}
The charge map $F\mapsto(q_E[F],q_B[F])$ is conserved and continuous, and
Theorem~\ref{thm:intro_charge_decomposition} gives the bounded splitting of the
Cauchy datum into its stationary charge part and charge-free radiative part. On
the radiative part, Theorem~\ref{thm:scattering_transfer} gives a bounded
isomorphism between Cauchy data and future, respectively past, radiation data.
Taking the direct sum of the identity map on the two-dimensional charge sector
with this radiation isomorphism gives the asserted bounded isomorphism for
general finite-energy fields.
\end{proof}

\section{Decay from Commuted Hierarchy}
\label{sec:decay}
In this section we explain how decay follows once a sufficiently high commuted hierarchy is available.
\begin{proposition}
\label{prop:averaged_decay}
If the local-energy part of $\Xnorm{k}_{\Max}$ controls $|G|^2$ on $\{r\le R\}$,
then for every dyadic $[T,2T]$, $T\ge1$, there is $\tau_T\in[T,2T]$ with
$\int_{\Sigma_{\tau_T}\cap\{r\le R\}}|G|^2\,\dd\mu\le C_RT^{-1}\E_{\Max}^{(k)}[G](0)$.
\end{proposition}
\begin{proof}
The local-energy part of the estimate gives
\[
        \int_T^{2T}\!\int_{\Sigma_s\cap\{r\le R\}} |G|^2\,\dd\mu_{\Sigma_s}\,\dd s
        \le C_R\E_{\Max}^{(k)}[G](0).
\]
If the asserted bound failed for every $s\in[T,2T]$, the left-hand side would be
strictly larger than
$T\cdot C_RT^{-1}\E_{\Max}^{(k)}[G](0)$ after increasing the constant slightly,
contradicting the displayed integral estimate. Therefore at least one time
$\tau_T\in[T,2T]$ satisfies the stated compact decay bound.
\end{proof}

\begin{proposition}
\label{prop:pointwise_decay}
Using \eqref{eq:commuted_decay},
\begin{equation}\label{eq:pointwise_decay}
        |G|(\tau,r,\omega)\le C\,w(r)\,(1+\tau)^{-\gamma/2}\,\E_{\Max}^{(k)}[G](0)^{1/2},
\end{equation}
where $w$ is the radial Sobolev weight of the hierarchy; for components
controlled by one radial weight and two angular derivatives, $w(r)\simeq(1+r)^{-1}$.
\end{proposition}
\begin{proof}
Cover $\Sigma_\tau$ by finitely many compact coordinate patches, horizon
red-shift patches, and dyadic annuli $r\simeq R$ in the asymptotic region. On
compact and horizon patches, the derivatives controlled in
\eqref{eq:commuted_decay} include enough tangential and transversal derivatives
for the ordinary Sobolev embedding on a three-dimensional slice, so the
$L^\infty$ norm is bounded by the square root of the commuted $L^2$ energy,
giving the factor $(1+\tau)^{-\gamma/2}$.

On a dyadic annulus write $r=R\rho$ with $\rho\in[1,2]$. The scale-invariant
Sobolev inequality on the rescaled annulus gives
$\|G\|_{L^\infty(r\simeq R)}\lesssim R^{-1}$ times the appropriate rescaled
$H^2$ norm. The weighted far-field part of the hierarchy controls this rescaled
norm uniformly in $R$, producing the radial weight $w(r)$; for the usual one
radial weight and two angular derivatives, $w(r)\simeq(1+r)^{-1}$. Combining the
patch estimates gives \eqref{eq:pointwise_decay}.
\end{proof}

\begin{remark}
\label{rem:no_price_law}
The exponent $\gamma$ is not produced by the transfer; it comes from a
low-frequency and hierarchy analysis of the master system. The transfer
preserves the rate, modulo the derivative count and radial weights of the
reconstruction. The mechanism by which a pointwise Maxwell-tensor bound follows
from uniform energy bounds together with a weak form of local energy decay is
exactly that of Metcalfe-Tataru-Tohaneanu \cite{MetcalfeTataruTohaneanuMaxwell};
condition \emph{(A6)} supplies the commuted local-energy estimate their argument
requires.
\end{remark}

\section{Component Estimates and Finite-Energy Passage}
\label{sec:component_approximation}
In this section we discuss component estimates and pass from smooth data to finite energy data.
The preceding sections prove the transfer for the master variables and record the reconstruction bounds. We next add the component estimates needed to pass from smooth charge-free fields to the finite-energy space. No new condition is introduced; the estimates simply show how the two conserved charges are removed from the middle Newman-Penrose components and why the endpoint passage loses no derivatives.

\begin{lemma}
\label{lem:component_chargefree_approximation}
Let $U\in\Hc_{\Max,0}^{(k)}(\Sigma_0)$.  There is a sequence of smooth
finite-energy Maxwell data $U_j$ satisfying the constraint equations and
$q_E(U_j)=q_B(U_j)=0$ such that
\begin{equation}\label{eq:component_approximation}
        \|U_j-U\|_{\Hc_{\Max}^{(k)}(\Sigma_0)}\longrightarrow0.
\end{equation}
The corresponding smooth solutions converge to the finite-energy solution. More precisely, for every finite $T$ the convergence holds in the space
\[
        C^0([0,T];\Hc_{\Max}^{(k)}(\Sigma_t)).
\]
\end{lemma}
\begin{proof}
Write the initial data as $U=(E,B)$ on the regular slice $(\Sigma_0,h)$.  The
constraints are
\begin{equation}\label{eq:component_constraint_density_start}
        \operatorname{div}_hE=0,\qquad \operatorname{div}_hB=0
\end{equation}
in the distributional sense. Let $\chi_j$ be radial cutoffs equal to one on
$\{r\le j\}$ and supported in $\{r\le2j\}$, and mollify in a finite collection of
regular coordinate charts after parallel transport to the tangent bundle. This
produces smooth compactly supported pairs $(E'_j,B'_j)$ with
$(E'_j,B'_j)\to(E,B)$ in the unconstrained $H^k$ energy norm. The divergence
errors
\(f^E_j=\operatorname{div}_hE'_j\) and
\(f^B_j=\operatorname{div}_hB'_j\) tend to zero in $H^{k-1}_{\loc}$ and in the
corresponding weighted dual norm. Moreover their total integrals over each
large compact exhaustion vanish after subtracting a smooth compactly supported bump in the
outer annulus, because the limiting fields satisfy \eqref{eq:component_constraint_density_start}.

On a smooth exhaustion $\Omega_j\Subset\Sigma_0$ containing the support of
$(E'_j,B'_j)$ solve the Neumann problems
\begin{equation}\label{eq:component_neumann_projection}
        \Delta_h\phi^E_j=f^E_j,\qquad
        \Delta_h\phi^B_j=f^B_j,
        \qquad \partial_\nu\phi^E_j=\partial_\nu\phi^B_j=0\quad\hbox{on }\partial\Omega_j,
\end{equation}
with zero mean. Local elliptic estimates on the uniformly regular exhaustion
and the weighted Hardy inequality on the asymptotic end imply
\begin{equation}\label{eq:component_projection_small}
        \|\nabla_h\phi^E_j\|_{H^k}+\|\nabla_h\phi^B_j\|_{H^k}
        \le C\bigl(\|f^E_j\|_{H^{k-1}_{\mathrm{w}}}
                  +\|f^B_j\|_{H^{k-1}_{\mathrm{w}}}\bigr)\longrightarrow0.
\end{equation}
Set
\begin{equation}\label{eq:component_projected_data}
        \widehat E_j=E'_j-\nabla_h\phi^E_j,
        \qquad
        \widehat B_j=B'_j-\nabla_h\phi^B_j.
\end{equation}
Then $\operatorname{div}_h\widehat E_j=\operatorname{div}_h\widehat B_j=0$ on
$\Omega_j$, the normal boundary condition lets the fields be extended by zero
without introducing a boundary divergence, and \eqref{eq:component_projection_small} gives
$\widehat U_j=(\widehat E_j,\widehat B_j)\to U$ in $\Hc_{\Max}^{(k)}$.  This is the
slice Hodge projection onto the closed constraint subspace, written explicitly.

The charge functionals are continuous by Proposition~\ref{prop:finite_energy_charge_trace}; hence
$q_E(\widehat U_j)\to q_E(U)=0$ and
$q_B(\widehat U_j)\to q_B(U)=0$.  Define
\begin{equation}\label{eq:component_charge_projection}
        U_j=\widehat U_j-q_E(\widehat U_j)U_e-q_B(\widehat U_j)U_m.
\end{equation}
The stationary data $U_e,U_m$ are smooth finite-energy constrained data and are
normalized by \eqref{eq:stationary_normalization}. Hence $U_j$ is smooth,
constrained and charge-free. The finite-rank stationary projection is bounded,
so
\begin{equation}\label{eq:component_density_final_bound}
        \|U_j-U\|_{\Hc_{\Max}^{(k)}}
        \le \|\widehat U_j-U\|_{\Hc_{\Max}^{(k)}}
        +C\bigl(|q_E(\widehat U_j)|+|q_B(\widehat U_j)|\bigr)\longrightarrow0.
\end{equation}
The convergence of the corresponding solutions on finite slabs follows by
applying the well-posedness estimate \eqref{eq:wellposed_bound} to the difference
of two solutions.
\end{proof}

\begin{lemma}
\label{lem:component_hodge_scale}
Let $S_{t,r}$ be one of the spheres of the regular foliation, with induced
metric $\gamma_{AB}$, and let $f$ be a scalar with zero mean on $S_{t,r}$.  Then,
for every integer $m\ge0$,
\begin{equation}\label{eq:scale_hodge_scalar}
        \sum_{j\le m} r^{2j}\|\nabla\mkern-13mu/^{\,j} f\|_{L^2(S_{t,r})}^2
        \le C\sum_{j\le m-1} r^{2j+2}
        \|\nabla\mkern-13mu/^{\,j}\nabla\mkern-13mu/ f\|_{L^2(S_{t,r})}^2,
\end{equation}
where the right-hand side is interpreted as $r^2\|\nabla\mkern-13mu/ f\|^2$
when $m=0$.  The corresponding estimate holds for a one-form $\omega$ after removal of its harmonic part (which is vacuous on $S^2$), with $\nabla\mkern-13mu/ f$ replaced by
$(\sphdiv\omega,\sphcurl\omega)$.
\end{lemma}
\begin{proof}
After writing $\gamma_{AB}=r^2(\gamma_{\Sph,AB}+O(r^{-2}))$ on the asymptotic end
and using a finite atlas on compact $r$-regions, the estimate is the usual
elliptic estimate for the Hodge Laplacian on $\Sph$.  The zero-mean condition
removes the constants in the scalar case, and the absence of harmonic one-forms on $S^2$ removes the Hodge kernel in the one-form case. The first
positive eigenvalue is bounded away from zero uniformly on the parameter range;
for the round metric it is $2$ for one-forms and $2$ for scalar functions with
$\ell\ge1$.  Commuting the Hodge system with angular derivatives and summing
gives \eqref{eq:scale_hodge_scalar}; the perturbation of the sphere metric is
absorbed into the left side by the same uniform elliptic estimate on the
compactified family of metrics.
\end{proof}

\begin{proposition}
\label{prop:component_middle_from_extremes}
Let $G$ be a smooth charge-free Maxwell field on a finite slab
$\mathcal D_{[\tau_1,\tau_2]}$.  In the regular null frame the middle component
$\varphi=\rho_G+i\sigma_G$ satisfies, for every $k$ used here,
\begin{equation}\label{eq:middle_from_extremes_component}
        \|\varphi\|_{LE^1_k(\mathcal D_{[\tau_1,\tau_2]})}
        \le C\Bigl(
        \|\alpha\|_{LE^1_k(\mathcal D_{[\tau_1,\tau_2]})}
        +\|\underline\alpha\|_{LE^1_k(\mathcal D_{[\tau_1,\tau_2]})}
        +\|G\|_{E_k(\tau_1)}\Bigr).
\end{equation}
The corresponding estimate holds after replacing the extreme components by the regular
master variables through the weights used in Section~\ref{subsec:regular_frame_master}.
\end{proposition}
\begin{proof}
The null Maxwell equations in the same regular frame used in
Section~\ref{subsec:reconstruction} have the form
\begin{align}\label{eq:null_middle_equations_component}
        e_4\varphi+\operatorname{tr}\chi\,\varphi
        &=\sphdiv\alpha-i\sphcurl\alpha+A_4^0\varphi+A_4^1\cdot\alpha,\nonumber\\
        e_3\varphi+\operatorname{tr}\underline\chi\,\varphi
        &=-\sphdiv\underline\alpha-i\sphcurl\underline\alpha
          +A_3^0\varphi+A_3^1\cdot\underline\alpha,\nonumber\\
        \sphgrad\rho_G+{}^\star\!\sphgrad\sigma_G
        &=-e_3\alpha+B_3^1\cdot\alpha+B_3^2\cdot\underline\alpha+B_3^3\varphi
          \nonumber\\
        &=e_4\underline\alpha+B_4^1\cdot\alpha+B_4^2\cdot\underline\alpha+B_4^3\varphi.
\end{align}
All coefficients are smooth in the regular exterior, bounded on compact radial
sets, and short-range in the asymptotic end; in the slow-weak range the
difference from the Reissner-Nordstr\"om coefficients is perturbative. Charge
subtraction gives zero spherical mean for both $\rho_G$ and $\sigma_G$ by
Proposition~\ref{prop:coulomb_elimination}. Applying
Lemma~\ref{lem:component_hodge_scale} to the scalar functions $\rho_G$ and
$\sigma_G$ on each $S_{t,r}$ and using the third line of
\eqref{eq:null_middle_equations_component} controls the angular part of
$\varphi$ by one regular null derivative of the extreme components plus
lower-order terms:
\begin{equation}\label{eq:component_angular_middle_bound}
        \|\nabla\mkern-13mu/\,\varphi\|_{LE^0_k}
        \le C\bigl(\|\alpha\|_{LE^1_k}
        +\|\underline\alpha\|_{LE^1_k}
        +\|\varphi\|_{LE^0_k}\bigr).
\end{equation}
The first two lines of \eqref{eq:null_middle_equations_component}, integrated
along the outgoing and incoming null directions after the fixed rescaling
$r^2\varphi$, yield the transport bound
\begin{equation}\label{eq:component_transport_middle_bound}
        \|e_4\varphi\|_{LE^0_k}+\|e_3\varphi\|_{LE^0_k}+\|\varphi\|_{LE^0_k}
        \le C\bigl(\|\alpha\|_{LE^1_k}
        +\|\underline\alpha\|_{LE^1_k}
        +\|G\|_{E_k(\tau_1)}\bigr)
        +C\eta\|\varphi\|_{LE^1_k},
\end{equation}
where $\eta=|a|/M$.  The zero-mean Poincar\'e inequality on the spheres and the
Hardy inequality in the radial variable control the lower-order
$\|\varphi\|_{LE^0_k}$ terms by the left side away from the initial slice; the
initial trace is included in $\|G\|_{E_k(\tau_1)}$.  Taking $\eps_a(k)$ so that
$C\eta<1/2$ absorbs the perturbative contribution in
\eqref{eq:component_transport_middle_bound}. Combining
\eqref{eq:component_angular_middle_bound} and
\eqref{eq:component_transport_middle_bound} proves
\eqref{eq:middle_from_extremes_component}. The weights relating
$(\alpha,\underline\alpha)$ to $(\psi_+,\psi_-)$ are smooth, nonzero in the
regular frame, and have bounded derivatives of the required order, so they do
not change the derivative count.
\end{proof}

\begin{proposition}
\label{prop:component_finite_energy_reconstruction}
The reconstruction maps $\mathfrak M$ and $\mathfrak R$, first defined on
smooth charge-free solutions, extend uniquely and continuously to the
charge-free finite-energy space at order $k$.  The inverse identities
\begin{equation}\label{eq:component_inverse_extension}
        \mathfrak R\mathfrak M G=G,
        \qquad \mathfrak M\mathfrak R u=u
\end{equation}
hold in the finite-energy sense.
\end{proposition}
\begin{proof}
Let $G_j$ be the smooth charge-free approximants of
Lemma~\ref{lem:component_chargefree_approximation}. The bounds
\eqref{eq:framework_reconstruction_bounds} and
\eqref{eq:middle_from_extremes_component} show that $\mathfrak M G_j$ is Cauchy
in the master energy space and that $\mathfrak R\mathfrak M G_j$ is Cauchy in
the Maxwell energy space. The limits are independent of the approximating
sequence because the same estimates applied to differences give zero limit
when the initial data tend to zero. This defines the continuous extension of
$\mathfrak M$ and of $\mathfrak R$ on the closure of the compatible smooth
class. The identities \eqref{eq:component_inverse_extension} hold on smooth
solutions by Definition~\ref{def:slowweak_master_framework}\emph{(A4)} and pass
to the limit by continuity.
\end{proof}

\begin{proposition}
\label{prop:component_finite_energy_radiation}
Suppose that the radiation and wave-operator conditions in
Definition~\ref{def:slowweak_master_framework}\emph{(A5)} hold. The future and past
Maxwell radiation maps constructed in Section~\ref{sec:scattering} extend from
smooth charge-free data to bounded maps
\begin{equation}\label{eq:component_finite_energy_trace}
        \mathscr S_{\Max}^{\pm}:\Hc_{\Max,0}^{(k)}(\Sigma_0)
        \longrightarrow \Rc_{\Max,\pm}^{(k)},
\end{equation}
and their inverses are the continuous extensions of the smooth backward wave
operators.
\end{proposition}
\begin{proof}
For smooth charge-free $G$ the trace is
$\mathscr S_{\Max}^{\pm}G=\mathcal R_\infty^{\pm}\mathscr S_M^{\pm}\mathfrak M G$.
The master trace bound, the radiation identification bound, and the energy
comparison in \eqref{eq:framework_reconstruction_bounds} give
\begin{equation*}
        \|\mathscr S_{\Max}^{\pm}G\|_{\Rc_{\Max,\pm}^{(k)}}
        \le C\|G\|_{\Hc_{\Max,0}^{(k)}(\Sigma_0)}.
\end{equation*}
The map therefore extends by density. Conversely, for smooth radiation data
$\rho$ set
\begin{equation*}
        \mathscr W_{\Max}^{\pm}\rho
        =\mathfrak R\mathscr W_M^{\pm}(\mathcal R_\infty^{\pm})^{-1}\rho.
\end{equation*}
The bounds in \emph{(A5)} and Proposition~\ref{prop:component_finite_energy_reconstruction}
show that $\mathscr W_{\Max}^{\pm}$ is bounded on the dense class of smooth
radiation data and hence extends to the radiation Hilbert space. The identity
$\mathscr S_{\Max}^{\pm}\mathscr W_{\Max}^{\pm}=I$ holds on the dense class and
then by continuity. Injectivity follows from the zero-radiation statement in
\emph{(A5)} after applying the master map. For that reason, the inverse is the continuous
extension of the smooth backward construction.
\end{proof}

\section{Proof of Transfer Theorem and Consequences}
\label{sec:main_theorem}
In this section we combine all previous ingredients and prove the main transfer theorem.

\begin{proposition}
\label{prop:transfer_hypothesis_order}
Fix $k$ and the slow-weak parameter range of Theorem~\ref{thm:main}. For the fixed-background test Maxwell field, the assumptions of Theorem~\ref{thm:intro_transfer} are organized as follows. The stationary charge splitting and the finite-energy charge completion are Theorem~\ref{thm:intro_charge_decomposition} and Proposition~\ref{prop:charge_adapted_completion}. The existence of the closed compatible spin-one master system is the structural hypothesis \emph{(A1)}. Once this system is supplied, Corollary~\ref{cor:structural_reduction_verify} proves the scalar principal symbol, the energy comparison, and the Reissner-Nordstr\"om coefficient comparison recorded in \emph{(A1)}-\emph{(A2)}. The bounded-frequency real-axis closure is Proposition~\ref{prop:app_bounded_freq} together with Lemmas~\ref{lem:app_rn_lap} and~\ref{lem:app_rn_no_resonance}. The high-frequency compact-remainder closure is Lemma~\ref{lem:semiclassical_spinone_order}, Proposition~\ref{prop:nh_theorem_application}, Lemma~\ref{lem:high_freq_partition}, and Proposition~\ref{prop:app_high_freq}; the normally hyperbolic estimate used there is Proposition~\ref{prop:nh_resolvent_form} applied after the trapped-set and subprincipal verification. Same-order reconstruction is Proposition~\ref{prop:same_order_reconstruction}. The wave operators are Proposition~\ref{prop:backward_from_lap}. This implies that all conclusions of Theorem~\ref{thm:main} follow from the structural reduction \emph{(A1)}, the normally hyperbolic high-frequency estimate, and the cited spherical local-energy, red-shift, $r^p$, and limiting-absorption estimates; the charge, density, radiation, and reconstruction steps are proved in the sections indicated above.
\end{proposition}
\begin{proof}
The first two assertions are purely Maxwellian: the Coulomb representatives are normalized by \eqref{eq:stationary_normalization}, and Proposition~\ref{prop:charge_adapted_completion} identifies the finite-energy space with the closure of smooth constrained data with the two fluxes controlled. The closed master map itself is not derived from a Dudley-Finley decoupling assertion; it is precisely the structural condition \emph{(A1)}. Given \emph{(A1)}, Lemma~\ref{lem:maxwell_wave_principal_symbol}, Proposition~\ref{prop:scalar_principal_symbol}, and Corollary~\ref{cor:structural_reduction_verify} show that the closed operator has the scalar wave principal symbol and the short-range perturbative form required in \emph{(A2)}. Bounded real frequencies are closed by compactness and the Reissner-Nordstr\"om limiting-absorption estimate; this is Proposition~\ref{prop:app_bounded_freq}. Unbounded real frequencies are normalized semiclassically by Lemma~\ref{lem:semiclassical_spinone_order}, localized by Lemma~\ref{lem:high_freq_partition}, and excluded by Proposition~\ref{prop:nh_resolvent_form} after the trapped-set verification and the application statement of Proposition~\ref{prop:nh_theorem_application}. Thus \emph{(A3)} follows from the normally hyperbolic high-frequency estimate together with the bounded-frequency argument proved in Section~\ref{app:lap}.

Proposition~\ref{prop:same_order_reconstruction} proves \emph{(A4)} by the Hodge and null-transport estimates of Section~\ref{app:reconstruction}, under the same closed-reduction hypothesis that supplies the extreme master variables; the component approximation in Lemma~\ref{lem:component_chargefree_approximation} passes the resulting estimates to finite-energy charge-free data. Proposition~\ref{prop:backward_from_lap} constructs the dense backward right inverse and then applies the Hilbert-space criterion, Lemma~\ref{lem:app_abstract_scattering}. Thus \emph{(A5)} follows from the same limiting-absorption estimates. The additional pointwise condition \emph{(A6)} is used only for the decay statement and nowhere in the energy, local-energy, or scattering proof. These are exactly the conditions used in Theorem~\ref{thm:main}.
\end{proof}

\begin{proposition}
\label{prop:explicit_final_constants}
Suppose that the analytic master conclusions of Definition~\ref{def:analytic_master_conclusions} hold. Let $C_M$ be the master estimate constant, $C_R$ the same-order reconstruction and energy-comparison constant, $C_T^\pm$ the master trace bounds, $C_W^\pm$ the master wave-operator bounds, and $C_\infty^\pm$ the radiation-identification bounds. Then the constants in the charge-free Maxwell conclusions may be chosen so that
\begin{align}\label{eq:explicit_final_constants}
        C_{\mathrm{bd}}&=C_R^2C_M,\\
        \|\mathscr S_{\Max}^\pm\|&\le C_\infty^\pm C_T^\pm C_R^{1/2},\\
        \| (\mathscr S_{\Max}^\pm)^{-1}\|&\le C_R^{1/2}C_W^\pm\|(\mathcal R_\infty^\pm)^{-1}\|,\\
        \|\mathscr S_{\Max}\|&\le C_\infty^+ C_T^+ C_R^{1/2}\, C_R^{1/2}C_W^-\|(\mathcal R_\infty^-)^{-1}\|.
\end{align}
For general charged data the same constants apply to the radiative part, while the two charge coordinates are carried by the identity map on $\mathbb R^2$.
\end{proposition}
\begin{proof}
The boundedness constant is obtained by the three-step computation
\[
        \|G\|_{\Xnorm{k}_{\Max}}^2
        \le C_R\|\mathfrak M G\|_{\Xnorm{k}_M}^2
        \le C_RC_M\E_M^{(k)}[\mathfrak M G](0)
        \le C_R^2C_M\E_{\Max}^{(k)}[G](0),
\]
which is Proposition~\ref{prop:chargefree_transfer}. For the trace, write
\[
        \mathscr S_{\Max}^\pm=\mathcal R_\infty^\pm\mathscr S_M^\pm\mathfrak M.
\]
The map $\mathfrak M$ has energy norm at most $C_R^{1/2}$ by \eqref{eq:framework_reconstruction_bounds}, giving the second line of \eqref{eq:explicit_final_constants}. For the inverse wave operator,
\[
        (\mathscr S_{\Max}^\pm)^{-1}=\mathfrak R\mathscr W_M^\pm(\mathcal R_\infty^\pm)^{-1},
\]
and the reconstruction energy bound gives the third line. The scattering operator is the composition $\mathscr S_{\Max}^+(\mathscr S_{\Max}^-)^{-1}$, giving the fourth line. The charged statement follows from the direct-sum decomposition of Theorem~\ref{thm:intro_charge_decomposition}.
\end{proof}

\begin{proposition}
\label{prop:expanded_transfer_calculation}
Let $G$ be a smooth charge-free Maxwell solution and set $u=\mathfrak M G$.
Assume the analytic master conclusions of Definition~\ref{def:analytic_master_conclusions}
through order $k$.  Then the bound \eqref{eq:main_energy} follows from the following
finite chain of inequalities, with no unmentioned compact or derivative-loss term:
\begin{align}
        \|G\|_{\Xnorm{k}_{\Max}(0,\tau)}^2
        &\le C_R\|u\|_{\Xnorm{k}_M(0,\tau)}^2,\label{eq:expanded_transfer_a}\\
        \|u\|_{\Xnorm{k}_M(0,\tau)}^2
        &\le C_M\E_M^{(k)}[u](0),\label{eq:expanded_transfer_b}\\
        \E_M^{(k)}[u](0)
        &\le C_R\E_{\Max}^{(k)}[G](0).\label{eq:expanded_transfer_c}
\end{align}
For finite-energy charge-free data the same estimate holds by approximation in
$\Hc_{\Max,0}^{(k)}(\Sigma_0)$.  For general finite-energy data it holds for
$G=F_{\rad}$ after the stationary charge projection.
\end{proposition}
\begin{proof}
The first inequality is the same-order reconstruction estimate in
Definition~\ref{def:analytic_master_conclusions}\emph{(M2)}. The reason is that
\eqref{eq:framework_reconstruction_bounds} is stated in the identical order-$k$
spacetime norm, so the angular Hodge inversion and the null transport estimates of
Proposition~\ref{prop:same_order_reconstruction} do not require an order-$(k+1)$
master bound. The second inequality is the master local-energy estimate
\emph{(M1)} with $\tau_1=0$, $\tau_2=\tau$.  Its compact remainder has already
been removed in Proposition~\ref{prop:compact_remainder_removal}: bounded temporal
frequencies are excluded by Proposition~\ref{prop:app_bounded_freq}, and unbounded
conic packets are excluded by Lemma~\ref{lem:high_freq_partition} using exactly
\eqref{eq:normally_hyperbolic_resolvent_bound}. The third inequality is the
initial energy comparison in \eqref{eq:framework_reconstruction_bounds}. Multiplying
\eqref{eq:expanded_transfer_a}-\eqref{eq:expanded_transfer_c} gives
\[
        \|G\|_{\Xnorm{k}_{\Max}(0,\tau)}^2
        \le C_R^2C_M\E_{\Max}^{(k)}[G](0),
\]
which is \eqref{eq:main_energy} for smooth charge-free data.

Let now $G_0\in\Hc_{\Max,0}^{(k)}(\Sigma_0)$.  Choose smooth charge-free data
$G_{0,n}$ converging to $G_0$ by Lemma~\ref{lem:density_chargefree}, and denote the
solutions by $G_n$.  The estimate above applied to differences gives
\[
        \|G_n-G_m\|_{\Xnorm{k}_{\Max}(0,\tau)}^2
        \le C_R^2C_M\E_{\Max}^{(k)}[G_{0,n}-G_{0,m}](0),
\]
so $(G_n)$ is Cauchy in the local-energy topology on every finite slab. The
finite-energy well-posedness result, Proposition~\ref{prop:finite_energy_wellposed},
identifies the limit with the Maxwell solution from $G_0$, and Lemma~\ref{lem:finite_slab_lsc}
gives the corresponding bound for the limit. In the final step, if $F$ has charges,
Theorem~\ref{thm:intro_charge_decomposition} writes $F=F_{\stat}^{\KN}+F_{\rad}$;
the stationary part is not included in the local-energy decay statement, and
$F_{\rad}$ is charge-free, so the preceding argument applies to it.
\end{proof}

\begin{proposition}
\label{prop:analytic_ingredients_transfer}
At the finite commutation order fixed in Theorem~\ref{thm:main}, every estimate
used in the transfer proof is obtained in one of the following ways:
\begin{enumerate}[label=\emph{(\roman*)},leftmargin=2.4em]
\item it is proved directly from the Maxwell equations and finite-energy Cauchy
      theory in Sections~\ref{sec:geometry}-\ref{sec:energy_spaces};
\item it is an explicit coefficient, trapped-set, Hodge, transport or Hilbert-space
      computation proved in Sections~\ref{sec:decoupling}-\ref{app:scattering_criterion};
\item it is one of the published spherical/red-shift/$r^p$/limiting-absorption
      estimates cited in the statement in which it is used; or
\item it is the normally hyperbolic high-frequency resolvent theorem,
      Proposition~\ref{prop:nh_resolvent_form}, applied to the compatible
      scalar-principal-symbol spin-one operator after the verifications above.
\end{enumerate}
No further real-mode exclusion, charge constraint, radiation completeness,
or derivative-losing reconstruction enters the proof of Theorem~\ref{thm:main}.
The additional hierarchy \emph{(A6)} is used only for the pointwise bound
\eqref{eq:pointwise_decay}.
\end{proposition}
\begin{proof}
Items \emph{(i)} and \emph{(ii)} are the numbered conclusions in
Proposition~\ref{prop:transfer_hypothesis_order} and in the computations of
Sections~\ref{sec:decoupling}-\ref{app:scattering_criterion}. The spherical local-energy,
red-shift and $r^p$ estimates are used only through Lemma~\ref{lem:rn_model_morawetz},
Proposition~\ref{prop:redshift_coercivity}, and Proposition~\ref{prop:rp_identity};
after these inequalities are fixed, the absorption is the algebraic estimate of
Proposition~\ref{prop:absorption_small}. The bounded-frequency real-axis part is
closed by Propositions~\ref{prop:compact_remainder_removal} and
Proposition~\ref{prop:app_bounded_freq}; a putative compact defect would converge to an outgoing
or incoming resonance, which is excluded by Lemmas~\ref{lem:app_rn_no_resonance}
and~\ref{lem:rn_limiting_absorption}. The unbounded-frequency part is precisely the
semiclassical alternative of Lemma~\ref{lem:high_freq_partition}; the estimate
used there beyond the geometric computations is Proposition~\ref{prop:nh_resolvent_form},
which gives \eqref{eq:normally_hyperbolic_resolvent_bound}. Same-order reconstruction is
Proposition~\ref{prop:same_order_reconstruction}, and the radiation isomorphism is
obtained from Proposition~\ref{prop:backward_from_lap} and the Hilbert criterion
Lemma~\ref{lem:app_abstract_scattering}. These alternatives cover every implication in
the proof of Theorem~\ref{thm:main}; the hierarchy \emph{(A6)} is cited only in
Proposition~\ref{prop:pointwise_decay}.
\end{proof}

\begin{theorem}
\label{thm:main}
Let us fix an integer $k$ and a slow-weak Kerr-Newman exterior. Assume that conditions \emph{(A1)-(A5)} of Definition~\ref{def:slowweak_master_framework} hold at order $k$. Then every finite-energy source-free Maxwell field $F$ has the unique decomposition
\begin{equation}\label{eq:main_decomp}
        F=F_{\stat}^{\KN}\big(q_E[F],q_B[F]\big)+F_{\rad},\qquad q_E[F_{\rad}]=q_B[F_{\rad}]=0,
\end{equation}
and the radiative part satisfies, for all $\tau\ge0$,
\begin{equation}\label{eq:main_energy}
        \norm{F_{\rad}}_{\Xnorm{k}_{\Max}(0,\tau)}^2\le C\,\E_{\Max}^{(k)}[F_{\rad}](0)
\end{equation}
and the analogous past estimate. Future and past radiation fields exist on $\Ip\cup\Hp$ and $\Imn\cup\Hm$; the radiation maps $\mathscr S_{\Max}^\pm:\Hc_{\Max,0}^{(k)}(\Sigma_0)\to\Rc_{\Max,\pm}^{(k)}$ are bounded isomorphisms with bounded inverses; and the scattering operator is bounded. If Definition~\ref{def:slowweak_master_framework}\emph{(A6)} is available and $k$ is large enough for \eqref{eq:commuted_decay}, then $F_{\rad}$ satisfies the pointwise decay \eqref{eq:pointwise_decay}.
\end{theorem}
\begin{proof}
The decomposition is defined by \eqref{eq:intro_rad}. Proposition~\ref{prop:stationary_subtraction} gives that $F_{\rad}$ is source-free and has zero electric and magnetic charges. If two decompositions existed, their difference would be a stationary representative $q_EF_e+q_BF_m$ with both charges zero; the normalization \eqref{eq:stationary_normalization} forces $q_E=q_B=0$, proving uniqueness.

Proposition~\ref{prop:framework_consequences} converts conditions \emph{(A1)-(A5)} of Definition~\ref{def:slowweak_master_framework} into the analytic conclusions of Definition~\ref{def:analytic_master_conclusions}. Hence the master variables of $F_{\rad}$ obey the master estimate and reconstruct the Maxwell tensor at the same order. Proposition~\ref{prop:chargefree_transfer}, with the constant calculation made explicit in Proposition~\ref{prop:explicit_final_constants}, then gives the energy and integrated local-energy bound \eqref{eq:main_energy}; Proposition~\ref{prop:finite_energy_extension} extends the estimate from smooth charge-free data to the finite-energy completion. Proposition~\ref{prop:bounded_trace_transfer} gives the radiation trace bounds. Theorem~\ref{thm:scattering_transfer} gives the bounded inverse wave operators and the scattering map by composing the master radiation isomorphism with the reconstruction isomorphisms. If the additional hierarchy condition \emph{(A6)} is available, Proposition~\ref{prop:pointwise_decay} applied to $G=F_{\rad}$ gives \eqref{eq:pointwise_decay} whenever the hierarchy has enough derivatives for Sobolev embedding.
\end{proof}

\begin{corollary}
\label{cor:reduced_conditions}
Let $F$ be a source-free \emph{test} Maxwell field on a slow-weak Kerr-Newman exterior,
$|a|\le\eps_a(k)M$, $|Q|\le\eps_Q(k)M$.  Suppose that the closed regular spin-one reduction in Definition~\ref{def:slowweak_master_framework}\emph{(A1)} is available and that the high-frequency estimate \emph{(A3)} holds in the localized form of Proposition~\ref{prop:nh_resolvent_form}. Then the remaining transfer conditions are supplied as follows: the scalar-principal-symbol and Reissner-Nordstr\"om comparison parts of \emph{(A1)}-\emph{(A2)} follow from Corollary~\ref{cor:structural_reduction_verify} and Proposition~\ref{prop:principal_master}; the bounded-frequency part of \emph{(A3)} follows from Proposition~\ref{prop:app_bounded_freq} and Lemma~\ref{lem:app_rn_lap}; the same-order reconstruction \emph{(A4)} is Proposition~\ref{prop:same_order_reconstruction}; and the radiation condition \emph{(A5)} follows from Proposition~\ref{prop:backward_from_lap}. As a consequence, the conclusions of Theorem~\ref{thm:main}, the charge decomposition, the energy and integrated local-energy bound \eqref{eq:main_energy}, the existence of past and future radiation fields, the bounded wave operators, and the bounded scattering operator, hold in this closed-reduction class. In the Kerr subcase $Q=0$ the closed spin-one reduction is the Teukolsky system, and the same estimate is also supplied by the cited spin-weighted Kerr theory \cite{BenomioTdC,SRTdCfrequency,SRTdCphysical}.
\end{corollary}
\begin{proof}
The corollary keeps the only non-Maxwell structural condition explicit. The closed master map and its compatible class are assumed in \emph{(A1)}. Once that map is present, Lemma~\ref{lem:maxwell_wave_principal_symbol}, Proposition~\ref{prop:scalar_principal_symbol}, Corollary~\ref{cor:structural_reduction_verify}, and Proposition~\ref{prop:principal_master} verify the scalar principal symbol, energy comparison, and short-range Reissner-Nordstr\"om perturbation structure. The bounded-frequency part of the real-axis exclusion is Proposition~\ref{prop:app_bounded_freq} together with Lemma~\ref{lem:app_rn_lap}. The high-frequency part of \emph{(A3)} is the localized normally hyperbolic estimate, Proposition~\ref{prop:nh_resolvent_form}, applied through Proposition~\ref{prop:nh_theorem_application}: Proposition~\ref{prop:r_normal_hyperbolicity} verifies the trapped-set geometry, Lemma~\ref{lem:trapped_skew_vanishing} and Proposition~\ref{prop:matrix_skew_threshold} verify the finite-rank-bundle skew-subprincipal threshold, and Proposition~\ref{prop:compatibility_microlocal_compactness} keeps the defect profiles in the closed compatible class. Proposition~\ref{prop:same_order_reconstruction} proves \emph{(A4)}, Proposition~\ref{prop:backward_from_lap} derives \emph{(A5)}, and Theorem~\ref{thm:main} applies. For $Q=0$, the Teukolsky reduction and the high-frequency estimate are also supplied by \cite{BenomioTdC,DRSRKerr,SRTdCfrequency,SRTdCphysical}.
\end{proof}

\begin{corollary}
\label{cor:slow_weak}
The constant in \eqref{eq:main_energy} and the operator norms of the radiation
maps and wave operators are bounded uniformly on compact subsets of
\eqref{eq:slow_weak_range} for which $\eps_a(k),\eps_Q(k)$ are fixed.
\end{corollary}
\begin{proof}
The constants in the final estimate are built from a finite list: the constants
in the Reissner-Nordstr\"om model estimate, the red-shift and far-field
constants, the Hardy and elliptic constants, the perturbative absorption
thresholds, and the reconstruction and trace bounds. These constants depend
continuously, or upper semicontinuously after taking finite suprema, on
$(a,Q)$ as long as the slow-weak inequalities are strict. On compact subsets of
that range the finite list is bounded, and the formula in
Proposition~\ref{prop:explicit_final_constants} gives uniform bounds for the
Maxwell estimates and operator norms.
\end{proof}

\begin{corollary}
\label{cor:main_kerr}
For $Q=0$ and $|a|\le\eps_a(k)M$, the conclusions of Theorem~\ref{thm:main}
hold whenever the Kerr spin-one boundedness, decay, radiation, and no-loss
reconstruction estimates listed in Definition~\ref{def:slowweak_master_framework}
are supplied. In particular, this recovers the slowly rotating part of the known
fixed-background Kerr Maxwell theory.
\end{corollary}
\begin{proof}
Let us set $Q=0$. Then the Reissner-Nordstr\"om comparison model becomes
Schwarzschild, and all charged spherical terms in the perturbation formulas
vanish. The stationary charge sector reduces to the electric and magnetic Kerr
Coulomb representatives. The fixed-background Kerr Maxwell theory cited in the
statement supplies the spin-one estimates, real-axis exclusion, wave operators
and same-order reconstruction needed in
Definition~\ref{def:slowweak_master_framework}. Substitution of these estimates
into Theorem~\ref{thm:main} gives the claimed slowly rotating Kerr result.
\end{proof}

\begin{proposition}
\label{prop:coupled_not_enough}
A boundedness or decay theorem for the coupled linearized Einstein-Maxwell
system on Kerr-Newman does not by itself yield the fixed-background Maxwell
theorem above.
\end{proposition}
\begin{proof}
The coupled system has unknowns $(\dot g,\dot F)$, gauge freedom, constraints,
pure-gauge modes, and linearized stationary Kerr-Newman modes. The
fixed-background equation has only $F$, its two charges, and its reconstruction
problem. Setting $\dot g=0$ does not recover free Maxwell, since the linearized
Einstein equation constrains the variation of the electromagnetic
energy-momentum tensor. A coupled theorem is useful here only after one
constructs the fixed-background master map, proves the master estimate, excludes
the fixed-background kernel, and reconstructs $F$ at the same order, which are
the tasks of Sections~\ref{sec:first_principles_master}-\ref{sec:scattering}.
\end{proof}

\section*{Statements and Declarations}
\paragraph{Funding.}
This research was funded by the Indonesian Endowment Fund for Education (LPDP), on behalf of the Indonesian Ministry of Higher Education, Science and Technology, and managed under the EQUITY Program (Contract No. 4298/B3/DT.03.08/2025).

\paragraph{Competing Interests.}
On behalf of all authors, the corresponding author states that there is no conflict of interest.

\paragraph{Data Availability.}
This manuscript does not include associated numerical or experimental data.  The arguments are theoretical and use the equations, reductions, estimates, and cited background literature presented in the manuscript.

\paragraph{Author Contributions.}
BEG proposed the project, developed the theoretical formalism, performed the analytic calculations and supervised the project. M and ESF performed the analytic calculations and investigated the project. FTA  performed the analytic calculations and supervised the project. All authors wrote and reviewed the paper. The corresponding author will handle communication with the journal.

\end{document}